\documentclass[10pt, letterpaper]{amsart}
\usepackage[margin=1in]{geometry}

\usepackage{amssymb}
\usepackage{mathrsfs}
\usepackage{colonequals}

\usepackage{enumitem}

\usepackage{tikz}
\usepackage{tikz-cd}
\usetikzlibrary{decorations.pathreplacing}

\usepackage{listings}
\usepackage{hyperref}
\hypersetup{
pdflang={en-US},
pdftitle={Interpolation for Brill--Noether curves},
pdfauthor={Eric Larson and Isabel Vogt},
pdfsubject={Algebraic Geometry},
pdfkeywords={Interpolation, Brill--Noether Theory}
}

\usepackage{axessibility}

\makeatletter
\newcommand{\oset}[3][0ex]{%
  \mathrel{\mathop{#3}\limits^{
    \vbox to#1{\kern-2\ex@
    \hbox{$\scriptstyle#2$}\vss}}}}
\makeatother

\newcommand{\pp}{\mathbb{P}}
\newcommand{\zz}{\mathbb{Z}}

\renewcommand{\O}{\mathcal{O}}

\newcommand{\sE}{\mathscr{E}}
\newcommand{\sF}{\mathscr{F}}
\newcommand{\sP}{\mathscr{P}}

\newcommand{\Sym}{\operatorname{Sym}}
\newcommand{\Pic}{\operatorname{Pic}}
\newcommand{\res}{\operatorname{res}}
\newcommand{\Bl}{\operatorname{Bl}}
\newcommand{\Spec}{\operatorname{Spec}}
\newcommand{\Aut}{\operatorname{Aut}}
\newcommand{\rk}{\operatorname{rk}}
\newcommand{\ev}{\operatorname{ev}}
\newcommand{\Span}{\operatorname{Span}}

\renewcommand{\Im}{\operatorname{Im}}

\renewcommand{\bar}{\overline}
\renewcommand{\setminus}{\smallsetminus}

\newcommand{\posmod}{\oset{+}{\rightarrow}}
\newcommand{\posmodwithplus}{\oset{+}{\rightarrow}}
\newcommand{\negmod}{\oset{-}{\rightarrow}}
\newcommand{\posmodalong}{\oset{+}{\leadsto}}
\newcommand{\negmodalong}{\oset{-}{\leadsto}}
\newcommand{\biposmod}{\oset{+}{\leftrightarrow}}

\newcommand{\redmod}{\%}

\newcommand{\qand}{\quad \text{and} \quad}
\newcommand{\qor}{\quad \text{or} \quad}

\newtheorem{thm}{Theorem}[section]
\newtheorem{lem}[thm]{Lemma}
\newtheorem{conj}[thm]{Conjecture}
\newtheorem{prop}[thm]{Proposition}
\newtheorem{cor}[thm]{Corollary}

\newcommand{\defi}[1]{\textsf{#1}} 

\theoremstyle{definition}
\newtheorem{defin}[thm]{Definition}
\newtheorem{example}[thm]{Example}
\newtheorem*{question}{Question}
\newtheorem{hypo}[thm]{Hypothesis}

\theoremstyle{remark}
\newtheorem{rem}[thm]{Remark}

\title[Interpolation for Brill--Noether curves]{Interpolation for Brill--Noether curves}

\author{Eric Larson}
\address{Department of Mathematics, Brown University}
\email{elarson3@gmail.com}
\author{Isabel Vogt}
\address{Department of Mathematics, Brown University}
\email{ivogt.math@gmail.com}
\thanks{During the preparation of this article,
E.L.\ and I.V.\ were supported by
NSF MSPRF grants DMS-1802908 and DMS-1902743.}

\begin{document}

\begin{abstract}
In this paper we determine the number of general points through which a Brill--Noether curve of fixed degree and genus in any projective space can be passed. 
\end{abstract}

\maketitle

\section{Introduction}

The \defi{interpolation} problem has occupied a central position in mathematics
for several millennia. Roughly speaking, it asks:

\begin{question}
When can a curve of a given type be drawn through a given collection of points?
\end{question}

The first results on the interpolation problem date to classical antiquity.
Two such results appear in
Euclid's \emph{Elements} (circa 300 B.C.):
A line can be drawn
through any two distinct
points in the plane (the First Postulate), and a circle can be drawn through any three
non-collinear points in the plane (Proposition~5 of Book~IV).
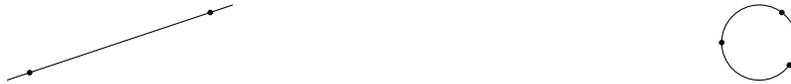
\begin{center}
\begin{tikzpicture}
\draw (0, 0) -- (3, 1);
\filldraw (0.3, 0.1) circle[radius=0.03];
\filldraw (2.7, 0.9) circle[radius=0.03];
\draw (10, 0.5) circle[radius=0.5];
\filldraw (10.3, 0.9) circle[radius=0.03];
\filldraw (10.4, 0.2) circle[radius=0.03];
\filldraw (9.5, 0.5) circle[radius=0.03];
\end{tikzpicture}
\end{center}
The study of the interpolation problem in antiquity
culminated in the work of Pappus (circa 340 A.D.),
who showed in his \emph{Collection} 
that a conic section can be drawn through any five points in the plane
(Proposition 14 of Book VIII).

The introduction of algebraic techniques
to geometry enabled a second wave of results
in the 18th century,
and cast the interpolation problem firmly
in the then-emerging field of algebraic geometry. 
For example, Cramer generalized Pappus's result
to plane curves of arbitrary degree \(n\), which he showed can pass through \(n(n + 3)/2\) general points in 1750 \cite{cramer}.
Then Waring solved the interpolation problem
for graphs of polynomial functions in 1779 \cite{waring}.
(Lagrange independently 
rediscovered this result in 1795 \cite{lagrange}
and thus it is often known as the ``Lagrange interpolation formula.'')
Cauchy \cite{cauchy}, Hermite and Borchardt \cite{hermite}, and Birkhoff \cite{birkhoff},
all subsequently generalized Waring's result in several different directions.
These results are of interest far outside algebraic geometry, and even outside of mathematics.
For example, they play essential roles in the Newton--Cotes method for numerical integration,
in Shamir's cryptographic secret sharing protocol \cite{shamir},
and in Reed--Solomon error-correcting codes \cite{rs}
(which currently power most digital storage media).

The key prerequisite to the modern study of the interpolation problem
was the development of 
\defi{Brill--Noether theory} in the 20th century,
which studies maps from general curves to projective space,
and thus identifies the most natural class of curves for which to study the interpolation problem.
Namely, let \(C\) be a general curve of genus \(g\).
From our perspective here, the two key facts are:
\begin{enumerate}
\item There exists a nondegenerate map \(C \to \pp^r\) of degree \(d\) if and only if
the quantity
\[\rho(d, g, r) \colonequals (r + 1)d - rg - r(r + 1)\]
satisfies \(\rho \geq 0\).
[Proven in 1980 by Griffiths and Harris \cite{bn}.]

\item In this case, the
\emph{universal} space of such maps has a unique component dominating \(\bar{M}_g\).
[Proven in the 1980s by Fulton and Lazarsfeld \cite{fl}, Gieseker \cite{gp}, and Eisenbud and Harris \cite{im}.]
\end{enumerate}
We call stable maps \(f \colon C \to \pp^r\)
corresponding to points in this unique component \defi{Brill--Noether curves (BN-curves)}.
(The general such curve is an embedding of a smooth curve for \(r \geq 3\).)
This language then gives us a \emph{precise} and \emph{natural}
formulation of the interpolation problem:

\begin{question}
Let \(d, g, r, n\) be nonnegative integers with \(\rho(d, g, r) \geq 0\).
When can we pass a BN-curve of degree \(d\) and genus \(g\) through \(n\) general points in \(\pp^r\)?
\end{question}


Equivalently, writing \(M^\circ_{g, n}(\pp^r, d)\) for the component
corresponding to BN-curves, this question is asking when the evaluation map \(M^\circ_{g, n}(\pp^r, d) \to (\pp^r)^n\) is dominant.
It is evidently necessary for:
\[rn = \dim (\pp^r)^n \leq \dim M^\circ_{g, n}(\pp^r, d) = (r + 1)d - (r - 3)(g - 1) + n,\]
or upon rearrangement,
\[(r - 1)n \leq (r + 1)d - (r - 3)(g - 1).\]
Despite cases where this is not sufficient,
it is a folklore conjecture that it \emph{usually} suffices:

\begin{conj} \label{conj:main}
Let \(d, g, r, n\) be nonnegative integers with \(\rho(d, g, r) \geq 0\).
Then there is a BN-curve of degree \(d\) and genus \(g\) through \(n\) general points in \(\pp^r\)
if and only if
\[(r - 1)n \leq (r + 1)d - (r - 3)(g - 1),\]
apart from finitely many exceptions.
\end{conj}

This conjecture has been studied intensely in recent years.
As in previous eras, this attention has been motivated by both intrinsic interest, as well as by striking
applications
to a wide range of other interesting geometric problems. Recent examples
of such problems include smoothing curve singularities \cite{stevens1}, 
constructing moving
curves in \(\bar{M}_g\) \cite{atan, CFM}, the first author's resolution of Severi's 1915
Maximal Rank Conjecture \cite{mrc}, as well as various generalizations thereof \cite{ballico-union}.

The easiest cases of this conjecture are when
\(d\) is large relative to \(g\) and \(r\), and such cases
have therefore been the focus of significant work.
For example, Sacchiero proved Conjecture~\ref{conj:main} for rational
curves in 1980 \cite{sacchiero};
Ran later gave an independent proof in this case in 2007 \cite{ran}.
Subsequently, the first author, in joint work with Atanasov and Yang, proved Conjecture~\ref{conj:main}
when \(d \geq g + r\) in characteristic zero \cite{aly}.
Another case of interest is the proof of Conjecture~\ref{conj:main} for canonical curves in 
characteristic zero in a pair
of papers by Stevens from 1989 and 1996 \cite{stevens, stevens1}.
Many authors have also considered this conjecture in low dimensions.
For example,
Ellingsrud and Hirschowitz in 1984 \cite{eh}, Perrin in 1987 \cite{perrin}, and later Atanasov in 2014 \cite{atan}, all made
significant progress on Conjecture~\ref{conj:main} for space curves,
but their analysis left infinitely many cases unsolved.
This effort culminated in the proof of Conjecture~\ref{conj:main} for space
curves in characteristic zero by the second
author in 2018 \cite{p3}, and for curves in \(\pp^4\) in characteristic zero by both authors in 2021 \cite{p4}.

Nevertheless, despite this significant interest, fundamental limitations of previous techniques
have prevented the resolution of Conjecture~\ref{conj:main}
in general, and limited even partial results largely to characteristic zero.
Our main result gives the first comprehensive answer to the interpolation problem.

\begin{thm} \label{cor:main}
Conjecture~\ref{conj:main} holds in full generality and in any characteristic.
More precisely: Let \(d, g, r, n\) be nonnegative integers with \(\rho(d, g, r) \geq 0\).
There is a BN-curve of degree \(d\) and genus \(g\) through \(n\) general points in \(\pp^r\)
if and only if
\begin{equation} \label{npoints}
(r - 1)n \leq (r + 1)d - (r - 3)(g - 1),
\end{equation}
except in the following four exceptional cases:
\[(d, g, r) \in \{(5, 2, 3), (6, 4, 3), (7, 2, 5), (10, 6, 5)\}.\]
\end{thm}

Since the normal bundle \(N_C\) controls the deformation theory of \(C\),
Conjecture~\ref{conj:main} is closely related to a certain
property, also known as \defi{interpolation},
for \(N_C\).

\begin{defin}
A vector bundle \(E\) on a curve \(C\) \defi{satisfies interpolation} if \(H^1(E) = 0\),
and for every \(n > 0\), there exists an effective divisor \(D\) of degree \(n\) such that
\begin{equation} \label{h0h1}
H^0(E(-D)) = 0 \qor H^1(E(-D)) = 0.
\end{equation}
For \(C\) irreducible, \(\Sym^n C\) is also irreducible, so interpolation is equivalent
to \eqref{h0h1} for \(D\) general.
\end{defin}

If \(p_1, \dots, p_n \in C\subset \pp^r\), then obstructions to lifting deformations of \(p_1, \dots, p_n\) to deformations of \(C\) lie in \(H^1(N_C(-p_1 - \cdots -p_n))\).  Therefore, interpolation for 
\(N_C\) implies Conjecture~\ref{conj:main}.
In fact, interpolation for \(N_C\) is a slightly stronger condition, the essential differences being:
\begin{enumerate}
\item It implies an analog of Conjecture~\ref{conj:main} where the general points are replaced
by general linear spaces: There is a BN-curve of degree \(d\) and genus \(g\) incident to general
linear spaces \(\Lambda_i\) of dimension \(\lambda_i\) if and only if
\begin{equation} \label{inter-linsp}
\sum (r - 1 - \lambda_i) \leq (r + 1)d - (r - 3)(g - 1).
\end{equation}
(This implication can be deduced from \cite[Theorem 8.1]{atan}, cf.\ the introduction to loc.\ cit.)
\item It implies that \(M^\circ_{g, n}(\pp^r, d) \to (\pp^r)^n\) is generically smooth,
rather than merely dominant. (This is of course equivalent in characteristic zero but is a
stronger statement in positive characteristic.)
\end{enumerate}
Theorem~\ref{cor:main} is a consequence of our main theorem, which asserts:

\begin{thm}\label{thm:main}
Let \(d, g, r\) be nonnegative integers with \(\rho(d, g, r) \geq 0\),
and \(C \subset \pp^r\) be a general BN-curve of degree \(d\) and genus \(g\).
Then \(N_C\) satisfies interpolation
if and only if neither of the following hold:
\begin{enumerate}
\item The tuple \((d, g, r)\) is one of the following five exceptions:
\begin{equation}\label{eq:counter_examples}
(d, g, r) \in \{(5,2,3), (6,4,3), (6,2,4), (7,2,5),(10,6,5)\}.
\end{equation}
\item The characteristic is \(2\), and \(g=0\), and \(d \not\equiv 1\) mod \(r - 1\).
\end{enumerate}
\end{thm}

\noindent There are several exceptions in Theorem~\ref{thm:main}
that are not exceptions for Theorem~\ref{cor:main}:
\begin{enumerate}
\item The case \((d, g, r) = (6, 2, 4)\): Such curves have the expected behavior for passing through points,
but \emph{not} for incidence to linear spaces of arbitrary dimension. More precisely, a naive dimension count
suggests that they can
pass through \(9\) general points and meet a general line, but this is not true.
\item The cases in characteristic \(2\): In these cases, the evaluation map
\(M^\circ_{g, n}(\pp^r, d) \to (\pp^r)^n\) is dominant but \emph{not} generically smooth.
\end{enumerate}
We discuss these two cases in more depth in Sections \ref{sec:XX} and \ref{sec:rational}.

\medskip

%
%

Our approach to Theorem~\ref{thm:main} will be via degeneration to reducible curves \(X \cup Y\).
In general, although the restrictions \(N_{X \cup Y}|_X\) and
\(N_{X \cup Y}|_Y\) admit nice descriptions, fitting these together
to describe \(N_{X \cup Y}\) is extremely challenging
outside a handful of special cases.
This fundamental obstacle has limited previous attempts to
study the interpolation problem. For example, the key innovation of \cite{aly}
was an essentially complete description of \(N_{X \cup Y}\)
in the special case that \(Y\) was a \(1\)- or \(2\)- secant line.
Considering only such degenerations leads to two severe limitations:
\begin{enumerate}
\item Only nonspecial curves can be obtained by successively adding
\(1\)- and \(2\)- secant lines.
\item Since the set of degenerations used is so limited,
only a few types of elementary modifications to the normal bundle appear.
Because there are only a few types of modifications, it is difficult
to produce desired modifications by combining them,
in a way reminiscent of the Frobenius coin problem.
Circumventing this difficulty requires additional
tools that work
only in characteristic zero.
\end{enumerate}
Previous attempts to overcome these difficulties were limited to ad-hoc constructions
in low-dimensional projective spaces. The present paper introduces two key innovations:
\begin{enumerate}
\item We consider a third degeneration, where \(Y\) is an \((r + 1)\)-secant
rational curve of degree \(r - 1\) contained in a hyperplane \(H\),
which allows us to obtain any BN-curve.
Describing how \(N_{X \cup Y}|_X\) and
\(N_{X \cup Y}|_Y\) fit together to give \(N_{X \cup Y}\) is intractable
even in a degeneration of this complexity.  Nevertheless, thanks to our detailed study of this setup in Sections~\ref{subsec:onion} and~\ref{sec:red-E},
we are able to reduce interpolation for \(N_{X \cup Y}\) to
interpolation for a certain modification of \(N_X\).

\item Although \(Y\) does not have many interesting degenerations in \(H\),
we show in Section~\ref{sec:onion_specialization} that as \(H\) becomes tangent to \(C\), a plethora of such degenerations appear.
As in the Frobenius coin problem, this plethora of additional degenerations
makes it possible to produce the desired modifications by combining them.
This eliminates the restriction to characteristic zero that plagued previous work.
\end{enumerate}

The techniques we develop here hold promise of application to other
problems about the geometry of normal bundles. Indeed, in \cite{canstab},
they have already been applied to prove semistability of normal bundles of canonical curves,
thereby establishing a (slightly weaker version of a) conjecture of Aprodu, Farkas, and Orgeta \cite{afo}.
(A vector bundle of integral slope which satisfies interpolation is necessarily semistable, c.f.\ \cite[Remark~1.6]{p3},
but in these cases the slope is not integral.  Nevertheless, our techniques may be applied.)


\subsection*{Structure of the paper}
We begin, in Section~\ref{sec:XX}, by discussing the various exceptional cases
appearing in Theorem~\ref{thm:main}.
Then in Section~\ref{sec:prelim}, we introduce the notation
we shall use for the remainder of the paper and discuss a few other preliminary points.

In Section~\ref{sec:overview}, we explain the basic strategy of proof for the
hard direction of Theorem~\ref{thm:main}, i.e., that there are no other exceptional
cases besides those mentioned in the statement of Theorem~\ref{thm:main} above.
After explaining the basic strategy, we give a roadmap to the proof, which occupies
the remainder of the paper.

\subsection*{Acknowledgements} We thank Dan Abramovich, Atanas Atanasov, Izzet Coskun, Joe Harris, Gavril Farkas, Dave Jensen, Ravi Vakil, and David Yang for many stimulating discussions about normal bundles of curves and comments on an earlier version of this manuscript. We particularly thank Izzet Coskun for many helpful discussions about the normal bundle of a canonical curve, and in particular, the Harder--Narasimhan filtration in the degeneration used in Section \ref{sec:canonical}.

\section{Counterexamples}\label{sec:XX}

\subsection{Counterexamples in all characteristics}
We start with the five counterexamples to Theorem~\ref{thm:main} that occur in all characteristics:
\[(d, g, r) \in \{(5,2,3), (6,4,3), (6,2,4), (7,2,5),(10,6,5)\}.\]
In each of these cases, we will construct a certain surface \(S\) containing \(C\),
and see that \(S\) prevents 
Theorem~\ref{cor:main} (or the generalization \eqref{inter-linsp} thereof) from holding.
Indeed, if \(S\) cannot be made to pass through the requisite number of points
(or be made incident to the requisite linear spaces), then \(C\) cannot either.
Since Theorem~\ref{thm:main} implies Theorem~\ref{cor:main} (and the generalization \eqref{inter-linsp}),
this implies that these five cases must also be counterexamples to Theorem~\ref{thm:main}.

\begin{rem}
An alternative approach, the details of which we leave to the interested reader,
would be to see directly that the geometry of \(S\) obstructs Theorem~\ref{thm:main}.
The basic idea is that,
for any effective divisor \(D\) on \(C\), we have
\(h^0(N_C(-D)) \geq h^0(N_{C/S}(-D))\);
in the five exceptional cases,
this inequality will prevent \(N_C\) from satisfying interpolation.
\end{rem}

\subsubsection{The family \((d,g,r)=(r+2, 2, r)\)}
Let \(C\) be a curve of genus \(2\) and \(L\) be a line bundle of degree \(r+2\) on \(C\).  Write \(f \colon C \to \pp^1\) for the hyperelliptic map.  Then \(E \colonequals f_*L\) is a vector bundle of rank \(2\) on \(\pp^1\) with
\[\chi(E) = \chi(L) = r+1.\]
By Riemann--Roch, \(c_1(E) = r-1\).  The inclusion \(L \to f^*f_* L\) embeds \(C\) in the projective bundle \(\pp E^\vee\) so that \(O_{\pp E^\vee}(1)|_C \simeq L\).
Therefore, the image of \(C\) in \(\pp^r\) under the complete linear series for \(L\) lies on the image \(S\) of \(\pp E^\vee\) under the complete linear series for \(\O_{\pp E^\vee}(1)\).  The surface \(S\) is a scroll of degree equal to 
\[[\O_{\pp E^\vee}(1)]^2 = - c_1(E^\vee) \cdot \O_{\pp E^\vee}(1) = r-1.\]
By \cite[Lemma 2.6]{izzet-scrolls}, the dimension of the space of such scrolls is \(r^2 + 2r - 6\).

\begin{description}
\item[\boldmath If \((d, g, r) = (5, 2, 3)\)]
Then \(r^2 + 2r - 6 = 9\). Since it is \(1\) condition for a surface in \(\pp^3\) to pass through a point,
\(S\) cannot pass through more than \(9\) general points.
This contradicts \eqref{npoints}, which
predicts that \(C\) should be able to pass through \(10\) general points.

\item[\boldmath If \((d, g, r) = (6, 2, 4)\)]
Then \(r^2 + 2r - 6 = 18\). Since it is \(2\) conditions for a surface in \(\pp^4\) to pass through a point
and \(1\) condition to meet a line,
\(S\) cannot pass through \(9\) general points points while meeting a general line.
This contradicts \eqref{inter-linsp}, which
predicts that \(C\) should be able to pass through \(9\) general points points while meeting a general line.

\item[\boldmath If \((d, g, r) = (7, 2, 5)\)]
Then \(r^2 + 2r - 6 = 29\). Since it is \(3\) conditions for a surface in \(\pp^5\) to pass through a point,
\(S\) cannot pass through more than \(9\) general points.
This contradicts \eqref{npoints}, which
predicts that \(C\) should be able to pass through \(10\) general points.
\end{description}

\subsubsection{The case \((d,g,r)=(6, 4, 3)\)}
A general canonical curve of genus \(4\) is a cubic section of a quadric surface \(S\).
There is a \(9\)-dimensional family of quadric surfaces, and it is \(1\) condition for a surface in \(\pp^3\)
to pass through a point, so \(S\) cannot pass through more than \(9\) points.
This contradicts \eqref{npoints}, which
predicts that \(C\) should be able to pass through \(12\) general points.

\subsubsection{The case \((d,g,r)=(10, 6, 5)\)}
A general canonical curve of genus \(6\) is a quadric section of a quintic del Pezzo surface \(S\).
There is a \(35\)-dimensional family of quintic del Pezzo surfaces, and it is \(3\) conditions for a surface in \(\pp^5\)
to pass through a point, so \(S\) cannot pass through more than \(11\) points.
This contradicts \eqref{npoints}, which
predicts that \(C\) should be able to pass through \(12\) general points.

\subsection{\boldmath Rational curves in characteristic \(2\)}
In this section we explain the final infinite family of counterexamples to Theorem~\ref{thm:main}
that occurs only in characteristic \(2\).  This phenomenon was already observed for space curves in \cite{stab_p3} in relation to semistability.  
We begin by describing which vector bundles on a rational curve satisfy interpolation in terms of the Birkhoff--Grothendieck classification.

\begin{lem}\label{lem:BG}
The bundle \(E = \bigoplus_i \O_{\pp^1}(e_i)\) satisfies interpolation if and only 
if for all \(i, j\), 
\[|e_i - e_j| \leq 1\qand e_i \geq -1.\]
\end{lem}
\begin{proof}
For any \(n \geq 0\), 
\begin{align*}
h^0(E(-n)) = 0 \quad &\Leftrightarrow\quad n \geq 1 + \max(e_i) \\ 
h^1(E(-n)) = 0 \quad &\Leftrightarrow \quad n \leq 1 + \min(e_i).
\end{align*}
One of these two conditions holds for all \(n \geq 0\) if and only if \(|e_i - e_j| \leq 1\) for all \(i\) and \(j\), and the first of these holds for \(n=0\) if and only if \(e_i \geq -1\) for all \(i\).
\end{proof}

\noindent
As a consequence of the Euler sequence, the conormal bundle of \(C\) sits in the exact sequence
\[0 \to N_C^\vee(1) \to \O_C^{\oplus r+1} \to \sP^1(\O_C(1)) \to 0, \]
where \(\sP^1(\O_C(1))\) is the bundle of first principle parts of \(\O_C(1)\).  If the characteristic is \(2\), and \(\pi \colon C \to C^{(2)}\) is the (relative) Frobenius morphism, then we have
\[\sP^1(\O_C(1)) \simeq \pi^*\pi_*\O_C(1).\]
Therefore \(N_C^\vee(1)\) is isomorphic to the pullback of a bundle under the Frobenius morphism,
so every entry of its splitting type is even.

\begin{lem}\label{lem:char2_xex}
Assume that the characteristic of the ground field is \(2\).  Let \(C \subset \pp^r\) be a rational curve of degree \(d\).  Then \(N_C\) satisfies interpolation only if 
\[d \equiv 1 \pmod{r-1}.\]
\end{lem}
\begin{proof}
Suppose that \(N_C \simeq \bigoplus_{i = 1}^{r - 1} \O_{\pp^1}(e_i)\).
Since \(N_C^\vee(1) \simeq \bigoplus_{i = 1}^{r - 1} \O_{\pp^1}(d - e_i)\),
and every entry of its splitting type is even,
every \(e_i\) satisfies \(e_i \equiv d\) mod \(2\).
Applying Lemma \ref{lem:BG}, we conclude that \(N_C\) can only satisfy interpolation if all \(e_i\) are equal.
This implies
\[(r + 1)d - 2 = c_1(N_C) \equiv (r - 1)d \pmod{2(r - 1)},\]
and therefore \(d \equiv 1\) mod \(r - 1\)
as desired.
\end{proof}

\section{Preliminaries} \label{sec:prelim}

\subsection{Elementary modifications of vector bundles}
In this section, we give a brief overview of the key properties of elementary modifications
of vector bundles. Our presentation will roughly follow the more detailed exposition given in Sections 2--4 of~\cite{aly}.

\begin{defin}\label{def:modification}
Let \(E\) be a vector bundle on a scheme \(X\), and \(D \subset X\) be a Cartier divisor, and \(F \subset E|_D\) be a subbundle
of the restriction of \(E\) to \(D\).
We define the
\defi{negative elementary modification of \(E\) along \(D\) towards \(F\)} by
\[E[D \negmod F] \colonequals \ker\left(E \to E|_D/F\right).\]
We then define the \defi{(positive) elementary modification of \(E\) along \(D\) towards \(F\)} as
\[E[D \posmod F] \colonequals E[D \negmod F](D).\]
\end{defin}

\begin{rem}
This notation differs slightly from \cite{aly}, in which negative modifications were denoted by \(E[D \to F]\)
(and no separate notation was given for positive modifications).
\end{rem}

By construction, a modification of \(E\) along \(D\) is naturally isomorphic to \(E\)
when restricted to the complement of \(D\).
If \(D_1\) and \(D_2\) are disjoint, then
we may easily make sense of multiple modifications
such as \(E[D_1 \posmod F_1][D_2 \posmod F_2]\) by working locally.
However, if \(D_1\) and \(D_2\) meet, then we do not have enough data to even define
multiple modifications: For example, if \(D_1 = D_2 = D\) and \(F_1 = F_2 = F\),
then we should have \(E[D_1 \posmod F_1][D_2 \posmod F_2] \simeq E[2D \posmod F]\),
so we must know how \(F\) extends over \(2D\).
To sidestep these issues, we suppose when defining multiple modifications --- at least
along divisors that meet ---
that we are given not just a subbundle \(F_i\) of \(E|_{D_i}\),
but a subbundle \(F_i\) of \(E|_{U_i}\) where \(U_i \subset X\) is an open neighborhood of \(D_i\).

We first construct the modification \(E[D_1 \posmod F_1]\),
which is naturally isomorphic to \(E\) on \(X \setminus D_1\),
and so in particular on \(U_2 \setminus D_1\).
If the data of a subbundle \(F_2 \subset E|_{U_2 \setminus D_1}\)
extends to \(U_2\), then it does so uniquely,
and we may modify along \(D_2\) towards this extension.
However, this subbundle may not extend to \(U_2\).
The following situation where it does will include
all situations we shall need in this paper:

\begin{defin}\label{def:tree_like}
Let \(M = \{(D_i, U_i, F_i)\}_{i \in I}\) be a collection of modification data.
For each point \(x \in X\), define \(I_x \subseteq I\) to be the set of indices for which \(x \in D_i\).
We say that \(M\) is \defi{tree-like} if for all \(x \in X\), and all subsets \(I' \subset I_x\), the following condition holds:
Whenever the fibers \(\{F_i|_x\}_{i \in I'}\) are dependent, there exist indices \(i,j \in I'\)
and an open \(U \subseteq U_i \cap U_j\) containing \(x\) such that \(F_i|_U \subseteq F_j|_U\).
\end{defin}

By \cite[Proposition 2.17]{aly}, we can transfer modification data as above when it is tree-like.
That is, given modification data \(M\) for \(E\), such that \(\{(D, U, F)\} \cup M\) is tree-like,
we obtain modification data \(M'\) for \(E[D \posmod F]\).
In this way we inductively define the multiple modification \(E[M]\) for tree-like modification data \(M\).
This is independent of the order in which the modifications from \(M\) are applied \cite[Proposition 2.20]{aly}.

\begin{example}
A simplifying special case is when \(F\) is a direct summand of \(E\).  Writing \(E \simeq F \oplus E'\), 
\[E[D \posmod F]  = \ker(F \oplus E' \to E'|_D)(D) \simeq (F \oplus E'(-D))(D) \simeq F(D) \oplus E'.\]
\end{example}

In nice cases, a short exact sequence of vector bundles
induces a short exact sequence of modifications. To make this more precise,
consider a short exact sequence
  \begin{equation}\label{eq:original_ses} 
 0 \to S \to E \to Q  \to 0.
  \end{equation}
For example, first suppose that \(F \cap S\) is flat over the base \(X\).
Then \eqref{eq:original_ses} induces the short exact sequence
\begin{equation}\label{eq:easy_ses}
0 \to S[D \posmod F \cap S] \to E[D \posmod F] \to Q[D \posmod F / (F \cap S)] \to 0.
\end{equation}

A more interesting example is when the base is a curve \(X = C\),
and \(F \subset E\) is a line subbundle, and \(D = np\) where \(p \in C\) is a smooth point.
Then we obtain an induced sequence for the modification \(E[np \posmod F]\), as follows.
Define \(k'\) to be the order to which \(F\) is contained in \(S\) in a neighborhood of \(p\). In other words,
if \(F\) is not contained in \(S\), this is the length of the subscheme \(\pp S \cap \pp F\) in \(\pp E\); if \(F\) is contained in \(S\), this is \(\infty\).
Let \(k = \max(k', n)\).  Then \eqref{eq:original_ses} induces the short exact sequence
\begin{equation}\label{eq:mods_ses}
0 \to S[kp \posmod F|_{kp}] \to E[np \posmod F] \to Q[(n-k)p \posmod \bar{F}] \to 0,
\end{equation}
where \(\bar{F}\) is the saturation of the image of \(F\) in \(Q\).
When \(k' = \infty\) or \(k' = 0\) this agrees with \eqref{eq:easy_ses}.

\subsection{Pointing bundles}

Given an unramified map \(f \colon C \to \pp^r\), the sheaf \(N_f = \ker(f^*\Omega_{\pp^r} \to \Omega_X)^\vee\) is a vector bundle, which we refer to as the \defi{normal bundle of the map \(f\)}.
In almost all cases that we shall consider, \(f\) will be an embedding,
in which case \(N_f = N_C\) coincides with the normal bundle of the image.

We will primarily deal with modifications of \(N_f\) towards pointing subbundles \(N_{f \to \Lambda}\), whose definition we now recall.  Let \(\Lambda \subset \pp^r\) be a linear space of dimension \(\lambda\).  Let \(\pi_{\Lambda} \circ f\) denote the composition of \(f\) with the projection map
\[\pi_{\Lambda} \colon \pp^r \dashrightarrow \pp^{r-\lambda -1}. \]
Let \(U =U_{\Lambda}\) denote the open locus of \(C \setminus (\Lambda \cap C)\) where \(\pi_{\Lambda} \circ f\) is unramified;
explicitly this is the locus of points of \(C\) whose tangent space does not meet \(\Lambda\).
Assuming that \(U\) is dense and contains the singular locus of \(C\), we may define
\(N_{f \to \Lambda}\) as the unique subbundle of \(N_f\) whose restriction to \(U\) is
\[\ker(N_f|_U \to N_{\pi_{\Lambda} \circ f}|_U).\]
The notation \(N_{f \to \Lambda}\) is evocative of the geometry of sections of \(N_{f \to \Lambda}\): Informally speaking,
they ``point towards'' the subspace \(\Lambda \subset \pp^r\).  
When \(f\) is an embedding, we write \(N_{C \to \Lambda} = N_{f \to \Lambda}\).
If \((\pi_\Lambda \circ f) \colon C \to \pp^{r-\lambda -1}\) is unramified, then \(N_{f \to \Lambda}\) sits in the \defi{pointing bundle exact sequence}
\begin{equation}\label{eq:pointing}
0 \to N_{f \to \Lambda} \to N_f \to N_{\pi_\Lambda \circ f} (\Lambda \cap C) \to 0.
\end{equation}
The same definitions work for families of curves in a projective bundle.  For a treatment in this more general setting, see \cite[Section 5]{aly}.

The simplest case, and our primary interest, is when \(\Lambda = p\) is a point in \(\pp^r\).  In this case, by \cite[Propositions 6.2 and 6.3]{aly}, we have the following explicit descriptions:
\begin{itemize}
\item If \(p \in \pp^r\) is a general point (in which case \(U_p = C\)), then \(N_{f \to p} \simeq f^*\O_{\pp^r}(1)\).
\item If \(p \in C\) is a general point (in which case \(U_p = C \setminus p\)), then \(N_{f \to p} \simeq f^*\O_{\pp^r}(1)(2p)\).
\end{itemize}
When modifying towards a pointing bundle, we use the simpler notation
\[N_f[D \posmod \Lambda] \colonequals N_f[D \posmod N_{f \to \Lambda}].\]

We now restate a foundational result of Hartshorne--Hirschowitz, which describes the normal bundle of a nodal curve in projective space,
in this language of pointing bundles. Let \(X \cup_{\Gamma} Y\) be a reducible nodal curve.  For each point \(p_i \in \Gamma\), let \(q_i\) denote any point on \(T_{p_i}Y \setminus p_i\). For simplicity, we introduce the following notation.  For any subset \(\Gamma' = \{p_1, \dots p_n\} \subseteq \Gamma\), we write
\[N_X[ \Gamma' \posmodalong Y] \colonequals N_X[p_1 \posmod q_1] \dots [p_n \posmod q_n].\]
When \(\Gamma' = \Gamma\) is the full set of points where \(X \) and \(Y\) meet, we simplify further and write
\[N_X[\posmodalong Y]\colonequals N_X[\Gamma \posmodalong Y].\]
(We analogously define \(N_X[ \Gamma' \negmodalong Y]\) and \(N_X[\negmodalong Y]\).)

\begin{prop}[{\cite[Corollary 3.2]{hh}}]\label{prop:hh} As above, let \(X \cup Y \subseteq \pp^r\) be a reducible nodal curve.  Then
\[N_{X \cup Y}|_X \simeq N_X[\posmodalong Y].\]
\end{prop}

\subsection{Interpolation for vector bundles}

\subsubsection{Interpolation for bundles on nodal curves}

Let \(C\) be a nodal curve.  
A nonspecial vector bundle \(E\) of rank \(r\) on \(C\)
satisfies interpolation if and only if, for every \(n > 0\), some collection of \(n\)
points impose the expected number of conditions:
\[h^0(E(-p_1 - \cdots - p_n)) = \max(0, h^0(E) - rn) \qquad \text{for some \(p_1, \dots, p_n\)}.\]

Since \(h^0\) and \(h^1\) can only change in a prescribed way when twisting down by a point, a nonspecial vector bundle \(E\) satisfies interpolation provided that there exist two divisors \(D_+\) and \(D_-\) for which
\begin{equation}\label{eq:DplusDminus}
h^0(E(-D_+)) = 0, \quad h^1(E( - D_-)) = 0, \qand \deg D_+ - \deg D_- \leq 1.
\end{equation}

More generally, we can use this idea to define interpolation for a space of sections of a vector bundle.  Given \(V \subseteq H^0(E)\), write
\[V(-p_1 - \dots - p_n) \colonequals \{ \sigma \in V  : \sigma|_{p_1} =  \cdots = \sigma|_{p_n} = 0\}. \]
We say that \(V \subseteq H^0(E)\) \defi{satisfies interpolation} if \(H^1(E) = 0\) and, for every \(n > 0\), there are \(n\) points
\(p_1, \dots, p_n\), such that
\[\dim V(-p_1 - \cdots - p_n) = \max(0,\dim V - rn).\]

The following basic result allows us to reduce interpolation for a vector bundle \(E\) on a reducible curve \(X \cup_{\Gamma} Y\) to interpolation for a space of sections on one component.

\begin{lem}[{\cite[Lemma 2.10]{p4}}]\label{lem:int_rest}
Let \(E\) be a vector bundle on  \(X \cup_{\Gamma} Y\).  If the restriction map on \(Y\)
\[ \res_{Y,\Gamma} \colon H^0(Y, E|_Y) \to E|_\Gamma\]
is injective, and the space of sections
\(\{\sigma \in H^0(X, E|_X) : \sigma|_\Gamma \in \Im(\res_{Y, \Gamma})\}\)
has dimension \(\chi(E)\) and satisfies interpolation, then \(E\) satisfies interpolation.
\end{lem}
\noindent
The main case of interest in this paper is when \(Y = R\) is a rational curve and \(E|_R\) is perfectly balanced.

\begin{lem}\label{lem:interpolation_rational_bal}
Let \(C \cup_{\Gamma} R\) be a nodal curve with \(R\) rational, and let \(E\) be a vector bundle on \(C \cup R\) with \(E|_R\) perfectly balanced of slope at least \(\#\Gamma - 1\).  If \(E|_C\) satisfies interpolation, then \(E\) satisfies interpolation.
\end{lem}
\begin{proof}
Let \(D\) be an effective divisor of degree \(\mu(E|_R) - \#\Gamma + 1\) supported on \(R \setminus \Gamma\).  It suffices to prove that \(E(-D)\) satisfies interpolation.  
The bundle \(E(-D)|_R\) is perfectly balanced of slope \(\#\Gamma - 1\), and so 
\(\res_{R, \Gamma}\) is an isomorphism.  The result now follows from Lemma \ref{lem:int_rest}.
\end{proof}

\subsubsection{Interpolation and twists}
If \(E\) satisfies interpolation, then the twist \(E(D)\) by any effective divisor \(D\) also satisfies interpolation.  Conversely, we have the following.

\begin{lem}[{\cite[Proposition 4.12]{aly}}] \label{lem:int_large_degree}
Suppose that \(E\) is a vector bundle on a genus \(g\) curve such that 
\[ \chi(E) \geq \rk(E) \cdot g.\]
If there exists an effective divisor \(D\) for which \(E(D)\) satisfies interpolation, then \(E\) also satisfies interpolation.
\end{lem}

\subsubsection{Interpolation and modifications}

Consider a vector bundle \(E\) and its modification \(E[p \posmod F]\).
Given sufficient generality of either \(p\) or \(F\), or if the slope \(\mu(E) \in \zz\), it is sometimes possible to deduce that \(E[p \posmod F]\) satisfies interpolation from the assumption that \(E\) satisfies interpolation.

\begin{lem} \label{lem:muzz} Let \(E\) be a vector bundle on \(C\), let \(p \in C\) be a smooth point, and let \(F \subseteq E|_p\).
If \(E\) satisfies interpolation and \(\mu(E) \in \zz\), then \(E[p \posmod F]\) satisfies interpolation.
\end{lem}
\begin{proof}
Since \(\mu(E) \in \zz\), there is an effective divisor \(D\) with
\[H^0(E(-D)) = 0 \qand H^1(E(-D)) = 0.\]
Then
\[H^0(E[p \posmod F](-D-p)) = 0 \qand H^1(E[p \posmod F](-D)) = 0. \qedhere\]
\end{proof}

\begin{defin}
We say that a collection of subspaces \(\{W_b\}_{b \in B}\) of a vector space \(V\) are \defi{linearly general} if, for any subspace \(U \subset V\), there is some \(b \in B\) so that \(W_b\) is transverse to \(U\).
\end{defin}

\begin{lem}[{\cite[Proposition 4.10]{aly}}]\label{lem:lin_general}
Let \(E\) be a vector bundle on \(C\) and let \(p \in C\) be a smooth point.  Let \(\{F_b\}_{b \in B}\) be a collection of  subspaces of \(E|_p\) that all contain a fixed subspace \(F_0\).  Suppose that both \(E\) and \(E[p \to F_0]\) satisfy interpolation.  If the collection \(\{F_b/F_0\}_{b \in B}\) is linearly general in \(E|_p/F_0\), then for some \(b\in B\), the positive modification
\(E[p \posmod F_b]\)
satisfies interpolation.
\end{lem}

\begin{lem}[{\cite[Proposition 4.21 for \(n = 1\)]{aly}}] \label{gen-to-fixed}
Suppose \(E\) satisfies interpolation, \(L \subset E\) is a nonspecial line subbundle,
and the quotient \(Q = E/L\) also satisfies interpolation.  If 
\(\mu(L) \leq \mu(H)\),
then \(E[p \posmod L]\) satisfies interpolation.
\end{lem}

\begin{rem}
While \cite{aly} assumes characteristic zero, none of the specific results we quote from \cite{aly} use this assumption --- except for Proposition 4.21.
This proposition states that if \(\mu(L) \leq \mu(H) + n - 1\), then
\(E[np \posmod L]\) satisfies interpolation.
The proof uses that vanishing at \(np\) imposes \(n\) conditions on sections of any linear series.
This is true in \emph{any} characteristic when \(n=1\), but \emph{fails in positive characteristic for \(n > 1\)}.
Since we will use Proposition 4.21 of \cite{aly} only when \(n = 1\),
we do not need a restriction on the characteristic.
\end{rem}

\begin{lem}\label{lem:monodromy}
Let \(E\) be a vector bundle on an irreducible curve \(C\).
Let \(p_1, \ldots, p_n \in C\) be points, and let
\(L_i \subseteq E|_{p_i}\) be \(1\)-dimensional subspaces.  Suppose that both 
\(E\) and \(E[p_1 \posmod L_1] \cdots[p_n \posmod L_n]\)
satisfy interpolation.  Then for any \(0 <m <n\), there is a collection of distinct indices \(i_1, \dots, i_m\) such that 
\[E[p_{i_1} \posmod L_{i_1}] \cdots[p_{i_m} \posmod L_{i_m}]\]
satisfies interpolation.
\end{lem}
\begin{proof}
By induction on \(n\) we reduce to the case \(m=n-1\). Write \(E' = E[p_1 \posmod L_1] \cdots[p_n \posmod L_n]\) and 
\(N = \lceil \chi(E') / \rk E'\rceil\).
Let \(D_N\) and \(D_{N-1}\) be general divisors of degrees \(N\) and \(N - 1\) respectively.
Since \(E\) and \(E'\) both satisfy interpolation, and \(\chi(E) < \chi(E')\), we have
\[h^0(E'(-D_N)) = 0, \quad h^0(E'(-D_{N-1})) \neq 0, \qand h^0(E(-D_{N-1})) < h^0(E'(-D_{N-1})).\]
Let \(E_i = E[p_1 \posmod L_1] \cdots[p_{i - 1} \posmod L_{i - 1}] [p_{i + 1} \posmod L_{i + 1}] \cdots [p_n \posmod L_n]\).  Since \(\chi(E_i) = \chi(E') - 1\), it suffices to show \(h^0(E_i(-D_{N-1})) < h^0(E'(-D_{N-1}))\) for some \(i\).  This follows from the fact that
\[\bigcap_i H^0(E_i(-D_{N-1}))  = H^0(E(-D_{N-1})) \subsetneq H^0(E'(-D_{N-1})). \qedhere\]
\end{proof}

\subsubsection{Interpolation and short exact sequences}

\begin{lem}\label{lem:intp_ses}
Consider an exact sequence
\[0 \to S \to E \to Q \to 0\]
of vector bundles on an irreducible curve \(C\).  Suppose that \(S\) and \(Q\) satisfy interpolation and 
\begin{equation}\label{eq:close_enough}\mu(S) \leq \lfloor \mu(Q) \rfloor + 1 \qand \mu(Q) \leq \lfloor \mu(S) \rfloor +1.\end{equation}
Then \(E\) also satisfies interpolation.  
\end{lem}
\begin{proof}
Since \(S\) and \(Q\) are nonspecial, \(E\) is nonspecial.
By \eqref{eq:close_enough}, there exists an integer \(n \in \zz\) such \(\mu(S)\) and \(\mu(Q)\) are contained in the closed interval \([n, n+1]\).
Since \eqref{eq:DplusDminus} is satisfied for \(D_+\) a general divisor of degree \(n + 1\),
and \(D_-\) a general divisor of degree \(n\), we conclude that \(E\) satisfies interpolation as desired.
\end{proof}

\noindent We will most often use this result in the special case in which \(S\) is a line subbundle of \(E\).

\begin{cor}\label{cor:intp_ses}
Suppose that \(S \subset E\) is a nonspecial line subbundle and \(|\mu(S) - \mu(E)| < 1\).  If the quotient \(Q = E/S\) satisfies interpolation, then \(E\) does as well.
\end{cor}
\begin{proof}
By assumption
\(-1 < \frac{\deg(Q) - \rk(Q)\deg(S)}{\rk(Q)+1} < 1.\)  Hence we have strict inequalities
\[\deg(S) < \mu(Q) + \frac{\rk(Q) + 1}{\rk(Q)} \qand \mu(Q) < \deg(S) + \frac{\rk(Q) + 1}{\rk(Q)},\]
which imply the required inequalities in Lemma \ref{lem:intp_ses}, since \(\mu(Q)\) is a \([1/\rk(Q)]\)-integer.
\end{proof}

\section{Overview} \label{sec:overview}

\subsection{Base cases}  
We can reduce the number of base cases by extending Theorem \ref{thm:main} to \(r=1\) and \(r=2\). 
For \(r=2\), we replace \(N_C\) with the normal sheaf \(N_f\), where \(f \colon C \to \pp^2\) is a general BN-curve.  In this case, adjunction implies that \(N_f = K_C \otimes f^*\O_{\pp^2}(3)\) is a nonspecial line bundle, and therefore satisfies interpolation.  For \(r=1\), we only consider the case where \(f \colon C \to \pp^1\) is an isomorphism, so \(N_f = 0 \) satisfies interpolation.

\subsection{\boldmath First Strategy: Degeneration of \(C\)}
The first inductive strategy we will use is
degeneration of \(C\) to reducible curves \(X \cup Y\).
In Section~\ref{sec:basicdegen} we will study certain such degenerations,
for which \(Y\) has a prescribed form, and
we can thus relate interpolation for \(N_C\) to interpolation
for certain modifications of \(N_X\).

While the following hypothesis does not encompass \emph{all}
modifications that might appear using this method,
it includes those modifications that will play the most central role
in our inductive argument:

\begin{hypo}[\(I(d, g, r, \ell, m)\)]\label{ind_hyp}
Let \(C \subset \pp^r\) be a general BN-curve of degree \(d\) and genus \(g\).  Let \(u_1, v_1, \dots, u_\ell, v_\ell\) be \(\ell\) pairs of general points on \(C\).  Let \(R_1, \dots, R_m\) be \(m\) general \((r+1)\)-secant rational curves of degree \(r - 1\) (contained in hyperplanes transverse to \(C\)).
Then the modification 
\[N_C[u_1 \biposmod v_1] \cdots [u_\ell \biposmod v_\ell] [\posmodalong R_1]\cdots[\posmodalong R_m]\]
of the normal bundle of \(C\)
satisfies interpolation.
\end{hypo}

A central complicating factor is that the inductive hypothesis \(I(d, g, r, \ell, m)\) is not always true.
The following definition describes a set of tuples \((d, g, r, \ell, m)\)
for which we will prove that it holds:

\begin{defin}
A tuple \((d, g, r, \ell, m)\) is called \defi{good} if it satisfies all the following conditions:
\begin{itemize}
\item The following inequalities hold:
\[d \geq g + r, \quad 0 \leq \ell \leq \frac{r}{2}, \quad \text{and} \quad 0\leq m \leq \rho(d,g,r).\]
\item If \(g=m=0\), then
\[2 \ell  \geq (1-d) \redmod (r-1),\]
where for integers \(a\) and \(b\) we write \(a \redmod b\) for the reduced residue of \(a\) modulo \(b\),
and
\item It is not the following set:
\begin{equation}\label{XX}\tag{XEx}
\left\{
\begin{array}{cccccc}
(5, 2, 3, 0, 0), &(4, 1, 3, 1, 0), & (4, 1, 3, 0, 1), & (4, 1, 3, 1, 1), \\
 (6, 2, 4, 0, 0), & (5, 1, 4, 1, 0),&& (5, 1, 4, 1, 1), & (5, 1, 4, 2, 1), & (6, 2, 4, 1, 1),  \\
 (7, 2, 5, 0, 0), &&(6, 1, 5, 0, 1), &(6, 1, 5, 1, 1).
 \end{array}
\right\}
\end{equation}
\end{itemize}
\end{defin}

\begin{rem}\label{rem:small_g}
By Clifford's theorem, the first inequality  \(d \geq g + r\) follows from \(g \leq r\).
\end{rem}

We conclude Section~\ref{sec:basicdegen} by using this first strategy
to show that if \(I(d, g, r, 0, m)\) holds for every good tuple \((d, g, r, 0, m)\),
then Theorem~\ref{thm:main} holds except possibly for rational curves
or canonical curves of even genus.

\subsection{Second Strategy: Limits of Modifications and Projection}
The basic issue with the first strategy described above
is that every time we apply it,
we get more modifications. In order to make an inductive
argument work, we need a second inductive strategy that \emph{decreases}
the number of modifications.  

Hypothesis \(I(d,g,r,\ell, m)\) assets that \(N_C' \colonequals N_C[u_1 \biposmod v_1] \cdots [u_\ell \biposmod v_\ell] [\posmodalong R_1]\cdots[\posmodalong R_m]\) satisfies interpolation.  Let \(p\) be a general point on \(C\).
The pointing bundle exact sequence induces the exact sequence
\begin{equation}\label{eq:pointing_seq}
0 \to N_{C \to p} \to N_C' \to N_{\bar{C}}(p)[u_1 \biposmod v_1] \cdots [u_\ell \biposmod v_\ell] [\posmodalong \bar{R}_1\cup \dots \cup \bar{R}_m] \to 0,
\end{equation}
where \(\bar{C}\) and \(\bar{R}_j\) denote the images of \(C\) and \(R_j\) respectively
under projection from \(p\).  In order to apply Corollary \ref{cor:intp_ses} to relate interpolation for the original bundle \(N_C'\) to interpolation for the quotient bundle, the sequence must be close to balanced.  The failure of the sequence to be balanced is related to the quantity
\[\delta = \delta(d, g, r, \ell, m) \colonequals \mu(N_C') - \mu(N_{C \to p}) = \frac{2d + 2g - 2r + 2\ell + (r+1)m}{r-1}.\]
We first apply these ideas in Section \ref{sec:delta1} to treat the family of good tuples \((d, g,r, 0, 0)\) with \(\delta(d, g, r, 0,0)=1\), which are difficult from the perspective of our more uniform inductive arguments.

More generally, in order to make this sequence sufficiently close to balanced, we will appropriately specialize the points on \(C\) at which the modifications occur.  To illustrate this idea in the simplest possible case, assume here that \(\ell \geq \lfloor \delta\rfloor\).  Since the points \(v_1, \dots, v_{\lfloor \delta \rfloor}\) are general on \(C\), we may specialize them all to the point \(p\).  This induces the specialization of \(N_C'\) to
\[N_C'' \colonequals N_C[u_1 + \dots + u_{\lfloor \delta\rfloor} \biposmod p][u_{\lfloor \delta\rfloor +1} \biposmod v_{\lfloor \delta\rfloor +1}] \cdots [u_\ell \biposmod v_\ell] [\posmodalong R_1\cup\cdots\cup R_m].\]
Using \eqref{eq:original_ses}, the pointing bundle exact sequence becomes
\begin{multline*}
0 \to N_{C \to p}(u_1 + \dots + u_{\lfloor \delta\rfloor}) \to N_C'' \\
\to N_{\bar{C}}(p)[p \posmod u_1 + \dots + u_{\lfloor \delta\rfloor}] [u_{\lfloor \delta\rfloor +1} \biposmod v_{\lfloor \delta\rfloor +1}] \cdots [u_\ell \biposmod v_\ell] [\posmodalong \bar{R}_1\cup \dots \cup \bar{R}_m] \to 0.
\end{multline*}
By our auspicious choice to specialize exactly \({\lfloor \delta \rfloor}\) points to \(p\), this sequence is now close enough to balanced to reduce to proving interpolation for the quotient bundle.  Furthermore, since \(u_1, \dots, u_{\lfloor \delta \rfloor}\) are general, the modification at \(p\) in the quotient  is linearly general and we can erase it by Lemma \ref{gen-to-fixed}.  It therefore suffices to prove interpolation for 
\[N_{\bar{C}} [u_{\lfloor \delta\rfloor +1} \biposmod v_{\lfloor \delta\rfloor +1}] \cdots [u_\ell \biposmod v_\ell] [\posmodalong \bar{R}_1\cup \dots \cup \bar{R}_m],\]
which evidently has fewer modifications.
However, there are two basic issues with this argument:
\begin{enumerate}
\item In general, we might not have \(\ell \geq \lfloor \delta \rfloor\).
\item
Since \(\bar{R}_i\) is still an \((r+1)\)-secant curve of degree \(r-1\), the argument does not reduce to another case of our inductive hypothesis.
\end{enumerate}
To surmount both of these two difficulties, we will need to specialize the \(R_i\) as well.

This second strategy will be fleshed out in Sections~\ref{sec:onion_specialization} and~\ref{sec:inductive}:
In Section~\ref{sec:onion_specialization}, we will study how to specialize the \(R_i\)
so that they can also contribute modifications to \(N_{C \to p}\).
Then in Section~\ref{sec:inductive}, we will refine the basic argument outlined above
to use these degenerations of the \(R_i\) as well,
and also explain further degenerations that will be necessary
to reduce to another case of our inductive hypothesis.

\subsection{Outline of the Remainder of Paper}
Section \ref{sec:rational} is a brief interlude in which we use the inductive arguments of Section \ref{sec:inductive} to treat the case of rational curves not implied by \(I(d,g,r,\ell,m)\) for good tuples.  We also explain the counterexamples to Theorem \ref{thm:main} that are not counterexamples to Theorem~\ref{cor:main}.  At this point we will have reduced both Theorem \ref{thm:main} and Theorem~\ref{cor:main} to \(I(d,g,r,\ell, m)\) for good tuples, as well as Theorem \ref{thm:main} for canonical curves of even genus, which we treat at the end of the paper in Section \ref{sec:canonical}.  The intervening sections \ref{sec:combinat}--\ref{sec:remaining-sporadic} inductively prove \(I(d,g,r,\ell,m)\) for good tuples.

In Section \ref{sec:combinat}, we complete a purely combinatorial analysis, in which
we show that the inductive arguments of Section \ref{sec:inductive} can be applied to 
reduce \(I(d, g, r, \ell, m)\)
for \emph{all} good tuples to
\begin{itemize}
\item \(I(d, g, r, \ell, m)\) for a certain large but \emph{finite} list of sporadic cases \((d, g, r, \ell, m)\) with \(r \leq 13\).
\item The infinite family of tuples \((d, g, r, 0, 0)\) with \(\delta = 1\), which was already treated in Section \ref{sec:delta1}.
\end{itemize}
In Section \ref{sec:most-sporadic} we give a more complicated, yet more flexible, inductive argument in the style of those in Section \ref{sec:inductive}, and verify by exhaustive computer search that
it reduces the finitely many sporadic cases identified above to a managable list of \(30\) base cases.
These base cases are treated by ad hoc techniques in Section \ref{sec:remaining-sporadic}.

\section{Basic Degenerations \label{sec:basicdegen}}

In this section, we discuss the three basic degenerations of BN-curves, to reducible curves \(C \cup D\),
that we will use in the proof of Theorem \ref{thm:main}.  In each subsection we will first show that these degenerations
lie in the Brill--Noether component. 
We will then relate interpolation for \(N_{C \cup D}\), or a modification thereof,
to interpolation for a particular modification of \(N_C\).

In what follows, write \(N_{C \cup D}'\) for a modification of \(N_{C \cup D}\) away from \(D\).
In other words, \(N_{C \cup D}'\) is a vector bundle on \(C \cup D\), equipped with an isomorphism
to \(N_{C \cup D}\) over a dense open subset of \(C \cup D\) containing the entire curve \(D\),
and in particular containing a neighborhood \(U\) of \(C \cap D\) in \(C\).
Write \(N_C'\) for the bundle obtained by making the same modifications to \(N_C\).
In other words, \(N_C'\) is obtained from \(N_{C \cup D}'|_{C \setminus (C \cap D)}\)
by gluing along \(U \setminus (C \cap D)\) via our given isomorphism to \(N_C|_U\).

\begin{center}
\begin{minipage}{.4\textwidth}
\begin{center}
\begin{tikzpicture}[scale=1.5]
\draw (1, 2) .. controls (0.5, 2) and (-0.5, 1.5) .. (0, 1);
\draw (0, 1) .. controls (1, 0) and (1, 2) .. (0.1, 1.1);
\draw (-0.1, 0.9) .. controls (-0.5, 0.5) and (0.5, -0.3) .. (1, -0.3);
\draw (0.8, 0.5) node[right]{\(D = L\)};
\draw (1.12, 1.9) node{\(C\)};
\draw (0.8, 0.5) -- (-0.3, 0.61);
\filldraw (0.36, 0.544) circle[radius=0.02];
\filldraw (-0.18, 0.598) circle[radius=0.02];
\draw (0.3, 0.45) node{\(v\)};
\draw (-0.25, 0.52) node{\(u\)};
\end{tikzpicture} \\
Peeling off a \(1\)-secant line.
\end{center}
\end{minipage}
\begin{minipage}{.4\textwidth}
\begin{center}
\begin{tikzpicture}[scale=1.5]
\draw (1, 2) .. controls (0.5, 2) and (-0.5, 1.5) .. (0, 1);
\draw (0, 1) .. controls (1, 0) and (1, 2) .. (0.1, 1.1);
\draw (-0.1, 0.9) .. controls (-0.5, 0.5) and (0.5, -0.3) .. (1, -0.3);
\draw (1.12, 1.9) node{\(C\)};
\filldraw (0.31, 0.767) circle[radius=0.02];
\draw (0.45, 0.814) -- (-0.6, 0.458);
\draw (-0.6, 0.45) node[left]{\(D = L\)};
\filldraw (-0.18, 0.598) circle[radius=0.02];
\draw (-0.29, 0.67) node{\(u\)};
\draw (0.31, 0.87) node{\(v\)};
\end{tikzpicture}\\
Peeling off a \(2\)-secant line.
\end{center}
\end{minipage}
\end{center}

\subsection{\boldmath Peeling off \(1\)-secant lines}\label{subsec:1sec}

Our most basic degeneration will be when \(D = L\) is a quasitransverse \(1\)-secant line.  If \(C\) has degree \(d\) and genus \(g\), then \(C \cup L\) has degree \(d+1\) and genus~\(g\). Write \(v \in L \setminus \{u\}\)
for any other point on the line \(L\).

\begin{lem}\label{lem:1secBN}
If \(C\) is a BN-curve, then \(C \cup L\) is also a BN-curve.
\end{lem}
\begin{proof}
Generalizing \(C\), we may suppose \(C\) is a general BN-curve.
We will show \(H^1(T_{\pp^r}|_{C \cup L}) =0\), which implies that the map \(C \cup L \to \pp^r\) may be lifted as \(C \cup L\) is deformed to a general curve. 

Since \(C\) is a general BN-curve, \(H^1(T_{\pp^r}|_C)=0\) by the Gieseker--Petri Theorem.  Furthermore, we have \(H^1(T_{\pp^r}|_L(-u)) =0\), because \(T_{\pp^r}|_L \simeq \O_{\pp^1}(2) \oplus \O_{\pp^1}(1)^{\oplus (r-1)}\).  This implies \(H^1(T_{\pp^r}|_{C \cup L}) = 0\) as desired, using
\[ 0 \to T_{\pp^r}|_L(-u) \to T_{\pp^r}|_{C \cup L} \to T_{\pp^r}|_C \to 0. \qedhere \]
\end{proof}

\begin{lem}[Lemma 8.5 of \cite{aly}]\label{lem:1sec}
If \(N_C'(u)[2u \negmod v]\) satisfies interpolation,
then so does \(N_{C \cup L}'\).
\end{lem}

When \(C\) is nonspecial with genus small relative to \(r\), we can
combine Lemma \ref{lem:1sec} with Lemma \ref{lem:int_large_degree} to reduce to a positive modification of \(N_C'\).

\begin{cor}\label{cor:1sec_posmod}
Suppose that \(N_C'\) is a positive modification of \(N_C\).  If
\(d \geq g + r\) and \(g \leq r + 6\)
and \(N_C'[2u \posmod v]\) satisfies interpolation, then \(N_{C \cup L}'\) satisfies interpolation.
\end{cor}
\begin{proof}
Since
\( N_C'[2u \posmod v] \simeq \left(N_C'(u)[2u \negmod v]\right)(u)\),
it
suffices by Lemma \ref{lem:int_large_degree} to show that 
\begin{equation}\label{eq:desired_ineq}
\chi\left(N_C'(u)[2u \negmod v]\right) \geq (r-1)g.
\end{equation}
We have
\begin{align*}
\chi(N_C'(u)[2u \negmod v]) &\geq \chi(N_C(u)[2u \negmod v]) \\
&= (r+1)d - (r-3)(g-1) - (r-3) \\
&\geq (r+1)(g+r) - (r-3)g \\
& = 4g + r(r+1).
\end{align*}
If \(g \leq r + 6\), then
\[(r-1)g - 4g = (r-5)g \leq (r-5)(r+6) \leq r(r+1),\]
and the desired inequality \eqref{eq:desired_ineq} holds.
\end{proof}

\subsection{\boldmath Peeling off \(2\)-secant lines}\label{subsec:2sec}

Our next basic degeneration will be to the union of a curve \(C\) and a quasitransverse \(2\)-secant line \(L\), meeting \(C\)
at points \(u\) and \(v\).  If \(C\) has degree \(d\) and genus \(g\), then \(C \cup L\) has degree \(d+1\) and genus \(g+1\).  

\begin{lem} \label{lem:2secBN}
If \(C\) is a BN-curve, then \(C \cup L\) is also a BN-curve.
\end{lem}
\begin{proof}
As in the proof of Lemma \ref{lem:1secBN}, we have that \(H^1(T_{\pp^r}|_{C \cup L})=0\) by combining \(H^1(T_{\pp^r}|_C)=0\) (from the Gieseker--Petri theorem) and \(H^1(T_{\pp^r}|_L(-u-v)) = 0\) (from \(T_{\pp^r}|_L \simeq \O_{\pp^1}(2) \oplus \O_{\pp^1}(1)^{\oplus r-1}\)).
\end{proof}

We generalize Lemma~\ref{lem:1sec} to \(2\)-secant lines.
In slightly greater generality,
let \(p \in L \setminus \{u, v\}\) be a point, and
\(\Lambda\) be a linear space disjoint from the span of the tangent lines to \(C\) at \(u\) and \(v\).

\begin{lem}[Slight generalization of Lemma 8.8 of \cite{aly}]\label{lem:2sec}
If \(N_C'(u + v)[u \negmod v][v \negmod u][v \negmod 2u + \Lambda]\)
satisfies interpolation,
then so does \(N_{C \cup L}'[p \posmod \Lambda]\).
\end{lem}
\begin{proof}
Imitate the proof Lemma 8.8 of \cite{aly}, mutatis mutandis.

(In the notation of \cite{aly}: Take \(T\) instead to have dimension \(r - 5 - \dim \Lambda\),
where by convention \(\dim \emptyset = -1\), and use instead the decomposition 
\(N_L \simeq N_{L \to x} \oplus N_{L \to y} \oplus N_{L \to \Lambda} \oplus N_{L \to T}\).)
\end{proof}

\noindent As in Corollary \ref{cor:1sec_posmod}, we can 
reduce to a positive modification of \(N_C'\).

\begin{cor}\label{cor:2sec_posmod}
Suppose that \(N_C'\) is a positive modification of \(N_C\).  If
\(d \geq g + r\) and \(g \leq r + 6\)
and \(N_C'[u \posmod v][v \posmod u][v \posmod 2u]\) satisfies interpolation, then \(N_{C \cup L}'\) satisfies interpolation.
\end{cor}
\begin{proof}
We have
\begin{align*}
\chi\left(N_C'(u+v)[u \negmod v][v \negmod u][v \negmod 2u]\right) &\geq \chi\left(N_C(u+v)[u \negmod v][v \negmod u][v \negmod 2u]\right) \\
&= (r+1)d - (r-3)(g-1) -(r-5) \\
& \geq 4g + r(r+1) + 2.
\end{align*}
As in the proof of Corollary \ref{cor:1sec_posmod}, if \(g \leq r+6\), then \(4g + r(r+1)+2 \geq (r-1)g\), and so
\begin{equation}\label{eq:desired_ineq2}\chi\left(N_C'(u+v)[u \negmod v][v \negmod u][v \negmod 2u]\right) \geq (r-1)g.\end{equation}
Hence, by Lemma \ref{lem:int_large_degree}, \(N_C'(u+v)[u \negmod v][v \negmod u][v \negmod 2u]\) satisfies interpolation.
\end{proof}

\subsection{Peeling off rational normal curves in hyperplanes}\label{subsec:onion}

The final of our basic degenerations is to the union of a BN-curve \(C\) of degree at least \(r+1\) in \(\pp^r\), and an \((r+1)\)-secant rational curve \(R\) of degree \(r-1\) contained in a hyperplane \(H\).  If \(C\) has degree \(d\) and genus \(g\), then \(C \cup R\) has degree \(d + r-1\) and genus \(g + r\).  Observe that
\[\rho(d,g,r) = \rho(d + r -1, g+r, r) + 1.\]

\begin{lem}\label{lem:onionBN}
If \(\rho(d, g, r) \geq 1\), and \(C\) is a BN-curve, then \(C \cup R\) is also a BN-curve.
\end{lem}
\begin{proof}
Generalizing \(C\), we may suppose \(C\) is a general BN-curve.

If \(g > 0\), then we can specialize \(C\) to the union \(C' \cup L\), where \(C'\)
is a general BN-curve of degree \(d-1\) and genus \(g-1\), and \(L\) is a general \(2\)-secant line.
Otherwise, if \(g = 0\), then we can specialize \(C\) to the union \(C' \cup L\), where \(C'\)
is a general BN-curve of degree \(d-1\) and genus \(g\), and \(L\) is a general \(1\)-secant line.
Either way, we can arrange for
one of the points \(p\) of \(C \cap R\) to specialize onto \(L\), and the rest to specialize onto \(C'\).
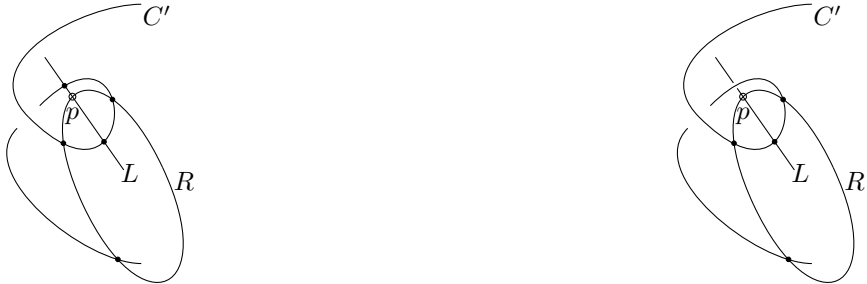
\begin{center}
\begin{minipage}{0.9\textwidth}
\begin{center}
\begin{minipage}{.4\textwidth}
\begin{center}
\begin{tikzpicture}[scale=1.5]
\draw (1, 2) .. controls (0.5, 2) and (-0.5, 1.5) .. (0, 1);
\draw (0, 1) .. controls (1, 0) and (1, 2) .. (0.1, 1.1);
\draw (-0.1, 0.9) .. controls (-0.5, 0.5) and (0.5, -0.3) .. (1, -0.3);
\draw (0.145, 1.53) -- (0.845, 0.53);
\draw (0.4838, 0.22) .. controls (-0.0362, 1.3328) and (0.67, 1.6628) .. (1.19, 0.55);
\draw (0.4838, 0.22) .. controls (1.0038, -0.8928) and (1.71, -0.5628) .. (1.19, 0.55);
\draw (0.39, 1.18) circle[radius=0.03];
\draw (0.9, 0.5) node{\(L\)};
\draw (1.38, 0.43) node{\(R\)};
\draw (0.39, 1.01) node{\(p\)};
\draw (1.14, 1.9) node{\(C'\)};
\filldraw (0.793, -0.263) circle[radius=0.02];
\filldraw (0.67, 0.78) circle[radius=0.02];
\filldraw (0.31, 0.767) circle[radius=0.02];
\filldraw (0.32, 1.277) circle[radius=0.02];
\filldraw (0.745, 1.155) circle[radius=0.02];
\end{tikzpicture}\\
\end{center}
\end{minipage}
\hfill
\begin{minipage}{.4\textwidth}
\begin{center}
\begin{tikzpicture}[scale=1.5]
\draw (0.145, 1.53) -- (0.845, 0.53);
\filldraw[color=white] (0.32, 1.277) circle[radius=0.025];
\draw (1, 2) .. controls (0.5, 2) and (-0.5, 1.5) .. (0, 1);
\draw (0, 1) .. controls (1, 0) and (1, 2) .. (0.1, 1.1);
\draw (-0.1, 0.9) .. controls (-0.5, 0.5) and (0.5, -0.3) .. (1, -0.3);
\draw (0.4838, 0.22) .. controls (-0.0362, 1.3328) and (0.67, 1.6628) .. (1.19, 0.55);
\draw (0.4838, 0.22) .. controls (1.0038, -0.8928) and (1.71, -0.5628) .. (1.19, 0.55);
\draw (0.39, 1.18) circle[radius=0.03];
\draw (0.9, 0.5) node{\(L\)};
\draw (1.38, 0.43) node{\(R\)};
\draw (0.39, 1.01) node{\(p\)};
\draw (1.14, 1.9) node{\(C'\)};
\filldraw (0.793, -0.263) circle[radius=0.02];
\filldraw (0.67, 0.78) circle[radius=0.02];
\filldraw (0.31, 0.767) circle[radius=0.02];
\filldraw (0.745, 1.155) circle[radius=0.02];
\end{tikzpicture}\\
\end{center}
\end{minipage}
\end{center}
\vspace{-25pt}
\end{minipage}
\end{center}

\noindent
Write \(\Gamma \colonequals (L \cup R) \cap C\).  Note that this is a set of \(r + 1\) or \(r+2\) points on \(C'\).
We will show the following:
\begin{enumerate}[label=(\alph*)]
\item\label{onionBNa} The curve \(C' \cup L \cup R\) is a smooth point of the Hilbert scheme.
\item\label{onionBNb} The curve \(L \cup R\) can be smoothed to a rational normal curve \(M\) while preserving the points of incidence with \(C'\).
\item\label{onionBNc} The curve \(C' \cup M\) is in the Brill--Noether component.
\end{enumerate}
By part \ref{onionBNa}, the curves \(C \cup R\) and \(C' \cup M\) are both generalizations of a smooth point of the Hilbert scheme.  Hence they are in the same component, which must be the Brill--Noether component by part \ref{onionBNc}.

Beginning with part \ref{onionBNb}, write \(N\) for the subsheaf of \(N_{L \cup R}\) whose sections fail to smooth the node \(p\).  It suffices to show that \(H^1(N(-\Gamma))=0\).  This follows from the exact sequence
\[ 0 \to N_{R \cup L}|_R(-p - \Gamma) \to N(-\Gamma) \to N|_L(-\Gamma) \to 0,\]
together with the isomorphisms
\begin{align*}
N_{R \cup L}|_R &\simeq N_{R/H} \oplus \O_R(1)(p) \simeq \O_{\pp^1}(r+2)^{\oplus r-2} \oplus \O_{\pp^1}(r), \\
N|_L &\simeq N_L \simeq \O_{\pp^1}(1)^{\oplus r-1}.
\end{align*}

Since \(N_{L \cup R}(-\Gamma)\) is a positive modification of \(N(-\Gamma)\), we also have \(H^1(N_{L \cup R}(-\Gamma)) = 0\).  Similarly, \(N_{C' \cup L \cup R}|_{C'}\) is a positive modification of \(N_{C'}\).  As \(C'\) is a general BN-curve, the Gieseker--Petri theorem implies that \(H^1(N_{C' \cup L \cup R}|_{C'})=0\).  Hence, using the exact sequence
\[0 \to N_{L \cup R}(-\Gamma) \to N_{C' \cup L \cup R} \to N_{C' \cup L \cup R}|_{C'} \to 0, \]
we see that \(H^1(N_{C' \cup L \cup R}) = 0\), and 
part \ref{onionBNa} follows.  

Finally, for part \ref{onionBNc}, we will show that \(H^1(T_{\pp^r}|_{C' \cup M})= 0\), and hence the map from \(C' \cup M\) to \(\pp^r\) can be lifted as \(C' \cup M\) is smoothed to a general curve.  This vanishing follows from the exact sequence
\[0 \to T_{\pp^r}|_M(-\Gamma) \to T_{\pp^r}|_{C' \cup M} \to T_{\pp^r}|_{C'} \to 0, \]
the isomorphism
\(T_{\pp^r}|_M \simeq \O_{\pp^r}(r+1)^{\oplus r} \), and the Gieseker--Petri theorem (\(H^1(T_{\pp^r}|_{C'}) = 0\)).
\end{proof}

Our next goal is to study the restricted normal bundle
\(N_{C \cup R}|_R \simeq N_R[\posmodalong C]\),
which is of slope \(r + 2\). In most cases, 
this bundle is perfectly balanced (equivalently semistable):

\begin{lem}\label{lem:rest_onion_bal}
Unless \(r\) is odd and \(C\) is an elliptic normal curve,
\(N_R[\posmodalong C]\) is perfectly balanced, i.e.
\begin{equation} \label{perf-bal}
N_R[\posmodalong C] \simeq \O_{\pp^1}(r + 2)^{\oplus (r - 1)}.
\end{equation}
If \(r\) is odd and \(C\) is an elliptic normal curve, then \(N_R[\posmodalong C]\) is ``almost balanced'', i.e., is isomorphic
to one of the two bundles:
\[\O_{\pp^1}(r + 2)^{\oplus (r - 1)} \quad \text{or} \quad \O_{\pp^1}(r + 3) \oplus \O_{\pp^1}(r + 2)^{\oplus (r - 3)} \oplus \O_{\pp^1}(r + 1).\]
\end{lem}
\begin{proof}
Write \(d\) and \(g\) for the degree and genus of \(C\).
First we reduce to the cases where \(C\) is nonspecial, i.e., where \(d \geq g + r\).
To do this, when \(d < g + r\),
we inductively specialize \(C\) to a union \(C' \cup D\), where \(C'\) is a general BN-curve of degree \(d' \geq r + 2\).
Since in particular \(d' \geq r + 1\), we may specialize the points where \(R\) meets \(C\) onto \(C'\),
thereby replacing \(C\) by \(C'\), which is not an elliptic normal curve since \(d' \geq r + 2\).
To find such a specialization, we break into cases as follows:

\begin{enumerate}
\item If \(d < g + r\) and \(\rho(d, g, r) > 0\) (which forces \(d \geq 2r + 1 \geq r + 3\)),
we apply Lemma~\ref{lem:2secBN} to degenerate \(C\)
to the union of a general BN-curve \(C'\) of degree \(d - 1\) and genus \(g - 1\), with a \(2\)-secant line \(D\).

\item If \(d < g + r\) and \(\rho(d, g, r) = 0\), but \(C\) is not a canonical curve
(which forces \(d \geq 3r \geq 2r + 2\)),
we claim that we may specialize \(C\)
to the union of a general BN-curve \(C'\) of degree \(d - r\) and genus \(g - r - 1\), with an \((r + 2)\)-secant rational
normal curve \(D\).
Indeed, as in the proof of Lemma \ref{lem:1secBN}, we have that \(H^1(T_{\pp^r}|_{C' \cup D})=0\) by combining \(H^1(T_{\pp^r}|_{C'})=0\) (from the Gieseker--Petri theorem) and \(H^1(T_{\pp^r}|_D(-D \cap C')) = 0\) (from \(T_{\pp^r}|_D \simeq \O_{\pp^1}(r + 1)^{\oplus r}\)).

\item If \(C\) is a canonical curve, we claim that we may specialize \(C\)
to the union of a general BN-curve \(C'\) of degree \(r + 2\) and genus \(2\), with a \(r\)-secant rational curve of degree \(r - 2\).
Indeed, we glue an abstract curve of genus \(2\) to \(\pp^1\) at \(r\) general points,
and map it to projective space via the complete linear series for the dualizing sheaf.
\end{enumerate}

For \(C\) nonspecial, we will prove the lemma by induction on \(r\).
Let \(\Gamma\) be a collection of \(r + 1\) points where \(C\) meets \(R\);
for \(r \geq 3\), this is exactly the intersection \(C \cap R\),
and so \(N_R[\Gamma \posmodalong C] = N_R[\posmodalong C]\).
However, we can extend the lemma to cover the case \(r = 2\) as well,
by replacing \eqref{perf-bal} with the assertion that
\(N_R[\Gamma \posmodalong C] \simeq \O_{\pp^1}(r + 2)^{\oplus (r - 1)}\).
With this formulation, the base case of \(r=2\) is clear since \(N_R[\Gamma \posmodalong C]\)
is a line bundle of degree \(4\).
The base case \(r=3\) is \cite[Lemma 4.2]{stab_p3}.

For the inductive step, we suppose \(r \geq 4\). Let \(H'\) be a general hyperplane transverse to \(H\).  
We will degenerate \(C\) to a union \(C' \cup L\), where \(C' \subset H'\) is a general BN-curve of degree \(d-1\) and genus \(g\).  By hypothesis, \(d\geq r+1\), and so \(C'\) meets \(H\) in at least \(r\) points, which are in linear general position
in \(H \cap H' \simeq \pp^{r - 2}\) since the sectional monodromy group of a general curve always contains the alternating group \cite{kadets}.  Since \(\Aut \pp^{r-2}\) acts transitively
on \(r\)-tuples of point in linear general position, we can apply an automorphism so that \(r-1\) of these points are on \(R\) and the final point \(p\) lies on a \(2\)-secant line \(L\) to \(R\).

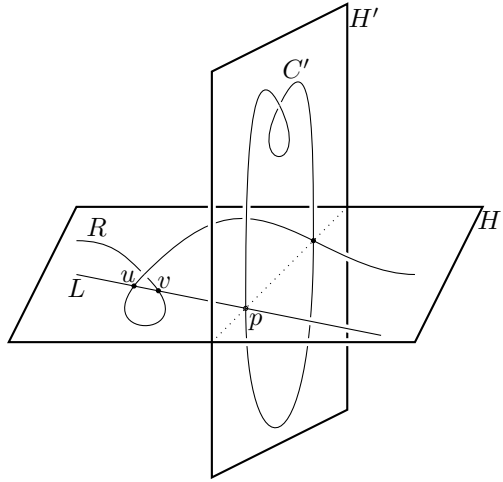
\begin{center}
\begin{tikzpicture}[scale=0.9]
\draw[thick] (3, 0) -- (0, 0) -- (1, 2) -- (5, 2);
\draw[thick] (3, 0) -- (3, -2) -- (5, -1) -- (5, 2);
\draw[dotted] (3, 0) -- (5, 2);
\draw (4, 3.5) .. controls (4.5, 2.5) and (3.5, 2.5) .. (3.975, 3.45);
\draw (3.5, 0.5) .. controls (3.5, -2) and (4.5, -2) .. (4.5, 1.5);
\draw (4.25, 4.05) node{\(C'\)};
\draw (2, 1) .. controls (1, 0) and (3, 0) .. (2.05, 0.95);
\draw (1.95, 1.05) .. controls (1.7, 1.3) and (1.5, 1.5) .. (1, 1.5);
\draw (2, 1) .. controls (3, 2) and (3.5, 2) .. (4.5, 1.5);
\draw (3.5, 0.5) -- (1, 1);
\draw (3.5, 0.5) circle[radius=0.03];
\filldraw (4.5, 1.5) circle[radius=0.03];
\draw[white, line width = 3pt] (3.1, 0) -- (6, 0) -- (7, 2) -- (5.1, 2);
\draw[thick] (3, 0) -- (6, 0) -- (7, 2) -- (5, 2);
\draw[white, line width = 3pt] (3.55, 0.49) -- (5.5, 0.1);
\draw (3.5, 0.5) -- (5.5, 0.1);
\draw[white, line width = 3pt] (4.54, 1.48) .. controls (5, 1.25) and (5.5, 1) .. (6, 1);
\draw (4.5, 1.5) .. controls (5, 1.25) and (5.5, 1) .. (6, 1);
\draw[white, line width = 3pt] (3, 0.1) -- (3, 4) -- (5, 5) -- (5, 2.1);
\draw[thick] (3, 0) -- (3, 4) -- (5, 5) -- (5, 2);
\draw[white, line width = 3pt] (3.5, 0.55) .. controls (3.5, 2) and (3.5, 4.5) .. (4, 3.5);
\draw[white, line width = 3pt] (4.5, 1.55) .. controls (4.5, 3) and (4.5, 4.5) .. (4.025, 3.55);
\draw (3.5, 0.5) .. controls (3.5, 2) and (3.5, 4.5) .. (4, 3.5);
\draw (4.5, 1.5) .. controls (4.5, 3) and (4.5, 4.5) .. (4.025, 3.55);
\draw (1, 0.8) node{\(L\)};
\draw (1.3, 1.7) node{\(R\)};
\draw (3.65, 0.3) node{\(p\)};
\draw (7.1, 1.8) node{\(H\)};
\draw (5.25, 4.8) node{\(H'\)};
\filldraw (1.85, 0.83) circle[radius=0.03];
\filldraw (2.21, 0.76) circle[radius=0.03];
\draw (1.76, 0.97) node{\(u\)};
\draw (2.3, 0.89) node{\(v\)};
\end{tikzpicture}
\end{center}

We claim that \(C' \cup L\) is a BN-curve of degree \(d\) and genus \(g\), as \(H^1(T_{\pp^r}|_{C' \cup L})=0\).  Indeed, \(H^1(T_{\pp^r}|_{C'}) = 0\) because \(C'\) is nonspecial (and \(T_{\pp^r}\) is a quotient of \(\O_{\pp^r}(1)^{\oplus r+1}\)).  Furthermore, \(T_{\pp^r}|_L \simeq \O_{\pp^1}(2) \oplus \O_{\pp^r}(1)^{\oplus r}\), so \(H^1(T_{\pp^r}|_L(-p)) = 0\).  The result now follows by considering
\[0 \to T_{\pp^r}|_L(-p) \to T_{\pp^r}|_{C' \cup L} \to T_{\pp^r}|_{C'} \to 0.\]
Call \(u\) and \(v\) the points on \(R\) where \(L\) is \(2\)-secant.
Projecting
from \(L\) 
induces an exact sequence
\[ 0 \to N_{R \to L}[u \posmod v][v \posmod u] \simeq \O_{\pp^1}(r + 2)^{\oplus 2} \to N_R[\posmodalong C' \cup L] \to N_{\bar{R}}[\posmodalong \bar{C'}] (u + v) \to 0.\]
The curve \(\bar{R}\) is again a rational curve of degree \(r-3\) in a hyperplane that is incident to \(\bar{C'}\) at \(r-1\) points.  Furthermore, if \(C\) is not an elliptic normal curve, then neither is \(\bar{C'}\).
Applying our inductive hypothesis (for \(\pp^{r - 2}\)),
in combination with the above exact sequence, completes the proof.
\end{proof}

It is natural to ask what happens when \(C\) is an elliptic normal
curve and \(r\) is odd.  This case is only necessary to treat the special family of canonical curves of even genus.  When we treat that case in Section~\ref{sec:canonical}, we will show that
\(N_R[\posmodalong C]\) is not perfectly balanced in this case.
Moreover, we will give a geometric construction of its Harder--Narasimhan
filtration.

\subsection{Reduction to good tuples}\label{sec:BN_to_inductive}

In this section we show that, apart from rational curves and canonical curves
of even genus, all other cases of
Theorem~\ref{thm:main} follow from \(I(d,g,r, 0, m)\) for good tuples with \(\ell = 0\).

\begin{lem}\label{lem:pull_off_onion}
Suppose that \(\rho(d,g,r) \geq 0\) and that \((d, g,r) \neq (2r, r+1, r)\) if \(r\) is odd.
If \(g \geq r\) and \(I(d-r+1, g-r, r,\ell,m+1)\) holds, then \(I(d,g,r,\ell,m)\) holds.
\end{lem}
\begin{proof}
Let \(C\) be a general BN-curve of degree \(d\) and genus \(g\) in \(\pp^r\) with \(g \geq r\).  Let \(u_1, v_1, \ldots, u_\ell, v_\ell\) be general points on \(C\).  Let \(R_1, \dots, R_m\) be general \((r+1)\)-secant rational curves of degree \(r-1\).
The statement \(I(d, g, r, \ell, m)\) asserts that
\[N_C[\posmodalong R_1 + \cdots + R_m][u_1 \biposmod v_1][u_2 \biposmod v_2] \cdots [u_\ell \biposmod v_\ell]\]
satisfies interpolation.  Combining the assumptions that \(\rho(d,g,r)\geq 0\) and \(g \geq r\), we see that
\[d \geq r + \frac{rg}{r+1} \geq r + \frac{r^2}{r + 1} > 2r - 1.\]
We may therefore prove \(I(d, g, r, \ell, m)\) by peeling off an additional \((r+1)\)-secant rational curve \(R_{m+1}\) of degree \(r-1\).  That is, we specialize the curve \(C\) as in Lemma \ref{lem:onionBN} to the union of a general BN-curve \(C'\) of degree \(d-r+1\) and genus \(g-r\) and a rational curve \(R_{m+1}\), in such a way that the points of \(C \cap (R_1 \cup \cdots \cup R_m)\) and the \(u_i\) and \(v_i\) specialize onto \(C'\).
By virtue of specializing the auxilliary points onto \(C'\), we have 
\[\big(N_{C'\cup R_{m+1}}[\posmodalong R_1 + \cdots + R_m][u_1 \biposmod v_1][u_2 \biposmod v_2] \cdots [u_\ell \biposmod v_\ell]\big)\big|_{R_{m+1}} \simeq N_{C' \cup R_{m+1}}|_{R_{m+1}}.
\]
Since we assume that \((d,g,r)\neq (2r, r+1, r)\) if \(r\) is odd, Lemma \ref{lem:rest_onion_bal} implies that \(N_{C' \cup R_{m+1}}|_{R_{m+1}}\) is perfectly balanced.
By Lemma \ref{lem:interpolation_rational_bal}, it suffices to prove that \(N_{C' \cup R_{m+1}}|_{C'}\) satisfies interpolation.  This restriction is
\[N_{C'}[ \posmodalong R_1  + \cdots + R_{m+1}][u_1 \biposmod v_1][u_2 \biposmod v_2] \cdots [u_\ell \biposmod v_\ell],\]
which satisfies interpolation by our assumption that \(I(d-r+1, g-r, r,\ell,m+1)\) holds.
\end{proof}

\begin{prop}\label{prop:I_implies_interpolation}
Suppose that \(I(d,g,r, 0, m)\) holds for all good \((d,g,r,0,m)\).
Then \(I(d, g, r, 0, 0)\) holds whenever \(\rho(d, g, r) \geq 0\), except if
\begin{itemize}
\item \((d,g,r)\) is in the list \eqref{eq:counter_examples}, or
\item \((d,g,r)= (2r, r+1, r)\) and \(r\) is odd, or
\item \(g=0\) and \(d\not\equiv 1 \pmod{r-1}\).
\end{itemize}
\end{prop}
\begin{proof}
When \(g = 0\), the tuple \((d, g, r, 0, 0)\) is good when \(d \equiv 1\) mod \(r - 1\),
so the result is a tautology. We therefore suppose \(g \geq 1\).

We will prove by induction on \(g\) that \(I(d,g,r,0,m)\) holds for \(g \geq 1\)
subject to the conditions that \(m \leq \rho(d, g, r)\)
and \((d, g, r, 0, m)\) is not in the list \eqref{XX}.
(If \((d, g, r, 0, 0)\) is in the list \eqref{XX}, then \((d, g, r)\) is in the list \eqref{eq:counter_examples},
so this is sufficient.)
Our base cases will be \(g \leq r\); in these cases, \(d \geq g + r\) by Remark \ref{rem:small_g},
and so \((d, g, r, 0, m)\) is good.
For the inductive step, we apply
Lemma \ref{lem:pull_off_onion} to reduce from
\(I(d,g,r,0,m)\) to \(I(d-r+1, g-r, r, 0, m + 1)\).
\end{proof}

\section{\texorpdfstring{The family with \(\delta = 1\) and \(\ell = m = 0\)}{The family with delta = 1 and l = m = 0}}\label{sec:delta1}

In this section, we establish \(I(d, g, r, 0, 0)\) for good tuples with \(\delta = 1\).
When \(\ell = m = 0\), the condition \(\delta = 1\) is equivalent to
\begin{equation} \label{dlm}
2d + 2g = 3r - 1,
\end{equation}
and \(I(d, g, r, 0, 0)\) asserts interpolation for \(N_C\) (with no modifications).
Our argument will be by induction on \(d - g - r\).

When \(d - g - r = 0\), then by \eqref{dlm}, we have
\((d, g, r) = (5g + 1, g, 4g + 1)\).
We may therefore conclude interpolation by \cite[Lemma 11.3]{aly}.

To complete the inductive step, observe that if \(g = 0\), then \(d = (3r - 1)/2\) by~\eqref{dlm},
and so such tuples \((d, 0, r, 0, 0)\) are never good. It thus suffices to prove the following proposition
in the case \(g > 0\).
(The \(g = 0\) case is included here too since it will be useful later,
when establishing Theorem~\ref{thm:main} for rational curves,
and it will follow via the same argument.)

\begin{prop}\label{prop:delta1}
Suppose that \(\ell = m = 0 \), and \(d > g + r\), and \eqref{dlm} holds.
If \(I(d-3, g, r-2, 0, 0)\) holds, and \(g>0\) or the characteristic is not \(2\), then \(I(d,g,r,0,0)\) also holds.
\end{prop}

\noindent
In order to prove this proposition, we first establish the following lemmas.

\begin{lem}\label{lem:pointing_tangent}
Let \(L\) be a line meeting \(C\) quasitransversely at a smooth point \(x\).
For any points \(y, y' \in L \setminus x\), the sections of \(\pp N_C\) corresponding to \(N_{C \to y}\) and \(N_{C \to y'}\) are tangent over \(x\).
\end{lem}
\begin{proof}
We prove this by a calculation in local coordinates.  We may choose an affine neighborhood of \(x\) in \(\pp^r\),
and a local coordinate \(t\)
on \(C\), so that \(x = C(0) = 0\), and
\(C\) is given parametrically by the power series \(C(t) = tC_1 + t^2C_2 + O(t^3)\).
By assumption \(y' = a \cdot y\) for some invertible scalar \(a\).
It suffices to show that the three vectors \(C(t)-y\), and \(C(t) - y'\), and \(C'(t)\), are dependent mod~\(t^2\).
The explicit dependence is
\[-a(C(t) - y) + (C(t) - y') + t(1-a)C'(t) \equiv -a (tC_1 - y) + (tC_1 - y')  + t(1-a)C_1 = 0 \pmod{t^2}. \qedhere\]
\end{proof}

\begin{lem}\label{lem:char2_line}
Let \(u\), \(v\), and \(x\) be general points on a BN-curve \(C\) of degree \(d\) and genus \(g\) with
\begin{equation}\label{eq:Ray_assump} 2(d+1) + 2g = 3r - 1 \qand d \geq g+r.  \end{equation}
Let \(y\) be a general point on the \(2\)-secant line \(\bar{uv}\).  Then the bundle
\[\left(N_{C \to u} \oplus N_{C \to v} \right)[2x \posmod y]\]
satisfies interpolation if and only if \(g > 0\) or the characteristic is not \(2\). 
\end{lem}
\begin{proof}
Subtracting the first equation in \eqref{eq:Ray_assump} from \(6\) times the second
implies \(2d + 4 - 2g \geq 2g\). By Lemma~\ref{lem:int_large_degree},
interpolation for \(\left(N_{C \to u} \oplus N_{C \to v} \right)[2x \posmod y]\) is thus
equivalent to interpolation for
\[N \colonequals \left(N_{C \to u} \oplus N_{C \to v} \right)[2x \negmod y].\]

By Lemma \ref{lem:pointing_tangent},
the subbundles \(N_{C \to y}\) and \(N_{C \to v}\) agree to second order at \(u\).
Therefore, the composition
\(N_{C \to y} \to N_{C \to u} \oplus N_{C \to v} \to N_{C \to u}\)
vanishes to order \(2\) at \(u\).  Similarly, \(N_{C \to y} \to N_{C \to v}\) vanishes to order \(2\) at \(v\).
Therefore, under the isomorphisms
\begin{equation} \label{Cypq}
N_{C\to y} \simeq \O_C(1), \quad N_{C \to u}\simeq \O_C(1)(2u), \qand N_{C \to v}\simeq \O_C(1)(2v),
\end{equation}
the map 
\(N_{C \to y} \to N_{C \to u} \oplus N_{C \to v}\)
is a diagonal inclusion
\[\O_C(1) \to \O_C(1)(2u) \oplus \O_C(1)(2v)\]
given by constant sections of \(\O_C(2u)\) and \(\O_C(2v)\) that vanish at \(2u\) and \(2v\) respectively.
Such a section is indexed by two nonzero constants \(c\) and \(d\).  
In other words,
\[N \simeq \left(\O_C(1)(2u) \oplus \O_C(1)(2v) \right)[2x \negmod \O_C(1)],\]
where \(O_C(1)\) is the diagonal subbundle identified above.

Since \(\mu(N)=d + 1\), it satisfies interpolation if and only if it has
no cohomology when twisted down by a general line bundle of degree \(d + 2 - g\).
Write such a line bundle as \(L^\vee(1)(2u+2v)\), where \(L\) is a general line bundle of degree \(g+2\).  We therefore want
\begin{equation}\label{eq:noh0}
N \otimes L(-1)(-2u-2v) = \left(L(-2v) \oplus L(-2u)\right)[2x \posmod L(-2u-2v)]
\end{equation}
to have no global sections, where the diagonal subbundle \(L(-2u-2v)\) is indexed by \([c:d] \in \mathbb{G}_m\) as above.

As \(L\) is a general line bundle of degree \(g+2\) and \(u\) and \(v\) are general points, \(h^0(L(-2v)) = h^0(L(-2u)) = 1\); write \(\sigma\) and \(\tau\) for the unique (up to scaling)
sections of \(L\) vanishing to order \(2\) at \(u\) and \(v\) respectively.
Every section of \(L(-2u) \oplus L(-2v)\) is a linear combination \(a \sigma \oplus b \tau\), viewed as a section of \(L \oplus L\).
Such a global section comes from the subsheaf \(\left(L(-2v) \oplus L(-2u)\right)[2x \posmod L(-2u-2v)]\) when it is dependent with the constant diagonal section \(c \oplus d\) at \(2x\), i.e., 
when the section \(ad \sigma - bc \tau\) of \(L\) vanishes at \(2x\).
Hence, \eqref{eq:noh0} has no cohomology if \(x\) is not a ramification point of the map
\(\varphi \colon C \to \pp^1\) determined by \(\langle \sigma, \tau \rangle \subseteq H^0(L)\).  

As \(x\) was a general point, this holds if and only if \(\varphi\) is separable.  If \(\varphi\) is not separable, then the characteristic \(p\) of the ground field is positive and \(\varphi\) factors through the Frobenius morphism \(F\):
\[C \xrightarrow{F} C' \to \pp^1.\]
In this case, \(L\) and the linear system \(\langle \sigma, \tau \rangle\)
are necessarily pulled back under \(F\). In other words,
\(L \simeq F^*M\)
for a line bundle \(M\) (necessarily general because \(L\) is general),
of degree \((g+2)/p\) with \(h^0(M) \geq 2\).  Therefore
\[\frac{g + 2}{p} + 1 - g = h^0(M) \geq 2,\]
or upon rearrangement,
\[p \leq \frac{g+2}{g+1}.\]
Thus \(g=0\) and \(p=2\).

Conversely, when \(g = 0\), there is a choice of coordinates \([t : s]\) on \(C\) so that
\(\langle \sigma, \tau \rangle = \langle t^2, s^2 \rangle\).
If in addition \(p = 2\), then the map \(\varphi\) is inseparable.
\end{proof}

\begin{proof}[Proof of Proposition~\ref{prop:delta1}]
Specialize \(C\) to the union \(C' \cup L\), where \(C'\) is a general BN-curve
of degree \(d - 1\) and genus \(g\) in \(\pp^r\), and \(L\) is a \(1\)-secant line \(\bar{xy}\) meeting \(C'\) at \(x\).
It suffices to show interpolation for
\[N_{C'}[2x \posmod y].\]
Let \(u,v \in C'\) be general points, and specialize \(y\) to a general point on the line \(\bar{uv}\).
Projection from \(\bar{uv}\) induces a pointing bundle exact sequence
\begin{equation}\label{eq:proj_L}
0 \to \left(N_{C' \to u} \oplus N_{C' \to v}\right) [2x \posmod y] \to N_{C'}[2x \posmod y] \to N_{\bar{C'}}(u+v) \to 0.
\end{equation}
By Lemma~\ref{lem:char2_line}, the subbundle \(\left(N_{C' \to u} \oplus N_{C' \to v}\right) [2x \posmod y]\) satisfies interpolation.
By hypothesis, \(N_{\bar{C'}}\), and hence the quotient \(N_{\bar{C'}}(u+v)\), also satisfies interpolation.
Finally, by \eqref{dlm},
\[\mu\left(\left(N_{C' \to u} \oplus N_{C' \to v}\right) [2x \posmod y]\right) = d + 2 = \frac{(r - 1)d + 2g - r - 5}{r - 3} = \mu(N_{\bar{C'}}(u+v)).\]
Thus \(N_{C'}[2x \posmod y]\) satisfies interpolation by Lemma~\ref{lem:intp_ses}.
\end{proof}

\section{Specializations of the \(R_i\)}\label{sec:onion_specialization}

In this section, for some integers \(n=n_i\), we construct a specialization of one of the rational curves \(R = R_i\) so that exactly \(n\) modifications point towards a point \(p\) on \(C\).  We then show that this specialization plays well with projection from \(p\).

\subsection{Setup}
Let \(n\) be an integer satisfying \(0 \leq n \leq r - 1\) and \(n \equiv r - 1\) mod \(2\).
Let \(p, q_1, q_2, \ldots, q_{r - 1} \in C\) be points such that \(2p + q_1 + \cdots + q_{r - 1}\)
lies in a hyperplane.
For \(i\) between \(1\) and \(n\), write \(L_i\) for the line joining \(p\) and \(q_i\).
For \(j\) from \(1\) to \(n' \colonequals (r - 1 - n)/2\), let \(Q_j\) be a plane conic passing through \(p\), \(q_{n + 2j - 1}\), and \(q_{n + 2j}\).
The following diagram illustrates the \(L_i\) and \(Q_j\):
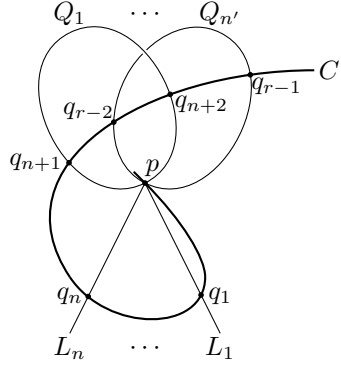
\begin{center}
\begin{tikzpicture}
\draw (0, 0) -- (-1, -2);
\draw (0, 0) -- (1, -2);
\draw (0, 0) .. controls (1, -0.5) and (2, 1.5) .. (1, 2);
\draw (0, 0) .. controls (-0.6, 0.3) and (-0.5, 1.25) .. (-0.03, 1.72);
\draw (0, 0) .. controls (-1, -0.5) and (-2, 1.5) .. (-1, 2);
\draw (0.02, 1.77) .. controls (0.2, 1.95) and (0.6, 2.2) .. (1, 2);
\draw (0, 0) .. controls (1, 0.5) and (0, 2.5) .. (-1, 2);
\draw[thick] (0.8, -1.25) .. controls (0.8, -2) and (-0.5, -2) .. (-1, -1.25);
\draw[thick] (-1, -1.25) .. controls (-2, 0.25) and (0, 1.5) .. (2.25, 1.5);
\draw[thick] (0, 0) .. controls (0.5, -0.5) and (0.8, -1) .. (0.8, -1.25);
\draw[thick] (0, 0) -- (-0.15, 0.15);
\draw (2.45, 1.5) node{\(C\)};
\draw (0.1, 0.2) node{\(p\)};
\draw (1, -2.2) node{\(L_1\)};
\draw (-1, -2.2) node{\(L_n\)};
\draw (0, -2.2) node{\(\cdots\)};
\draw (-1, 2.25) node{\(Q_1\)};
\draw (0, 2.25) node{\(\cdots\)};
\draw (1, 2.25) node{\(Q_{n'}\)};
\draw (1, -1.5) node{\(q_1\)};
\draw (-1, -1.5) node{\(q_n\)};
\draw (-1.45, 0.27) node{\(q_{n + 1}\)};
\draw (0.75, 1.05) node{\(q_{n + 2}\)};
\draw (-0.75, 0.95) node{\(q_{r - 2}\)};
\draw (1.75, 1.29) node{\(q_{r - 1}\)};
\filldraw (0, 0) circle[radius=0.03];
\filldraw (0.745, -1.49) circle[radius=0.03];
\filldraw (-0.755, -1.51) circle[radius=0.03];
\filldraw (-1.01, 0.27) circle[radius=0.03];
\filldraw (-0.415, 0.81) circle[radius=0.03];
\filldraw (0.335, 1.18) circle[radius=0.03];
\filldraw (1.4, 1.44) circle[radius=0.03];
\end{tikzpicture}
\end{center}
Define
\[R^\circ \colonequals L_1 \cup L_2 \cup \cdots \cup L_n \cup Q_1 \cup Q_2 \cup \cdots \cup Q_{n'}.\]

We will study when \(R^\circ\) is a limit of \((r + 1)\)-secant rational normal curves \(R^t\) in hyperplanes
that are \((r + 1)\)-secant to \(C\), in such a way that exactly two points of secancy limit together to \(p\)
while the remaining points of secancy limit to \(q_1, q_2, \ldots, q_{r - 1}\).
For this, it is evidently necessary to have a containment of Zariski tangent spaces \(T_p C \subset T_p R^\circ\).
In what follows, we will show that this condition is sufficient.

Suppose that \(T_p C \subset T_p R^\circ\)
and \(m \colonequals n + n' = (r - 1 + n)/2 \geq 3\).
Then the tangent line to \(p\) at~\(L_i\) (respectively \(Q_j\))
gives a distinguished point \(a_i\) (respectively \(b_j\))
in \(\Lambda \colonequals\pp(T_p R^\circ / T_p C) \simeq \pp^{m - 2} \subset \pp N_C|_p\).
Write
\(\Gamma = \{a_1, \ldots, a_n, b_1, \ldots, b_{n'}\}\),
which is a collection of \(m\) points in \(\Lambda\).
Let \(T \subset \Lambda\)
be a general rational normal curve in \(\Lambda\) passing through \(\Gamma\),
and let \(M\) be a general \(2\)-secant line to \(T\).
Our argument will furthermore show that we can choose the family \(R^t\)
so that the modifications along \(R^t\)
at the points approaching \(p\) limit to \(M\), i.e., so that
\(N_C[\posmodalong R^t]\) fits into a flat family whose central fiber is
\(N_C[q_1 + \cdots + q_{r - 1} \posmodalong R^\circ][p \posmod M]\).

\subsection{The Construction}
To construct the desired family \(R^t\), let \(B\) be the spectrum of a DVR,
with special point \(t = 0\).
Consider the blowup of \(\pp^r \times B\) along \(C \times 0\).
The special fiber \(X\) over \(0\) has two components: The first is isomorphic to the blowup \(\Bl_C \pp^r\),
and contains the proper transform \(\hat{R}^\circ = \hat{L}_1 \cup \cdots \cup \hat{L}_n \cup \hat{Q}_1 \cup \cdots \cup \hat{Q}_{n'}\) of \(R^\circ\).
The second is isomorphic to the normal cone \(\pp(N_C \oplus \O_C)\),
and contains the special fiber of the proper transform of \(C \times B\),
which coincides with \(\pp \O_C \subset \pp(N_C \oplus \O_C)\) and is isomorphic to \(C\).
The two components meet along \(\pp N_C\), and the intersection \(\hat{R}^\circ \cap \pp N_C\)
is the finite set of points \(\Gamma \cup \Gamma'\), where \(\Gamma = \{a_1, \ldots, a_n, b_1, \ldots, b_{n'}\}\) lies in the fiber over \(p\),
and \(\Gamma' = \{q_1', q_2', \ldots, q_{r - 1}'\}\) contains one point \(q_i'\) in the fiber over each \(q_i\).
The following diagram illustrates the central fiber \(X\):

\begin{center}
\begin{tikzpicture}[scale=1.3]
\draw (-0.644, 2.576)--(0.857, 3.428);
\draw (-0.644, 2.576) circle[radius=0.03];
\draw (0.857, 3.428) circle[radius=0.03];
\draw (-0.5, 1) -- (-1.5, 2) -- (0.5, 4) -- (1.5, 3) -- (-0.5, 1);
\draw[dashed] (-0.5, 1) -- (0, 0);
\draw[dashed] (-1.5, 2) -- (0, 0);
\draw[dashed] (1.5, 3) -- (0, 0);
\draw (-4, 4) -- (2, 4) -- (1, 1) -- (-5, 1) -- (-4, 4);
\draw (-4, 4) -- (-4, -0.5);
\draw[dashed] (-4, -0.5) -- (-4, -1.25);
\draw (2, 4) -- (2, -0.5);
\draw[dashed] (2, -0.5) -- (2, -1.25);
\draw (1, 1) -- (1, -1.5);
\draw[dashed] (1, -1.5) -- (1, -2.25);
\draw (-5, 1) -- (-5, -1.5);
\draw[dashed] (-5, -1.5) -- (-5, -2.25);
\draw[thick] (-4.5, 0) -- (1.5, 0);
\draw (-4, 4) -- (-4.3, 5);
\draw[dashed] (-4.3, 5) -- (-4.6, 6);
\draw (2, 4) -- (2.9, 5);
\draw[dashed] (2.9, 5) -- (3.8, 6);
\draw (1, 1) -- (2.75, 2.0);
\draw[dashed] (2.75, 2.0) -- (3.625, 2.5);
\draw (-5, 1) -- (-5.9375, 1.75);
\draw[dashed] (-5.9375, 1.75) -- (-6.875, 2.5);
\draw (-0.9, 1.6) .. controls (-1.4, 2.1) and (0.4, 3.9) .. (0.9, 3.4);
\draw[white, line width = 3pt] (-0.445, 1.55) .. controls (-1.5, 5.4) and (-2, 6) .. (-3.5, 3.5);
\draw[thick] (-0.445, 1.55) .. controls (-1.5, 5.4) and (-2, 6) .. (-3.5, 3.5);
\draw[white, line width = 3pt] (-0.965, 1.99) .. controls (-1.5, 5.5) and (-2.5, 6) .. (-2.5, 2);
\draw[thick] (-0.965, 1.99) .. controls (-1.5, 5.5) and (-2.5, 6) .. (-2.5, 2);
\draw[white, line width = 3pt] (0.415, 2.17) -- (-0.25, 5);
\draw[thick] (0.415, 2.17) -- (-0.1835, 4.717);
\draw[thick, dotted] (-0.1835, 4.717) -- (-0.25, 5);
\draw[white, line width = 3pt] (0.94, 2.92) .. controls (2, 4) and (1.5, 4.5) .. (1.2, 4.8);
\draw[thick] (0.94, 2.92) .. controls (2, 4) and (1.5, 4.5) .. (1.2, 4.8);
\draw[thick, dotted] (1.2, 4.8) -- (1, 5);
\draw (1.6, 4.6) node{\(\hat{Q}_j\)};
\draw (0.03, 4.6) node{\(\hat{L}_i\)};
\filldraw (-3.5, 3.5) circle[radius=0.03];
\filldraw (-2.5, 2) circle[radius=0.03];
\draw[densely dotted] (-3.5, 3.5) -- (-3.5, 0);
\draw[densely dotted] (-2.5, 2) -- (-2.5, 0);
\draw (-3.65, 3.5) node{\(q_1'\)};
\draw (-2.2, 2) node{\(q_{r-1}'\)};
\filldraw (-3.5, 0) circle[radius=0.03];
\filldraw (-2.5, 0) circle[radius=0.03];
\filldraw (-0.965, 1.99) circle[radius=0.03];
\filldraw (-0.445, 1.55) circle[radius=0.03];
\filldraw (0.415, 2.17) circle[radius=0.03];
\filldraw (0.94, 2.92) circle[radius=0.03];
\draw (0.27, 2.21) node{\(a_i\)};
\draw (1.01, 2.74) node{\(b_j\)};
\filldraw (0, 0) circle[radius=0.03];
\draw[densely dotted] (-1.4, 2) .. controls (-0.5, 2) and (-0.5, 2) .. (0, 0);
\draw[densely dotted] (1.4, 3) .. controls (0.2, 3) and (0.5, 2) .. (0, 0);
\draw[densely dotted] (0, 0) .. controls (-0.25, -1) and (0.25, -1) .. (0, 0);
\draw (-0.9, 1.6) .. controls (-0.4, 1.1) and (1.4, 2.9) .. (0.9, 3.4);
\draw [decorate, decoration={brace, amplitude=0.75ex}] (-2.5, 5.1) -- (1.1, 5.1);
\draw (-0.65, 5.33) node{\(\hat{R}^\circ\)};
\draw (-1.75, 3) node{\(\pp N_C\)};
\draw (-1.75, 6) node{\(\Bl_C \pp^r\)};
\draw (-1.75, -2) node{\(\pp(N_C \oplus \O_C)\)};
\draw (1.5, -0.17) node{\(C\)};
\draw (0.5, 3.6) node{\(T\)};
\draw (-0.2, 2.65) node{\(M\)};
\draw (0.15, -0.15) node{\(p\)};
\draw (0, -0.9) node{\(f\)};
\draw (-3.5, -0.15) node{\(q_1\)};
\draw (-2.5, -0.15) node{\(q_{r - 1}\)};
\draw (-3, -0.15) node{\(\cdots\)};
\draw (-2.23, 0.5) node{\(\ell_{r-1}\)};
\draw (-3.65, 0.5) node{\(\ell_1\)};
\draw (-0.2, 3.5) node{\(\Lambda\)};
\end{tikzpicture}
\end{center}

Let \(\ell_k \subset \pp(N_C \oplus \O_C)|_{q_k}\) denote the line joining \(q_k\) to \(q_k'\).
These lines are pictured as dotted vertical lines in the above diagram.

Write \(\Delta = M \cap T\).
The linear series \(V \colonequals H^0(\O_T(1)) \oplus H^0(\O_T(1)(\Delta - \Gamma)) \subset H^0(\O_T(1)(\Delta))\)
defines a map \(\pp^1 \simeq T \to \pp V \simeq \pp^{m-1}\) of degree \(m\),
which identifies the two points of \(\Delta\) to a common point.
Fix an embedding \(\pp V \hookrightarrow \pp(N_C \oplus \O_C)\),
which agrees with the identification \(\pp H^0(\O_T(1)) \simeq \Lambda\),
and sends this common point to \(p \colonequals \pp \O_C|_p \in \pp(N_C \oplus \O_C)|_p\).
Composing these maps, we obtain a map \(f \colon \pp^1 \simeq T \to \pp(N_C \oplus \O_C)\),
which is pictured as the dotted curve in the above diagram.
By construction, \(f\) passes through \(\Gamma\) and is nodal at \(p\).
Moreover, composing projection from \(p\) with \(f\) is the identity on \(T\),
and projection from \(p\) sends the Zariski tangent space of the image of \(f\) at \(p\) to \(M\).

We now glue \(\hat{R}^\circ\) to \(f\) and the \(\ell_i\), i.e., we
consider the map \(F \colon \hat{R}^\circ \cup T \cup \ell_1 \cup \cdots \cup \ell_{r - 1} \to X\)
defined by the natural inclusions on \(\hat{R}^\circ\) and the \(\ell_i\),
and by \(f\) on \(T\).
To complete the argument, it suffices to deform
\(F\) to the general fiber
in such a way that preserves its incidence to \(C\).
To check that this is possible, we just need to check that the corresponding
obstruction space vanishes, i.e., that \(H^1(N_F[\negmodalong C]) = 0\).

\subsection{\boldmath The Normal Space \(\pp N_C|_p\)}
One tool that we will use --- both to show in the next section that
\(H^1(N_F[\negmodalong C]) = 0\), and in the following section to analyze
the transformation \([p \posmod M]\) --- is the natural identification of \(\pp N_C|_p\)
with the projection of \(\pp^r\) from \(T_p C\).
Under this projection, \(\bar{q}_1, \bar{q}_2, \ldots, \bar{q}_{r - 1}\)
are a collection of \(r - 1\) points, which are general
subject to the condition of being contained in a hyperplane \(H\).
The conics \(Q_j\) project to the lines through \(\bar{q}_{n + 2j - 1}\) and \(\bar{q}_{n + 2j}\).
The image \(\bar{p}\) of \(p\) is identified with the osculating
\(2\)-plane, which coincides with \(\pp N_{C\to p}|_p \in \pp N_C|_p\).
Under this identification, the points \(a_i\) are identified with \(\bar{q}_i\),
and the points \(b_j\) lie on \(\bar{Q}_j\).
The following diagram illustrates this setup. 
\begin{center}
\begin{tikzpicture}[scale=1.1]
\draw[densely dotted] (3, 0) .. controls (3, -1) and (3, -3) .. (0, -3);
\draw[densely dotted] (0, -3) .. controls (-3, -3) and (-3, 0) .. (0, 0);
\draw[densely dotted] (0, 0) .. controls (2.25, 0) and (2.25, -2) .. (0, -2);
\draw[densely dotted] (0, -2) .. controls (-1.5, -2) and (-3, -1) .. (-3, 0);
\draw (-3, -1) -- (4.5, -1) -- (3, -3.5) -- (-4.5, -3.5) -- (-3, -1);
\filldraw (0, 0) circle[radius=0.02];
\filldraw (-1.125, -1.82) circle[radius=0.02];
\filldraw (-1.875, -2.37) circle[radius=0.02];
\filldraw (1.875, -2.6) circle[radius=0.02];
\filldraw (1.125, -1.77) circle[radius=0.02];
\draw (-1.125, -1.82) -- (-1.875, -2.37);
\draw (1.875, -2.6) -- (1.125, -1.77);
\draw (-1.77, -2.05) node{\(b_1\)};
\draw (1.85, -2.25) node{\(b_{n'}\)};
\filldraw (-1.625, -2.18666) circle[radius=0.02];
\filldraw (1.625, -2.32333) circle[radius=0.02];
\filldraw (-0.75, -2.93) circle[radius=0.02];
\filldraw (0.75, -2.95) circle[radius=0.02];
\draw (0, -0.17) node{\(\bar{p}\)};
\draw (-0.75, -3.15) node{\(\bar{q}_1 = a_1\)};
\draw (0.75, -3.15) node{\(\bar{q}_n = a_n\)};
\draw (0, -3.15) node{\(\cdots\)};
\draw (-2.1, -2.45) node{\(\bar{q}_{n+1}\)};
\draw (2.12, -2.8) node{\(\bar{q}_{r - 2}\)};
\draw (-1, -1.55) node{\(\bar{q}_{n + 2}\)};
\draw (0.9, -1.55) node{\(\bar{q}_{r - 1}\)};
\draw (0, -1.55) node{\(\cdots\)};
\draw (-3.15, -0.1) node{\(\bar{C}\)};
\draw (-1.1, -2.15) node{\(\bar{Q}_1\)};
\draw (1.25, -2.15) node{\(\bar{Q}_{n'}\)};
\draw (4.65, -1.1) node{\(H\)};
\end{tikzpicture}
\end{center}

\subsection{\texorpdfstring{\boldmath Vanishing of \(H^1(N_F[\negmodalong C])\)}{Vanishing of H1(NF[-> C])}}

Given a vector bundle \(E\) on a reducible curve \(X \cup_{\Gamma} Y\), recall that the \defi{Mayer--Vietoris} sequence is
\[0 \to E \to E|_X \oplus E|_Y \to E|_{\Gamma} \to 0.\]
Applying this to \(E = N_F[\negmodalong C]\), we obtain
\begin{equation} \label{mv}
0 \to N_F[\negmodalong C] \to N_{\hat{R}^\circ / \Bl_C \pp^r} \oplus N_f[\negmodalong C] \oplus \bigoplus_{k = 1}^{r - 1} N_{\ell_k/\pp(N_C \oplus \O_C)}[\negmodalong C] \to T_{\pp N_C}|_{\Gamma \cup \Gamma'} \to 0.
\end{equation}
For each of the direct summands in the middle term, we both show that \(H^1\) vanishes,
and extract information about the image of its global sections in \(T_{\pp N_C}|_{\Gamma \cup \Gamma'}\).
We then combine this information to show that the rightmost map is surjective on global sections.

\begin{lem} \label{lem:Nell} We have
\(H^1(N_{\ell_k/\pp(N_C \oplus \O_C)}[\negmodalong C]) = 0\), and
\(H^0(N_{\ell_k/\pp(N_C \oplus \O_C)}[\negmodalong C])\) surjects onto \(T_{\pp N_C}|_{q_k'}\).
\end{lem}
\begin{proof} Both statements follow from
\(H^1(N_{\ell_k/\pp(N_C \oplus \O_C)}[\negmodalong C](-q_k')) = 0\), which can in turn
be deduced from
\[0 \to N_{\ell_k / \pp(N_C \oplus \O_C)|_{q_k}}(-q_k - q_k') \to N_{\ell_k/\pp(N_C \oplus \O_C)}[\negmodalong C](-q_k') \to N_{\pp(N_C \oplus \O_C)|_{q_k} / \pp(N_C \oplus \O_C)}|_{\ell_k} (-q_k') \to 0,\]
using the isomorphisms
\[N_{\ell_k / \pp(N_C \oplus \O_C)|_{q_k}} \simeq \O_{\ell_k}(1)^{r-1} \quad \text{and} \quad N_{\pp(N_C \oplus \O_C)|_{q_k} / \pp(N_C \oplus \O_C)} \simeq \O_{\pp(N_C \oplus \O_C)|_{q_k}}. \qedhere\]
\end{proof}

\begin{lem} \label{lem:Nf}
We have \(H^1(N_f[\negmodalong C]) = 0\). Moreover, the image \(H^0(N_f[\negmodalong C]) \to T_{\pp N_C}|_\Gamma\)
consists of those deformations of \(\Gamma\) that can be lifted to deformations of \(\Lambda\), i.e., this image
coincides with the full preimage in \(T_{\pp N_C}|_\Gamma\) of the image of \(H^0(N_{\Lambda/\pp N_C}) \to N_{\Lambda/\pp N_C}|_\Gamma\).
\end{lem}
\begin{proof} 
The map \(\pi_p \circ f \colon T \to T\) is the identity map, so the pointing bundle exact sequence yields a surjection
\[N_f[\negmodalong C] \to N_T,\]
which we may further compose with the surjection \(N_T \to N_{\Lambda/ \pp N_C}|_T\) coming from the inclusion \(T \subset \Lambda\).  Define \(K\) via the exact sequence
\[0 \to K \to N_f[\negmodalong C] \to N_{\Lambda / \pp N_C}|_T \to 0.\]
By considering the diagram
\begin{center}
\begin{tikzcd}[column sep=small]
&&& H^1(K(-\Gamma)) \arrow[d] \\
H^0(K) \arrow[r] \arrow[d]& H^0(N_f[\negmodalong C]) \arrow[r] \arrow[d]& H^0(N_{\Lambda/\pp N_C}|_T) \arrow[r] \arrow[d]& H^1(K) \arrow[r] \arrow[d]& H^1(N_f[\negmodalong C]) \arrow[r] & H^1(N_{\Lambda/\pp N_C} |_T)\\
K|_\Gamma \arrow[r]\arrow[d] & T_{\pp N_C}|_\Gamma \arrow[r] & N_{\Lambda/\pp N_C}|_\Gamma \arrow[r] & 0 \\
H^1(K(-\Gamma))
\end{tikzcd}
\end{center}
and noting that \(H^0(N_\Lambda) \simeq H^0(N_\Lambda|_T)\), it
 suffices to show that
\[H^1(K(-\Gamma)) = H^1(N_{\Lambda / \pp N_C}|_T) = 0.\]
The first vanishing statement follows from the pointing bundle sequence
\[0 \to N_{f \to C}(-\Delta) \simeq \O_T(\Delta + \Gamma) \to K \to N_{T/\Lambda} \simeq \O_T(\Gamma)^{\oplus (m - 3)} \to 0.\]
The second vanishing statement follows from the sequence:
\[0 \to N_{\Lambda / \pp N_C|_p}|_T \simeq \O_T(1)^{r - m} \to N_{\Lambda / \pp N_C}|_T \to N_{\pp N_C|_p / \pp N_C}|_T \simeq \O_T \to 0. \qedhere\]
\end{proof}

\noindent
We finally consider the bundle
\[N_{\hat{R}^\circ / \Bl_C \pp^r} \simeq \bigoplus_{i = 1}^n N_{\hat{L}_i / \Bl_C \pp^r} \oplus \bigoplus_{j = 1}^{n'} N_{\hat{Q}_j / \Bl_C \pp^r}.\]
To describe the images of
\[H^0(N_{\hat{L}_i / \Bl_C \pp^r}) \to T_{\pp N_C}|_{a_i} \qand H^0(N_{\hat{Q}_j / \Bl_C \pp^r}) \to T_{\pp N_C}|_{b_j},\]
we define
\[A_i = T_{q_i} C \qand B_j = \langle T_{q_{n+2j - 1}} C, T_{q_{n+2j}} C\rangle,\]
for the tangent line or span of tangent lines,
and \(\bar{A}_i\) and \(\bar{B}_j\) for their projections from \(T_p C\).

\begin{lem} \label{lem:NLi} We have \(H^1(N_{\hat{L}_i / \Bl_C \pp^r}) = 0\).
Moreover the image of \(H^0(N_{\hat{L}_i / \Bl_C \pp^r}) \to T_{\pp N_C}|_{a_i}\) has the following two properties:
\begin{enumerate}
\item It surjects onto \(T_C|_{p}\).
\item\label{part2} The kernel of the map from the image to \(T_C|_{p}\) is precisely
\(T_{a_i} \bar{A}_i\).
\end{enumerate}
\end{lem}
\begin{proof}
We have an exact sequence
\[0 \to N_{L_i / \pp^r}[q_i \negmodalong C](-p) \to N_{\hat{L}_i / \Bl_C \pp^r} \to T_C|_{p} \to 0.\]
Using this, the kernel in \eqref{part2} is the image of \(H^0(N_{L_i / \pp^r}[q_i \negmodalong C](-p))\) in  \(T_{\pp N_C}|_{a_i}\).
Moreover, the normal bundle exact sequence for \(L_i\) in the span \(\bar{L_i A_i}\) gives
\[0 \to N_{L_i / \bar{L_i A_i}}(-p) \to N_{L_i / \pp^r}[q_i \negmodalong C](-p) \to N_{\bar{L_i A_i} / \pp^r}|_{L_i} (- q_i-p) \to 0,\]
and \(N_{L_i / \bar{L_i A_i}}(-p)|_p\) is identified with \(T_{a_i}\bar{A}_i\) under projection from \(T_pC\).

\vspace{5pt}

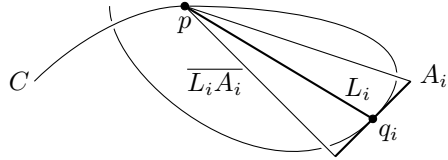
\begin{minipage}{.95\textwidth}
\begin{center}
\begin{tikzpicture}
\filldraw (1.5, -.5) circle[radius=0.05];
\draw (1.45, -.45) node [below right] {\(q_i\)};
\draw (2,-0.2) node [above right] {\({A_i}\)};
\draw (1.3,-.1) node {\({L_i}\)};
\draw (-0.6,0) node {\({\bar{L_iA_i}}\)};
\draw (1.5, -.5) .. controls (2.5, 0.5) and (1, 1) .. (-1, 1);
\draw (1.5, -.5) .. controls (0.25, -1.75) and (-2, 0) .. (-2, 1);
\draw[white, line width = 3pt] (-1.05, 1) .. controls (-1.7, 1) and (-2.5, 0.5) .. (-3, 0);
\draw (-1, 1) .. controls (-1.7, 1) and (-2.5, 0.5) .. (-3, 0);
\filldraw (-1, 1) circle[radius=0.05];
\draw (-1,0.75) node{\(p\)};
\draw[white, line width = 3pt] (0,0) -- (1, -1);
\draw[white, line width = 3pt] (2,0) -- (0.5, 0.5);
\draw (-1,1) -- (1, -1) -- (2,0) -- (-1, 1);
\draw[thick] (-1, 1) -- (1.5, -.5);
\draw[thick] (1, -1) -- (2,0);
\draw (-3.2, 0) node{\(C\)};
\end{tikzpicture}
\end{center}
\end{minipage}

\vspace{-10pt}

\noindent
The same diagram chase as in Lemma \ref{lem:Nf} implies that it suffices to show:
\[H^1(N_{L_i / \bar{L_i A_i}}(-2p)) = H^0(N_{\bar{L_i A_i} / \pp^r}|_{L_i} (- q_i-p)) = H^1(N_{\bar{L_i A_i} / \pp^r}|_{L_i} (- q_i-p)) = 0.\]
These statements follow from the isomorphisms
\[N_{L_i / \bar{L_i A_i}}(-2p) \simeq \O_{L_i}(-1) \quad \text{and} \quad N_{\bar{L_i A_i} / \pp^r}|_{L_i} (-p - q_i) \simeq \O_{L_i}(-1)^{r - 2}. \qedhere\]
\end{proof}

\begin{lem} \label{lem:NQj} We have \(H^1(N_{\hat{Q}_j / \Bl_C \pp^r}) = 0\).
Moreover the image of \(H^0(N_{\hat{Q}_j / \Bl_C \pp^r}) \to T_{\pp N_C}|_{b_j}\) has the following two properties:
\begin{enumerate}
\item It surjects onto \(T_C|_{p}\).
\item The kernel of the map from the image to \(T_C|_{p}\) is precisely
\(T_{b_j} \bar{B}_j\).
\end{enumerate}
\end{lem}
\begin{proof}  We will imitate the proof of Lemma \ref{lem:NLi}.
We have an exact sequence
\[0 \to N_{Q_j / \pp^r}[q_{n+2j-1} + q_{n+2j} \negmodalong C](-p) \to N_{\hat{Q}_j / \Bl_C \pp^r} \to T_C|_{b_j} \to 0.\]
Moreover, the normal bundle exact sequence for \(Q_j\) in the span \(\bar{Q_j B_j}\) gives
\begin{multline*}
0 \to N_{Q_j / \bar{Q_j B_j}}[q_{n+2j-1} + q_{n+2j} \negmodalong C](-p) \to N_{Q_j / \pp^r}[q_{n+2j-1} + q_{n+2j} \negmodalong C](-p) \\
\to N_{\bar{Q_j B_j} / \pp^r}|_{Q_j} (- q_{n+2j-1} - q_{n+2j}-p) \to 0,
\end{multline*}
and \(N_{Q_j / \bar{Q_j B_j}}[q_{n+2j-1} + q_{n+2j} \negmodalong C](-p)|_p\) is identified with \(T_{b_j}\bar{B}_j\) under projection from \(T_pC\).
Our goal is therefore to show both
\[H^1(N_{Q_j / \bar{Q_j B_j}}[q_{n+2j-1} + q_{n+2j} \negmodalong C](-2p)) = 0\]
and
\[H^0(N_{\bar{Q_j B_j} / \pp^r}|_{Q_j} (- q_{n+2j-1} - q_{n+2j}-p)) = H^1(N_{\bar{Q_j B_j} / \pp^r}|_{Q_j} (- q_{n+2j-1} - q_{n+2j}-p)) = 0.\]
The first vanishing statement follows from the exact sequence:
\begin{multline*}
0 \to \big[N_{Q_j/\bar{Q_j}}(-q_{n+2j-1}  - q_{n+2j} - 2p) \simeq \O_{\pp^1}\big] \to N_{Q_j / \bar{Q_j B_j}}[q_{n+2j-1} + q_{n+2j} \negmodalong C](-2p) \\
\to \big[N_{\bar{Q_j} / \bar{Q_j B_j}}|_{Q_j}[q_{n+2j-1} + q_{n+2j} \negmodalong C](-2p) \simeq \O_{\pp^1}(-1)^{\oplus 2}\big] \to 0.
\end{multline*}
The second vanishing statement follows from the isomorphism
\[N_{\bar{Q_j B_j} / \pp^r}|_{Q_j} (- q_{n+2j-1} - q_{n+2j}-p) \simeq O_{\pp^1}(-1)^{\oplus (r - 4)}. \qedhere\]
\end{proof}

Combining these lemmas, we immediately see that \(H^1\) of the middle terms in the sequence \eqref{mv}
vanish. We now see that the rightmost map of \eqref{mv} is surjective on global sections,
as follows. First we apply Lemma~\ref{lem:Nell} to handle the points of \(\Gamma'\); this reduces our problem
to showing the surjectivity of
\[H^0(N_{\hat{R}^\circ / \Bl_C \pp^r}) \oplus H^0(N_f[\negmodalong C]) \to T_{\pp N_C}|_\Gamma.\]
Applying Lemmas~\ref{lem:NLi} and~\ref{lem:NQj}, we see that the composition to \((T_C|_p)^m\) is surjective.
It thus suffices to show that the image contains the kernel of \(T_{\pp N_C}|_\Gamma \to (T_C|_p)^m\),
i.e., \(T_{\pp N_C|_p}|_\Gamma\).

Since removing any point from \(\Gamma\) yields a linearly independent collection of points, any deformation in \(\pp N_C|_p\) of all but one point of \(\Gamma\) lifts to a deformation of \(\Lambda\).
Combining Lemma~\ref{lem:Nf} with Lemma~\ref{lem:NLi}, we therefore conclude that the image contains each
\(T_{\langle \Lambda, \bar{A}_i\rangle}|_\Gamma\). Similarly, combining
Lemma~\ref{lem:Nf} with Lemma~\ref{lem:NQj}, we conclude that the image contains each
\(T_{\langle \Lambda, \bar{B}_j\rangle}|_\Gamma\).
Since the \(T_{\langle \Lambda, \bar{A}_i\rangle}|_\Gamma\) and \(T_{\langle \Lambda, \bar{B}_j\rangle}|_\Gamma\)
span \(T_{\pp N_C|_p}|_\Gamma\), the desired result follows.

\subsection{\texorpdfstring{\boldmath The transformation \([p \posmod M]\)}{The transformation [p -> M]}}
We next show that \( M\) is ``suitably generic''
in \(\pp N_C|_p\).  

\begin{lem} \label{M:unprojected} Fix a general BN-curve \(C\) and a general point \(p \in C\).
\begin{enumerate}
\item \label{n2} If \(n \geq 2\): As \(q_1, q_2, \ldots, q_{r-1}\) vary, \(M\) is linearly general in \(\pp N_C|_p\).
\item \label{n3} If \(n \geq 3\): This remains true if we fix \(q_{n + 1}, q_{n + 2}, \ldots, q_{r - 1} \in C\) to be general.
In other words, as just the remaining points \(q_1, q_2, \ldots, q_n\) vary, \(M\) is still linearly general in \(\pp N_C|_p\).
\end{enumerate}
\end{lem}
\begin{proof}
Fix \(\Lambda \subset \pp N_C|_p\) of codimension~\(2\); we want to
show that \(M\) can be disjoint from \(\Lambda\).
In first case, \(\bar{q}_1\) and \(\bar{q}_2\) are general points on \(\bar{C}\).
In the second case, since any 
\((r - 1 - n) + 2 = r + 1 - n\leq r-2\)
points in \(\pp^{r-2}\) lie in a hyperplane,
the points \(\bar{q}_1\) and \(\bar{q}_2\) remain general even as  \(q_{n + 1}, q_{n + 2}, \ldots, q_{r - 1} \in C\) are fixed.
In either case, the line between \(\bar{q}_1\) and \(\bar{q}_2\) is therefore disjoint from \(\Lambda\).
This completes the proof because
\(M\) can be specialized to the line between \(\bar{q}_1\) and \(\bar{q}_2\).
\end{proof}

Lemma~\ref{M:unprojected}\eqref{n3} is sharp, in the sense that the conclusion is always false if \(n = 2\)
(the subspace \(M\) is never transverse to \(\Lambda = \bar{q_3 q_4 \cdots q_{r - 1}}\)).
Nevertheless, there is a variant that does hold for \(n = 2\).
By part~\eqref{n2}, the general such \(M\) is disjoint from \(\bar{p} = \pp N_{C \to p}|_p \in \pp N_C|_p\).
We may therefore ask for the weaker conclusion that the image of \(M\) is linearly general
in the quotient \(\pp (N_C / N_{C \to p})|_p\),
i.e., that \(M\) is transverse to any \(\Lambda\) containing \(\bar{p}\).
In this case, the analog of Lemma~\ref{M:unprojected}\eqref{n3} holds apart from a single counterexample:

\begin{lem} \label{M:projected}
Suppose \(n = 2\), and fix a general BN-curve \(C\) and general points \(p, q_3, q_4, \ldots, q_{r - 1} \in C\).
If \(C\) is not an elliptic normal curve, then as \(q_1, q_2\) vary, \(M\) is linearly general in \(\pp (N_C / N_{C \to p})|_p\).
\end{lem}
\begin{proof}
By assumption, \(r-1\equiv n=2\) mod \(2\); hence \(r\) is odd.
If \(C\) is not an elliptic curve, then since \(d \geq r + 1\), either \(d \geq r + 2\)
or \((d, g) = (r + 1, 0)\).
We consider these two cases separately.

\medskip

\noindent
\textbf{\boldmath Case 1: \(d \geq r + 2\).}
Let \(\Lambda\subset \pp N_C|_p\) be any codimension \(2\) plane containing \(\bar{p}\).  We will show that \(M\) can be chosen disjoint from \(\Lambda\).
Since any \((r - 1 - n) + 1 = r - n=r-2\) points lie in a hyperplane,
\(\bar{q}_1\) is a general point on \(\bar{C}\), and is therefore not contained in \(\Lambda\).
Let \(H \simeq \pp^{r-3}\) be a general hyperplane containing \(\bar{q}_3, \bar{q}_4, \ldots, \bar{q}_{r - 1}\).
Since \(\bar{p} \notin H\) and \(\bar{p} \in \Lambda\), it follows that \(\Lambda\) is transverse to \(H\).
As \(d \geq r + 2\), the hyperplane section \(H \cap \bar{C}\) contains two points \(\{x, y\}\) distinct from
\(\bar{q}_1, \bar{q}_3, \bar{q}_4, \ldots, \bar{q}_{r - 1}\).
Since the sectional monodromy group of a general curve always contains the alternating group \cite{kadets}, the points \(\{x, y, \bar{q}_1, \bar{q}_3, \bar{q}_4, \ldots, \bar{q}_{r - 1}\}\)
are in linear general position.
For any \(k \geq 0\) there is a unique \(k\)-plane in \(\pp^{2k+2}\) meeting each of \(k+2\) lines in linear general position.  Applying this with \(k = (r-7)/2\), we see that
there is a unique \([(r - 3)/2]\)-plane
\(\Lambda_x\subset H\) containing \(\bar{q}_1\) and \(x\), and meeting each of the lines \(\bar{Q}_j\).
If \(\bar{q}_2 = x\), then by this uniqueness, \(\Lambda_x\) coincides with the projection of \(T_p R^\circ\).
Similarly define \(\Lambda_y\).
Because \(M\) can be linearly general in either \(\Lambda_x\) or \(\Lambda_y\), it suffices to show that one of \(\Lambda_x\) or \(\Lambda_y\)
contains a line disjoint from \(\Lambda\).

Note that \(\Lambda_x \cap \bar{Q}_1\) is the projection of \(x\) from \(\langle \bar{q}_1, \bar{Q}_2, \ldots, \bar{Q}_{n'}\rangle\)
onto \(\bar{Q}_1\), and similarly for \(\Lambda_y \cap \bar{Q}_1\).
It follows that \(\Lambda_x \cap \bar{Q}_1 \neq \Lambda_y \cap \bar{Q}_1\),
and thus that \(\langle \Lambda_x, \Lambda_y \rangle\) contains \(\bar{Q}_1\).
Similarly, \(\langle \Lambda_x, \Lambda_y \rangle\) contains all other \(\bar{Q}_i\).
By inspection, \(\langle \Lambda_x, \Lambda_y \rangle\) contains \(\bar{q}_1\), \(x\), and \(y\).
Therefore \(\langle \Lambda_x, \Lambda_y \rangle = H\).
In particular, \(\langle \Lambda_x, \Lambda_y \rangle\) is a distinct hyperplane from \(\langle \Lambda, \bar{q}_1\rangle\).
Without loss of generality, \(\Lambda_x\) contains a point \(z \notin \langle \Lambda, \bar{q}_1\rangle\).
Then \(\langle z, \bar{q}_1 \rangle\) gives the desired line contained in \(\Lambda_x\)
and disjoint from \(\Lambda\).

\medskip

\noindent
\textbf{\boldmath Case 2: \((d, g) = (r + 1, 0)\).}
Since \(\bar{f} \colon \pp^1 \simeq C \to \pp^{r - 2}\)
is a general rational curve of degree \(r - 1\), it suffices to verify
that \(M\) is linearly general for a particular choice of \(\bar{f}\).
We may therefore take
\[\bar{f}(t) = \left[t^2 + 1 : t : \frac{t - p_3}{t - q_3} : \frac{t - p_4}{t - q_4} : \cdots : \frac{t - p_{r - 1}}{t - q_{r - 1}}\right],\]
where \(p_i \in \pp^1\) are general.  For \(3 \leq i \leq r-1\), we have \(\bar{f}(q_i) = [0: \cdots : 0 : 1 : 0 : \cdots : 0]\), where the \(1\) occurs in the \(i\)th position.
(The interested reader may verify that this is not actually a specialization, i.e., the general
rational curve of degree \(r - 1\) in \(\pp^{r - 2}\) is of this form
after applying automorphisms of the source and target.)

Let \(H = H_s\) be a generic hyperplane passing through \(\bar{q}_3, \bar{q}_4, \ldots, \bar{q}_{r - 1}\),
defined by the ratio of the first two coordinates being equal to \(s\).
Note that \(H\) meets \(\bar{f}(\pp^1)\) at two other points \(\bar{q}_1 = \bar{f}(q_1)\) and \(\bar{q}_2 = \bar{f}(q_2)\).
The parameters \(q_1\) and \(q_2\) are the solutions of the equation \(t + t^{-1} = (t^2 + 1)/t = s\).
The projection \(\Lambda_s\) of \(T_p R^\circ\) is the
unique \([(r - 3)/2]\)-plane \(\Lambda_s\) containing \(\bar{q}_1\) and \(\bar{q}_2\) and meeting
each of the lines \(\bar{Q}_i\).
We will show that, for \(s \in \pp^1\) generic, \(\Lambda_s\)
is transverse to any fixed subspace \(\Lambda\) of codimension \(2\) containing \(\bar{p}\).  Hence a general line \(M \subseteq \Lambda_s\) is disjoint from \(\Lambda\).

To show this, we calculate \(\Lambda_s\) explicitly.
Since \(\Lambda_s\) is unique, it suffices to exhibit
a particular \([(r - 3)/2]\)-plane containing \(\bar{q}_1\) and \(\bar{q}_2\) and meeting
each of the lines \(\bar{Q}_i\). We claim that we may take:
\[\Lambda_s = \langle \alpha(s), \beta_1(s), \beta_2(s), \ldots, \beta_{(r - 3)/2}(s)\rangle,\]
where
\begin{align*}
\alpha(s) &= \left[s : 1 : \frac{p_3 q_4 s - p_3 - q_4}{q_3 q_4 - 1} : \frac{p_4 q_3 s - p_4 - q_3}{q_3 q_4 - 1} : \cdots : \frac{p_{r-2} q_{r-1} s - p_{r-2} - q_{r-1}}{q_{r-2} q_{r-1} - 1} : \frac{p_{r-1} q_{r-2} s - p_{r-1} - q_{r-2}}{q_{r-2} q_{r-1} - 1} \right] \\
\beta_i(s) &= \left[0 : 0 : \cdots : 0 : 0 : \frac{p_{2i+1}}{q_{2i + 1}} \cdot \frac{s - q_{2i+1} - p_{2i + 1}^{-1}}{s - q_{2i+1} - q_{2i+1}^{-1}} : \frac{p_{2i+2}}{q_{2i + 2}} \cdot \frac{s - q_{2i+2} - p_{2i + 2}^{-1}}{s - q_{2i+2} - q_{2i+2}^{-1}} : 0 : 0 : \cdots : 0 : 0\right].
\end{align*}
Here, the nonzero entries of \(\beta_i(s)\) occur in the \((2i + 1)\)st and \((2i + 2)\)nd
coordinates.
Indeed, \(\Lambda_s\) meets \(\bar{Q}_i\) at \(\beta_i(s)\),
so it suffices to check that \(\Lambda_s\) contains \(\bar{f}(t)\)
when \(s = t + t^{-1}\).  This follows from the following identity,
which may be verified by separately considering the first coordinate,
the second coordinate, the \((2i + 1)\)st coordinate, and the \((2i + 2)\)nd coordinate:
\[\bar{f}(t) = t \cdot \alpha(t + t^{-1}) - \sum_i \frac{(q_{2i + 1} t - 1)(q_{2i + 2} t - 1)}{q_{2i + 1} q_{2i + 2} - 1} \cdot \beta_i(t + t^{-1}).\]
This establishes that \(\Lambda_s\) is given by the above explicit formula, as claimed.

From the above explicit formulas for \(\alpha\) and the \(\beta_i\), it is evident that
\(\alpha\) is an isomorphism from \(\pp^1\) onto a line \(L\),
and the \(\beta_i\) are quadratic maps from \(\pp^1\) onto lines \(M_i\),
such that \(L, M_1, M_2, \ldots, M_{(r - 3)/2}\) are linearly independent
and span \(\pp^{r - 2}\). In fact, the above formulas for the \(\beta_i\) imply that,
up to changing coordinates on the \(M_i\), the \(\beta_i\) are \emph{independently general} quadratic maps
--- so in particular distinct (from themselves and from \(\alpha\)).
Since the image of \(\bar{f}\) does not lie in any union of proper linear subspaces,
and \(\Lambda\) must meet \(\bar{p}\) (which is a general point on the image of \(\bar{f}\)),
all that remains is to prove Lemma~\ref{lm:gen-quad} below.
\end{proof}

\begin{lem} \label{lm:gen-quad} Let \(L_1, L_2, \ldots, L_k \subset \pp^{2k - 1}\)
be linearly independent lines, and \(\beta_i \colon \pp^1 \to L_i\)
be maps which are pairwise distinct (under every possible identification of \(L_i\) with \(L_j\)).

If \(\Lambda \subset \pp^{2k + 1}\) is a fixed codimension~\(2\) subspace that
is not transverse to \(\langle \beta_1(s), \beta_2(s), \ldots, \beta_k(s) \rangle\)
for \(s \in \pp^1\) general, then \(\Lambda\) is the span of \(k - 1\) of the \(k\)
given lines \(L_1, L_2, \ldots, L_k\).
\end{lem}
\begin{proof} 
We argue by induction on \(k\). For the base case, we take \(k = 1\), which is vacuous.

For the inductive step, we suppose \(k \geq 2\), and
we divide into cases based on how \(\Lambda\) meets \(\langle L_1, L_2, \ldots, L_{k - 1}\rangle\).
If \(\Lambda = \langle L_1, L_2, \ldots, L_{k - 1}\rangle\),
then the desired conclusion evidently holds.

Next consider the case when
\(\Lambda\) meets \(\langle L_1, L_2, \ldots, L_{k - 1}\rangle\) in codimension \(1\).
Fix \(s \in \pp^1\) general.
Then the intersection \(\Lambda \cap \langle L_1, L_2, \ldots, L_{k - 1}\rangle\)
does not contain, and hence is transverse to, \(\langle \beta_1(s), \beta_2(s), \ldots, \beta_{k-1}(s) \rangle\)
inside of \(\langle L_1, L_2, \ldots, L_{k - 1}\rangle \simeq \pp^{2k - 3}\).
Also, \(\beta_k(s) \notin \Lambda + \langle L_1, L_2, \ldots, L_{k - 1}\rangle\), since 
\(\langle L_1, L_2, \ldots, L_k\rangle = \pp^{2k-1}\).
Combining these, \(\Lambda\) is transverse to
\(\langle \beta_1(s), \beta_2(s), \ldots, \beta_k(s) \rangle\)
in violation of our assumption.

Finally, consider the case when \(\Lambda\) is transverse to \(\langle L_1, L_2, \ldots, L_{k - 1}\rangle\).
Applying our inductive hypothesis, \(\Lambda \cap \langle L_1, L_2, \ldots, L_{k - 1}\rangle\)
is the span of \(k - 2\) of the \(k - 1\) given lines \(L_1, L_2, \ldots, L_{k-1}\).
If \(k \geq 3\), then \(\Lambda\) contains some \(L_i\),
and projecting from this \(L_i\) and applying our inductive hypothesis completes the proof.

It thus remains only to rule out the case when \(k = 2\) and \(\Lambda\) is transverse to \(L_1\);
exchanging the roles of \(L_1\) and \(L_2\), we may also suppose \(\Lambda\) is transverse to \(L_2\).
Projection from \(\Lambda\) then defines an isomorphism \(L_1 \simeq L_2\).
By assumption, \(\beta_1 \neq \beta_2\) with respect to this identification of \(L_1\) with \(L_2\),
i.e., \(\Lambda\) is disjoint from 
\(\langle \beta_1(s), \beta_2(s)\rangle\) for \(s \in \pp^1\) generic,
in violation of our assumption.
\end{proof}

\section{Inductive arguments}\label{sec:inductive}

In this section, we suppose that \((d,g,r,\ell,m)\) is good and give several inductive arguments 
that reduce \(I(d, g, r, \ell, m)\) to cases where \(d\) is smaller or where \(d\) is the same and \(m\) is smaller.
In the next section, we will show that these arguments reduce all allowed instances \(I(d, g, r, \ell, m)\) to
the already considered infinite family of cases with \((\delta, \ell, m) = (1, 0, 0)\),
plus finitely many sporadic base cases in small projective spaces.

\subsection{Outline of inductive arguments}\label{subsec:over_ind}
In order to indicate the specializations and projections of the original BN-curve \(C\), we introduce the following notation.  Write \(C(0,0;0) = C\) for our original general BN-curve of degree \(d\) and genus \(g\) in \(\pp^r\).  More generally, the notation \(C(a, b; c)\) will denote a curve obtained from \(C(0,0;0)\) by peeling off \(a\) one-secant lines (as described in \eqref{peel1sec} below), peeling off \(b\) two-secant lines (as described in \eqref{peel2sec} below), and projecting from \(c\) general points
on the curve (as described in \eqref{project} below).  In particular, \(C(a,b;c)\) is a BN-curve of degree \(d-a-b-c\) and genus \(g-b\) in \(\pp^{r-c}\). 
The inductive arguments we will give will
make use of the following six key ingredients:
\begin{enumerate}
\item\label{peel1sec} (cf.\ Section \ref{subsec:1sec}) We peel off a \(1\)-secant line, i.e., we degenerate \(C(a, b; c)\) to \(C(a + 1, b; c) \cup L\),
where \(L\) is a \(1\)-secant line to \(C(a + 1, b; c)\), meeting \(C(a + 1, b; c)\)
at a point we will call \(x\).  In this case, we write \(y\) for some point in \(L \setminus \{x\}\).
We always do this specialization so that all marked points determining the modification data
specialize onto \(C(a + 1, b; c) \setminus \{x\}\).

\item\label{peel2sec} (cf.\ Section \ref{subsec:2sec}) We peel off a \(2\)-secant line, i.e., we degenerate \(C(a, b; c)\) to \(C(a, b + 1; c) \cup L\),
where \(L\) is \(2\)-secant to \(C(a, b + 1; c)\), meeting \(C(a, b + 1; c)\)
at points we will denote \(\{z, w\}\).
We always do this specialization so that all marked points determining the modification data
specialize onto \(C(a, b + 1; c) \setminus \{z, w\}\).

\item We specialize the modification data. For the modifications \([\posmodalong R_i]\),
we use the technology developed in Section \ref{sec:onion_specialization}. For the remaining modifications,
we specialize the marked points determining the modification data
(which start out general).

\item\label{project} We project from a point \(p \in C(a, b; c)\). Namely, if we
write \(C(a, b; c + 1)\) for the projection of \(C(a, b; c)\) from \(p\), then the pointing bundle
exact sequence induces (cf.~\eqref{eq:mods_ses}) an exact sequence
\[\qquad 0 \to N_{C(a, b; c) \to p}(\text{mods to \(p\)}) \to N_{C(a, b; c)}[\text{mods}] \to N_{C(a, b; c + 1)}[\text{residual mods}](p) \to 0.\]
If the number \(n\) of modifications towards \(p\) satisfies \(|n - \delta| < 1\),
then by Corollary \ref{cor:intp_ses} interpolation for \(N_{C(a, b; c)}[\text{mods}]\) follows from
interpolation for \(N_{C(a, b; c + 1)}[\text{residual mods}]\).
More generally, if \(n\) satisfies \(|n - \delta| \leq 1 - \frac{\epsilon}{r - 1}\),
then we may iterate this construction (i.e., first specialize as desired and then project)
a total of \(\epsilon\) times.

\item We erase modifications that are linearly general.
Namely, suppose that one of our modifications \([p \posmod M]\)
is linearly general. Then interpolation for \(N[p \posmod M]\) follows from interpolation for \(N\) by Lemma \ref{lem:lin_general}.
More generally, if \(M\) is not linearly general, but contains some subspace \(M_0\)
and is linearly general in the quotient \(N|_p / M_0\),
then interpolation for \(N[p \posmod M]\)
follows from interpolation for \(N\) and \(N[p \posmod M_0]\) by Lemma \ref{lem:lin_general}.

\item\label{indstrat:center_proj_R} We specialize any remaining \(R_i\) to pass through the center of projection.
In more detail, suppose that we projected from a point \(p\), and that prior to this step, \(R_i\) remains general;
write \(\bar{R}_i\) for the projection of \(R_i\) from \(p\).
Specializing \(R_i\) to pass through \(p\) then induces
the specialization of \(\bar{R}_i\)
to a union \(R'_i \cup L\), where \(R'_i\) is an \(r\)-secant rational normal curve
in a hyperplane (the projection from \(p\) of the hyperplane containing \(R_i\)),
and \(L\) is a line passing through \(p\) and a point \(t \in R'_i\).
This has the effect of replacing the modification
\([\posmodalong \bar{R}_i]\) with the modifications
\([\posmodalong R'_i] [p \posmod t]\).
By Lemma~\ref{proj-general} below, \(R'_i\) is a general \(r\)-secant rational normal curve in a hyperplane,
and \(t \in R_i'\) is a general point.
The modification \([p \posmod t]\) is therefore in a linearly general direction, and
can be erased as above.
In other words, at least when no other modifications are made at \(p\),
the combined effect of these steps is to replace 
\([\posmodalong \bar{R}_i]\) with \([\posmodalong R'_i]\)
(which fits well with our inductive hypothesis).
\end{enumerate}

\begin{lem}\label{proj-general}
Let \(p, q_1, \ldots, q_{r} \in \pp^{r-1}\) be a general set of points,
and write \(\bar{q}_i \in \pp^{r - 2}\) for the projection of \(q_i\) from \(p\).
Let \(\bar{R} \subset \pp^{r - 2}\) be a general rational normal curve
passing through \(\bar{q}_1, \bar{q}_2, \ldots, \bar{q}_{r}\),
and \(x \in \bar{R}\) be a general point.
Then there exists a rational normal curve \(R\) through \(p, q_1, \ldots, q_{r}\)
whose tangent direction at \(p\) corresponds to \(x\),
and whose projection from \(p\) is \(\bar{R}\).
\end{lem}
\begin{proof}
Such a rational curve, if it exists, is unique.
We can therefore simply compare the dimension of the space
of rational curves through \(p, q_1, \ldots, q_{r}\),
to the dimension of the space of rational curves through \(\bar{q}_1, \bar{q}_2, \ldots, \bar{q}_{r}\)
together with a choice of point on that rational curve.
Visibly both are equal to \(r - 2\).
\end{proof}

\subsection{Main inductive arguments}

We begin with the following proposition, which applies this method without utilizing specialization \eqref{peel2sec} (peeling off a \(2\)-secant line), and which specializes the \(R_i\) as in Section~\ref{sec:onion_specialization}. Since this is the first application of the method
described above, we include some additional explanations which serve to clarify
this method, and will be omitted in subsequent applications.

\begin{prop} \label{master} Let \(\ell'\) and \(m'\) be integers satisfying
\(0 \leq \ell' \leq \ell\) and \(0 \leq m' \leq m\), with \(m' = 0\) if \(r = 3\). Let \(d'\) be an integer satisfying
\(g + r \leq d' \leq d\), with \(d' > g + r\) if both \(g = 0\) and \(m \neq 0\).
For \(1 \leq i \leq m'\), let \(n_i\) be an integer satisfying
\(n_i \equiv r - 1\) mod \(2\) and \(2 \leq n_i \leq r - 1\), with \(n_i \neq 2\) if \((d', g) = (r + 1, 1)\).
Define
\[\bar{\ell} = \ell - \ell' + \frac{(r - 1)m' - \sum n_i}{2} \quad \text{and} \quad \bar{m} = m - m'.\]
If
\[2m' + \ell' \leq r - 2 \quad \text{and} \quad \left|\delta - \left[\ell' + 2(d - d') + \sum n_i \right] \right| \leq 1 - \frac{1}{r - 1},\]
and \(I(d' - 1, g, r - 1, \bar{\ell}, \bar{m})\) holds, then so does
\(I(d, g, r, \ell, m)\).
\end{prop}
\begin{proof} Our goal is to establish \(I(d, g, r, \ell, m)\), which asserts interpolation for
\[N_{C(0,0;0)}[u_1 \biposmod v_1]\cdots[u_\ell \biposmod v_\ell][\posmodalong R_1 \cup \cdots \cup R_m].\]
Our assumption that \(g + r \leq d' \leq d\), with \(d' > g + r\) if both \(g = 0\) and \(m \neq 0\),
implies that we may peel off \(d - d'\) one-secant lines.
(Recall from the discussion at the beginning of the section that this means
we specialize \(C(0, 0; 0)\) to the union of a BN-curve \(C(d - d', 0; 0) \subset \pp^r\)
of degree \(d'\) and genus \(g\), with \(d - d'\) one-secant lines, in such a way that
all \(u_i\) and \(v_i\), and all points of intersection with the \(R_i\),
specialize onto \(C(d - d', 0; 0)\).)
This reduces our problem to showing interpolation for
\[N \colonequals N_{C(d - d', 0; 0)} [u_1 \biposmod v_1]\cdots[u_\ell \biposmod v_\ell][\posmodalong R_1 \cup \cdots \cup R_m][2x_1 \posmod y_1]\cdots[2x_{d-d'} \posmod y_{d-d'}].\]

For \(1 \leq i \leq m'\), write \(n_i' = (r-1-n_i)/2\),
and degenerate \(R_i\) as in Section \ref{sec:onion_specialization} to the union \(R_i^\circ\), of \(n_i\) lines \(L_{i,j}\) meeting \(C\) at \(p_i\) and \(q_{i,j}\), and \(n_i'\) conics \(Q_{i,j}\) meeting \(C\) at \(p_i\) and \(q_{i,n_i + 2j-1}\) and \(q_{i,n_i + 2j}\).
This induces a specialization of \(N\) to
\begin{multline*}
N_{C(d - d', 0; 0)}[u_1 \biposmod v_1]\cdots[u_\ell \biposmod v_\ell][\posmodalong R_{m' + 1} \cup \cdots \cup R_m][2x_1 \posmod y_1]\cdots[2x_{d-d'} \posmod y_{d-d'}] \\
[q_{1,1} + \cdots + q_{1,r - 1} \posmodalong R_1^\circ] \cdots [q_{m',1} + \cdots + q_{m',r - 1} \posmodalong R_{m'}^\circ] [p_1 \posmod M_1] \cdots [p_{m'} \posmod M_{m'}].
\end{multline*}
Now fix a general point \(p \in C\), and specialize \(p_1, p_2, \ldots, p_{m'}, v_1, v_2, \ldots, v_{\ell'}, y_1, y_2, \ldots, y_{d - d'}\)
all to \(p\).
Because \(2m' + \ell' \leq r - 2 \leq r-1\) by assumption, the limiting directions \(M_1, \dots, M_{m'}, u_1, \dots, u_{\ell'}\) are linearly independent in \( \pp N_C|_p\), and the limit is therefore tree-like (cf.~Defintion \ref{def:tree_like}).  Hence, this induces a further specialization of \(N\) to
\begin{multline*}
N^\circ \colonequals N_{C(d - d', 0; 0)}[u_{\ell'+1} \biposmod v_{\ell'+1}]\cdots[u_\ell \biposmod v_\ell][\posmodalong R_{m' + 1} \cup \cdots \cup R_m] [u_1 + \cdots + u_{\ell'} + 2x_1 + \cdots + 2 x_{d - d'} \posmod p] \\
[q_{1,1} + \cdots + q_{1,r - 1} \posmodalong R_1^\circ] \cdots [q_{m',1} + \cdots + q_{m',r - 1} \posmodalong R_{m'}^\circ] [p \posmod M],
\end{multline*}
where \(M = \Span(M_1, \dots, M_{m'}, u_1, \dots, u_{\ell'}) \subset \pp N_C|_p\).  Furthermore, \(M\) is disjoint from \(\pp N_{C \to p}|_p\) by combining the assumption \(2m' + \ell' \leq r - 2 \) with Lemma \ref{M:unprojected}\eqref{n2}.  Finally, \(M\) linearly general in
\(\pp (N_C / N_{C \to p})|_p\) by Lemmas \ref{M:unprojected}\eqref{n3} and \ref{M:projected}.

It remains to see that \(N^\circ\) satisfies interpolation. For this,
we project from \(p\). In other words, as described at the beginning of the section,
we use the following pointing bundle exact sequence:
\begin{multline} \label{ncircqseq}
0 \to N_{C(d - d', 0; 0) \to p}(u_1 + \cdots + u_{\ell'} + 2x_1 + \cdots + 2x_{d-d'} + (q_{1,1} + \cdots + q_{1, n_1}) + \cdots + (q_{m',1} + \cdots + q_{m', n_{m'}})) \\
\to N^\circ \to Q(p) \to 0,
\end{multline}
where 
\begin{multline*}
Q \colonequals N_{C(d - d', 0; 1)}[u_{\ell' + 1} \biposmod v_{\ell' + 1}] \cdots[u_\ell \biposmod v_\ell][\posmodalong \bar{R}_{m' + 1} \cup \cdots \cup \bar{R}_m] [p \posmod \bar{M}] \\
[q_{1,n_1 + 1} \biposmod q_{1,n_1 + 2}] \cdots [q_{1, r-2} \biposmod q_{1, r-1}] \cdots [q_{m',n_{m'} + 1} \biposmod q_{m',n_{m'} + 2}] \cdots [q_{m', r-2} \biposmod q_{m', r-1}].
\end{multline*}
The number of transformations towards \(p\) is \(\ell' + 2(d - d') + \sum n_i\).
(These transformations occur at \(u_1, \ldots, u_{\ell'}, x_1, \ldots, x_{d - d'}, q_{1,1}, \ldots, q_{1, n_1}, \ldots, q_{m',1}, \ldots, q_{m', n_{m'}}\), c.f.\ \eqref{ncircqseq}.
In particular, this specialization 
does \emph{not} produce a positive transformation at \(p\) in the direction of \(p\), because
\(M\) is disjoint from \(\pp N_{C \to p}|_p\) as explained above.)
Our assumption that \(|\delta - [\ell' + 2(d - d') + \sum n_i]| \leq 1 - \frac{1}{r - 1}\) therefore implies
that interpolation for \(N^\circ\) follows from interpolation for \(Q\) by Corollary \ref{cor:intp_ses}.

We next erase the transformation at \(p\).
In other words,
the only way that \(Q\) depends on the points \(u_1, \dots, u_{\ell'}, q_{1,1}, \dots, q_{1, n_1}, \dots, q_{m', 1}, \dots, q_{m', n_{m'}}\) is via the dependence of \(\bar{M}\) on these points.  As only these points vary, 
 \(\bar{M}\) is linearly general.  Thus interpolation for \(Q\) follows, by Lemma \ref{lem:lin_general}, from interpolation for
\begin{multline*}
N_{C(d - d', 0; 1)}[u_{\ell' + 1} \biposmod v_{\ell' + 1}] \cdots[u_\ell \biposmod v_\ell][\posmodalong \bar{R}_{m' + 1} \cup \cdots \cup \bar{R}_m] \\
[q_{1,n_1 + 1} \biposmod q_{1,n_1 + 2}] \cdots [q_{1, r-2} \biposmod q_{1, r-1}] \cdots [q_{m',n_{m'} + 1} \biposmod q_{m',n_{m'} + 2}] \cdots [q_{m', r-2} \biposmod q_{m', r-1}].
\end{multline*}

Finally, we specialize the remaining \(R_i\), for \(m' + 1 \leq i \leq m\), to pass through \(p\).
Namely, we first specialize \(R_{m' + 1}\) to pass through \(p\), which induces the specialization of \(\bar{R}_{m' + 1}\)
to a union \(R'_{m' + 1} \cup L\) as described in Subsection \ref{subsec:over_ind}\eqref{indstrat:center_proj_R}.
The effect of this specialization on the above bundle is to replace the modification
\([\posmodalong \bar{R}_{m' + 1}]\) with the modifications
\([\posmodalong R'_{m' + 1}] [p \posmod t]\), where \(t\) is a general point on \(R'_{m'+1}\).  In other words, the above bundle specializes to
\begin{multline*}
N_{C(d - d', 0; 1)}[u_{\ell' + 1} \biposmod v_{\ell' + 1}] \cdots[u_\ell \biposmod v_\ell][\posmodalong R'_{m' + 1}][\posmodalong \bar{R}_{m' + 2} \cup \cdots \cup \bar{R}_m] [p \posmod t] \\
[q_{1,n_1 + 1} \biposmod q_{1,n_1 + 2}] \cdots [q_{1, r-2} \biposmod q_{1, r-1}] \cdots [q_{m',n_{m'} + 1} \biposmod q_{m',n_{m'} + 2}] \cdots [q_{m', r-2} \biposmod q_{m', r-1}].
\end{multline*}
The modification \([p \posmod t]\) is in a linearly general direction,
and may therefore be erased by Lemma \ref{lem:lin_general}. In other words, interpolation for this bundle follows from interpolation for
\begin{multline*}
N_{C(d - d', 0; 1)}[u_{\ell' + 1} \biposmod v_{\ell' + 1}] \cdots[u_\ell \biposmod v_\ell][\posmodalong R'_{m' + 1}][\posmodalong \bar{R}_{m' + 2} \cup \cdots \cup \bar{R}_m] \\
[q_{1,n_1 + 1} \biposmod q_{1,n_1 + 2}] \cdots [q_{1, r-2} \biposmod q_{1, r-1}] \cdots [q_{m',n_{m'} + 1} \biposmod q_{m',n_{m'} + 2}] \cdots [q_{m', r-2} \biposmod q_{m', r-1}].
\end{multline*}
Similarly specializing \(R_{m' + 2}\), then \(R_{m' + 3}\), and so on until \(R_m\), we reduce to interpolation for the bundle
\begin{multline*}
N_{C(d - d', 0; 1)}[u_{\ell' + 1} \biposmod v_{\ell' + 1}] \cdots[u_\ell \biposmod v_\ell] [\posmodalong R'_{m' + 1} \cup \cdots \cup R'_m] \\
[q_{1,n_1 + 1} \biposmod q_{1,n_1 + 2}] \cdots [q_{1, r-2} \biposmod q_{1, r-1}] \cdots [q_{m',n_{m'} + 1} \biposmod q_{m',n_{m'} + 2}] \cdots [q_{m', r-2} \biposmod q_{m', r-1}],
\end{multline*}
which is just the assertion
\(I(d' - 1, g, r - 1, \bar{\ell}, \bar{m})\).
\end{proof}

The \(n_i\) appearing in 
Proposition~\ref{master} are constrained mod \(2\).
It is thus often difficult to apply Proposition~\ref{master} in situations
where \(\delta\) is an integer with the ``wrong'' parity.
We introduce the following variant, which has the advantage
that its difficult parity is the opposite of the difficult parity
for Proposition~\ref{master}.

\begin{prop} \label{master-111}
Let \(\ell'\), \(m'\), \(d'\), the \(n_i\), \(\bar{\ell}\), and \(\bar{m}\), be as in Proposition~\ref{master}.
If
\[m' < m, \quad 2m' + \ell' < r - 2, \quad \text{and} \quad \left|\delta - \left[1 + \ell' + 2(d - d') + \sum n_i\right]\right| \leq 1 - \frac{1}{r - 1},\]
and
\[I(d' - 1, g, r - 1, \bar{\ell}, \bar{m}), \quad I(d' - 1, g, r - 1, \bar{\ell}, \bar{m} - 1), \quad \text{and} \quad I(d' - 2, g, r - 2, \bar{\ell}, \bar{m})\]
all hold, then so does \(I(d, g, r, \ell, m)\).
\end{prop}
\begin{proof}
As in the proof of Proposition~\ref{master}, it suffices to show that \(N^\circ\)
satisfies interpolation, where
\begin{multline*}
N^\circ \colonequals N_{C(d - d', 0; 0)}[u_{\ell'+1} \biposmod v_{\ell'+1}]\cdots[u_\ell \biposmod v_\ell][\posmodalong R_{m' + 1} \cup \cdots \cup R_m] [u_1 + \cdots + u_{\ell'} + 2x_1 + \cdots + 2 x_{d - d'} \posmod p] \\
[q_{1,1} + \cdots + q_{1,r - 1} \posmodalong R_1^\circ] \cdots [q_{m',1} + \cdots + q_{m',r - 1} \posmodalong R_{m'}^\circ] [p \posmod M].
\end{multline*}
Write \(R_{m' + 1} \cap C = \{s_0, s_1, s_2, \ldots, s_{r - 1}, s_r\}\).
We first specialize \(R_{m' + 1}\) to a union \(R \cup L\), where \(L\)
is the line through \(s_0\) and \(s_r\), and \(R\) is a rational curve of degree \(r - 2\)
passing through \(s_1, s_2, \ldots, s_{r - 1}\) and meeting \(L\) at a single point. 
We then specialize \(s_r\) to \(p\).
These specializations induce a specialization of \(N^\circ\) to
\begin{multline*}
N_{C(d - d', 0; 0)}[u_{\ell'+1} \biposmod v_{\ell'+1}]\cdots[u_\ell \biposmod v_\ell][\posmodalong R][\posmodalong R_{m' + 2} \cup \cdots \cup R_m] [s_0 + u_1 + \cdots + u_{\ell'} + 2x_1 + \cdots + 2 x_{d - d'} \posmod p] \\
[q_{1,1} + \cdots + q_{1,r - 1} \posmodalong R_1^\circ] \cdots [q_{m',1} + \cdots + q_{m',r - 1} \posmodalong R_{m'}^\circ] [p \posmod M'],
\end{multline*}
where \(M' = \langle M,  \pp N_{C \to s_0}|_p\rangle \).
We then project from \(p\), thereby reducing to interpolation for
\begin{multline*}
N_{C(d - d', 0; 1)} [u_{\ell' + 1} \biposmod v_{\ell' + 1}] \cdots[u_\ell \biposmod v_\ell][s_1 + \cdots + s_{r - 1} \posmodalong \bar{R}][\posmodalong \bar{R}_{m' + 2} \cup \cdots \cup \bar{R}_m] [p \posmod \bar{M}'] \\
[q_{1,n_1 + 1} \biposmod q_{1,n_1 + 2}] \cdots [q_{1, r-2} \biposmod q_{1, r-1}] \cdots [q_{m',n_{m'} + 1} \biposmod q_{m',n_{m'} + 2}] \cdots [q_{m', r-2} \biposmod q_{m', r-1}].
\end{multline*}
Here we have written out the modification \([s_1 + \cdots + s_{r - 1} \posmodalong \bar{R}]\) because \(\bar{R}\) also meets \(C(d-d',0;1)\) at \(s_0\).
Note that \(\bar{M}'\) is not linearly general since it contains the fixed direction \(\pp N_{C(d-d', 0;1) \to s_0}|_p\); however, it is linearly general in the quotient by \(\pp N_{C(d-d', 0;1) \to s_0}|_p\).  Using Lemma \ref{lem:lin_general}, we reduce to interpolation for the pair of bundles
\begin{align*}
& Q \colonequals N_{C(d - d', 0; 1)}[u_{\ell' + 1} \biposmod v_{\ell' + 1}] \cdots[u_\ell \biposmod v_\ell][s_1 + \cdots + s_{r - 1} \posmodalong \bar{R}][\posmodalong \bar{R}_{m' + 2} \cup \cdots \cup \bar{R}_m]  \\
& \qquad\qquad [q_{1,n_1 + 1} \biposmod q_{1,n_1 + 2}] \cdots [q_{1, r-2} \biposmod q_{1, r-1}] \cdots [q_{m',n_{m'} + 1} \biposmod q_{m',n_{m'} + 2}] \cdots [q_{m', r-2} \biposmod q_{m', r-1}] \\
\text{and} \ & Q[p \posmod s_0].
\end{align*}
Specializing \(R_{m' + 2}, R_{m' + 3}, \ldots, R_m\) to pass through \(p\),
interpolation for these two bundles follows from interpolation for the two bundles
\begin{align*}
Q^- &\colonequals N_{C(d - d', 0; 1)}[u_{\ell' + 1} \biposmod v_{\ell' + 1}] \cdots[u_\ell \biposmod v_\ell][s_1 + \cdots + s_{r - 1} \posmodalong \bar{R}][\posmodalong R'_{m' + 2} \cup \cdots \cup R'_m]  \\
&\qquad 
[q_{1,n_1 + 1} \biposmod q_{1,n_1 + 2}] \cdots [q_{1, r-2} \biposmod q_{1, r-1}] \cdots [q_{m',n_{m'} + 1} \biposmod q_{m',n_{m'} + 2}] \cdots [q_{m', r-2} \biposmod q_{m', r-1}] \\
Q^+ &\colonequals Q^- [p \posmod s_0].
\end{align*}

By Lemma \ref{lem:monodromy},
interpolation for \(Q^-\) follows from interpolation for the two closely related
vector bundles where all (respectively none) of the transformations along \(\bar{R}\) are performed:
\begin{multline*}
N_{C(d - d', 0; 1)}[u_{\ell' + 1} \biposmod v_{\ell' + 1}] \cdots[u_\ell \biposmod v_\ell][\posmodalong \bar{R} \cup R'_{m' + 2} \cup \cdots \cup R'_m]  \\
[q_{1,n_1 + 1} \biposmod q_{1,n_1 + 2}] \cdots [q_{1, r-2} \biposmod q_{1, r-1}] \cdots [q_{m',n_{m'} + 1} \biposmod q_{m',n_{m'} + 2}] \cdots [q_{m', r-2} \biposmod q_{m', r-1}]
\end{multline*}
\begin{multline*}
N_{C(d - d', 0; 1)}[u_{\ell' + 1} \biposmod v_{\ell' + 1}] \cdots[u_\ell \biposmod v_\ell][\posmodalong R'_{m' + 2} \cup \cdots \cup R'_m]  \\
[q_{1,n_1 + 1} \biposmod q_{1,n_1 + 2}] \cdots [q_{1, r-2} \biposmod q_{1, r-1}] \cdots [q_{m',n_{m'} + 1} \biposmod q_{m',n_{m'} + 2}] \cdots [q_{m', r-2} \biposmod q_{m', r-1}].
\end{multline*}
But these are the assertions \(I(d' - 1, g, r - 1, \bar{\ell}, \bar{m})\) and \(I(d' - 1, g, r - 1, \bar{\ell}, \bar{m} - 1)\) respectively, which hold by assumption.

It remains to see that \(Q^+\) satisfies interpolation.
Applying Lemma~\ref{gen-to-fixed},
and noting that we have already established interpolation for \(Q^-\) above,
it suffices to check interpolation for \(Q^- / N_{C(d - d', 0; 1) \to s_0}\),
which after twisting down by \(s_0\) is isomorphic to
\begin{multline*}
N_{C(d - d', 0; 2)}[u_{\ell' + 1} \biposmod v_{\ell' + 1}] \cdots[u_\ell \biposmod v_\ell][\posmodalong \bar{\bar{R}} \cup \bar{R}'_{m' + 2} \cup \cdots \cup \bar{R}'_m]  \\
[q_{1,n_1 + 1} \biposmod q_{1,n_1 + 2}] \cdots [q_{1, r-2} \biposmod q_{1, r-1}] \cdots [q_{m',n_{m'} + 1} \biposmod q_{m',n_{m'} + 2}] \cdots [q_{m', r-2} \biposmod q_{m', r-1}].
\end{multline*}
Specializing \(R'_{m' + 2}, R'_{m' + 3}, \ldots, R'_m\) to pass through \(s_0\),
we reduce to interpolation for
\begin{multline*}
N_{C(d - d', 0; 2)}[u_{\ell' + 1} \biposmod v_{\ell' + 1}] \cdots[u_\ell \biposmod v_\ell][\posmodalong \bar{\bar{R}} \cup R''_{m' + 2} \cup \cdots \cup R''_m]  \\
[q_{1,n_1 + 1} \biposmod q_{1,n_1 + 2}] \cdots [q_{1, r-2} \biposmod q_{1, r-1}] \cdots [q_{m',n_{m'} + 1} \biposmod q_{m',n_{m'} + 2}] \cdots [q_{m', r-2} \biposmod q_{m', r-1}].
\end{multline*}
But this is just the assertion
\(I(d' - 2, g, r - 2, \bar{\ell}, \bar{m})\),
which holds by assumption.
\end{proof}

\subsection{\label{subsec:largeparams} Large parameters}
Both of the main inductive arguments above impose upper bounds on \(2m' + \ell'\)
(depending on \(r\)).
It is thus difficult to apply them when any of the remaining parameters, i.e., \(d\), \(g\), or \(m\), is large.
(Note that \(\ell\) is already bounded in terms of \(r\) by construction.)
We therefore next give three inductive arguments that apply
for large values of \(d\), \(g\), and \(m\), respectively.

\begin{prop}\label{gather_lines} Suppose that \(d \geq g + 2r - 1\). If \(I(d - (r - 1), g, r, \ell, m)\) holds, then so does
\(I(d, g, r, \ell, m)\).
\end{prop}
\begin{proof}
We want to show interpolation for
\[N_C[u_1 \biposmod v_1]\cdots[u_\ell \biposmod v_\ell][\posmodalong R_1 \cup \cdots \cup R_m].\]
Peeling off \(r - 1\) one-secant lines, it suffices to show interpolation for
\[N_{C(r - 1, 0; 0)} [u_1 \biposmod v_1]\cdots[u_\ell \biposmod v_\ell][\posmodalong R_1 \cup \cdots \cup R_m][2x_1 \posmod y_1]\cdots[2x_{r - 1} \posmod y_{r - 1}].\]
Specializing \(x_1, x_2, \ldots, x_{r - 1}\) to a common point \(x \in C\)
(while leaving \(y_1, y_2, \ldots, y_{r - 1}\) general) reduces to interpolation for
\[N_{C(r - 1, 0; 0)} [u_1 \biposmod v_1]\cdots[u_\ell \biposmod v_\ell][\posmodalong R_1 \cup \cdots \cup R_m](2x).\]
Removing the twist, this bundle satisfies interpolation provided that
\[N_{C(r - 1, 0; 0)} [u_1 \biposmod v_1]\cdots[u_\ell \biposmod v_\ell][\posmodalong R_1 \cup \cdots \cup R_m]\]
satisfies interpolation, which is the assertion \(I(d - (r - 1), g, r, \ell, m)\) that holds by assumption.
\end{proof}

\begin{prop} \label{peel_onion} Suppose that \(g \geq r\). If \(I(d - (r - 1), g - r, r, \ell, m + 1)\) holds, then so does
\(I(d, g, r, \ell, m)\).
\end{prop}
\begin{proof}
Since \((d,g,r,\ell, m)\) is good, \((d,g,r) \neq (2r, r+1, r)\), and so this is a special case of Lemma \ref{lem:pull_off_onion}.
\end{proof}

\begin{prop}\label{pancake_onions} Suppose that \(m \geq r - 1\). If \(I(d, g, r, \ell, m - (r - 1))\) holds, then so does
\(I(d, g, r, \ell, m)\).
\end{prop}
\begin{proof}
We want to show interpolation for 
\[N_C[u_1 \biposmod v_1]\cdots[u_\ell \biposmod v_\ell][\posmodalong R_1 \cup \cdots \cup R_m].\]
Fix points \(q_1, q_2, \ldots, q_{r + 1}\) lying in a general hyperplane section of \(C\).
For \(m - (r - 2) \leq i \leq m\), specialize \(R_i\) to a general rational curve of degree \(r - 1\) meeting
\(C\) at \(q_1, q_2, \ldots, q_{r + 1}\).
This induces a specialization of the above bundle to
\[N_C[u_1 \biposmod v_1]\cdots[u_\ell \biposmod v_\ell][\posmodalong R_1 \cup \cdots \cup R_{m - (r - 1)}](q_1 + \cdots + q_{r + 1}).\]
Removing the twist, this bundle satisfies interpolation provided that
\[N_C[u_1 \biposmod v_1]\cdots[u_\ell \biposmod v_\ell][\posmodalong R_1 \cup \cdots \cup R_{m - (r - 1)}]\]
satisfies interpolation, which is the assertion \(I(d, g, r, \ell, m - (r - 1))\) that holds by assumption.
\end{proof}

\subsection{Small parameters}
The remaining cases where our main inductive arguments do not apply
are when various parameters are small
(which deprives us of flexibility in choosing \(\ell'\) and the \(n_i\)).
Some of the arguments we give here readily generalize to larger values
of various parameters, but since we will not need them in that regime,
we opt to simplify the exposition as far as possible.
We first consider two cases where \(\ell = 0\) and \(m = 1\).

\begin{prop}\label{two_proj} Suppose that \(\ell = 0\) and \(m = 1\).
Let \(\epsilon\) be an integer satisfying \(0 \leq \epsilon \leq (d - g - r) / 2\),
with \(\epsilon < (d - g  - r)/2\) if \(g = 0\).
If
\[\left| \delta - (2\epsilon + 1)\right| \leq 1 - \frac{2}{r - 1},\]
and \(I(d - 2\epsilon - 2, g, r - 2, 0, 1)\) holds, then so does \(I(d, g, r, 0, 1)\).
\end{prop}
\begin{proof}
We want to show that \(N_C[\posmodalong R_1]\) satisfies interpolation.
Peeling off \(2\epsilon\) one-secant lines, we reduce to interpolation for
\[N_{C(2\epsilon, 0; 0)} [\posmodalong R_1][2x_1 \posmod y_1]\cdots[2x_{2\epsilon} \posmod y_{2\epsilon}].\]
Write \(R_1 \cap C = \{s_0, s_1, s_2, \ldots, s_{r - 1}, s_r\}\).
As in the proof of Proposition~\ref{master-111}, 
specialize \(R_1\) to a union \(R \cup L\), where \(L\)
is the line through \(s_0\) and \(s_r\), and \(R\) is a rational curve of degree \(r - 2\)
passing through \(s_1, s_2, \ldots, s_{r - 1}\) and meeting \(L\) at a single point.
Then specialize \(y_1, y_2, \ldots, y_\epsilon\) to \(s_0\), and
\(y_{\epsilon + 1}, y_{\epsilon + 2}, \ldots, y_{2\epsilon}\) to \(s_r\).
This reduces our problem to interpolation for
\[N_{C(2\epsilon, 0; 0)}[\posmodalong R][s_r + 2 x_1 + \cdots + 2 x_\epsilon \posmod s_0][s_0 + 2 x_{\epsilon + 1} + \cdots + 2 x_{2\epsilon} \posmod s_r].\]
Projecting from \(s_0\) and then from \(s_r\), we reduce to interpolation for
\(N_{C(2\epsilon, 0; 2)}[\posmodalong \bar{R}]\),
which is the assertion \(I(d - 2\epsilon - 2, g, r - 2, 0, 1)\)
that holds by assumption.
\end{proof}

\begin{prop}\label{delta5} Let \(k \geq 3\).  If \(I(4k - 3, 2k - 2, 2k - 1, k - 3, 0)\) holds, then so does
\(I(4k + 1, 2k - 1, 2k + 1, 0, 1)\).
\end{prop}
\begin{proof}
Note that \(\delta(4k + 1, 2k - 1, 2k + 1, 0, 1) = 5\).
Our goal is to show interpolation for \(N_C[\posmodalong R_1]\).
Peeling off a \(1\)-secant line and a \(2\)-secant line, we reduce to interpolation for
\[N_{C(1, 1; 0)} [2x \posmod y] [z \biposmod w] [z \posmod 2w] [\posmodalong R_1].\]
Degenerate \(R_1\) as in Section \ref{sec:onion_specialization} to the union \(R_1^\circ\), of \(4\) lines \(L_j\) meeting \(C\) at \(p\) and \(q_j\), and \(k - 2\) conics \(Q_j\) meeting \(C\) at \(p\) and \(q_{2j + 3}\) and \(q_{2j + 4}\).
This induces a specialization of the above bundle to
\[N_{C(1, 1; 0)} [2x \posmod y] [z \biposmod w] [z \posmod 2w] [q_1+q_2+q_3+q_4 \posmod p][q_5 + \cdots + q_{2k} \posmodalong R_1^\circ][p \posmod M],\]
where \(M\) is linearly general as \(q_1, q_2, q_3, q_4\) vary.
Specializing \(z\) to \(p\), we reduce to interpolation for
\[N_{C(1, 1; 0)} [2x \posmod y] [p \posmod w] [p \posmod 2w] [w+q_1+q_2+q_3+q_4 \posmod p][q_5 + \cdots + q_{2k} \posmodalong R_1^\circ][p \posmod M].\]
Projecting from \(p\), we reduce to interpolation for
\[N_{C(1, 1; 1)} [2x \posmod y] [q_5 \biposmod q_6] \cdots [q_{2k - 1} \biposmod q_{2k}][p \posmod w][p \posmod \bar{M} + 2w].\]
Specializing \(y\) and \(q_5\) to \(w\), we reduce to interpolation for
\[N_{C(1, 1; 1)} [q_7 \biposmod q_8] \cdots [q_{2k - 1} \biposmod q_{2k}][p + q_6 + 2x \posmod w][p \posmod \bar{M} + 2w][w \posmod q_6].\]
Projecting from \(w\), we reduce to interpolation for
\[N_{C(1, 1; 2)} [q_7 \biposmod q_8] \cdots [q_{2k - 1} \biposmod q_{2k}][p \posmod \bar{M} + w][w \posmod q_6].\]
Erasing the transformation \([w \posmod q_6]\), and then \([p \posmod \bar{M} + w]\),
we reduce to interpolation for
\[N_{C(1, 1; 2)} [q_7 \biposmod q_8] \cdots [q_{2k - 1} \biposmod q_{2k}],\]
which is the assertion \(I(4k - 3, 2k - 2, 2k - 1, k - 3, 0)\) that holds by assumption.
\end{proof}

\noindent
We finally consider several arguments that are adapted to the case \(m = 0\).

\begin{prop} \label{m0-delta35}
Suppose \(m = 0\), and \(g \geq 3\), and \(r \geq 6\).
Let \(\epsilon\) be an integer with \(0 \leq \epsilon \leq (d - g - r) / 3\).
If
\[\left|\delta - (2\epsilon + 3)\right| \leq 1 - \frac{3}{r - 1},\]
and \(I(d - 3\epsilon - 6, g - 3, r - 3, \ell + 1, 0)\) and \(I(d - 3\epsilon - 6, g - 3, r - 3, \ell, 0)\)
hold, then so does \(I(d, g, r, \ell, 0)\).
\end{prop}
\begin{proof}
Our goal is to show interpolation for
\[N_C [u_1 \biposmod v_1] \cdots [u_\ell \biposmod v_\ell].\]
Peeling off \(3\) two-secant lines and \(3\epsilon\) one-secant lines, we reduce to interpolation for
\begin{multline*}
N_{C(3\epsilon, 3; 0)} [u_1 \biposmod v_1] \cdots [u_\ell \biposmod v_\ell] [2x_1 \posmod y_1] \cdots [2x_{3\epsilon} \posmod y_{3\epsilon}] \\
[z_1 \biposmod w_1][z_1 \posmod 2w_1] [z_2 \biposmod w_2][z_2 \posmod 2w_2] [z_3 \biposmod w_3][z_3 \posmod 2w_3].
\end{multline*}
Specializing \(y_1, y_2, \ldots, y_\epsilon, w_1\) to \(z_2\), and
\(y_{\epsilon + 1}, y_{\epsilon + 2}, \ldots, y_{2\epsilon}, w_2\) to \(z_3\), 
we reduce to interpolation for
\begin{multline*}
N_{C(3\epsilon, 3; 0)} [u_1 \biposmod v_1] \cdots [u_\ell \biposmod v_\ell] [2x_{2\epsilon + 1} \posmod y_{2\epsilon + 1}] \cdots [2x_{3\epsilon} \posmod y_{3\epsilon}] [z_3 \posmod w_3][z_3 \posmod 2w_3] [z_2 \posmod z_1]\\
[2x_1 + \cdots + 2x_\epsilon + z_1 + z_3 \posmod z_2] [2x_{\epsilon + 1} + \cdots + 2x_{2\epsilon} + z_2 + w_3 \posmod z_3] [z_1 \posmod 2z_2] [z_2 \posmod 2z_3].
\end{multline*}
Projecting from \(z_2\), and then \(z_3\), we reduce to interpolation for
\[N_{C(3\epsilon, 3; 2)} [u_1 \biposmod v_1] \cdots [u_\ell \biposmod v_\ell] [2x_{2\epsilon + 1} \posmod y_{2\epsilon + 1}] \cdots [2x_{3\epsilon} \posmod y_{3\epsilon}] [z_3 \posmod w_3][z_3 \posmod 2w_3] [z_1 \biposmod z_2] [z_2 \posmod z_3].\]
Specializing \(y_{2\epsilon + 1}, y_{2\epsilon + 2}, \ldots, y_{3\epsilon}, w_3\) to \(z_2\), we reduce to interpolation for
\[N_{C(3\epsilon, 3; 2)} [u_1 \biposmod v_1] \cdots [u_\ell \biposmod v_\ell] [z_2 \posmod z_1+z_3] [2x_{2\epsilon + 1} +  \cdots + 2x_{3\epsilon} + z_1 + z_3 \posmod z_2] [z_3 \posmod 2z_2].\]
Projecting from \(z_2\) (again), we reduce to interpolation for
\[N_{C(3\epsilon, 3; 3)} [u_1 \biposmod v_1] \cdots [u_\ell \biposmod v_\ell] [z_3 \posmod z_2] [z_2 \posmod z_1 + z_3].\]
Erasing the transformation \([z_2 \posmod z_1 + z_3]\), we reduce to interpolation for the pair of bundles
\[N_{C(3\epsilon, 3; 3)} [u_1 \biposmod v_1] \cdots [u_\ell \biposmod v_\ell] [z_3 \biposmod z_2] \quad \text{and} \quad N_{C(3\epsilon, 3; 3)} [u_1 \biposmod v_1] \cdots [u_\ell \biposmod v_\ell][z_3 \posmod z_2].\]
The first is our assumption \(I(d - 3\epsilon - 6, g - 3, r - 3, \ell + 1, 0)\).
For the second, we erase the transformation \([z_3 \posmod z_2]\) to reduce to interpolation for
\[N_{C(3\epsilon, 3; 3)} [u_1 \biposmod v_1] \cdots [u_\ell \biposmod v_\ell],\]
which is our assumption \(I(d - 3\epsilon - 6, g - 3, r - 3, \ell, 0)\).
\end{proof}

\begin{prop} \label{m0-delta2}
Suppose \(m = 0\) and \(g \geq 1\), and that
\[|\delta - 2| \leq 1 - \frac{1}{r - 1}.\]
If \(I(d - 2, g - 1, r - 1, \ell + 1, 0)\) holds, then so does
\(I(d, g, r, \ell, 0)\).
\end{prop}
\begin{proof}
Our goal is to show interpolation for
\(N_C [u_1 \biposmod v_1] \cdots [u_\ell \biposmod v_\ell]\).
Peeling off a \(2\)-secant line, we reduce to interpolation for
\[N_{C(0, 1; 0)} [u_1 \biposmod v_1] \cdots [u_\ell \biposmod v_\ell] [z \biposmod w][z \posmod 2w].\]
Projecting from \(w\), we reduce to interpolation for
\[N_{C(0, 1; 1)} [u_1 \biposmod v_1] \cdots [u_\ell \biposmod v_\ell] [z \biposmod w],\]
which is our assumption \(I(d - 2, g - 1, r - 1, \ell + 1, 0)\).
\end{proof}

\begin{prop}\label{m0-delta4}
Suppose \(m = 0\) and \(g \geq 3\) and \(r \geq 6\), and that
\[ |\delta - 4| \leq 1 - \frac{2}{r - 1}.\]
If \(I(d - 5, g - 3, r - 2, \ell + 1, 0)\) and \(I(d - 5, g - 3, r - 2, \ell, 0)\) hold, then so does
\(I(d, g, r, \ell, 0)\).
\end{prop}
\begin{proof}
Our goal is to show interpolation for
\(N_C [u_1 \biposmod v_1] \cdots [u_\ell \biposmod v_\ell]\).
Peeling off \(3\) two-secant lines, we reduce to interpolation for
\[N_{C(0, 3; 0)} [u_1 \biposmod v_1] \cdots [u_\ell \biposmod v_\ell] [z_1 \biposmod w_1][z_1 \posmod 2w_1][z_2 \biposmod w_2][z_2 \posmod 2w_2][z_3 \biposmod w_3][z_3 \posmod 2w_3].\]
Specializing \(w_2\) to \(w_1\), we reduce to interpolation for
\[N_{C(0, 3; 0)} [u_1 \biposmod v_1] \cdots [u_\ell \biposmod v_\ell] [z_1 + z_2 \biposmod w_1][z_1 + z_2 \posmod 2w_1][z_3 \biposmod w_3][z_3 \posmod 2w_3].\]
Projecting from \(w_1\), we reduce to interpolation for
\[N_{C(0, 3; 1)} [u_1 \biposmod v_1] \cdots [u_\ell \biposmod v_\ell] [z_1 + z_2 \biposmod w_1][z_3 \biposmod w_3][z_3 \posmod 2w_3].\]
Specializing \(w_3\) to \(w_1\), we reduce to interpolation for
\[N_{C(0, 3; 1)} [u_1 \biposmod v_1] \cdots [u_\ell \biposmod v_\ell] [z_1 + z_2 + z_3 \biposmod w_1] [z_3 \posmod 2w_1].\]
Projecting from \(w_1\) (again), we reduce to interpolation for
\[N_{C(0, 3; 2)} [u_1 \biposmod v_1] \cdots [u_\ell \biposmod v_\ell] [z_3 \posmod w_1][w_1 \posmod z_1 + z_2+z_3].\]
Erasing the transformation \([w_1 \posmod z_2 + z_2 +z_3]\), we reduce to interpolation for the pair of bundles
\[N_{C(0, 3; 2)} [u_1 \biposmod v_1] \cdots [u_\ell \biposmod v_\ell] [z_3 \biposmod w_1] \quad \text{and} \quad N_{C(0, 3; 2)} [u_1 \biposmod v_1] \cdots [u_\ell \biposmod v_\ell] [z_3 \posmod w_1].\]
The first is our assumption \(I(d - 5, g - 3, r - 2, \ell + 1, 0)\).
For the second, we erase the transformation \([z_3 \posmod w_1]\) to reduce to interpolation for
\[N_{C(0, 3; 2)} [u_1 \biposmod v_1] \cdots [u_\ell \biposmod v_\ell],\]
which is our assumption \(I(d - 5, g - 3, r - 2, \ell, 0)\).
\end{proof}

\section{\texorpdfstring{Interlude: Some cases not implied by \(I(d, g, r, \ell, m)\)}{Interlude: Some cases not implied by I(d, g, r, l, m)} \label{sec:rational}}

As explained in Section~\ref{sec:overview},
our main inductive argument will establish
\(I(d, g, r, \ell, m)\) for all good tuples. We have already seen that:
\begin{itemize}
\item \(I(d, g, r, \ell, m)\) for all good tuples implies
Theorem~\ref{thm:main} except in a couple of cases.
\item Theorem~\ref{thm:main} implies Theorem~\ref{cor:main}
except in a couple of cases.
\end{itemize}
Of course, we must also check
Theorem~\ref{thm:main} and Theorem~\ref{cor:main}
respectively in these couple of cases.
The most difficult of these is
Theorem~\ref{thm:main} for canonical curves of even genus \(g \geq 8\),
which we defer to Section~\ref{sec:canonical}.
Here we quickly take care of all the others.

\subsection{Theorem~\ref{thm:main} for rational curves}
Consulting Proposition~\ref{prop:I_implies_interpolation},
we may assume \(d \not\equiv 1\) mod \(r - 1\).
By assumption, this implies the characteristic is distinct from \(2\).
It thus suffices to argue that Theorem~\ref{thm:main} holds for rational curves
in characteristic distinct from \(2\), which we do by induction on \(d\) as follows:

\begin{description}
\item[\boldmath If \(\delta < 1\)] We apply Proposition~\ref{master} with \(d' = d\).
\item[\boldmath If \(\delta = 1\)] We apply Proposition \ref{prop:delta1} (the characteristic assumption enters here).
\item[\boldmath If \(1 < \delta < 2\)] We apply Proposition~\ref{master} with \(d' = d - 1\).
\item[\boldmath If \(2 \leq \delta\)] Upon rearrangement this implies \(d \geq 2r - 1\).
We may thus apply Proposition \ref{gather_lines}.
\end{description}

\subsection{Theorem~\ref{cor:main} for rational curves}
Using Theorem~\ref{thm:main}, we deduce Theorem~\ref{cor:main}
for rational curves when the characteristic is distinct from \(2\).
Here we show that Theorem~\ref{cor:main} also holds for rational curves in characteristic \(2\).

\begin{lem} \label{tozero} Suppose the evaluation map \(\bar{M}_{g, n}(\pp^r, d) \to (\pp^r)^n\) is dominant
in characteristic \(0\). Then it is dominant in all characteristics.
\end{lem}

\begin{proof}
Because \(\bar{M}_{g, n}(\pp^r, d)\) is proper over \(\Spec \zz\),
and the evaluation map is dominant in characteristic \(0\),
the evaluation map is therefore surjective over \(\Spec \zz\).
\end{proof}

If \(g = 0\), then \(\bar{M}_{0, n}(\pp^r, d)\) is \emph{irreducible} in any characteristic, and so we conclude
the truth of Theorem~\ref{cor:main} in characteristic \(2\) from
the truth of Theorem~\ref{cor:main} in characteristic \(0\).

\begin{rem}
The reader might hope to apply Lemma~\ref{tozero} to higher genus curves.
Unfortunately, all we learn is that \emph{some} component of \(\bar{M}_{g, n}(\pp^r, d)\)
dominates \((\pp^r)^n\) in positive characteristic. This is a fatal flaw
when the genus is positive, because there are other components, not corresponding to BN-curves,
which would tell us nothing about the interpolation problem for positive-genus curves.
For example, consider the component
containing those stable maps which contract a smooth curve of genus \(g\)
to a point and map a rational tail to \(\pp^r\) with degree \(d\).
\end{rem}

\subsection{\boldmath Theorem~\ref{cor:main} for \((d, g, r) = (6, 2, 4)\)}
We want to show such a BN-curve can pass through \(9\) general points.
It suffices to show \(H^1(N_C(-D)) = 0\) when \(D\) is a general divisor of degree \(9\) on \(C\).
Peeling off a \(2\)-secant line and specializing one of the points of \(D\) onto the \(2\)-secant line,
this reduces to
\[H^1(N_{C(0, 1; 0)}[u \posmod v][v \posmod u][v \posmod 2u](-v-D')) = 0,\]
where \(D'\) is now a general divisor of degree \(8\) on \(C(0, 1; 0)\).
This follows in turn from
\[H^1(N_{C(0, 1; 0)}[2v \posmod u](-v-D')) = 0,\]
because this is a subsheaf with punctual quotient.
Since \(v + D'\) is a general divisor of degree \(9\) on \(C(0, 1; 0)\),
this follows from interpolation for
\[N_{C(0, 1; 0)}[2v \posmod u].\]
Projecting from \(u\), we reduce to interpolation for \(N_{C(0, 1; 1)}\),
which is Theorem~\ref{thm:main} for \((d, g, r) = (4, 1, 3)\).

\section{Combinatorics}\label{sec:combinat}

In this section we show, by a purely combinatorial argument,
that the inductive arguments in Section \ref{sec:inductive} apply to all good tuples \((d, g,r, \ell, m)\) except for:
\begin{itemize}
\item The infinite family \((d, g,r, 0,0)\) with \(\delta = 1\) already treated in Section \ref{sec:delta1}; and
\item A \emph{finite} number of other cases.
\end{itemize}

We begin by showing that these
inductive arguments apply to all but finitely many tuples
for each projective space, i.e., for each value of \(r\).
To reduce casework, define:
\[\epsilon_0 = \epsilon_0(g) = \begin{cases}
1 & \text{if \(g = 0\);} \\
0 & \text{if \(g \neq 0\).}
\end{cases} \quad \text{and} \quad \epsilon_1 = \epsilon_1(d, g) = \begin{cases}
1 & \text{if \(d > g + r\);} \\
0 & \text{if \(d = g + r\).}
\end{cases}
\]

\begin{prop} \label{prop:largeparams}
Let \((d, g, r, \ell, m)\) be a good tuple. Then one of the
arguments of Section~\ref{subsec:largeparams} may be applied unless
\begin{equation} \label{box}
d \leq g + 2r - 1, \quad g \leq r - 1, \quad \text{and} \quad m \leq r - 2 + \epsilon_0,
\end{equation}
or unless
\begin{equation} \label{to-exc}
(d, g, r, \ell, m - (r - 1)) \ \text{lies in \eqref{XX}}.
\end{equation}
\end{prop}
\begin{proof}
If \(g \geq r\), then we may apply Lemma~\ref{peel_onion}.
If \(m \geq r - 1 + \epsilon_0\), then we may apply Lemma~\ref{pancake_onions},
unless \((d, g, r, \ell, m - (r - 1))\) lies in \eqref{XX}.

We may thus assume \(m \leq r - 2 + \epsilon_0 \leq r - 1\).
For any \(d' > g + r\), this implies \(\rho(d', g, r) \geq r + 1 \geq m\).
Therefore, if \(d \geq g + 2r\), we may apply Lemma~\ref{gather_lines}.
\end{proof}

For any \emph{fixed} \(r\), conditions \eqref{box} and \eqref{to-exc} describe
a \emph{finite} set of tuples \((d, g, r, \ell, m)\) as promised.
It therefore suffices to prove:

\begin{thm} \label{thm:14}
If \(r \geq 14\), one of the arguments in Section~\ref{sec:inductive}
may be applied, unless \(\ell = m = 0\) and \(\delta = 1\).
\end{thm}

The remainder of this section is devoted to a proof of Theorem~\ref{thm:14},
which is a purely combinatorial exercise.
Since all tuples in \eqref{XX} have \(r \leq 5\), by Proposition~\ref{prop:largeparams},
we may  suppose \eqref{box} is satisfied.

\subsection{\texorpdfstring{\boldmath The Cases With \(m \neq 0\)}{The Cases With m != 0}}
Our first step will be to show that
Proposition~\ref{master} by itself handles the majority of these cases. 
This consists of showing that we may assign integer values to the various parameters
appearing in Proposition~\ref{master} that satisfy the desired inequalities.
We shall accomplish this using the following lemma, which gives a sufficient criterion for a system
of inequalities to have an integer solution.

\begin{lem} \label{lm:eliminate}
Let \(a_i/b_i\) and \(c_j/d_j\) be rational numbers.
There is an integer \(n\) satisfying
\[n \geq \frac{a_i}{b_i} \ \text{for all \(i\)} \quad \text{and} \quad n \leq \frac{c_j}{d_j} \ \text{for all \(j\)},\]
provided that, for all \(i\) and \(j\), we have
\[\frac{a_i}{b_i} \leq \frac{c_j}{d_j} - \frac{(b_i - 1)(d_j - 1)}{b_i d_j}.\]
\end{lem}
\begin{proof}
The collection of intervals \([a_i/b_i, c_j/d_j]\) is closed under intersection,
so it suffices to check that there is an integer \(n\)
satisfying
\begin{equation} \label{abncd}
\frac{a}{b} \leq n \leq \frac{c}{d}
\end{equation}
provided that
\begin{equation} \label{abncd-target}
\frac{a}{b} \leq \frac{c}{d} - \frac{(b - 1)(d - 1)}{bd}.
\end{equation}
For this, we note that \eqref{abncd} is equivalent to
\[\frac{a - 1}{b} < n < \frac{c + 1}{d}.\]
Since any interval of length greater than \(1\) contains an integer, it suffices to have
\[\frac{c + 1}{d} - \frac{a - 1}{b} > 1,\]
or equivalently,
\[\frac{c + 1}{d} - \frac{a - 1}{b} \geq 1 + \frac{1}{bd}.\]
Upon rearrangement this yields \eqref{abncd-target} as desired.
\end{proof}

\noindent
The following simple observations will be used repeatedly in what follows.

\begin{lem} \label{deltam} If \(r\) is even and \(\delta\) is an integer, then \(\delta \equiv m\) mod \(2\).
\end{lem}
\begin{proof}
This follows directly from examining the formula
\[\delta = \frac{2d + 2g - 2r + 2\ell + (r + 1)m}{r - 1}. \qedhere\]
\end{proof}

\begin{lem} \label{lm:mbar-rho}
In Propositions~\ref{master} and~\ref{master-111}, suppose that \(d' \neq g + r\) if \(d \neq g + r\). Then
\[\bar{m} \leq \rho(d' - 1, g, r - 1) \quad \text{and} \quad \bar{m} \leq \rho(d' - 2, g, r - 2).\]
\end{lem}
\begin{proof}
We divide into cases as follows.

\medskip
\noindent
\textbf{\boldmath Case 1: \(d = g + r\).} This implies \(d' = d = g + r\); thus
\(g = \rho(d, g, r) = \rho(d' - 1, g, r - 1) = \rho(d' - 2, g, r - 2)\).
On the other hand,
because \(m \leq \rho(d, g, r)\), we have
\(\bar{m} = m - m' \leq m \leq \rho(d, g, r)\).

\medskip
\noindent
\textbf{\boldmath Case 2: \(d > g + r\).} This implies \(d' \geq g + r + 1\); thus \(\rho(d' - 1, g, r - 1) \geq g + r\)
and \(\rho(d' - 2, g, r - 2) \geq g + r - 1\).
On the other hand, because
\(m \leq r - 2 + \epsilon_0\), we have
\(\bar{m} = m - m' \leq m \leq r - 1\).
\end{proof}

\noindent
The first main step of our combinatorial analysis is the following:

\begin{prop} \label{lm:master-imp} Let \((d, g, r, \ell, m)\) be a good tuple satisfying \eqref{box} with \(m \neq 0\)
and \(r \geq 14\).
Then the conditions of Proposition~\ref{master}
can be satisfied unless one of the following holds:
\begin{enumerate}
\item \label{nm-1} \(\ell = 0\), and \(\delta\) is an integer with the same parity as \(r\),
and \(\delta < r\) if \(r\) is even.
\item \label{nm-2} \(\ell < \delta < \ell + 2\) and \(g > 0\).
\item \label{r-3} \((d, g, r, \ell, m) = (3k + 1, k, 2k, 0, 2k - 3)\) for some \(k\).
\item \label{r-4} \((d, g, r, \ell, m) = (k + 1, 0, k, 0, 1)\) for some \(k\).
\end{enumerate}
\end{prop}
\begin{proof}
We will show a slightly stronger statement: The conditions of Proposition~\ref{master}
can be satisfied, together with the additional conditions that
\[m' \neq m \ \text{if} \ g = 0, \qand d' \neq g + r \ \text{if} \ d \neq g + r,\]
unless either one of the above-mentioned conditions holds or
\[(d, g, r, \ell, m) = (4k - 2, 0, 2k, 0, 1) \ \text{or} \ (4k + 1, 2k - 1, 2k, 0, 2k - 3) \quad \text{for some \(k\)}.\]

This is indeed a stronger statement because if \((d, g, r, \ell, m) = (4k - 2, 0, 2k, 0, 1)\),
then the conditions of Proposition~\ref{master} can be satisfied by taking:
\[\ell' = 0, \quad m' = m = 1, \quad d' = d = 4k - 2, \quad n_1 = 3,\]
and if \((d, g, r, \ell, m) = (4k + 1, 2k - 1, 2k, 0, 2k - 3)\), then the conditions of Proposition~\ref{master} can be satisfied by taking:
\[\ell' = 0, \quad m' = 1, \quad d' = 4k - 1, \quad n_1 = 2k - 1.\]

The advantage of this first additional condition is that
\(m' \neq m\) implies \(\bar{m} \neq 0\).
In combination with Lemma~\ref{lm:mbar-rho} (which applies because of the second additional condition),
these conditions therefore imply that
 \((d' - 1, g, r - 1, \bar{\ell}, \bar{m})\) is good provided only that
\[\bar{\ell} \leq \frac{r - 1}{2}.\]

A further advantage of this second additional condition is that
\(\sum n_i\) can be any integer of the form \((r - 1)m' - 2n\) where
\[0 \leq n \leq \kappa m' \quad \text{where} \quad \kappa = \kappa(d, g, r) \colonequals \begin{cases}
\frac{r - 4}{2} & \text{if \(r\) is even;} \\
\frac{r - 5}{2} & \text{if \(r\) is odd and \((d, g) = (r + 1, 1)\);} \\
\frac{r - 3}{2} & \text{if \(r\) is odd and \((d, g) \neq (r + 1, 1)\).}
\end{cases}\]

We next write down a system of inequalities such that an integer solution
(for \(\ell'\), \(m'\), \(d'\), and \(n\)) to this system
guarantees that the conditions of Proposition~\ref{master}
can be satisfied:
\begin{gather}
0 \leq m' \leq m - \epsilon_0 \label{gath1_1}\\
2m' + \ell' \leq r - 2 \label{gath1_2}\\
g + r + \epsilon_1 \leq d' \leq d \\
0 \leq n \leq \kappa m' \\
0 \leq \ell' \leq \ell \label{gath1_5}\\
\left|\delta - \left[\ell' + 2(d - d') + (r - 1)m' - 2n \right] \right| \leq 1 - \frac{1}{r - 1} \\
\ell - \ell' + n \leq \frac{r - 1}{2}.
\end{gather}
Using \eqref{gath1_5}, the inequality \eqref{gath1_2} follows from \(2m' + \ell \leq r - 2\).
We introduce a new variable \(s = d' + n\).
Replacing \eqref{gath1_2} with \(2m' + \ell \leq r - 2\) and
rewriting the resulting system in terms of \(s\) and \(n\), we obtain:
\begin{gather}
0 \leq m' \leq m - \epsilon_0  \label{gath2_1} \\
2m' + \ell \leq r - 2  \label{gath2_2}\\
s - d \leq n \leq s - g - r - \epsilon_1 \label{gath2_3}\\
0 \leq n \leq \kappa m' \label{gath2_4}\\
0 \leq \ell' \leq \ell  \label{gath2_5}\\
\delta - \left[2d - 2s + (r - 1)m' \right] - \frac{r - 2}{r - 1} \leq \ell' \leq \delta - \left[2d - 2s + (r - 1)m'\right] + \frac{r - 2}{r - 1}  \label{gath2_6} \\
n \leq \frac{r - 1}{2} + \ell' - \ell \label{gath2_7}.
\end{gather}

We use Lemma~\ref{lm:eliminate} to eliminate the variable \(n\).
In other words, \eqref{gath2_3}, \eqref{gath2_4}, and \eqref{gath2_7} involve \(n\).
Applying Lemma~\ref{lm:eliminate}, there is such an integer \(n\) provided that:
\begin{align}
s - d &\leq s - g - r - \epsilon_1 \label{gath3_1}\\
s - d &\leq \kappa m' \\
s - d &\leq \frac{r - 1}{2} + \ell' - \ell \\
0 &\leq s - g - r - \epsilon_1 \\
0 &\leq \kappa m' \label{gath3_5}\\
0 &\leq \frac{r - 1}{2} + \ell' - \ell.
\end{align}
Inequalities \eqref{gath3_1} and \eqref{gath3_5} are immediate
(they follow from \(d \geq g + r + \epsilon_1\) and \(m' \geq 0\)
respectively).
Rearranging the remaining inequalities, and including the inequalities  \eqref{gath2_1}, \eqref{gath2_2}, \eqref{gath2_5}, and \eqref{gath2_6} that do not involve \(n\),
it therefore suffices to show that there is an integer solution to the following system:
\begin{gather*}
0 \leq m' \leq m - \epsilon_0 \\
2m' + \ell \leq r - 2 \\
0 \leq \ell' \leq \ell \\
\delta - \left[2d - 2s + (r - 1)m' \right] - \frac{r - 2}{r - 1} \leq \ell' \leq \delta - \left[2d - 2s + (r - 1)m'\right] + \frac{r - 2}{r - 1} \\
s \leq d + \kappa m' \\
s - d + \ell - \frac{r - 1}{2} \leq \ell' \\
g + r + \epsilon_1 \leq s \\
\ell - \frac{r - 1}{2} \leq \ell'.
\end{gather*}
Using Lemma~\ref{lm:eliminate} to eliminate the variable \(\ell'\)
replaces the inequalities involving \(\ell'\) with:
\begin{align}
0 &\leq \ell,\label{gath4_1}\\
0 &\leq \delta - \left[2d - 2s + (r - 1)m'\right] + \frac{r - 2}{r - 1} \\
\delta - \left[2d - 2s + (r - 1)m' \right] - \frac{r - 2}{r - 1} &\leq \ell \\
\delta - \left[2d - 2s + (r - 1)m' \right] - \frac{r - 2}{r - 1} &\leq \delta - \left[2d - 2s + (r - 1)m'\right] + \frac{r - 2}{r - 1} - \frac{(r - 2)^2}{(r - 1)^2} \label{gath4_4}\\
s - d + \ell - \frac{r - 1}{2} &\leq \ell \\
s - d + \ell - \frac{r - 1}{2} &\leq \delta - \left[2d - 2s + (r - 1)m'\right] + \frac{r - 2}{r - 1} - \frac{r - 2}{2r - 2}\\
\ell -  \frac{r - 1}{2} &\leq \ell \label{gath4_7}\\
\ell - \frac{r - 1}{2} &\leq \delta - \left[2d - 2s + (r - 1)m'\right] + \frac{r - 2}{r - 1} - \frac{r - 2}{2r - 2}. 
\end{align}
Inequalities \eqref{gath4_1}, \eqref{gath4_4}, and \eqref{gath4_7} are immediate.
Simplifying the remaining ones, and including the inequalities that do not involve \(\ell'\),
we obtain:
\begin{align}
s &\leq d + \frac{(r - 1)m' - \delta}{2} + \frac{\ell}{2} + \frac{r - 2}{2r - 2}\label{gath5_1up}\\
s &\leq d + \frac{r - 1}{2} \\
s &\leq d + \kappa m' \\
s &\geq d + \frac{(r - 1)m' - \delta}{2} - \frac{r - 2}{2r - 2} \label{gath5_1low}\\
s &\geq d + (r - 1)m' - \delta  + \ell - \frac{r^2 - r - 1}{2r - 2} \\
s &\geq d + \frac{(r - 1)m' - \delta}{2} + \frac{\ell}{2} - \frac{r^2 - r - 1}{4r - 4} \\
s &\geq g + r + \epsilon_1 \\
0 \leq m' &\leq m - \epsilon_0 \\
2m' + \ell &\leq r - 2.
\end{align}

We now eliminate the variable \(s\). Mostly, we will accomplish this by using
Lemma~\ref{lm:eliminate}, except we will compare \eqref{gath5_1up} and \eqref{gath5_1low}
by ad-hoc methods. Namely, for \eqref{gath5_1up} and \eqref{gath5_1low}, we want there to be an integer
between 
\[d + \frac{(r - 1)m' - \delta}{2} - \frac{r - 2}{2r - 2} \quad \text{and} \quad d + \frac{(r - 1)m' - \delta}{2} + \frac{\ell}{2} + \frac{r - 2}{2r - 2}.\]
By direct inspection, such an integer exists if and only if
\begin{equation} \label{lzoi}
\ell \neq 0 \quad \text{or} \quad (r - 1)m' - \delta \ \text{is not an odd integer}.
\end{equation}
Eliminating \(s\), we therefore have condition~\eqref{lzoi} plus the following system of inequalities:
\begin{align}
d + \frac{(r - 1)m' - \delta}{2} - \frac{r - 2}{2r - 2} &\leq d + \frac{r - 1}{2} - \frac{2r - 3}{4r - 4} \label{gath6_1}\\
d + \frac{(r - 1)m' - \delta}{2} - \frac{r - 2}{2r - 2} &\leq d + \kappa m' \label{gath6_2}\\
d + (r - 1)m' - \delta  + \ell - \frac{r^2 - r - 1}{2r - 2} &\leq d + \frac{(r - 1)m' - \delta}{2} + \frac{\ell}{2} + \frac{r - 2}{2r - 2} - \frac{(2r - 3)^2}{(2r - 2)^2} \label{gath6_3}\\
d + (r - 1)m' - \delta  + \ell - \frac{r^2 - r - 1}{2r - 2} &\leq d + \frac{r - 1}{2} - \frac{2r - 3}{4r - 4} \label{gath6_4}\\
d + (r - 1)m' - \delta  + \ell - \frac{r^2 - r - 1}{2r - 2} &\leq d + \kappa m' \label{gath6_5}\\
d + \frac{(r - 1)m' - \delta}{2} + \frac{\ell}{2} - \frac{r^2 - r - 1}{4r - 4} &\leq d + \frac{(r - 1)m' - \delta}{2} + \frac{\ell}{2} + \frac{r - 2}{2r - 2} - \frac{(4r - 5)(2r - 3)}{(4r - 4)(2r - 2)} \label{gath6_6}\\
d + \frac{(r - 1)m' - \delta}{2} + \frac{\ell}{2} - \frac{r^2 - r - 1}{4r - 4} &\leq d + \frac{r - 1}{2} - \frac{4r - 5}{8r - 8} \label{gath6_7}\\
d + \frac{(r - 1)m' - \delta}{2} + \frac{\ell}{2} - \frac{r^2 - r - 1}{4r - 4} &\leq d + \kappa m' \label{gath6_8}\\
g + r + \epsilon_1 &\leq d + \frac{(r - 1)m' - \delta}{2} + \frac{\ell}{2} + \frac{r - 2}{2r - 2} \label{gath6_9}\\
g + r + \epsilon_1 &\leq d + \frac{r - 1}{2} \label{gath6_10}\\
g + r + \epsilon_1 &\leq d + \kappa m'\label{gath6_11}\\
0 \leq m' &\leq m - \epsilon_0 \label{gath6_12}\\
2m' + \ell &\leq r - 2. \label{gath6_13}
\end{align}
Inequalities \eqref{gath6_6}, \eqref{gath6_10}, and \eqref{gath6_11} are immediate.
Moreover, \eqref{gath6_1}, \eqref{gath6_3}, \eqref{gath6_4}, and \eqref{gath6_7} all follow from
\[(r - 1)m' - \delta  + \ell \leq \frac{r^2 - 2r}{r - 1},\]
and \eqref{gath6_5} and \eqref{gath6_8} follow from
\[(r - 1)m' - \delta  + \ell \leq \kappa m' + \frac{r^2 - r - 2}{2r - 2}.\]
The above system of inequalities
therefore follows from the following system:
\begin{align}
(r - 1)m' - \delta &\leq 2\kappa m' + \frac{r - 2}{r - 1} \label{eqs_final_list_start} \\
(r - 1)m' - \delta  + \ell &\leq \frac{r^2 - 2r}{r - 1} \\
(r - 1)m' - \delta  + \ell &\leq \kappa m' + \frac{r^2 - r - 2}{2r - 2}  \\
g + r + \epsilon_1 &\leq d + \frac{(r - 1)m' - \delta}{2} + \frac{\ell}{2} + \frac{r - 2}{2r - 2} \\
0 &\leq m' \leq m - \epsilon_0 \\
2m' + \ell &\leq r - 2.\label{eqs_final_list_end}
\end{align}
All that remains is therefore to show that there is an integer \(m'\)
satisfying \eqref{eqs_final_list_start}--\eqref{eqs_final_list_end} plus condition~\eqref{lzoi}.
For this we divide into three cases as follows.

\medskip

\noindent
\textbf{\boldmath Case 1: \(\ell = 0\) and \(r\) is even and \(\delta\) is an even integer.}
In this case we will take \(m' = 2\), which evidently satisfies \eqref{lzoi}.
Substituting \(\ell = 0\) and \(m' = 2\) into \eqref{eqs_final_list_start}--\eqref{eqs_final_list_end}, it remains only to verify:
\begin{align}
2(r - 1) - \delta &\leq \frac{2r^2 - 9r + 6}{r - 1} \label{gath7_1} \\
2(r - 1) - \delta &\leq \frac{r^2 - 2r}{r - 1} \label{gath7_2} \\
2(r - 1) - \delta &\leq \frac{3r^2 - 11r + 6}{2r - 2} \label{gath7_3}\\
g + r + \epsilon_1 &\leq d + \frac{2(r - 1) - \delta}{2} + \frac{r - 2}{2r - 2}\label{gath7_4}\\
0 &\leq 2 \leq m - \epsilon_0 \label{gath7_5}\\
4 &\leq r - 2.\label{gath7_6}
\end{align}
Note that \(\delta \geq r\) by our exclusion of the cases \(\delta < r\) in Proposition \ref{lm:master-imp}\eqref{nm-1}.
This implies \eqref{gath7_1}, \eqref{gath7_2}, and \eqref{gath7_3}.
Since \(d \geq g + r + \epsilon_1\), inequality \eqref{gath7_4} follows from
\(\delta \leq 2r - 2\). Inequality \eqref{gath7_5} follows from
\(m \geq 2 + \epsilon_0\), and \eqref{gath7_6} is immediate.
All that remains is therefore to check the following pair of inequalities:
\begin{align}
\delta &\leq 2r - 2 \label{gath8_1}\\
m &\geq 2 + \epsilon_0.\label{gath8_2}
\end{align}
For \eqref{gath8_1}, we note that
\[\delta = \frac{2(d - g - 2r + 1) + 4g + (r + 1)m + 2r - 2}{r - 1} \leq \frac{4(r - 1) + (r + 1)(r - 1) + 2r - 2}{r - 1} \leq 2r - 2.\]
For \eqref{gath8_2}, we note that \(m\) is even by Lemma~\ref{deltam};
in particular, \(m \geq 2\).
Inequality \eqref{gath8_2} thus holds unless \(g = 0\) and \(m = 2\). But in this case,
\[\delta = \frac{2d + 2}{r - 1} \leq \frac{2(2r - 1) + 2}{r - 1} < r,\]
contradicting our assumption that \(\delta \geq r\).

\medskip

\noindent
\textbf{\boldmath Case 2: \(\ell = 0\) and \(r\) is even and \(\delta\) is an odd integer.}
In this case we will take \(m' = 1\), which again evidently satisfies \eqref{lzoi}.
Substituting \(\ell = 0\) and \(m' = 1\) into \eqref{eqs_final_list_start}--\eqref{eqs_final_list_end}, it remains only to verify:
\begin{align}
(r - 1) - \delta &\leq \frac{r^2 - 4r + 2}{r - 1} \label{gath9_1}\\
(r - 1) - \delta &\leq \frac{r^2 - 2r}{r - 1} \label{gath9_2}\\
(r - 1) - \delta &\leq \frac{r^2 - 3r + 1}{r - 1} \label{gath9_3}\\
g + r + \epsilon_1 &\leq d + \frac{(r - 1) - \delta}{2} + \frac{r - 2}{2r - 2}\label{gath9_4}\\
0 &\leq 1 \leq m - \epsilon_0 \label{gath9_5}\\
2 &\leq r - 2.\label{gath9_6}
\end{align}
Inequalities \eqref{gath9_1}, \eqref{gath9_2}, and \eqref{gath9_3} follow from \(\delta \geq 3\), inequality 
\eqref{gath9_4} from \(\delta \leq 2(d - g - r - \epsilon_1) + (r - 1)\),
inequality 
\eqref{gath9_5} from \(m \geq 1 + \epsilon_0\), and  
\eqref{gath9_6} is immediate.
All that remains is therefore to check the following system of inequalities:
\begin{align}
\delta &\geq 3 \label{gath10_1}\\
\delta &\leq 2(d - g - r - \epsilon_1) + (r - 1)\label{gath10_2} \\
m &\geq 1 + \epsilon_0. \label{gath10_3}
\end{align}

For \eqref{gath10_1}, since
\(m \geq 1\), we have \(\delta > 1\). Since \(\delta\) is an odd integer,  \(\delta \geq 3\)
as desired.
For \eqref{gath10_3}, since \(m \geq 1\),
the inequality holds unless \(g = 0\) and \(m = 1\).
But in this case,
\[\delta = \frac{2d}{r - 1} - 1 \leq \frac{2(2r - 1)}{r - 1} - 1 < 5,\]
and so \(\delta = 3\), and so \(d = 2r - 2\).
In other words, writing \(r = 2k\), we have
\[(d, g, r, \ell, m) = (4k - 2, 0, 2k, 0, 1),\]
which one of the cases excluded by assumption.

All that remains is to verify \eqref{gath10_2}.
Note that \(m\) is odd by Lemma~\ref{deltam};
in particular, since \(m \leq r-2+\epsilon_0\) by \eqref{box}, we have one of:
\[m \leq r - 5, \quad m = r - 3, \quad \text{or} \quad m = r - 1,\]
where the final case can only occur if \(g = 0\).
Our argument will be via casework as follows.

\smallskip
\noindent
\textit{Subcase 2.1: \(d \geq g + r + 2\).}
By separately considering the cases \(g = 0\) (in which case \(m \leq r - 1\)) and \(g > 0\)
(in which case \(m \leq r - 3\)), we have
\(4g + (r + 1)m \leq r^2 + 2r - 7\),
with equality only if \(g = r - 1\) and \(m = r - 3\).
Therefore
\begin{align*}
\delta &= 2(d - g - r - 1) + (r + 1) - \frac{(r^2 + 2r - 7) - 4g - (r + 1)m + (2r - 4)(d - g - r - 2)}{r - 1} \\
&\leq 2(d - g - r - \epsilon_1) + (r + 1),
\end{align*}
with equality only if \(g = r - 1\) and \(m = r - 3\) and \(d = g + r + 2 = 2r + 1\).
Since \(\delta\) is an odd integer, we therefore have 
\(\delta \leq 2(d - g - r - \epsilon_1) + (r - 1)\) unless, writing \(r = 2k\), we have
\[(d, g, r, \ell, m) = (2r + 1, r - 1, r, 0, r - 3) = (4k + 1, 2k - 1, 2k, 0, 2k - 3),\]
which is again one of the cases excluded by assumption.

\smallskip
\noindent
\textit{Subcase 2.2: \(d \leq g + r + 1\) and \(m \leq r - 5\).}
We have
\[\delta = \frac{2(d - g - r - 1) + 4g + 2 + (r + 1)m}{r - 1} \leq \frac{4(r - 1) + 2 + (r + 1)(r - 5)}{r - 1} = r + 1 - \frac{6}{r - 1} < r + 1.\]
Since \(\delta\) is an odd integer, this implies \(\delta \leq r - 1\) as desired.

\smallskip
\noindent
\textit{Subcase 2.3: \(d = g + r\) and \(m = r - 3\).}
We have
\[\delta = r - 1 + \frac{4(g - 1)}{r - 1}.\]
Since \(\delta\) is an integer, and \(r - 1\) is odd,
this implies \(g \equiv 1\) mod \(r - 1\),
which since \(0 \leq g \leq r - 1\) in turn implies \(g = 1\),
and so \(\delta = r - 1\).

\smallskip
\noindent
\textit{Subcase 2.4: \(d = g + r + 1\) and \(m = r - 3\).}
We have
\[\delta = r - 1 + \frac{2(2g - 1)}{r - 1}.\]
Since \(\delta\) is an integer, and \(r - 1\) is odd,
this implies \(2g \equiv 1 \equiv r\) mod \(r - 1\),
which since \(0 \leq g \leq r - 1\) in turn implies \(g = r/2\).
Writing \(r = 2k\), we therefore have \(g = k\) and \(d = g + r + 1 = 3k + 1\),
i.e., we have
\[(d, g, r, \ell, m) = (3k + 1, k, 2k, 0, 2k - 3),\]
which is again one of the cases excluded by assumption.

\smallskip
\noindent
\textit{Subcase 2.5: \(d \leq g + r + 1\) and \(m = r - 1\).}
Since \(m \leq r - 2 + \epsilon_0\), we would have \(g = 0\),
and thus \(d = r + 1\). But this would imply
\(\delta = (r^2 + 1)/(r - 1) \notin \zz\), in contradiction to our assumption that \(\delta\)
is an odd integer.

\medskip
\noindent
\textbf{\boldmath Case 3: \(\ell \neq 0\) or \(r\) is odd or \(\delta\) is not an integer.}
If \(\ell \neq 0\) or \(\delta\) is not an integer, then \eqref{lzoi} holds.
Otherwise, the current assumption implies \(r\) is odd,
so \(\delta\) is even (the cases where \(\delta\) is also odd are excluded),
and so \eqref{lzoi} again holds.
We conclude that \eqref{lzoi} is automatic.

All that remains is therefore to check that
there exists an integer \(m'\) satisfying \eqref{eqs_final_list_start}--\eqref{eqs_final_list_end}.
Rearranging to make the bounds on \(m'\) explicit, this is the system:
\begin{align}
m' &\leq \frac{1}{r - 1 - 2\kappa} \cdot \left(\delta + \frac{r - 2}{r - 1}\right) \label{gath11_1}\\
m' &\leq \frac{1}{r - 1} \cdot \left(\delta - \ell + \frac{r^2 - 2r}{r - 1}\right) \\
m' &\leq \frac{1}{r - 1 - \kappa} \cdot \left(\delta - \ell + \frac{r^2 - r - 2}{2r - 2}\right) \label{gath11_3}\\
m' &\leq m - \epsilon_0 \\
m' &\leq \frac{r - 2 - \ell}{2} \\
m' &\geq \frac{1}{r - 1} \cdot \left(\delta - 2(d - g - r - \epsilon_1) - \ell - \frac{r - 2}{r - 1}\right) \label{gath11_1low}\\
m' &\geq 0.
\end{align}
We will compare the \eqref{gath11_1} and \eqref{gath11_3} to \eqref{gath11_1low} using ad-hoc
methods. But first we handle all of the other comparisons
using Lemma~\ref{lm:eliminate}, by verifying the following system of inequalities:
\begin{align*}
0 &\leq \frac{1}{r - 1 - 2\kappa} \cdot \left(\delta + \frac{r - 2}{r - 1}\right) \\
0 &\leq \frac{1}{r - 1} \cdot \left(\delta - \ell + \frac{r^2 - 2r}{r - 1}\right) \\
0 &\leq \frac{1}{r - 1 - \kappa} \cdot \left(\delta - \ell + \frac{r^2 - r - 2}{2r - 2}\right) \\
0 &\leq m - \epsilon_0 \\
0 &\leq \frac{r - 2 - \ell}{2} \\
\frac{1}{r - 1} \cdot \left(\delta - 2(d - g - r - \epsilon_1) - \ell - \frac{r - 2}{r - 1}\right) &\leq \frac{1}{r - 1} \cdot \left(\delta - \ell + \frac{r^2 - 2r}{r - 1}\right) - \frac{(r - 2)^4}{(r - 1)^4} \\
\frac{1}{r - 1} \cdot \left(\delta - 2(d - g - r - \epsilon_1) - \ell - \frac{r - 2}{r - 1}\right) &\leq m - \epsilon_0 \\
\frac{1}{r - 1} \cdot \left(\delta - 2(d - g - r - \epsilon_1) - \ell - \frac{r - 2}{r - 1}\right) &\leq \frac{r - 2 - \ell}{2} - \frac{(r - 1)^2 - 1}{2(r - 1)^2}.
\end{align*}
Substituting in the definition of \(\delta\) and rearranging, these inequalities are equivalent to:
\begin{align}
2(d - g - r) + 4g + 2\ell + (r + 1)m &\geq -r + 2  \\
4(d - g - r) + 8g + (r - 3)(r-2\ell) + 2(r + 1)m &\geq -r^2 + r \\
4(d - g - r) + 8g + (r - 3)(r-2\ell) + 2(r + 1)m &\geq -2r + 2 \\
m - 1 &\geq \epsilon_0 - 1 \\
r-2\ell &\geq -r + 4 \\
2(r - 1)^3 (d - g - r - \epsilon_1) &\geq -5r^3 + 23r^2 - 35r + 18 \\
(2r - 4)(d - g - r - \epsilon_1)  + 4(r - 1 - g) + (r - 3)\ell   \qquad & \label{gath12_7}\\
\phantom{.}  + (r^2 - 3r)(m - 1) + 2(1 - \epsilon_1) &\geq (r - 1)^2 \epsilon_0 - r^2 + 6r \notag \\
(8r - 16)(d - g - r) + 16(r - 1 - g) + (r^2 - 4r + 7)(r - 2\ell) \qquad & \\
\phantom{.} + (4r + 4)(r - 1 - m) + 8(1 - \epsilon_1) &\geq -r^3 + 10r^2 + 5r. \notag
\end{align}
From these expressions we see that all but \eqref{gath12_7} is immediate,
and that \eqref{gath12_7} holds when \(\epsilon_0 = 0\).
But when \(\epsilon_0 = 1\), then \(g = 0\) and \(\epsilon_1 = 1\), and so \eqref{gath12_7} becomes
\[(2r - 4)(d - r - 1) + (r^2 - 3r)(m - 1) + (r - 3)\ell \geq 5,\]
which holds unless \(d = r + 1\) and \(m = 1\) and \(\ell = 0\), or equivalently unless
\[(d, g, r, \ell, m) = (k + 1, 0, k, 0, 1),\]
which is again one of the cases excluded by assumption.

\medskip

All that remains is our promised ad-hoc comparison \eqref{gath11_1} and \eqref{gath11_3} to \eqref{gath11_1low}. That is, we want to show that there are integers between:
\begin{align*}
\frac{1}{r - 1} \cdot \left(\delta - 2(d - g - r - \epsilon_1) - \ell - \frac{r - 2}{r - 1}\right) \quad &\text{and} \quad \frac{1}{r - 1 - 2\kappa} \cdot \left(\delta + \frac{r - 2}{r - 1}\right) \\
\frac{1}{r - 1} \cdot \left(\delta - 2(d - g - r - \epsilon_1) - \ell - \frac{r - 2}{r - 1}\right) \quad &\text{and} \quad \frac{1}{r - 1 - \kappa} \cdot \left(\delta - \ell + \frac{r^2 - r - 2}{2r - 2}\right)
\end{align*}

\smallskip
\noindent
\textit{Remark \(\ast\): If \((d, g) = (r + 1, 1)\) then \(\delta < \ell + 2\).}
Indeed, if \((d, g) = (r + 1, 1)\), then \(m \leq \rho(d, g, r) = 1\), and so
\[\delta = \frac{4 + 2\ell + (r + 1)m}{r - 1} \leq \frac{r + 5 + 2\ell}{r - 1} < \ell + 2.\]

\smallskip
\noindent
\textit{Subcase 3.1: \(\delta < \ell + 1 + 2(d - g - r - \epsilon_1)\).}
In this case the lower bound is nonpositive. We have already shown that both upper
bounds are nonnegative above, so there is nothing more to check.

\smallskip
\noindent
\textit{Subcase 3.2: \(\delta \geq \ell + 3\).} By Remark \(\ast\) above,
\(\kappa \geq \frac{r - 4}{2}\).
It suffices to show that there are integers between:
\begin{align*}
\frac{1}{r - 1} \cdot \left(\delta - \ell - \frac{r - 2}{r - 1}\right) \quad &\text{and} \quad \frac{1}{r - 1 - 2 \cdot \frac{r - 4}{2}} \cdot \left(\delta - \ell + \frac{r - 2}{r - 1}\right) = \frac{1}{3} \cdot \left(\delta - \ell + \frac{r - 2}{r - 1}\right)\\
\frac{1}{r - 1} \cdot \left(\delta - \ell - \frac{r - 2}{r - 1}\right) \quad &\text{and} \quad \frac{1}{r - 1 - \frac{r - 4}{2}} \cdot \left(\delta - \ell + \frac{r^2 - r - 2}{2r - 2}\right) = \frac{1}{r + 2} \cdot \left(2(\delta - \ell) + \frac{r^2 - r - 2}{r - 1}\right)
\end{align*}
By Lemma~\ref{lm:eliminate}, this follows from the following inequalities:
\begin{align*}
\frac{1}{r - 1} \cdot \left(\delta - \ell - \frac{r - 2}{r - 1}\right) &\leq \frac{1}{3} \cdot \left(\delta - \ell + \frac{r - 2}{r - 1}\right) - \frac{(3(r - 1) - 1) ((r - 1)^2 - 1)}{3(r - 1)^3} \\
\frac{1}{r - 1} \cdot \left(\delta - \ell - \frac{r - 2}{r - 1}\right) &\leq \frac{1}{r + 2} \cdot \left(2(\delta - \ell) + \frac{r^2 - r - 2}{r - 1}\right) - \frac{((r + 2)(r - 1) - 1)((r - 1)^2 - 1)}{(r + 2)(r - 1)^3}
\end{align*}
But these are immediate for \(\delta - \ell \geq 3\).

\smallskip
\noindent
\textit{Subcase 3.3: \(\ell + 1 + 2(d - g - r - \epsilon_1) \leq \delta < \ell + 3\).}
These inequalities force \(d = g + r + \epsilon_1\), or equivalently
\begin{equation} \label{dgr}
d = g + r \quad \text{or} \quad d = g + r + 1.
\end{equation}
The inequality \(\delta < \ell + 3\) also implies
\[\frac{1}{r - 1} \cdot \left(\delta - \ell - \frac{r - 2}{r - 1}\right) \leq 1.\]
It therefore suffices to show
\[\frac{1}{r - 1 - 2\kappa} \cdot \left(\delta + \frac{r - 2}{r - 1}\right) \geq 1 \quad \text{and} \quad \frac{1}{r - 1 - \kappa} \cdot \left(\delta - \ell + \frac{r^2 - r - 2}{2r - 2} \right) \geq 1,\]
or upon rearrangement
\[\delta \geq r - 2 - 2\kappa + \frac{1}{r - 1} \quad \text{and} \quad \delta \geq \ell + \frac{r}{2} - 1 - \kappa + \frac{1}{r - 1}.\]

\smallskip
\noindent
Subsubcase 3.3.1: \(g = 0\). By \eqref{dgr}, we have \(d = r + 1\).
Since \(g = 0\), we have \((d, g) \neq (r + 1, 1)\).
Our goal is thus to show:
\[\delta \geq \begin{cases}
1 + \frac{1}{r - 1} & \text{if \(r\) is odd,} \\
2 + \frac{1}{r - 1} & \text{if \(r\) is even;}
\end{cases}
\quad \text{and} \quad
\delta \geq \ell + \begin{cases}
\frac{1}{2} + \frac{1}{r - 1} & \text{if \(r\) is odd,} \\
1 + \frac{1}{r - 1} & \text{if \(r\) is even.}
\end{cases}\]
When \(\ell = 0\), the first inequality implies the second.  In this case, recall that  \((d, g, r, \ell, m) = (k + 1, 0, k, 0, 1)\) is excluded
by assumption. Therefore \(m \geq 2\), which implies the first inequality because
\[\delta = \frac{2 + (r + 1)m}{r - 1} \geq \frac{2 + 2(r + 1)}{r - 1} \geq 2 + \frac{1}{r - 1}.\]

Now suppose \(\ell \geq 1\).  Note that \(\delta \geq \ell + 1 \geq 2\),
which implies the first inequality unless \(r\) is even and \(\delta = \ell + 1\).
Similarly, \(\delta \geq \ell + 1\) implies the second inequality
unless \(r\) is even and \(\delta = \ell + 1\).
It thus remains only to show that it is impossible to have \(\delta = \ell + 1\) when \(r\) is even.
To see this, observe that
\[\ell + 1 = \delta = \frac{2 + 2\ell + (r + 1)m}{r - 1}\]
implies
\[(r - 3)(\ell + 1) = (r + 1)m.\]
But if \(r\) were even, then this would imply \((r + 1) \mid (\ell + 1)\),
which forces \(\ell + 1 \geq r + 1\), contradicting our assumption that \(\ell \leq r/2\).

\smallskip
\noindent
Subsubcase 3.3.2: \(g > 0\). We excluded the cases \(\ell < \delta < \ell + 2\) in Proposition \ref{lm:master-imp}\eqref{nm-2}.
Since \(\ell + 1 \leq \delta < \ell + 3\), we therefore have \(\ell + 2 \leq \delta < \ell + 3\).
Moreover, by Remark \(\ast\), we have \((d, g) \neq (r + 1, 1)\).
As in the previous subsubcase, our goal thus is to show:
\[\delta \geq \begin{cases}
1 + \frac{1}{r - 1} & \text{if \(r\) is odd,} \\
2 + \frac{1}{r - 1} & \text{if \(r\) is even;}
\end{cases}
\quad \text{and} \quad
\delta \geq \ell + \begin{cases}
\frac{1}{2} + \frac{1}{r - 1} & \text{if \(r\) is odd,} \\
1 + \frac{1}{r - 1} & \text{if \(r\) is even.}
\end{cases}\]

Since \(\delta \geq \ell + 2\), the second inequality is immediate in all cases.
Also the first inequality is immediate if \(r\) is odd.
To see the first inequality when \(r\) is even, note that \(\delta \geq \ell + 2 \geq 2\),
so the first inequality holds unless we have equality everywhere, i.e., unless \(\ell = 0\) and \(\delta = 2\).
But this possibility is excluded by assumption
(recall that in Case~3 we have \(\ell \neq 0\) or \(r\) odd or \(\delta\) is not an integer).
\end{proof}

\noindent
The majority of the remaining cases are handled by
Proposition~\ref{master-111}. More precisely:

\begin{prop} \label{lm:cleanup} Let \((d, g, r, \ell, m)\) be a good tuple satisfying \eqref{box}
with \(m \neq 0\) and \(r \geq 14\). Suppose in addition that either condition~\eqref{nm-1} or~\eqref{nm-2} of Lemma~\ref{lm:master-imp}
is satisfied.
Then the conditions of Proposition~\ref{master-111} can be satisfied
unless one of the following holds:
\begin{enumerate}
\item \label{r-1} \((d, g, r, \ell, m) = (4k, 0, 2k + 1, 0, 1)\) for some \(k\).
\item \label{r-2} \((d, g, r, \ell, m) = (4k + 1, 2k - 1, 2k + 1, 0, 1)\) for some \(k\).
\end{enumerate}
\end{prop}
\begin{proof}
We separately consider the following three cases:

\medskip
\noindent
\textbf{\boldmath Case 1: \(\ell < \delta < \ell + 2\) and \(g > 0\).}
We take
\(m' = 0\), \(\ell' = \ell\), and \(d' = d\),
which satisfy the conditions of Proposition~\ref{master-111} by Lemma~\ref{lm:mbar-rho}.

\medskip
\noindent
\textbf{\boldmath Case 2: \(\ell = 0\), and \(\delta\) and \(r\) are even integers with \(\delta < r\).}
We take
\[m' = 1, \quad \ell' = 0, \quad d' = d, \quad \text{and} \quad n_1 = \delta - 1.\]
Applying Lemma~\ref{lm:mbar-rho},
the conditions of Proposition~\ref{master-111} are satisfied provided that
\[m > 1 \quad \text{and}\]
\[2 \leq n_1 = \delta - 1 \leq r - 1 \quad \text{with \(n_1 = \delta - 1 \neq 2\) if \((d', g) = (r + 1, 1)\)}.\]
Since \(\delta\) is an even integer, \(\delta - 1 \neq 2\);
since \(\delta < r\) we have \(\delta - 1 \leq r - 1\).
All that remains to check is therefore that \(m \geq 2\) and that \(2 \leq \delta - 1\), which since \(\delta\)
is an even integer is equivalent to \(\delta > 2\).

To see this, we first apply Lemma~\ref{deltam} to conclude that \(m\) is even.
Since \(m \neq 0\), this implies \(m \geq 2\) as desired. This in turn implies \(\delta > 2\) because
\[\delta = \frac{2d + 2g - 2r + 2\ell + (r + 1)m}{r - 1} \geq \frac{2(r + 1)}{r - 1} > 2.\]

\medskip
\noindent
\textbf{\boldmath Case 3: \(\ell = 0\), and \(\delta\) and \(r\) are odd integers.}
As in the proof of Proposition~\ref{lm:master-imp}, we show a slightly stronger statement:
The conditions of Proposition~\ref{master-111}
can be satisfied, together with the additional conditions that
\[m' \neq m - 1 \ \text{if} \ g = 0, \quad \text{and} \quad d' \neq g + r \ \text{if} \ d \neq g + r,\]
unless either one of the above-mentioned conditions holds or
\[(d, g, r, \ell, m) = (3k + 1, k - 1, 2k + 1, 0, 1) \quad \text{for some \(k\)}.\]
This is indeed a stronger statement because if \((d, g, r, \ell, m) = (3k + 1, k - 1, 2k + 1, 0, 1)\),
then the conditions of Proposition~\ref{master-111} can be satisfied by taking:
\[\ell' = 0, \quad m' = 0, \quad \text{and} \quad d' = d - 1 = 3k.\]

Again
as in the proof of
Proposition~\ref{lm:master-imp}, 
Lemma \ref{lm:mbar-rho} guarantees that
the tuples
\((d' - 1, g, r - 1, \bar{\ell}, \bar{m})\), \((d' - 1, g, r - 1, \bar{\ell}, \bar{m} - 1)\), and \((d' - 2, g, r - 2, \bar{\ell}, \bar{m})\), are all good provided only that
\[\bar{\ell} \leq \frac{r - 3}{2}.\]

With \(\kappa\) as in the proof of
Proposition~\ref{lm:master-imp}, our task is thus to show that the following
system of inequalities can be satisfied
for integers \(\ell'\), \(m'\), \(d'\), and \(n\):
\begin{gather}
0 \leq m' \leq m - 1 - \epsilon_0 \\
2m' + \ell' \leq r - 3 \\
g + r + \epsilon_1 \leq d' \leq d \\
0 \leq n \leq \kappa m' \\
0 \leq \ell' \leq \ell = 0 \label{gath13_5} \\
\left|\delta - \left[1 + \ell' + 2(d - d') + (r - 1)m' - 2n \right] \right| \leq 1 - \frac{1}{r - 1} \label{gath13_6} \\
\ell - \ell' + n \leq \frac{r - 3}{2}.
\end{gather}
Inequalities \eqref{gath13_5} and \eqref{gath13_6} are satisfied by taking
\[\ell' = 0 \quad \text{and} \quad n = d - d' + \frac{r - 1}{2} m' - \frac{\delta - 1}{2}.\]
Substituting these into the remaining inequalities and rearranging,
we reduce to the system of inequalities:
\begin{gather}
0 \leq m' \leq m - 1 - \epsilon_0 \label{gath14_1}\\
m' \leq \frac{r - 3}{2} \label{gath14_2} \\
g + r + \epsilon_1 \leq d' \leq d \label{gath14_3} \\
d + \frac{r - 1}{2} m' - \frac{\delta - 1}{2} - \kappa m' \leq d' \leq d + \frac{r - 1}{2} m' - \frac{\delta - 1}{2} \label{gath14_4} \\
d + \frac{r - 1}{2} m' - \frac{\delta - 1}{2} - \frac{r - 3}{2} \leq d'. \label{gath14_5}
\end{gather}
All bounds on \(d'\) are integers because \(r\) and \(\delta\) are both odd integers.
Using Lemma~\ref{lm:eliminate} to eliminate \(d'\) replaces \eqref{gath14_3}--\eqref{gath14_5} with:
\begin{align}
g + r + \epsilon_1 &\leq d \label{gath15_1} \\
g + r + \epsilon_1 &\leq d + \frac{r - 1}{2} m' - \frac{\delta - 1}{2} \\
d + \frac{r - 1}{2} m' - \frac{\delta - 1}{2} - \kappa m' &\leq d \\
d + \frac{r - 1}{2} m' - \frac{\delta - 1}{2} - \kappa m' &\leq d + \frac{r - 1}{2} m' - \frac{\delta - 1}{2} \label{gath15_4} \\
d + \frac{r - 1}{2} m' - \frac{\delta - 1}{2} - \frac{r - 3}{2} &\leq d \\
d + \frac{r - 1}{2} m' - \frac{\delta - 1}{2} - \frac{r - 3}{2} &\leq d + \frac{r - 1}{2} m' - \frac{\delta - 1}{2}. \label{gath15_6}
\end{align}
Inequalities \eqref{gath15_1} and \eqref{gath15_6} always hold, while \eqref{gath15_4} is implied by \(m' \geq 0\).
Rearranging the others, and including \eqref{gath14_1} and \eqref{gath14_2},
we arrive at the system:
\begin{align*}
m' &\geq 0 \\
m' &\geq \frac{\delta - 1 - 2(d - g - r - \epsilon_1)}{r - 1} \\
m' &\leq \frac{\delta - 1}{r - 1 - 2\kappa} \\
m' &\leq \frac{\delta + r - 4}{r - 1} \\
m' &\leq m - 1 - \epsilon_0 \\
m' &\leq \frac{r - 3}{2}.
\end{align*}
Applying Lemma~\ref{lm:eliminate} to eliminate \(m'\), we reduce to the system:
\begin{align}
0 &\leq \frac{\delta - 1}{r - 1 - 2\kappa} \label{gath16_1}\\
0 &\leq \frac{\delta + r - 4}{r - 1} \label{gath16_2} \\
0 &\leq m - 1 - \epsilon_0 \label{gath16_3} \\
0 &\leq \frac{r - 3}{2} \label{gath16_4} \\
\frac{\delta - 1 - 2(d - g - r - \epsilon_1)}{r - 1} &\leq \frac{\delta - 1}{r - 1 - 2\kappa} - \frac{\left(\frac{r - 1}{2} - 1\right)\left(\frac{r - 1 - 2\kappa}{2} - 1\right)}{\left(\frac{r - 1}{2}\right)\left(\frac{r - 1 - 2\kappa}{2}\right)} \label{gath16_5} \\
\frac{\delta - 1 - 2(d - g - r - \epsilon_1)}{r - 1} &\leq \frac{\delta + r - 4}{r - 1} - \frac{\left(\frac{r - 1}{2} - 1\right)^2}{\left(\frac{r - 1}{2}\right)^2} \label{gath16_6} \\
\frac{\delta - 1 - 2(d - g - r - \epsilon_1)}{r - 1} &\leq m - 1 - \epsilon_0 \label{gath16_7} \\
\frac{\delta - 1 - 2(d - g - r - \epsilon_1)}{r - 1} &\leq \frac{r - 3}{2}. \label{gath16_8}
\end{align}
Since \(\delta\) is an odd integer, we have \(\delta \geq 1\),
which implies \eqref{gath16_1} and \eqref{gath16_2}.
The inequality \eqref{gath16_4} is immediate.
Inequality \eqref{gath16_6} follows from \(d \geq g + r + \epsilon_1\), which holds by construction.
For the remaining inequalities \eqref{gath16_3}, \eqref{gath16_5}, \eqref{gath16_7}, and \eqref{gath16_8}, we use the inequality \(\epsilon_1 \leq 1\) to reduce to the system:
\begin{align}
m &\geq 1 + \epsilon_0 \label{gath17_1} \\
\frac{\delta - 1 - 2(d - g - r - 1)}{r - 1} &\leq \frac{\delta - 1}{r - 1 - 2\kappa} - \frac{(r - 3)(r - 3 - 2\kappa)}{(r - 1)(r - 1 - 2\kappa)} \label{gath17_2} \\
\frac{\delta - 1 - 2(d - g - r - 1)}{r - 1} &\leq m - 1 - \epsilon_0 \label{gath17_3} \\
\frac{\delta - 1 - 2(d - g - r - 1)}{r - 1} &\leq \frac{r - 3}{2}. \label{gath17_4}
\end{align}
We divide our analysis as follows:

\smallskip
\noindent
\textit{Inequality \eqref{gath17_1}:}
This inequality asserts that
we do not simultaneously have \(g = 0\) (hence \(\epsilon_0 = 1\)) and \(m = 1\).
So assume \(g = 0\) and \(m = 1\). Then
\[\delta = \frac{2d + 2g - 2r + 2\ell + (r + 1)m}{r - 1} = \frac{2d - r + 1}{r - 1}.\]
Since \(r = g + r \leq d \leq g + 2r - 1 = 2r - 1\), we would have
\[1 < \frac{2r - r + 1}{r - 1} \leq \delta \leq \frac{2(2r - 1) - r + 1}{r - 1} < 5.\]
Since \(\delta\) is an odd integer, \(\delta = 3\), and so \(d = 2r - 2\).
In other words, writing \(r = 2k + 1\), we would have \((d, g, r, \ell, m) = (4k, 0, 2k + 1, 0, 1)\).
But this case is excluded by assumption.

\smallskip
\noindent
\textit{Inequality \eqref{gath17_2} when \((d, g) = (r + 1, 1)\):}
In this case, \(\kappa = (r - 5)/2\)
so upon rearrangement, \eqref{gath17_2} becomes
\[\delta \geq \frac{3r - 3}{r - 5}.\]
However,
\[\delta = \frac{2d + 2g - 2r + 2\ell + (r + 1)m}{r - 1} = \frac{(r + 1)m + 4}{r - 1};\]
since \(\delta\) is an integer, this implies \(2m + 4 \equiv (r + 1)m + 4 \equiv 0\) mod \(r - 1\),
and so \(m \equiv -2\) mod \((r - 1)/2\),
which implies \(m \geq (r - 1)/2 - 2 = (r - 5)/2\).
Therefore,
\[\delta = \frac{(r + 1)m + 4}{r - 1} \geq \frac{(r + 1) \cdot (r - 5)/2 + 4}{r - 1} = \frac{r - 3}{2}.\]
As \(r \geq 14\),
this implies \(\delta \geq (3r - 3)/(r - 5)\) as desired.

\smallskip
\noindent
\textit{Inequality \eqref{gath17_3} when \(m \leq 1 + \epsilon_0\):}
As we have already established \(m \geq 1 + \epsilon_0\), this implies \(m = 1 + \epsilon_0\).
In this case, \eqref{gath17_3} asserts
\[\delta \leq 2d - 2g - 2r - 1.\]
By definition,
\[\delta = \frac{2d + 2g - 2r + (r + 1)(1 + \epsilon_0)}{r - 1} = 2d - 2g - 2r + 5 - \frac{(2r - 4)(d - g - r) + (4r - 6) - (4g + (r + 1)\epsilon_0)}{r - 1}.\]

Since \(d \geq g + r\), we have \((2r - 4)(d - g - r) \geq 0\), with \((2r - 4)(d - g - r) \geq 2r - 4 \geq 4\)
unless equality holds.
Similarly, since \(g \leq r - 1\), we have \(4g + (r + 1)\epsilon_0 \leq 4r - 4\), with \(4g + (r + 1)\epsilon_0 \leq 4r - 8\)
unless equality holds.
Putting this together, we have \(\delta \leq 2d - 2g - 2r + 5 + 2/(r - 1)\),
with \(\delta \leq 2d - 2g - 2r + 5 - 2/(r - 1)\) unless equality holds.
As \(\delta\) is an odd integer, this implies \(\delta \leq 2d - 2g - 2r + 3\).

If \(g = 0\), then \(4g + (r + 1)\epsilon_0 = r + 1\).
Therefore
\(\delta \leq 2d - 2g - 2r + 5 - (3r - 7)/(r - 1)\),
with \(\delta \leq 2d - 2g - 2r + 5 - (5r - 11)/(r - 1) < 2d - 2g - 2r + 1\) unless equality holds.
As \(\delta\) is an odd integer, this implies \(\delta \leq 2d - 2g - 2r - 1\) as desired.

It thus remains only to rule out the cases where \(g > 0\) and
\(\delta = 2d - 2g - 2r + 3\) or \(\delta = 2d - 2g - 2r + 1\).
Upon rearrangement, this is equivalent to
\[(r - 2)(d - g - r) - 2g = -(r - 2) \ \text{or} \ 1\]
Since \(r\) is odd, considering the above equation mod \(2\) implies that \(d - g - r\)
must also be odd.
If \(d - g - r \geq 3\), then since \(g \leq r - 1\), the left-hand side is at least \(3(r - 2) - 2(r - 1) = r - 4\),
which is impossible. Therefore in this case we must have \(d - g - r = 1\). Solving for \(g\)
we obtain \(g = r - 2\) or \(g = (r - 3)/2\).
In other words, writing \(r = 2k + 1\), we would have
\((d, g, r, \ell, m) = (4k + 1, 2k - 1, 2k + 1, 0, 1)\) or
\((d, g, r, \ell, m) = (3k + 1, k - 1, 2k + 1, 0, 1)\).
But these cases are excluded by assumption.

\smallskip
\noindent
\textit{Inequalities \eqref{gath17_2} when \((d, g) \neq (r + 1, 1)\), and \eqref{gath17_3} when \(m \geq 2 + \epsilon_0\), and \eqref{gath17_4} (in all cases):}
Since \((d, g) \neq (r + 1, 1)\) for \eqref{gath17_2}, we may substitute \(\kappa = (r - 3)/2\).
Substituting in the definition of \(\delta\) and rearranging \eqref{gath17_2}, \eqref{gath17_3}, and \eqref{gath17_4}, we obtain:
\begin{align*}
(6r - 10)(d - g - r) + (r^2 - 2r - 3)(m - 1 - \epsilon_0) + (4r - 12)(g + \epsilon_0 - 1) + (r^2 - 6r + 9)\epsilon_0 + 2r - 14 &\geq 0\\
(2r - 4)(d - g - r) + (r^2 - 3r)(m - 2 - \epsilon_0) + 4(r - 1 - g) + (r + 1)(1 - \epsilon_0) + r^2 - 10r + 3 &\geq 0 \\
(4r - 8)(d - g - r) + (2r + 2)(r - 1 - m) - 8(r - 1 - g) + r^3 - 7r^2 - 3r + 9 &\geq 0.
\end{align*}
This establishes the desired inequalities
(using that \(m \geq 2 + \epsilon_0\) for \eqref{gath17_3}).
\end{proof}

Finally, we complete our analysis of the case \(m \neq 0\) by verifying
the desired result in the four remaining one-parameter infinite families of cases:
\begin{description}
\item[\boldmath Case \eqref{r-3} of Lemma~\ref{lm:master-imp}] This follow from Proposition~\ref{master-111} with the following
parameters:
\[\ell' = 0, \quad m' = 2, \quad d' = d = 3k + 1, \quad \text{and} \quad (n_1, n_2) = (3, 2k - 3).\]
\item[\boldmath Case \eqref{r-4} of Lemma~\ref{lm:master-imp}] This follows from Proposition~\ref{two_proj} with \(\epsilon = 0\).
\item[\boldmath Case \eqref{r-1} of Lemma~\ref{lm:cleanup}] This follows from Proposition~\ref{two_proj} with \(\epsilon = 1\).
\item[\boldmath Case \eqref{r-2} of Lemma~\ref{lm:cleanup}] This follows from Proposition~\ref{delta5}.
\end{description}

\subsection{\texorpdfstring{\boldmath The Cases With \(m = 0\) and \(g \neq 0\)}{The Cases With m = 0 and g != 0}}
As in the case \(m \neq 0\), we will begin by showing that Proposition~\ref{master}
handles ``most'' of the cases by itself.

\begin{prop} \label{lm:master-imp-0} Let \((d, g, r, \ell, 0)\) be a good tuple (with \(m = 0\)) satisfying \eqref{box},
such that \(g \neq 0\) and \(r \geq 14\).
Then the conditions of Proposition~\ref{master}
can be satisfied unless one of the following holds:
\begin{enumerate}
\item \label{db} \(\delta \geq \ell + 1 + 2(d - g - r)\).
\item \label{zzoi} \(\ell = 0\) and \(\delta\) is an odd integer.
\end{enumerate}
\end{prop}
\begin{proof}
Our goal is to show the existence of integers \(d'\) and \(\ell'\),
such that \((d' - 1, g, r - 1, \ell - \ell', 0)\) is good
(which is equivalent to \(\ell - \ell' = \bar{\ell} \leq (r - 1)/2\)),
and the inequalities of Proposition~\ref{master} are satisfied:
\begin{gather}
\ell - \ell' \leq \frac{r - 1}{2} \label{gath18_1} \\
0 \leq \ell' \leq \ell \label{gath18_2} \\
g + r \leq d' \leq d \\
|\delta - [\ell' + 2(d - d')]| \leq 1 - \frac{1}{r - 1} \label{gath18_4} \\
\ell' \leq r - 2. \label{gath18_5}
\end{gather}
Inequality \eqref{gath18_5} follows from \eqref{gath18_2} and the hypothesis
\(\ell \leq r/2\).
Rewriting \eqref{gath18_1}--\eqref{gath18_4}, we obtain the system:
\begin{gather*}
\ell - \frac{r - 1}{2} \leq \ell' \\
0 \leq \ell' \leq \ell \\
\delta - 2(d - d') - \frac{r - 2}{r - 1} \leq \ell' \leq \delta - 2(d - d') + \frac{r - 2}{r - 1}\\
g + r \leq d' \leq d.
\end{gather*}
Applying Lemma~\ref{lm:eliminate} to eliminate \(\ell'\), it suffices to show there is an integer solution \(d'\)
to the system:
\begin{align}
\ell - \frac{r - 1}{2} &\leq \ell \label{gath19_1} \\
\ell - \frac{r - 1}{2} &\leq \delta - 2(d - d') + \frac{r - 2}{r - 1} - \frac{r - 2}{2r - 2} \\
0 &\leq \ell \label{gath19_3} \\
0 &\leq \delta - 2(d - d') + \frac{r - 2}{r - 1} \\
\delta - 2(d - d') - \frac{r - 2}{r - 1} &\leq \ell \\
\delta - 2(d - d') - \frac{r - 2}{r - 1} &\leq \delta - 2(d - d') + \frac{r - 2}{r - 1} - \frac{(r - 2)^2}{(r - 1)^2} \label{gath19_6} \\
g + r &\leq d' \leq d.
\end{align}
Inequalities \eqref{gath19_1}, \eqref{gath19_3}, and \eqref{gath19_6} are immediate. 
Rearranging the remaining inequalities, we obtain:
\begin{align}
d' &\geq d - \frac{\delta}{2} + \frac{\ell}{2} - \frac{r^2 - r - 1}{4r - 4} \\
d' &\geq d - \frac{\delta}{2} - \frac{r - 2}{2r - 2} \label{gath20_sl} \\
d' &\geq g + r \\
d' &\leq d - \frac{\delta}{2} + \frac{\ell}{2} + \frac{r - 2}{2r - 2} \label{gath20_fu} \\
d' &\leq d.
\end{align}
We next eliminate \(d'\). Comparing \eqref{gath20_sl} to \eqref{gath20_fu},
we want there to be an integer between
\[d - \frac{\delta}{2} - \frac{r - 2}{2r - 2} \quad \text{and} \quad d - \frac{\delta}{2} + \frac{\ell}{2} + \frac{r - 2}{2r - 2}.\]
By inspection, such an integer exists unless \(\delta\) is an odd integer and \(\ell = 0\),
which is excluded by assumption Proposition~\ref{lm:master-imp-0}\eqref{zzoi}.
Applying Lemma~\ref{lm:eliminate} for the remaining pairs of inequalities, we reduce to verifying:
\begin{align}
d - \frac{\delta}{2} + \frac{\ell}{2} - \frac{r^2 - r - 1}{4r - 4} &\leq d - \frac{\delta}{2} + \frac{\ell}{2} + \frac{r - 2}{2r - 2} - \frac{(2r - 3)(4r - 5)}{(2r - 2)(4r - 4)} \\
d - \frac{\delta}{2} + \frac{\ell}{2} - \frac{r^2 - r - 1}{4r - 4} &\leq d \\
d - \frac{\delta}{2} - \frac{r - 2}{2r - 2} &\leq d \\
g + r &\leq d - \frac{\delta}{2} + \frac{\ell}{2} + \frac{r - 2}{2r - 2} \label{gath21_4} \\
g + r &\leq d.
\end{align}
Upon rearrangement, \eqref{gath21_4} is equivalent to
\[\delta \leq \ell + 1 + 2(d - g - r) - \frac{1}{r - 1},\]
which holds by our assumption Proposition~\ref{lm:master-imp-0}\eqref{db}. The remaining inequalities rearrange to:
\begin{align*}
2r^3 - 8r^2 + 10r - 5 &\geq 0 \\
(r - 3)(r - 2\ell) + 4(d - g - r) + 8g + 2r - 1 &\geq 0 \\
2\ell + 2(d - g - r) + 4g + r - 2 &\geq 0 \\
d - g - r &\geq 0,
\end{align*}
which hold because \((d, g, r, \ell, 0)\) is good and \(r \geq 14\).
\end{proof}

\begin{lem} Suppose that condition \eqref{db} or \eqref{zzoi} of Proposition~\ref{lm:master-imp-0} is satisfied,
but \((\ell, \delta) \neq (0, 1)\), and \(r \geq 14\). Then one of the following two conditions holds:
\begin{equation} \label{db-c}
\ell \leq 3, \quad g \geq 4, \quad \text{and} \quad 1 + \frac{1}{r - 1} \leq \delta \leq 5 - \frac{2}{r - 1}
\end{equation}
\centerline{or}
\begin{equation} \label{zzoi-c}
\ell = 0, \quad g \geq 4, \quad d \geq g + r + 3, \quad \text{and} \quad \delta = 5.
\end{equation}
\end{lem}
\begin{proof} We divide into cases according to whether \eqref{db} or \eqref{zzoi} is satisfied.

\medskip
\noindent
\textbf{\boldmath Case 1: \eqref{db} holds.} In this case we establish \eqref{db-c}.
Because \(\ell\) and \(g\) are integers, the first two inequalities
follow from \(\ell < 4\) and \(g > 3\) respectively. Upon rearrangement,
these inequalities become:
\begin{align*}
(r - 1)(\delta - (\ell + 1 + 2(d - g - r))) + (2r - 4)(d - g - r) + 4(r - 1 - g) + (r - 9) &> 0 \\
(r - 1)(\delta - (\ell + 1 + 2(d - g - r))) + (2r - 4)(d - g - r) + (r - 3)\ell + (r - 13) &> 0,
\end{align*}
and therefore hold for \(r \geq 14\) as desired.
Since \(\ell \geq 0\) and \(d \geq g + r\), we have
\[\delta \geq \ell + 1 + 2(d - g - r) \geq 1,\]
with equality only if \(\ell = 0\).
But equality is excluded by assumption (as \((\ell, \delta) \neq (0, 1)\)).
Finally,
the inequality \(\delta \leq 5 - 2/(r - 1)\) becomes upon rearrangement
\[(2r - 2)(\delta - (\ell + 1 + 2(d - g - r))) + (2r - 2)(d - g - r) + (4r - 4)(r - 1 - g) + (r^2 - 12r + 15) \geq 0.\]

\medskip
\noindent
\textbf{\boldmath Case 2: \eqref{zzoi} holds.} In this case
\[(r - 1)(7 - \delta) = 2(g + 2r - 1 - d) + 4(r - 1 - g) + (r - 1) > 0.\]
Since \(\delta\) is an odd integer, but \(\delta \neq 1\) (because \(\ell = 0\) so
\(\delta = 1\) is excluded by assumption),
we therefore have \(\delta = 3\) or \(\delta = 5\).
In particular,
\[4(g - 3) = (r - 1)(\delta - 3) + 2(g + 2r - 1 - d) + (r - 13) > 0;\]
since \(g\) is an integer, this implies \(g \geq 4\).

\smallskip
\noindent
\textit{Subcase 2.1: \(\delta = 3\).} Then \eqref{db-c} is satisfied.

\smallskip
\noindent
\textit{Subcase 2.2: \(\delta = 5\).} In this case
\(2(d - g - r - 3) = (r - 1)(\delta - 5) + 4(r - 1 - g) + (r - 7) \geq 0\),
so \eqref{zzoi-c} is satisfied.
\end{proof}

Recall that the case \(\delta = 1\) and \(\ell = m = 0\) is excluded in Theorem~\ref{thm:14}. Therefore,
to complete our analysis of the case \(m = 0\) and \(g \neq 0\),
we just have to handle the following two cases:
\begin{description}
\item[\boldmath If \eqref{db-c} holds] Then we apply one of the following propositions according to the value of \(\delta\):
\begin{itemize}
\item If \(1 + \frac{1}{r - 1} \leq \delta \leq 3 - \frac{1}{r - 1}\): Proposition~\ref{m0-delta2}.
\item If \(2 + \frac{3}{r - 1} \leq \delta \leq 4 - \frac{3}{r - 1}\): Proposition~\ref{m0-delta35} with \(\epsilon = 0\).
\item If \(3 + \frac{2}{r - 1} \leq \delta \leq 5 - \frac{2}{r - 1}\): Proposition~\ref{m0-delta4}.
\end{itemize}
Note that the union of these intervals covers the entire interval for \(\delta\)
given by the final inequality of \eqref{db-c}.
Moreover the conditions \(\ell \leq 3\) and \(g \geq 4\)
imply that all tuples appearing in these lemmas are good (they have positive genus and at most \(\ell + 1 \leq 5\) lines in a projective
space of dimension at least \(r - 3 \geq 11\)).

\item[\boldmath If \eqref{zzoi-c} holds] Then we apply
Proposition~\ref{m0-delta35} with \(\epsilon = 1\).
\end{description}

\subsection{\boldmath The Cases With \(m = g = 0\)}
Since \(m = 0\), Lemma~\ref{gather_lines} can be applied unless
\(d \leq 2r - 2\).
In the remaining cases, we will show that Proposition~\ref{master}
\emph{always} applies.

\begin{prop} \label{lm:master-imp-00} Let \((d, 0, r, \ell, 0)\) be a good tuple (with \(m = g = 0\)) satisfying \(d \leq 2r - 2\)
and \(r \geq 14\).
Then the conditions of Proposition~\ref{master}
can be satisfied.
\end{prop}

\begin{proof}
As in the proof of Lemma~\ref{lm:master-imp-0}, our goal is to show the existance
of certain integers \(d'\) and \(\ell'\) which in particular must satisfy:
\begin{gather}
\ell - \frac{r - 1}{2} \leq \ell' \label{gath22_1} \\
0 \leq \ell' \leq \ell \\
\delta - 2(d - d') - \frac{r - 2}{r - 1} \leq \ell' \leq \delta - 2(d - d') + \frac{r - 2}{r - 1}\\
r \leq d' \leq d, \label{gath22_4}
\end{gather}
plus possibly some additional conditions to guarantee that \((d' - 1, 0, r - 1, \ell - \ell', 0)\)
is good. For this we divide into cases as follows:

\medskip
\noindent
\textbf{\boldmath Case 1: \(d = r\).}
In this case we take \(d' = r\).
With this choice \((1 - (d' - 1)) \redmod ((r - 1) - 1) = 0\),
and so \eqref{gath22_1}--\eqref{gath22_4} are sufficient for
\((d' - 1, g, r - 1, \ell - \ell', 0)\) to be good.
Substituting \(d = d' = r\) and \(\delta = 2\ell / (r - 1)\),
our goal is thus to show that there is an integer \(\ell'\) satisfying
\begin{gather*}
\ell - \frac{r - 1}{2} \leq \ell' \\
0 \leq \ell' \leq \ell \\
\frac{2\ell - r + 2}{r - 1} \leq \ell' \leq \frac{2\ell + r - 2}{r - 1}.
\end{gather*}
Applying Lemma~\ref{lm:eliminate}, it suffices to verify:
\begin{align}
0 &\leq \ell \label{gath23_1} \\
0 &\leq \frac{2\ell + r - 2}{r - 1} \\
\frac{2\ell - r + 2}{r - 1} &\leq \ell \label{gath23_3} \\
\frac{2\ell - r + 2}{r - 1} &\leq \frac{2\ell + r - 2}{r - 1} - \frac{(r - 2)^2}{(r - 1)^2} \label{gath23_4} \\
\ell - \frac{r - 1}{2} &\leq \ell \label{gath23_5} \\
\ell - \frac{r - 1}{2} &\leq \frac{2\ell + r - 2}{r - 1} - \frac{r - 2}{2r - 2}. \label{gath23_6}
\end{align}
Inequalities \eqref{gath23_1}--\eqref{gath23_3} follow from \(\ell \geq 0\),
and \eqref{gath23_4} and \eqref{gath23_5}
are automatic, and \eqref{gath23_6} follows from \(\ell \leq r/2\).

\medskip
\noindent
\textbf{\boldmath Case 2: \(d \geq r + 1\).}
Since \(r + 1 \leq d \leq 2r - 2\), we have \((1 - d) \redmod (r - 1) = (1 - d) + 2(r - 1) = 2r - 1 - d\).
Therefore, as \((d, g, r, \ell, 0)\) is good by assumption, \(\ell\) satisfies
\begin{equation}\label{zzlb}
\frac{2r - 1 - d}{2} \leq \ell \leq \frac{r}{2}.
\end{equation}
Similarly, because \(d' - 1 \leq 2r - 3\), we have
\((1 - (d' - 1)) \redmod ((r - 1) - 1) \leq (2 - d') + 2(r - 2) = 2r - 2 - d'\).
Therefore \((d' - 1, g, r - 1, \ell - \ell', 0)\) is good provided both
\eqref{gath22_1}--\eqref{gath22_4}
are satisfied, and also \(2(\ell - \ell') \geq 2r - 2 - d'\).
In other words, we want to show that there are integers \(\ell'\) and \(d'\)
satisfying the following system (here we have substituted in \(\delta = (2d - 2r + 2\ell)/(r - 1)\)):
\begin{gather}
\ell - \frac{r - 1}{2} \leq \ell' \label{gath24_1} \\
0 \leq \ell' \leq \ell \\
\frac{2d - 2r + 2\ell}{r - 1} - 2(d - d') - \frac{r - 2}{r - 1} \leq \ell' \leq \frac{2d - 2r + 2\ell}{r - 1} - 2(d - d') + \frac{r - 2}{r - 1}\\
r \leq d' \leq d \\
\ell' \leq \ell - \frac{2r - 2 - d'}{2}. \label{gath24_5}
\end{gather}

\smallskip
\noindent
\textit{Subcase 2.1: \(d = r + 1\).}
In this case, \eqref{zzlb} becomes \(r/2 - 1 \leq \ell \leq r/2\).
In other words, \((d, r, \ell)\) is of one of the following forms:
\[(d, r, \ell) = (2k + 1, 2k, k - 1), \quad (d, r, \ell) = (2k + 1, 2k, k), \quad \text{or} \quad (d, r, \ell) = (2k + 2, 2k + 1, k).\]
In these cases way may satisfy \eqref{gath24_1}--\eqref{gath24_5} by taking:
\[(d', \ell') = (2k, 0), \quad (d', \ell') = (2k + 1, 1), \quad \text{respectively} \quad (d', \ell') = (2k + 1, 0).\]

\smallskip
\noindent
\textit{Subcase 2.2: \(d \geq r + 2\).}
In this case we take \(\ell' = 2 - 2(d - d')\),
in which case \eqref{gath24_1}--\eqref{gath24_5} becomes:
\begin{gather*}
d' \geq \frac{4d + 2\ell - r - 3}{4} \\
d - 1 \leq d' \leq d - 1 + \frac{\ell}{2} \\
3r \leq 2d + 2\ell \leq 5r - 4 \\
r \leq d' \leq d \\
d' \leq \frac{4d + 2\ell - 2r - 2}{3}.
\end{gather*}
The inequality \(d' \geq r\) follows from \(d' \geq d - 1\) since \(d = r + 2\).
Deleting the inequality \(d' \geq r\), and eliminating \(d'\) via Lemma~\ref{lm:eliminate}, we reduce to the system:
\begin{align}
\frac{4d + 2\ell - r - 3}{4} &\leq d - 1 + \frac{\ell}{2} - \frac{3}{8} \label{gath25_1} \\
\frac{4d + 2\ell - r - 3}{4} &\leq d \\
\frac{4d + 2\ell - r - 3}{4} &\leq \frac{4d + 2\ell - 2r - 2}{3} - \frac{1}{2} \\
d - 1 &\leq d - 1 + \frac{\ell}{2} \label{gath25_4} \\
d - 1 &\leq d \label{gath25_5} \\
d - 1 &\leq \frac{4d + 2\ell - 2r - 2}{3} \\ 
3r &\leq 2d + 2\ell \\
2d + 2\ell &\leq 5r - 4.
\end{align}
Inequalities \eqref{gath25_1}, \eqref{gath25_4}, and \eqref{gath25_5}
are immediate. Rearranging the others, we obtain:
\begin{align*}
(r - 2\ell) + 3 &\geq 0 \\
(2\ell - 2r + 1 + d) + 3(d - r - 2) &\geq 0 \\
2\ell - 2r + 1 + d &\geq 0 \\
(2\ell - 2r + 1 + d) + (d - r - 2) + 1 &\geq 0 \\
(r - 2\ell) + 2(2r - 2 - d) + 8 &\geq 0, 
\end{align*}
which all follow from \eqref{zzlb}, our assumption \(r + 2 \leq d \leq 2r - 2\), and the hypotheses in \eqref{box}.
\end{proof}

\section{Most of the Sporadic Cases \label{sec:most-sporadic}}

The finite set of sporadic cases identified in the previous section
is unfortunately rather large.
Our next task is to introduce an additional argument
that, in combination with the arguments of Section~\ref{sec:inductive}, applies
to handle most of the sporadic cases, i.e., all but a
list that is short enough to \emph{write down explicitly}.

This argument will, essentially, be a variant on Proposition~\ref{master},
but where we allow transformations to come together at \(p\) in a less restricted way.  In particular, we will weaken the hypothesis ``\(2m' + \ell \leq r-2\)'' in the statement of Proposition~\ref{master} by allowing more modifications to limit to the point \(p\) than the rank of the normal bundle.  In this regime, the limiting bundle can depend on how the points are specialized into \(p\).  We will be able to give a description of some possible limits by limiting the marked points into \(p\) one at a time
inductively.  At each step we will be able to identify what the limiting modifications are at \(p\).
Suppose that, after limiting some collection of marked points into \(p\),
we have a transformation at \(p\) of the form
\begin{equation} \label{atp-form}
(np)[p \posmod \Lambda_1][p \posmod \Lambda_2] \quad \text{where} \quad \Lambda_1 \supseteq \Lambda_2, \tag{\(\dagger\)}
\end{equation}
where \(\Lambda_1\) and \(\Lambda_2\) are linear spaces in \(\pp N_C|_p\).
Let \(S\) and \(W\) be sets of parameters varying in irreducible bases with \(S \subseteq W\).
(For us, \(W\) will be the collection of all corresponding marked points, and \(S\)
will be those marked points at which the projected normal bundle is \emph{not} modified.)
Assume that \(\Lambda_2\) is linearly general as the parameters \(S\) vary,
and assume that \(\Lambda_1\) is either:
\begin{itemize}
\item Linearly general as the parameters \(W\) vary (``weakly general''); or
\item Weakly general and its image in \(\pp (N_C / N_{C \to p})|_p\) is
linearly general as only the parameters \(S\) vary (``strongly general'').
\end{itemize}

We will summarize this situation
by three pieces of data: the linear dimensions \(t_1 = \rk \Lambda_1 = \dim \Lambda_1 + 1\) and \(t_2 = \rk \Lambda_2 = \dim \Lambda_2 + 1\),
and whether we are in the weak or strong case.
Note that we always have \(t_2 \leq t_1 \leq r - 2\).  (We do not need to keep track of the integer \(n\) since this is can be deduced from the Euler characteristic of the limit bundle, which is the same as the original bundle.)

Modifications of type \eqref{atp-form} occur naturally when considering the degenerations that make up the key inductive argument outlined in Section \ref{subsec:over_ind}.  We review them here (and add one additional argument that we will use in this section) in order to motivate the shape of \eqref{atp-form}.
\begin{enumerate}
\item\label{ind1} If we peel off a \(1\)-secant line \(\bar{xy}\) and limit the point \(x\) to \(p\), we obtain the modification \([p \posmod y][p \posmod y]\) at \(p\).  This is of the form \eqref{atp-form} with \(\Lambda_1 = \Lambda_2 = y\); hence, we have \(t_1 = t_2 = 1\).  Since \(y\) is a general point, and no modifications occur at \(y\) in the quotient by projection from \(p\), both \(\Lambda_1\) and \(\Lambda_2\) are strongly linearly general.

\smallskip

\item\label{ind2} If we peel off a \(2\)-secant line \(\bar{zw}\) and limit the point \(z\) to \(p\), we obtain the modification
\([p \posmod 2w][p \posmod w]\)
at \(p\).  This is of the form \eqref{atp-form} with \(\Lambda_1 = T_wC\) and \(\Lambda_2 = w\); hence, we have \((t_1, t_2) = (2, 1)\).  Since \(w\) is a general point on the curve, at which no modifications occur in the quotient by projection from \(p\), both \(\Lambda_1\) and \(\Lambda_2\) are strongly linearly general.

\smallskip

\item\label{ind3} If we specialize \(R\) as in Section \ref{sec:onion_specialization} to contain \(n\) lines through at \(p\), then at \(p\) we obtain the modification \([p \posmod \Lambda]\), where \(\rk \Lambda = 2\).  This is of the form \eqref{atp-form} with \((t_1, t_2) = (2, 0)\).  By Lemmas \ref{M:unprojected} and \ref{M:projected}, the subspace \(\Lambda\) is
strongly linearly general if 
\[n \geq \begin{cases} 3 & \text{if \(C\) is an elliptic normal curve} \\ 2 & \text{otherwise.} \end{cases}\] 
\end{enumerate}

\begin{enumerate}[label=(\arabic*\('\))]
\addtocounter{enumi}{2}
\item\label{ind3p} If we specialize one of the \(v_i\) to \(p\), we obtain the modification \([p \posmod u_i]\) at \(p\).  This is of the form \eqref{atp-form} with \((t_1, t_2) = (1, 0)\), and is strongly linearly general,
since \(u_i\) is a general point, at which no modifications occur in the quotient by projection from \(p\).
\end{enumerate}

\begin{enumerate}
\addtocounter{enumi}{6}
\item\label{ind7} We allow ourselves one new degeneration in our more general inductive step, which is similar to \eqref{indstrat:center_proj_R} from Section \ref{subsec:over_ind}, but crucially different in that we specialize \(R\) to pass through \(p\)
\emph{before} we project. We first specialize \(R\) to the union of a line \(L\) through \(2\) points \(s_0, s_r\) on \(C\) and a rational curve \(R'\) of degree \(r-2\) through \(r-1\) points on \(C\) and meeting \(L\) at one point.  Then we specialize \(s_0\) to \(p\).  This results in the modification \([p \posmod s_r]\) at \(p\), which is of type \((t_1, t_2) = (1,0)\).  This modification is linearly general as \emph{all} the points of contact between \(C\) and \(R'\) move.
However, \(s_r\) is constrained to be one of the points at which the \(r\)-secant rational curve \(\bar{R'}\) meets \(\bar{C}\),
and modifications occur at the remainder of these
points, so it is only weakly linearly general.
\end{enumerate}

Our first goal is to understand what happens when we limit into \(p\) another point \(p'\), at which we have
another transformation
\((n' p') [p' \posmod \Lambda_1'][p' \posmod \Lambda_2']\) of the same form \eqref{atp-form}
(depending on sets of parameters \(S' \subseteq W'\) disjoint from \(W\)).
In the following five cases, which we consider separately, we will see that in the limit we obtain another transformation
of the form \eqref{atp-form}
(depending on parameters \(S \cup S' \subseteq W \cup W'\)).  

Most of the subspaces whose generality we must assess are of the form \(\Lambda + \Lambda'\) (the span of \(\Lambda\) and \(\Lambda'\)).  If  \(\Lambda\) and \(\Lambda'\) are both linearly general as independent parameters \(X\) and \(X'\) vary,
then their span is linearly general,
as we now show.  Let \(M\) be a fixed subspace; there is a choice of the parameters \(X\) for which the corresponding subspace \(\Lambda\) meets \(M\) transversely. Then there is a choice of the parameters \(X'\)
for which the corresponding subspace \(\Lambda'\) meets \(M + \Lambda\) transversely.  For this choice of \(X \cup X'\), the subspace \(M\) meets \(\Lambda + \Lambda'\) transversely. The only case where the resulting modification is not of this form is \ref{hardcase} below.

\begin{enumerate}[label=(\alph*)]
\item 
If \(t_1 + t_1' < r - 1\): In this case, the limiting transformation is
\[((n + n')p) [p \posmod \Lambda_1 + \Lambda_1'] [p \posmod \Lambda_2 + \Lambda_2'].\]
This transformation is
of the desired form.  The subspace \(\Lambda_1 + \Lambda_1'\) is strongly general if 
\(\Lambda_1\) and \(\Lambda_1'\) are both strongly general, and weakly general otherwise.  

\smallskip

\item
If \(t_2 + t_2' < t_1 + t_1' = r - 1\): In this case, the limiting transformation is
\[((n + n' + 1)p) [p \posmod \Lambda_2 + \Lambda_2'][p \posmod \emptyset].\]
This transformation is
of the desired form, and the subspace \(\Lambda_2 + \Lambda_2'\) is always strongly general.

\smallskip

\item\label{hardcase} 
If \(t_2' = 0\) and \(t_1' + t_2 \leq r - 1 \leq t_1' + t_1\):
In this case, the limiting transformation is
\[((n + n' + 1)p) [p \posmod \Lambda_2 + (\Lambda_1 \cap \Lambda_1')][p \posmod \emptyset].\]
We now show that \(\Lambda_2 + (\Lambda_1 \cap \Lambda_1')\) is strongly linearly general if both \(\Lambda_1\) and \(\Lambda_1'\) are strongly general,
and weakly linearly general otherwise.
Indeed, let \(M\) be any fixed subspace; we want to show that \(M\) is transverse to \(\Lambda_2  + (\Lambda_1 \cap \Lambda_1')\). Both \(\Lambda_1\) and \(\Lambda_2\) are linearly general; since the parameters \(W\) vary in an irreducible base, there is a single choice of the parameters \(W\) for which \(M\) is simultaneously transverse to both \(\Lambda_1\) and \(\Lambda_2\).  Since \(\Lambda_1\) 
is transverse to \(M\) for our choice of parameters \(W\), and \(\Lambda_2 \subset \Lambda_1\), we can restrict to \(\Lambda_1\) and consider transversality as subspaces of \(\Lambda_1\).  The subspace \(M \cap \Lambda_1\) is transverse to
 \(\Lambda_2\) for our choice of parameters \(W\).  The subspace \(\Lambda_1' \cap \Lambda_1\) is transverse to \(\Lambda_2\) since \(\Lambda_1'\) is linearly general, varying with independent parameters.  Therefore the transversality of \(M \cap \Lambda_1\) and \(\Lambda_2 + (\Lambda_1' \cap \Lambda_1)\) is equivalent to the transversality of 
  \(\Lambda_1 \cap \Lambda_1'\)
and \(\Lambda_2 + (M \cap \Lambda_1)\), which again follows by the linear generality of \(\Lambda_1'\).

\smallskip

\item
If \(t_1' + t_2 < t_1 + t_2' = r - 1\):
In this case, the limiting transformation is
\[((n + n' + 1)p) [p \posmod \Lambda_2 + \Lambda_1'][p \posmod \emptyset].\]
This transformation is
of the desired form, with \(\Lambda_2 + \Lambda_1'\) strongly general if \(\Lambda_1'\)
is strongly general, and weakly general otherwise.

\smallskip

\item If \(t_1 + t_2' = t_1' + t_2 = r - 1\):
In this case, the limiting transformation is
\[((n + n' + 2)p) [p \posmod \emptyset][p \posmod \emptyset].\]
This transformation is
of the desired form, with \(\emptyset\) always strong general.
\end{enumerate}

\begin{defin}
For integers \(0 \leq i,j < r-1\), let \(\{s_{ij}\}\) and \(\{w_{ij}\}\) be collections of nonnegative integers,
and consider \(s_{ij}\) (respectively \(w_{ij}\)) marked points 
decorated with modifications of type \eqref{atp-form} with \((t_1, t_2) = (i, j)\)
and \(\Lambda_1\) strongly (respectively weakly) general.
Consider all ways of limiting these marked points into \(p\), one at a time
in some order,
such that at every step of the process, we are in one of the five cases discussed above.
If there is such an order for which the final resulting transformation at \(p\) satisfies
\(t_2 = 0\) and \(\Lambda_1\) is strongly general, then we say that \(\{s_{ij}\}\) and \(\{w_{ij}\}\)
is \defi{erasable}.
\end{defin}

We are now ready to state our more flexible variant on Proposition~\ref{master}.  The high-level overview is that, in some order, we do the following specializations:
\begin{itemize}
\item Peel off \(g - g'\) two-secant lines. Specialize all of them into \(p\) as in \eqref{ind1}.
\item Peel off \(\epsilon_{\text{in}} + \epsilon_{\text{out}} = d - g - d' + g'\) one-secant lines.
Specialize \(\epsilon_{\text{in}}\) of these into \(p\) as in the proof of Proposition~\ref{master},
and the remaining \(\epsilon_{\text{out}}\) of them into \(p\) as in \eqref{ind2}.
\item Specialize \(m'\) of the rational curves \(R_i\) as in Section \ref{sec:onion_specialization} to lines and conics through \(p_i\).
Specialize all of the \(p_i\) to \(p\) as in \eqref{ind3}.
\item Specialize \(\ell'\) of the points \(v_i\) to \(p\) as in \ref{ind3p}.
\item Specialize \(m''\) of the rational curves \(R_i\) to the union \(L_i \cup R_i'\) as in \eqref{ind7}.
Specialize one of the points where \(L_i\) meets \(C\) into \(p\).
\item Specialize the remaining \(m - m' - m''\) rational curves \(R_i\) to pass through \(p\) as in
Section~\ref{subsec:over_ind}\eqref{indstrat:center_proj_R}.
\end{itemize}
After projecting from \(p\) we will reduce to a case of our inductive hypothesis plus a single linearly general modification at \(p\) precisely when the modifications at \(p\) above are erasable.

\begin{prop} \label{master-erasable}
Let \(\ell'\), \(m'\), and \(m''\) be nonnegative integers satisfying \(\ell' \leq \ell\) and \(m' + m'' \leq m\),
with \(m' = 0\) if \(r = 3\).
Let \(d'\) and \(g'\) be integers satisfying \(0 \leq g' \leq g\) and \(g' + r \leq d' \leq d - g + g'\),
with \(d' > g' + r\) if both \(g' = 0\) and \(m \neq 0\).
Let \(\epsilon_{\text{in}}\) and \(\epsilon_{\text{out}}\) be nonnegative integers with
\(\epsilon_{\text{in}} + \epsilon_{\text{out}} = d - g - d' + g'\).
For \(1 \leq i \leq m'\), let \(n_i\) be an integer satisfying \(n_i \equiv r - 1\) mod \(2\)
and \(2 \leq n_i \leq r - 1\), with \(n_i \neq 2\) if \((d', g') = (r + 1, 1)\).
Define
\[\bar{\ell} = \ell - \ell' + \frac{(r - 1)m' - \sum n_i}{2} \quad \text{and} \quad \bar{m}_\text{max} = m - m' \quad \text{and} \quad \bar{m}_\text{min} = m - m' - m''.\]
Suppose that the following collection is erasable:
\begin{align}
s_{10} &= \ell' + m - m' - m'' \notag \\
s_{11} &= \epsilon_{\text{out}} \notag \\
s_{20} &= m' \label{erase-ass} \\
s_{21} &= g - g' \notag \\
w_{10} &= m''. \notag
\end{align}
If
\[\left|\delta - \left[2\epsilon_{\text{in}} + g - g' + m'' + \ell' + \left \lfloor\frac{2 \epsilon_{\text{out}} + 3 (g - g') + m + m' + \ell'}{r - 1} \right \rfloor + \sum n_i \right] \right| \leq 1 - \frac{1}{r - 1},\]
and \(I(d' - 1, g', r - 1, \bar{\ell}, \bar{m})\) holds for all \(\bar{m}\) with \(\bar{m}_\text{min} \leq \bar{m} \leq \bar{m}_\text{max}\), then so does
\(I(d, g, r, \ell, m)\).
\end{prop}
\begin{proof}
Our goal is to show interpolation for
\[N_{C(0,0;0)}[u_1 \biposmod v_1]\cdots[u_\ell \biposmod v_\ell][\posmodalong R_1 \cup \cdots \cup R_m].\]
Peeling off
\(\epsilon_{\text{in}} + \epsilon_{\text{out}} = d - g - d' + g'\) one-secant lines and \(g - g'\) two-secant lines
reduces to interpolation for
\begin{multline*}
N_{C(d - g - d' + g', g - g';0)}[u_1 \biposmod v_1]\cdots[u_\ell \biposmod v_\ell][\posmodalong R_1 \cup \cdots \cup R_m] \\
[2x_1 \posmod y_1] \cdots [2x_{\epsilon_{\text{in}} + \epsilon_{\text{out}}} \posmod y_{\epsilon_{\text{in}} + \epsilon_{\text{out}}}] [z_1 \biposmod w_1][z_1 \posmod 2w_1] \cdots [z_{g-g'} \biposmod w_{g-g'}][z_{g-g'} \posmod 2w_{g-g'}].
\end{multline*}

For \(1 \leq i \leq m'\), write \(n_i' = (r-1-n_i)/2\),
and degenerate \(R_i\) as in Section \ref{sec:onion_specialization} to the union \(R_i^\circ\), of \(n_i\) lines \(L_{i,j}\) meeting \(C\) at \(p_i\) and \(q_{i,j}\), and \(n_i'\) conics \(Q_{i,j}\) meeting \(C\) at \(p_i\) and \(q_{i,n_i + 2j-1}\) and \(q_{i,n_i + 2j}\).
For \(m' + 1 \leq i \leq m\), write
\(R_i \cap C = \{s^i_0, s^i_1, s^i_2, \ldots, s^i_{r - 1}, s^i_r\}\).
For \(m' + 1 \leq i \leq m' + m''\), specialize \(R_i\) to a union \(R_i^- \cup L_i\), where \(L_i\)
is the line through \(s^i_0\) and \(s^i_r\), and \(R_i^-\) is a rational curve of degree \(r - 2\)
passing through \(s^i_1, s^i_2, \ldots, s^i_{r - 1}\) and meeting \(L_i\) at a single point.
This induces a specialization of the above bundle to
\begin{multline*}
N_{C(d - g - d' + g', g - g';0)}[u_1 \biposmod v_1]\cdots[u_\ell \biposmod v_\ell][\posmodalong R_{m' + 1}^- \cup \cdots \cup R_{m' + m''}^-] [\posmodalong R_{m' + m'' + 1} \cup \cdots \cup R_m] \\
[q_{1,1} + \cdots + q_{1,r - 1} \posmodalong R_1^\circ] \cdots [q_{m',1} + \cdots + q_{m',r - 1} \posmodalong R_{m'}^\circ] \\
[s^{m' + 1}_0 \biposmod s^{m' + 1}_r] \cdots [s^{m' + m''}_0 \biposmod s^{m' + m''}_r] [p_1 \posmod M_1] \cdots [p_{m'} \posmod M_{m'}] \\
[2x_1 \posmod y_1] \cdots [2x_{\epsilon_{\text{in}} + \epsilon_{\text{out}}} \posmod y_{\epsilon_{\text{in}} + \epsilon_{\text{out}}}] [z_1 \biposmod w_1][z_1 \posmod 2w_1] \cdots [z_{g-g'} \biposmod w_{g-g'}][z_{g-g'} \posmod 2w_{g-g'}].
\end{multline*}
Fix a general point \(p \in C\), and
specialize \(p_1, p_2, \ldots, p_{m'}, v_1, v_2, \ldots, v_{\ell'}, y_1, y_2, \ldots, y_{\epsilon_{\text{in}}}, x_{\epsilon_{\text{in}} + 1}, x_{\epsilon_{\text{in}} + 2}, \ldots,\)
\(x_{\epsilon_{\text{in}} + \epsilon_{\text{out}}}, z_1, z_2, \ldots, z_{g - g'}, s^{m' + 1}_0, s^{m' + 2}_0, \ldots, s^m_0\) all to \(p\)
in some order. Our assumption that \eqref{erase-ass}
is erasable implies that we may choose the order so that the limiting bundle is
\begin{multline*}
N_{C(d - g - d' + g', g - g';0)}[u_{\ell' + 1} \biposmod v_{\ell' + 1}]\cdots[u_\ell \biposmod v_\ell] [\posmodalong R_{m' + 1}^- \cup \cdots \cup R_{m' + m''}^-] \\
[s^{m' + m'' + 1}_1 + s^{m' + m'' + 1}_2 + \cdots + s^{m' + m'' + 1}_r  \posmodalong R_{m' + m'' + 1}] \cdots [s^m_1 + s^m_2 + \cdots + s^m_r  \posmodalong R_m] \\
[q_{1,1} + \cdots + q_{1,r - 1} \posmodalong R_1^\circ] \cdots [q_{m',1} + \cdots + q_{m',r - 1} \posmodalong R_{m'}^\circ] \\
[2x_1 + \cdots + 2x_{\epsilon_{\text{in}}} + w_1 + \cdots + w_{g - g'} + s^{m' + 1}_r + \cdots + s^{m' + m''}_r + u_1 + \cdots + u_{\ell'} \posmod p] (np) [p \posmod \Lambda],
\end{multline*}
for some integer \(n\) and subspace \(\Lambda \subset \pp N_{C(d - g - d' + g', g - g';0)}|_p\),
disjoint from \(\pp N_{C(d - g - d' + g', g - g';0) \to p}|_p\) and
whose image \(\bar{\Lambda}\) in \(\pp (N_{C(d - g - d' + g', g - g';0)}/N_{C(d - g - d' + g', g - g';0) \to p})|_p\) is linearly general.
Computing the Euler characteristic, we obtain
\[(r - 1) n + \rk \Lambda = 2 \epsilon_{\text{out}} + 3 (g - g') + m + m' + \ell',\]
and so
\[n = \left\lfloor \frac{2 \epsilon_{\text{out}} + 3 (g - g') + m + m' + \ell'}{r - 1}\right\rfloor.\]
Projecting from \(p\), we reduce to interpolation for
\begin{multline*}
N_{C(d - g - d' + g', g - g';1)}[u_{\ell' + 1} \biposmod v_{\ell' + 1}]\cdots[u_\ell \biposmod v_\ell] [\posmodalong \bar{R}_{m' + m'' + 1} \cup \cdots \cup \bar{R}_{m}] \\
[s^{m' + 1}_1 + \cdots + s^{m' + 1}_{r - 1} \posmodalong \bar{R_{m' + 1}^-}] \cdots [s^{m' + m''}_1 + \cdots + s^{m' + m''}_{r - 1} \posmodalong \bar{R_{m' + m''}^-}] \\
[q_{1,n_1 + 1} \biposmod q_{1,n_1 + 2}] \cdots [q_{1, r-2} \biposmod q_{1, r-1}] \cdots [q_{m',n_{m'} + 1} \biposmod q_{m',n_{m'} + 2}] \cdots [q_{m', r-2} \biposmod q_{m', r-1}] (np) [p \posmod \bar{\Lambda}].
\end{multline*}
Erasing the transformation at \(p\), we reduce to interpolation for
\begin{multline*}
N_{C(d - g - d' + g', g - g';1)}[u_{\ell' + 1} \biposmod v_{\ell' + 1}]\cdots[u_\ell \biposmod v_\ell] [\posmodalong \bar{R}_{m' + m'' + 1} \cup \cdots \cup \bar{R}_{m}] \\
[s^{m' + 1}_1 + \cdots + s^{m' + 1}_{r - 1} \posmodalong \bar{R_{m' + 1}^-}] \cdots [s^{m' + m''}_1 + \cdots + s^{m' + m''}_{r - 1} \posmodalong \bar{R_{m' + m''}^-}] \\
[q_{1,n_1 + 1} \biposmod q_{1,n_1 + 2}] \cdots [q_{1, r-2} \biposmod q_{1, r-1}] \cdots [q_{m',n_{m'} + 1} \biposmod q_{m',n_{m'} + 2}] \cdots [q_{m', r-2} \biposmod q_{m', r-1}].
\end{multline*}
By Lemma~\ref{lem:monodromy}, this follows in turn from interpolation for the bundles
\begin{multline*}
N_{C(d - g - d' + g', g - g';1)}[u_{\ell' + 1} \biposmod v_{\ell' + 1}]\cdots[u_\ell \biposmod v_\ell] [\posmodalong \bar{R}_{m' + m'' + 1} \cup \cdots \cup \bar{R}_{m}] [\posmodalong \bar{R_{i_1}^-} \cup \bar{R_{i_2}^-} \cup \cdots \cup \bar{R_{i_j}^-}] \\
[q_{1,n_1 + 1} \biposmod q_{1,n_1 + 2}] \cdots [q_{1, r-2} \biposmod q_{1, r-1}] \cdots [q_{m',n_{m'} + 1} \biposmod q_{m',n_{m'} + 2}] \cdots [q_{m', r-2} \biposmod q_{m', r-1}],
\end{multline*}
with \(m' + 1 \leq i_1 < i_2 < \cdots < i_j \leq m' + m''\).
Finally, specializing \(R_{m' + m'' + 1}, R_{m' + m'' + 2}, \ldots, R_m\) to pass through \(p\) (as in the proof of Proposition~\ref{master}), we reduce
to interpolation for
\begin{multline*}
N_{C(d - g - d' + g', g - g';1)}[u_{\ell' + 1} \biposmod v_{\ell' + 1}]\cdots[u_\ell \biposmod v_\ell] [\posmodalong R'_{m' + m'' + 1} \cup \cdots \cup R'_m] [\posmodalong \bar{R_{i_1}^-} \cup \bar{R_{i_2}^-} \cup \cdots \cup \bar{R_{i_j}^-}] \\
[q_{1,n_1 + 1} \biposmod q_{1,n_1 + 2}] \cdots [q_{1, r-2} \biposmod q_{1, r-1}] \cdots [q_{m',n_{m'} + 1} \biposmod q_{m',n_{m'} + 2}] \cdots [q_{m', r-2} \biposmod q_{m', r-1}].
\end{multline*}
But these are precisely our assumptions \(I(d' - 1, g', r - 1, \bar{\ell}, \bar{m})\)
for \(\bar{m}_{\text{min}} \leq \bar{m} \leq \bar{m}_{\text{max}}\).
\end{proof}

We then write a computer program in \texttt{python} \cite{python} (see Appendix~\ref{app:code}) which
iterates over all of the finitely many sporadic cases identified in the previous section,
i.e., those tuples \((d, g, r, \ell, m)\) satisfying \(r \leq 13\) and \eqref{box} or \eqref{to-exc},
but excluding those tuples with \((\delta, \ell, m) = (1, 0, 0)\).
In each case, all possible parameters for every inductive argument in Section~\ref{sec:inductive},
as well as all possible parameters for Proposition~\ref{master-erasable}, are tried.
In all but the following \(30\) cases, one of these arguments applies:
\begin{center}
\begin{tabular}{cccccc}
(4, 0, 3, 0, 1) & (4, 0, 3, 0, 2) & (4, 0, 3, 1, 1) & (5, 0, 3, 0, 1) & (5, 1, 3, 0, 1) & (5, 1, 3, 1, 1) \\
(5, 2, 3, 0, 1) & (5, 2, 3, 0, 2) & (5, 2, 3, 1, 1) & (6, 2, 3, 0, 1) & (5, 0, 4, 0, 1) & (5, 0, 4, 2, 0) \\
(6, 2, 4, 0, 2) & (7, 3, 4, 0, 1) & (7, 3, 4, 1, 1) & (7, 1, 5, 0, 1) & (7, 2, 5, 0, 1) & (7, 2, 5, 2, 2) \\
(9, 2, 5, 0, 0) & (8, 3, 5, 2, 0) & (9, 4, 5, 0, 0) & (9, 4, 5, 1, 0) & (7, 0, 6, 0, 1) & (7, 1, 6, 2, 1) \\
(7, 1, 6, 3, 1) & (8, 2, 6, 2, 0) & (11, 5, 6, 0, 0) & (8, 1, 7, 0, 1) & (8, 1, 7, 1, 1) & (11, 4, 7, 1, 0)
\end{tabular}
\end{center}
Our remaining task is therefore to verify \(I(d, g, r, \ell, m)\) in these \(30\) base cases
(as well as to prove Theorem~\ref{thm:main} for canonical curves of even genus).

\section{The Remaining Sporadic Cases}\label{sec:remaining-sporadic}

\subsection{\texorpdfstring{\boldmath The cases \((d, g, r, \ell, m) = (7, 1, 6, 2, 1), (7, 1, 6, 3, 1)\), and \((8, 2, 6, 2, 0)\)}{The cases (d, g, r, l, m) = (7, 1, 6, 2, 1), (7, 1, 6, 3, 1), and (8, 2, 6, 2, 0)}}
In these three cases, our previous arguments apply provided that \(I(6, 1, 5, 3, 0)\) holds.
(Note that \((6, 1, 5, 3, 0)\) is, however, not good, which is why we were not able to deal with these
cases in the previous section and need to separately consider them here.)
Indeed:

\begin{description}
\item[\boldmath If \((d, g, r, \ell, m) = (7, 1, 6, 2, 1)\)]
We apply Proposition~\ref{master} with parameters:
\[\ell' = 0, \quad m' = 1, \quad d' = 7, \quad \text{and} \quad n_1 = 3.\]

\item[\boldmath If \((d, g, r, \ell, m) = (7, 1, 6, 3, 1)\)]
We apply Proposition~\ref{master} with parameters:
\[\ell' = 1, \quad m' = 1, \quad d' = 7, \quad \text{and} \quad n_1 = 3.\]

\item[\boldmath If \((d, g, r, \ell, m) = (8, 2, 6, 2, 0)\)]
We apply Proposition~\ref{m0-delta2}.
\end{description}

\noindent
It thus remains to check \(I(6, 1, 5, 3, 0)\). For this, we simply apply
Proposition~\ref{master} with parameters:
\[\ell' = 3, \quad m' = 0, \quad \text{and} \quad d' = 6,\]
thereby reducing \(I(6, 1, 5, 3, 0)\) to \(I(5, 1, 4, 0, 0)\),
which suffices because \((5, 1, 4, 0, 0)\) is good.

\subsection{\texorpdfstring{\boldmath The cases \((d, g, r, \ell, m) = (4, 0, 3, 1, 1), (5, 1, 3, 1, 1)\), and \((5, 2, 3, 1, 1)\)}{The cases (d, g, r, l, m) = (4, 0, 3, 1, 1), (5, 1, 3, 1, 1), and (5, 2, 3, 1, 1)}}
In each of these cases, we
want to show interpolation for
\[N_C[u \biposmod v][\posmodalong R].\]
Write \(R \cap C = \{q_1, q_2, q_3, q_4\}\). Specializing \(u\) to \(q_1\) and \(v\) to \(q_2\)
induces a specialization of this bundle to
\[N_C[q_3 + q_4 \posmodalong R](q_1 + q_2).\]
Removing the twists at \(q_1\) and \(q_2\), we reduce to interpolation for
\[N_C[q_3 + q_4 \posmodalong R].\]
Specializing \(R\) to the union of the lines \(\bar{q_1 q_2} \cup \bar{q_3 q_4}\)
induces a specialization of this bundle to
\[N_C[q_3 \biposmod q_4].\]
Interpolation for this bundle is the assertion \(I(d, g, 3, 1, 0)\),
and \((d, g, 3, 1, 0)\) is good in each of these cases.

\subsection{\texorpdfstring{\boldmath The cases \((d, g, r, \ell, m) = (5, 1, 3, 0, 1)\) and \((6, 2, 3, 0, 1)\)}{The cases (d, g, r, l, m) = (5, 1, 3, 0, 1) and (6, 2, 3, 0, 1)}}
In both of these cases, we
we want to show interpolation for
\(N_C[\posmodalong R]\).
Peeling off a \(2\)-secant line, we reduce to interpolation for
\[N_{C(1, 0; 0)}[\posmodalong R][z \biposmod w][z \posmod 2w] \simeq N_{C(1, 0; 0)}[\posmodalong R][z \biposmod w](z).\]
Removing the twist at \(z\), we reduce to interpolation for
\[N_{C(1, 0; 0)}[\posmodalong R][z \biposmod w].\]
Interpolation for this bundle is the assertion \(I(d - 1, g - 1, 3, 1, 1)\),
and \((d - 1, g - 1, 3, 1, 1)\) is good in both of these cases.

\subsection{\texorpdfstring{\boldmath The cases \((d, g, r, \ell, m) = (4, 0, 3, 0, 2)\) and \((5, 2, 3, 0, 2)\)}{The cases (d, g, r, l, m) = (4, 0, 3, 0, 2) and (5, 2, 3, 0, 2)}}
Let \(x_1, y_1, x_2, y_2 \in C\) be four general points. Projection from \(\bar{x_i y_i}\)
defines a general map \(\pi_i \colon C \to \pp^1\) of degree \(d - 2\), which is in particular separable.
Since \(x_2\) and \(y_2\) are general, \(\bar{x_2 y_2}\) does not meet
the tangent line to \(C\) at either \(x_1\), \(y_1\), or any of the ramification points of \(\pi_1\).
Thus \((\pi_1, \pi_2) \colon C \to \pp^1 \times \pp^1\) is birational onto its image, an isomorphism near \(x_1\)
and \(y_1\) (and by symmetry near \(x_2\) and \(y_2\)), and its image is nodal.

The number of nodes is the difference between the arithmetic and geometric genus, which is \((d - 3)^2 - g \neq 0\).
Therefore, there is a pair of points \(z, w \in C\), distinct from eachother and \(x_1, x_2, y_1, y_2\),
with \(\pi_i(z) = \pi_i(w)\) for both \(i\).
Geometrically, \(x_i, y_i, z, w\) are four distinct coplanar points.
Since \(x_i\) and \(y_i\) are general, \(\bar{x_i y_i}\) is not a trisecant to \(C\), so in particular,
\((x_i, y_i, z)\) and \((x_i, y_i, w)\) are not collinear.
Because \(x_i\) and \(y_i\) can be exchanged via monodromy, this implies
no three of \(x_i, y_i, z, w\) are collinear.

Our goal is to show interpolation for \(N_C[\posmodalong R_1 \cup R_2]\).
Specializing \(R_i\) to meet \(C\) at \(x_i, y_i, z, w\), this bundle specializes to
\[N_C[x_1 + y_1 \posmodalong R_1][x_2 + y_2 \posmodalong R_2](z + w).\]
Removing the twists at \(z\) and \(w\), we reduce to interpolation for
\[N_C[x_1 + y_1 \posmodalong R_1][x_2 + y_2 \posmodalong R_2].\]
Specializing \(R_i\) to the union of lines \(\bar{x_i y_i} \cup \bar{zw}\), this
bundle specializes to
\[N_C[x_1 \biposmod y_1][x_2 \biposmod y_2].\]
Interpolation for this bundle is the assertion \(I(d, g, 3, 2, 0)\).
Although the \((d, g, 3, 2, 0)\) are not good, our previous arguments still apply
in these cases:

\begin{description}
\item[\boldmath If \((d, g) = (4, 0)\)] We apply Proposition~\ref{master} with parameters:
\[\ell' = 1, \quad m' = 0, \quad \text{and} \quad d' = 3.\]
\item[\boldmath If \((d, g) = (5, 2)\)] We apply Proposition~\ref{master-erasable} with parameters:
\[\ell' = 1, \quad m' = m'' = \epsilon_{\text{in}} = \epsilon_{\text{out}} = 0, \quad d' = 3, \quad \text{and} \quad g' = 0.\]
(The required erasability of \((s_{10}, s_{11}, s_{20}, s_{21}, w_{10}) = (1, 0, 0, 2, 0)\)
can be checked by specializing the points in any order.)
\end{description}

\subsection{\texorpdfstring{\boldmath The cases \((d, g, r, \ell, m) = (4, 0, 3, 0, 1)\) and \((5, 0, 3, 0, 1)\)}{The cases (d, g, r, l, m) = (4, 0, 3, 0, 1) and (5, 0, 3, 0, 1)}}
In both of these cases, we want to show interpolation for 
\(N_C[\posmodalong R]\).
Write \(C \cap R = \{q_1, q_2, q_3, q_4\}\). Peel off a \(1\)-secant line, i.e., degenerate \(C\)
to \(C(1, 0; 0) \cup L\) --- but in such a way that \(q_4\) specializes onto \(L\),
while \(q_1\), \(q_2\), and \(q_3\) specialize onto \(C(1, 0; 0)\).
The restriction of the modified normal bundle to \(L\) is perfectly balanced of slope \(2\), so by
Lemma~\ref{lem:interpolation_rational_bal},
this reduces interpolation for \(N_C[\posmodalong R]\) to interpolation for
\[N_{C(1, 0; 0)}[q_1 + q_2 + q_3 \posmodalong R][z \posmod q_4].\]
Erasing the transformation \([z \posmod q_4]\), we reduce to interpolation for
\[N_{C(1, 0; 0)}[q_1 + q_2 + q_3 \posmodalong R].\]
Specializing \(R\) to the union of lines \(\bar{q_1 q_2} \cup \bar{q_3 q_4}\)
induces a specialization of this bundle to
\[N_{C(1, 0; 0)}[q_1 \biposmod q_2][q_3 \posmod q_4].\]
Erasing the transformation at \(q_3\), we reduce to interpolation for
\[N_{C(1, 0; 0)}[q_1 \biposmod q_2],\]
which is the assertion \(I(d - 1, 0, 3, 1, 0)\). Both
\((3, 0, 3, 1, 0)\) and \((4, 0, 3, 1, 0)\) are good.

\subsection{\texorpdfstring{\boldmath The case \((d, g, r, \ell, m) = (5, 2, 3, 0, 1)\)}{The case (d, g, r, l, m) = (5, 2, 3, 0, 1)}}
In this case, we want to show interpolation for 
\(N_C[\posmodalong R]\).
Write \(C \cap R = \{q_1, q_2, q_3, q_4\}\). Peel off a \(2\)-secant line, i.e., degenerate \(C\)
to \(C(0, 1; 0) \cup L\) --- but in such a way that \(q_3\) and \(q_4\) specialize onto \(L\),
while \(q_1\) and \(q_2\) specialize onto \(C(0, 1; 0)\).
The restriction of the modified normal bundle to \(L\) is perfectly balanced of slope \(3\), so by
Lemma~\ref{lem:interpolation_rational_bal},
this reduces interpolation for \(N_C[\posmodalong R]\) to interpolation for
\[N_{C(0, 1; 0)}[q_1 + q_2 \posmodalong R][z \biposmod w].\]
Let \(Q\) be the unique quadric containing \(C(0, 1; 0)\) and the line \(\bar{zw}\).
Then interpolation for this bundle follows from the balanced exact sequence
\[0 \to N_{C(0, 1; 0) / Q}(z + w) \to N_{C(0, 1; 0)}[q_1 + q_2 \posmodalong R][z \biposmod w] \to N_Q|_{C(0, 1; 0)}(q_1 + q_2) \to 0.\]

\subsection{\texorpdfstring{\boldmath The case \((d, g, r, \ell, m) = (5, 0, 4, 2, 0)\)}{The case (d, g, r, l, m) = (5, 0, 4, 2, 0)}} \label{ss:50420}
In this case, we want to show interpolation for
\[N_C[u_1 \biposmod v_1][u_2 \biposmod v_2].\]
Peel off a \(1\)-secant line, i.e., degenerate \(C\)
to \(C(1, 0; 0) \cup L\) --- but in such a way that \(v_1\) and \(v_2\) specialize onto \(L\),
while \(u_1\) and \(u_2\) specialize onto \(C(1, 0; 0)\).
By Lemma~\ref{lem:interpolation_rational_bal}, this reduces to interpolation for
\[N_{C(1, 0; 0)}[u_1 \posmod v_1][u_2 \posmod v_2][z \posmod v_2].\]
Specializing \(v_1\) to \(z\), we reduce to interpolation for
\[N_{C(1, 0; 0)}[u_1 \posmod z][u_2 \posmod v_2][z \posmod v_2].\]
Projecting from \(z\), we reduce to interpolation for
\[N_{C(1, 0; 1)}[u_2 \posmod v_2][z \posmod v_2].\]
Specializing \(v_2\) onto the line \(\bar{zu_2}\), we reduce to interpolation for
\[N_{C(1, 0; 1)}[z \biposmod u_2].\]
This is the assertion \(I(3, 0, 3, 1, 0)\), and \((3, 0, 3, 1, 0)\) is good.

\subsection{\texorpdfstring{\boldmath The cases \((d, g, r, \ell, m) = (6, 2, 4, 0, 2)\) and \((7, 3, 4, 0, 1)\)}{The cases (d, g, r, l, m) = (6, 2, 4, 0, 2) and (7, 3, 4, 0, 1)}} \label{ss:73401}
In these cases, we want to show interpolation for \(N_C[\posmodalong R_1 \cup \cdots \cup R_m]\).
Write \(R_i \cap C = \{q_{i1}, q_{i2}, q_{i3}, q_{i4}, q_{i5}\}\).
Note that \(m \leq g\) in both cases, so we may peel off \(m\) two-secant lines, i.e., degenerate \(C\)
to \(C(0, m; 0) \cup L_1 \cup \cdots \cup L_m\) --- but in such a way that \(q_{i5}\) specializes onto \(L_i\),
while the remaining \(q_{ij}\) specialize onto \(C(0, m; 0)\).
By Lemma~\ref{lem:interpolation_rational_bal},
this reduces interpolation for \(N_C[\posmodalong R_1 \cup \cdots \cup R_m]\) to interpolation for
\[N_{C(0, m; 0)}[q_{11} + q_{12} + q_{13} + q_{14} \posmodalong R_1] \cdots [q_{m1} + q_{m2} + q_{m3} + q_{m4} \posmodalong R_m] [z_1 \biposmod w_1] \cdots [z_m \biposmod w_m].\]
Specialize \(R_i\) to the union of lines \(\bar{q_{i1} q_{i2}}\), \(\bar{q_{i3}q_{i4}}\), and the unique line
through \(q_{i5}\) meeting both of these two lines.
This induces a specialization of the above bundle to
\[N_{C(0, m; 0)}[q_{11} \biposmod q_{12}][q_{13} \biposmod q_{14}] \cdots [q_{m1} \biposmod q_{m2}][q_{m3} \biposmod q_{m4}] [z_1 \biposmod w_1] \cdots [z_m \biposmod w_m].\]
In other words, all that remains is to check the assertion \(I(d - m, g - m, 4, 3m, 0)\).

\begin{description}
\item[\boldmath If \((d, g, m) = (6, 2, 2)\)] In this case, writing \(C\) for a curve of degree \(d - m = 4\) and genus
\(g - m = 0\), we want to establish interpolation for
\[N_C[u_1 \biposmod v_1]\cdots [u_6 \biposmod v_6].\]
Specializing ``to a tetrahedron'', i.e., specializing
\(u_1, u_2\) to \(u_3\), and \(u_4, v_6\) to \(v_1\), and \(v_4, u_5\) to \(v_2\), and \(v_5, u_6\) to \(v_3\),
this bundle specializes to
\[N_C(u_3 + v_1 + v_2 + v_3).\]
Removing the twists at \(u_3, v_1, v_2\), and \(v_3\), we reduce to interpolation for
\(N_C\), which is the assertion \(I(4, 0, 4, 0, 0)\).
Note that \((4, 0, 4, 0, 0)\) is good.

\item[\boldmath If \((d, g, m) = (7, 3, 1)\)] In this case, writing \(C\) for a curve of degree \(d - m = 6\) and genus
\(g - m = 2\), we want to establish interpolation for
\[N_C[u_1 \biposmod v_1][u_2 \biposmod v_2] [u_3 \biposmod v_3].\]
Note that \(\delta = 4\frac{2}{3}\).
Peeling off two \(2\)-secant lines, we reduce to interpolation for
\[N_{C(0, 2; 0)}[u_1 \biposmod v_1][u_2 \biposmod v_2] [u_3 \biposmod v_3][z_1 \biposmod w_1][z_1 \posmod 2w_1][z_2 \biposmod w_2][z_2 \posmod 2w_2].\]
Limiting \(w_1\) to \(w_2\), this bundle specializes to
\[N_{C(0, 2; 0)}[u_1 \biposmod v_1][u_2 \biposmod v_2] [u_3 \biposmod v_3][z_1 + z_2 \biposmod w_2][z_1 + z_2 \posmod 2w_2].\]
Projecting from \(w_2\), we reduce to interpolation for
\[N_{C(0, 2; 1)}[u_1 \biposmod v_1][u_2 \biposmod v_2] [u_3 \biposmod v_3][z_1 + z_2 \posmod w_2].\]
Limiting \(v_3\) to \(w_2\), and \(v_2\) to \(v_1\), we reduce to interpolation for
\[N_{C(0, 2; 1)}[u_1 + u_2 \posmod v_1] [z_1 + z_2 + u_3 \posmod w_2][w_2 \posmod u_3].\]
Projecting from \(w_2\) again, we reduce to interpolation for
\[N_{C(0, 2; 2)}(u_1 + u_2 + w_2),\]
which is a nonspecial line bundle and therefore satisfies interpolation.
\end{description}

\subsection{\texorpdfstring{\boldmath The cases \((d, g, r, \ell, m) = (5, 0, 4, 0, 1)\), and \((7, 3, 4, 1, 1)\)}{The cases (d, g, r, l, m) = (5, 0, 4, 0, 1), and (7, 3, 4, 1, 1)}}
We want to show interpolation for a vector bundle of rank \(3\),
and degree \(28\) and \(46\) respectively.
By Lemma~\ref{lem:muzz}, it suffices to check interpolation for corresponding vector bundles of rank \(27\) and \(45\)
where one positive transformation is omitted.

For \((d, g, r, \ell, m) = (5, 0, 4, 0, 1)\), our bundle
of degree \(28\) is \(N_C[\posmodalong R]\).
Write \(C \cap R = \{q_1, q_2, q_3, q_4, q_5\}\). It suffices to establish interpolation for
the degree \(27\) vector bundle \(N_C[q_1 + q_2 + q_3 + q_4 \posmodalong R]\).
Specialize \(R\) to the union of the lines \(\bar{q_1q_2}\), \(\bar{q_3q_4}\), and the unique line
through \(q_5\) meeting both of these two lines. This induces a specialization of this bundle to
\(N_C[q_1 \biposmod q_2][q_3 \biposmod q_4]\),
which is the assertion \(I(5, 0, 4, 2, 0)\).
Observe that \((5, 0, 4, 2, 0)\) was already considered above in Section~\ref{ss:50420}.

For \((d, g, r, \ell, m) = (7, 3, 4, 1, 1)\), our bundle
of degree \(46\) is \(N_C[u \biposmod v][\posmodalong R]\), and we can reduce to interpolation
for the degree \(45\) vector bundle
\(N_C[u \posmod v][\posmodalong R]\). Erasing the transformation at \(u\), this reduces to interpolation for
\(N_C[\posmodalong R]\), which is the assertion \(I(7, 3, 4, 0, 1)\).
Observe that \((7, 3, 4, 0, 1)\) was already considered above in Section~\ref{ss:73401}.

\subsection{\texorpdfstring{\boldmath The case \((d, g, r, \ell, m) = (7, 1, 5, 0, 1)\)}{The case (d, g, r, l, m) = (7, 1, 5, 0, 1)}}
In this case
we want to show interpolation for \(N_C[\posmodalong R]\).
Write \(R\cap C = \{q_1, q_2, q_3, q_4, q_5, q_6\}\).
Peel off a \(2\)-secant line, i.e., degenerate \(C\) to \(C(0, 1; 0) \cup L\) --- but in such a way that \(q_2\) and \(q_4\)
specialize onto \(L\) and the remaining points specialize onto \(C(0,1;0)\).  
By Lemma~\ref{lem:interpolation_rational_bal},
this reduces interpolation for \(N_C[\posmodalong R]\) to interpolation for
\[N_{C(0, 1; 0)}[q_1 + q_3 + q_5 + q_6 \posmodalong R][u \biposmod v].\]
Specializing \(R\) to the union of the three lines \(\bar{q_1 q_2}\), \(\bar{q_3 q_4}\), \(\bar{q_5 q_6}\), and the unique fourth line in \(\pp^5\) meeting these three lines, we reduce to interpolation for
\[N_{C(0,1;0)}[q_1 \posmod q_2][q_3 \posmod q_4][q_5 \biposmod q_6][u \biposmod v].\]
Limiting \(q_2\) to \(u\) and \(q_4\) to \(v\), we reduce to proving interpolation for
\[N_{C(0,1;0)}[q_1+v \posmod u][q_3+u \posmod v][q_5 \biposmod q_6]. \]
Projecting from \(u\) and then \(v\),
we reduce to interpolation for
\[N_{C(0,1;1)}[q_5 \biposmod q_6].\]
This is the assertion \(I(4, 0, 3, 1, 0)\), and \((4, 0, 3, 1, 0)\) is good.

\subsection{\texorpdfstring{\boldmath The case \((d, g, r, \ell, m) = (7, 2, 5, 0, 1)\)}{The case (d, g, r, l, m) = (7, 2, 5, 0, 1)}}
We want to show that
\(N_C[\posmodalong R]\) satisfies interpolation.  Peeling off a \(2\)-secant line,
we reduce to interpolation for
\[N_{C(0,1;0)}[\posmodalong R][z \biposmod w][z \posmod 2w].\]
We now specialize \(R\) as in Section \ref{sec:onion_specialization} to the union of two lines \(\bar{q_1p}\) and \(\bar{q_2 p}\), and a \(3\)-secant conic through \(q_3, q_4, p\).
Then limit \(w\) to \(p\).  This induces a specialization of our bundle to
\[N_{C(0,1;0)}[q_3 + q_4 \posmodalong R^\circ][z+q_1 + q_2 \posmod p][z \posmod 2p][p \posmod z+q_1 + q_2].\]
Projecting from \(p\) we reduce to interpolation for
\[N_{C(0,1;1)}[q_3 \biposmod q_4][z \posmod p].\]
Limiting \(q_4\) to \(p\), we reduce to interpolation for
\[N_{C(0,1;1)}[q_3 + z \posmod p][p \posmod q_3].\]
Projecting from \(p\), we reduce to interpolation for
\(N_{C(0,1;2)}[p \posmod q_3]\). Erasing the transformation \([p \posmod q_3]\),
this reduces to \(I(4, 1, 3, 0, 0)\), and \((4, 1, 3, 0, 0)\) is good.

\subsection{\texorpdfstring{\boldmath The case \((d, g, r, \ell, m) = (7, 2, 5, 2, 2)\)}{The case (d, g, r, l, m) = (7, 2, 5, 2, 2)}}
This case asserts interpolation for
\[N_C[\posmodalong R_1 + R_2][u_1 \biposmod v_1][u_2 \biposmod v_2].\]
We first specialize each \(R_i\) as in Section \ref{sec:onion_specialization} to the union of two lines \(\bar{q_{i1}p_i}, \bar{q_{i2}p_i}\) and a \(3\)-secant conic through \(\{p_i, q_{i3}, q_{i4}\}\).
We then specialize \(p_1\) and \(p_2\) together to a common point \(p\).
This induces a specialization of our bundle to
\[N_C(p)[q_{11} + q_{12} + q_{21} + q_{22} \posmod p][q_{13} + q_{14} \posmodalong R_1^\circ][q_{23} + q_{24} \posmodalong R_2^\circ][u_1 \biposmod v_1][u_2 \biposmod v_2].\]
Limiting \(u_1\) to \(p\) and removing the overall twist at \(p\) reduces us to interpolation for
\[N_C[v_1 + q_{11} + q_{12} + q_{21} + q_{22} \posmod p][q_{13} + q_{14} \posmod Q_1][q_{23} + q_{24} \posmod Q_2][p \posmod v_1][u_2 \biposmod v_2].\]
Projecting from \(p\), we reduce to interpolation for
\[N_{C(0,0;1)}[q_{13} \biposmod q_{14} ][q_{23} \biposmod q_{24}][p \posmod v_1][u_2 \biposmod v_2].\]
Erasing the transformation \([p \posmod v_1]\), and peeling off two \(2\)-secant lines, we reduce to
\[N_{C(0,2;1)}[q_{13} \biposmod q_{14} ][q_{23} \biposmod q_{24}][u_2 \biposmod v_2] [z_1 \biposmod w_1][z_1 \posmod 2w_1][z_2 \biposmod w_2][z_2 \posmod 2w_2].\]
Limiting \(w_1\) and \(w_2\) to \(p\), we reduce to interpolation for
\[N_{C(0,2;1)}[q_{13} \biposmod q_{14}][q_{23} \biposmod q_{24}][u_2 \biposmod v_2] [z_1 + z_2 \biposmod p][z_1 + z_2 \posmod 2p].\]
Projecting from \(p\), we reduce to interpolation for
\[N_{C(0,2;2)}[q_{13} \biposmod q_{14}][q_{23} \biposmod q_{24}][u_2 \biposmod v_2] [z_1 + z_2 \posmod p].\]
Limiting \(q_{24}\) to \(q_{14}\), and \(v_2\) to \(p\), and removing the resulting twist at \(q_{14}\), we reduce to
\[N_{C(0,2;2)}[q_{13} + q_{23} \posmod q_{14}][u_2 + z_1 + z_2 \posmod p][p \posmod u_2].\]
Projecting from \(p\), we reduce to interpolation for \(N_{C(0,2;3)}\),
which is a nonspecial line bundle.

\subsection{\texorpdfstring{\boldmath The case \((d, g, r, \ell, m) = (9, 2, 5, 0, 0)\)}{The case (d, g, r, l, m) = (9, 2, 5, 0, 0)}}
Peeling off two \(2\)-secant lines reduces to interpolation for
\[N_{C(0,2;0)}[ z_1 \biposmod w_1][z_1 \posmod 2 w_1][ z_2 \biposmod w_2][z_2 \posmod 2 w_2]. \]
Limit the points \(z_1\) and \(w_2\) to a common point \(p\).  This induces the specialization of our bundle to
\[N_{C(0,2;0)}[ p \biposmod w_1 + z_2][p \posmod 2 w_1][z_2 \posmod 2 p]. \]
Projection from \(p\) reduces to interpolation for
\[N_{C(0,2;1)}[p \posmod w_1][z_2 \posmod p].\]
Erasing the transformation \([p \posmod w_1]\) and then projecting from \(p\),
we reduce to interpolation for \(N_{C(0,2;2)}\).
This is \(I(5, 0, 3, 0, 0)\), and \((5, 0, 3, 0, 0)\) is good.

\subsection{\texorpdfstring{\boldmath The cases \((d, g, r, \ell, m) = (9, 4, 5, 0, 0)\) and \((9, 4, 5, 1, 0)\)}{The cases (d, g, r, l, m) = (9, 4, 5, 0, 0) and (9, 4, 5, 1, 0)}}
We want that both \(N_C\) and \(N_C[u \biposmod v]\) satisfy interpolation.
Peeling off four \(2\)-secant lines, we reduce to interpolation for
\begin{gather*}
N_{C(0,4;0)}[z_1 \biposmod w_1][z_1 \posmod 2w_1][z_2 \biposmod w_2][z_2 \posmod 2w_2] [z_3 \biposmod w_3][z_3 \posmod 2w_3][z_4 \biposmod w_4][z_4 \posmod 2w_4] \text{ and} \\
N_{C(0,4;0)}[z_1 \biposmod w_1][z_1 \posmod 2w_1][z_2 \biposmod w_2][z_2 \posmod 2w_2] [z_3 \biposmod w_3][z_3 \posmod 2w_3][z_4 \biposmod w_4][z_4 \posmod 2w_4][u \biposmod v].
\end{gather*}
Specializing \(w_2\) to \(w_1\), and \(w_4\) to \(w_3\), we reduce to interpolation for
\begin{gather*}
N_{C(0,4;0)}[z_1 + z_2 \biposmod w_1][z_1 + z_2 \posmod 2w_1] [z_3 + z_4 \biposmod w_3][z_3 + z_4 \posmod 2w_3] \text{ and} \\
N_{C(0,4;0)}[z_1 + z_2 \biposmod w_1][z_1 + z_2 \posmod 2w_1] [z_3 + z_4 \biposmod w_3][z_3 + z_4 \posmod 2w_3][u \biposmod v].
\end{gather*}
Projecting from \(w_1\) and then \(w_3\), we reduce to interpolation for
\[N_{C(0,4;2)}[z_1 + z_2 \posmod w_1][z_3 + z_4 \posmod w_3] \qand N_{C(0,4;2)}[z_1 + z_2 \posmod w_1][z_3 + z_4 \posmod w_3] [u \biposmod v].\]
Specializing \(v\) to \(w_1\), we reduce to interpolation for
\[N_{C(0,4;2)}[z_1 + z_2 \posmod w_1][z_3 + z_4 \posmod w_3] \qand N_{C(0,4;2)}[z_1 + z_2 + u \posmod w_1][z_3 + z_4 \posmod w_3] [w_1 \posmod u].\]
Projecting from \(w_1\), we reduce to interpolation for
\(N_{C(0,4;3)}\), which is a nonspecial line bundle.

\subsection{\texorpdfstring{\boldmath The cases \((d, g, r, \ell, m) = (8, 3, 5, 2, 0)\) and \((11, 5, 6, 0, 0)\)}{The cases (d, g, r, l, m) = (8, 3, 5, 2, 0) and (11, 5, 6, 0, 0)}}
We first reduce interpolation for both of these bundles to the same statement.

\begin{description}
\item[\boldmath For \((d, g, r, \ell, m) = (11, 5, 6, 0, 0)\)] Note that \(\delta = 4\). Our goal is to establish
interpolation for \(N_C\). We first peel off \(2\) two-secant lines,
which reduces our problem to interpolation for 
\[N_{C(2, 0; 0)}[z_1 \biposmod w_1][z_1 \posmod 2w_1] [z_2 \biposmod w_2][z_2 \posmod 2w_2].\]
Limiting \(w_2\) to \(w_1\) induces a specialization of this bundle to
\[N_{C(2, 0; 0)}[z_1 \biposmod w_1][z_1 \posmod 2w_1] [z_2 \biposmod w_1][z_2 \posmod 2w_1].\]
Projecting from \(w_1\), we reduce to interpolation for
\[N_{C(2, 0; 1)}[z_1 + z_2 \biposmod w_1].\]

\item[\boldmath For \((d, g, r, \ell, m) = (8, 3, 5, 2, 0)\)] Our goal is to establish interpolation
for
\[N_C[u_1 \biposmod v_1][u_2 \biposmod v_2].\]
Limiting \(v_2\) to \(v_1\), we reduce to interpolation for
\[N_C[u_1 + u_2 \biposmod v_1].\]
\end{description}

To finish the argument, let \(C\) be a general BN-curve of degree \(8\) and genus \(3\) in \(\pp^5\),
and \(p, q_1, q_2 \in C\) be general points. Above we have shown that both of the desired assertions
reduce to interpolation for the modified normal bundle
\[N_C[q_1 + q_2 \biposmod p].\]
We next peel off two \(2\)-secant lines, i.e., degenerate
\(C\) to \(C \cup L_1 \cup L_2\), where \(L_1\) and \(L_2\)
are \(2\)-secant lines to \(C\) --- but in such a way that
\(q_i\) limits onto \(L_i\), and \(p\) limits onto \(C\).
Applying Lemma \ref{lem:2sec}, 
we reduce to interpolation for
\[N_{C(0, 2; 0)} [z_1 \biposmod w_1][z_1 \posmod 2w_1 + p] [z_2 \biposmod w_2][z_2 \posmod 2w_2 + p] [p \posmod q_1 + q_2].\]
Over the function field of the moduli space of \emph{unordered} pairs of triples \(\{(z_1, w_1, q_1), (z_2, w_2, q_2)\}\),
the transformation \([p \posmod q_1 + q_2]\) is linearly general as just \(q_1\) and \(q_2\) vary.
Indeed, geometrically, it is transverse to any subspace of the normal space at \(p\) \emph{except}
for the two subspaces \(N_{C \to L_1}|_p\) and \(N_{C \to L_2}|_p\) --- but neither of these subspaces
is rational over this function field. Therefore, we may erase the transformation at \(p\), thereby
reducing to interpolation for
\[N_{C(0, 2; 0)} [z_1 \biposmod w_1][z_1 \posmod 2w_1 + p] [z_2 \biposmod w_2][z_2 \posmod 2w_2 + p].\]
Note that \(\delta = 3\frac{1}{2}\) for this bundle.
Peeling off a \(2\)-secant line, we reduce to interpolation for
\[N_{C(0, 3; 0)} [z_1 \biposmod w_1][z_1 \posmod 2w_1 + p] [z_2 \biposmod w_2][z_2 \posmod 2w_2 + p] [z_3 \biposmod w_3][z_3 \posmod 2w_3].\]
Specializing \(w_3\) to \(p\), and \(w_2\) to \(w_1\), we reduce to interpolation for
\[N_{C(0, 3; 0)} [z_1 + z_2 \biposmod w_1][z_1 + z_2 \posmod 2w_1] [z_1 + z_2 + z_3 \posmod p] [z_3 \posmod 2p][p \posmod z_3].\]
Projecting from \(w_1\) and then \(p\), we reduce to interpolation for
\[N_{C(0, 3; 2)} [z_1 + z_2 \posmod w_1] [z_3 \biposmod p].\]
Finally, projecting from \(w_1\) again, we reduce to interpolation for
\(N_{C(0, 3; 3)}\), which is a nonspecial line bundle.

\subsection{\texorpdfstring{\boldmath The case \((d, g, r, \ell, m) = (7, 0, 6, 0, 1)\)}{The case (d, g, r, l, m) = (7, 0, 6, 0, 1)}}
Arguing as in the proof of Proposition \ref{master-111}, it suffices to show that \(Q^-\) and \(Q^+\) satisfy interpolation, where
\[Q^- = N_{C(0, 0;1)}[s_1 + \cdots + s_5 \posmodalong \bar{R}], \qand Q^+ = Q^-[p \posmod s_0].\]
As in the proof of Proposition \ref{master-111}, interpolation for \(Q^+\) follows from interpolation for \(Q^-\) given the assertion \(I(5, 0, 4, 0, 1)\).  Since \((5, 0, 4, 0, 1)\) is good, it suffices to prove interpolation for \(Q^-\).  By Lemma \ref{lem:int_large_degree}, this follows in turn from interpolation for 
\[Q^-(s_0) = Q^-[s_0 \posmodalong \bar{R}][s_0\posmod \Lambda],\]
where \(\Lambda \subset Q^-|_{s_0}\) is codimension \(1\).
By Lemma~\ref{lem:muzz}, since \(\mu(Q^-[s_0 \posmodalong \bar{R}]) \in \zz\),
interpolation for \(Q^-(s_0)\) follows from interpolation for 
\[Q^-[s_0 \posmodalong \bar{R}] = N_{C(0,0;1)}[\posmodalong \bar{R}].\]
This is the assertion \(I(6, 0, 5, 0,1)\), and \((6,0,5,0,1)\) is good.

\subsection{\texorpdfstring{\boldmath The cases \((d, g, r, \ell, m) = (8, 1, 7, 0, 1)\) and \((8, 1, 7, 1, 1)\)}{The cases (d, g, r, l, m) = (8, 1, 7, 0, 1) and (8, 1, 7, 1, 1)}}
In both of these cases, our goal is to show interpolation for
\[N_C[u_1 \biposmod v_1]\cdots[u_\ell \biposmod v_\ell][\posmodalong R_1].\]
We specialize \(R_1\) to the union of the lines \(q_1 q_2\), \(q_3 q_4\), \(q_5 q_6\), \(q_7 q_8\),
together with a plane conic meeting each of these four lines.
This induces a specialization of this bundle to
\[N_C[u_1 \biposmod v_1]\cdots[u_\ell \biposmod v_\ell] [q_1 \biposmod q_2] [q_3 \biposmod q_4] [q_5 \biposmod q_6] [q_7 \biposmod q_8].\]
Note that the points \(q_1, q_2, q_3, q_4, q_5, q_6, q_7, q_8\) are \emph{not} general,
as they are constrained to lie in a hyperplane.

Let \(p_1, p_2, p_3, p_4 \in C\) be points with \(\O_C(1) = 2p_1 + 2p_2 + 2p_3 + 2p_4\)
(such points exist by Riemann--Roch because \(C\) is an elliptic curve). By construction, \(H^0(\O_C(1)(-2p_1 - 2p_2 - 2p_3 - 2p_4)) = 1\);
as \(C\) is embedded by a complete linear series, we conclude that the tangent lines to \(C\)
at \(p_1, p_2, p_3, p_4\) span a hyperplane~\(H\).
Specializing the hyperplane containing \(q_1, q_2, q_3, q_4, q_5, q_6, q_7, q_8\)
to \(H\), in such a way that \(q_1\) and \(q_8\) specialize to \(p_1\), and \(q_2\) and \(q_3\) specialize to \(p_2\),
and \(q_4\) and \(q_5\) specialize to \(p_3\), and \(q_6\) and \(q_7\) specialize to \(p_4\), 
we obtain a further specialization of the above bundle to
\begin{equation}\label{eq:bundle} N_C[u_1 \biposmod v_1]\cdots[u_\ell \biposmod v_\ell] [p_1 + p_3 \biposmod p_2 + p_4].\end{equation}
Note that, since \(H^0(\O_C(1)(-p_1 - p_2 - p_3 - p_4)) = 4\),
the points \(p_1, p_2, p_3, p_4\) are linearly independent.
We claim interpolation for \eqref{eq:bundle} reduces to interpolation for
\(N_{C(0, 0; 4)}\). To see this, we divide into cases as follows:

\begin{description}
\item[\boldmath If \(\ell = 0\)] Note that \(\delta = 2\). Our goal is to show interpolation for
\[N_C [p_1 + p_3 \biposmod p_2 + p_4].\]
Projecting from each of \(p_1, p_2, p_3, p_4\) in turn,
we reduce to interpolation for
\(N_{C(0, 0; 4)}\).

\item[\boldmath If \(\ell = 1\)] Note that \(\delta = 2\frac{1}{3}\). Our goal is to show interpolation for
\[N_C[u_1 \biposmod v_1] [p_1 + p_3 \biposmod p_2 + p_4].\]
Specializing \(u_1\) to \(p_1\) and \(v_1\) to \(p_3\), we reduce to interpolation for
\[N_C[p_1 \biposmod p_3] [p_1 + p_3 \biposmod p_2 + p_4].\]
Projecting from \(p_1\), \(p_2\), \(p_3\), and \(p_4\), we reduce to interpolation for
\(N_{C(0, 0; 4)}\).
\end{description}

\noindent
It remains to check interpolation for
\(N_{C(0, 0; 4)}\), which is the assertion \(I(4, 1, 3, 0, 0)\),
and \((4, 1, 3, 0, 0)\) is good.

\subsection{\texorpdfstring{\boldmath The case \((d, g, r, \ell, m) = (11, 4, 7, 1, 0)\)}{The case (d, g, r, l, m) = (11, 4, 7, 1, 0)}}
Our goal is to show interpolation for \(N_C[u \biposmod v]\).
Note that \(\delta = 3\). Peeling off a \(2\)-secant line, we reduce to interpolation for
\[N_{C(0, 1; 0)}[z \biposmod w][z \posmod 2w][u \biposmod v].\]
Specializing \(v\) to \(w\), we reduce to interpolation for
\[N_{C(0, 1; 0)}[z + u \biposmod w][z \posmod 2w].\]
Projecting from \(w\), we reduce to interpolation for
\[N_{C(0, 1; 1)}[z \posmod w][w \posmod z+u].\]
Erasing the transformation \([w \posmod z+u]\), we reduce to interpolation for
the two bundles
\[N_{C(0, 1; 1)}[z \biposmod w] \quad \text{and} \quad N_{C(0, 1; 1)}[z \posmod w].\]
The first of these statements is the assertion \(I(9, 3, 6, 1, 0)\).
For the second, erasing the transformation at \(z\)
reduces it to interpolation for \(N_{C(0, 1; 1)}\), which is the assertion
\(I(9, 3, 6, 0, 0)\).
Note that both \((9, 3, 6, 0, 0)\) and \((9, 3, 6, 1, 0)\) are good.

\section{Canonical curves of even genus}\label{sec:canonical}

In this section we prove interpolation for the normal bundle of a general canonical curve of even genus \(g \geq 8\),
which is the last remaining case (cf.\ Proposition~\ref{prop:I_implies_interpolation} and Section~\ref{sec:rational}).
These cases are difficult, in part, because interpolation does not hold for canonical curves of genus \(4\) and \(6\), i.e., when \(r=3\) or \(r=5\).

We will do this via degeneration to \(E \cup R\), as in Section~\ref{subsec:onion}.
That is, \(E\) is an elliptic normal curve in \(\pp^r\),
and \(R\) is a general \((r+1)\)-secant rational curve of degree \(r-1\),
where \(r = g - 1\) is odd.

\subsection{\boldmath Reduction to a bundle on \(E\)} \label{sec:red-E}
Recall that, due to the exceptional case of elliptic normal curves in odd-dimensional projective spaces in Lemma \ref{lem:rest_onion_bal}, we cannot reduce interpolation for \(N_{E \cup R}\) to interpolation for \(N_{E \cup R}|_E\).
Instead, we will reduce
interpolation for \(N_{E \cup R}\) to interpolation for a certain \emph{modification}
of \(N_{E \cup R}|_E\).
Our first step will be to show that \(N_{E \cup R}|_R\) is not perfectly balanced,
and give a geometric description of its Harder--Narasimhan (HN) filtration.

\begin{lem}\label{map_of_ell_curve}
Let \(q_1, \dots, q_{2n+2}\) be a general collection of points on \(\pp^1\).  Let \(p_1, \dots, p_{2n+2}\) be a general collection of points
on a general elliptic curve \(E\).  Then there exist exactly two maps
of degree \(n+1\) from \(E\) to \(\pp^1\) that send \(p_i\) to \(q_i\).
\end{lem}
\begin{proof}
If \(f \colon E \to \pp^1\) is a general map of degree \(n+1\),
then \(f^*T_{\pp^1}(-p_1 - \cdots -p_{2n+2})\) has vanishing cohomology.
Therefore, deformations of \(f\) are in bijection with deformations of the \(f(p_i)\).
The number of maps
of degree \(n+1\) from \(E\) to \(\pp^1\) that send \(p_i\) to \(q_i\)
is therefore finite and nonzero.

To calculate this number, we degenerate the target \(\pp^1\) to a binary curve,
with \(q_1, q_2, \ldots, q_{n+1}\) on the left \(\pp^1\), and \(q_{n + 2}, q_{n + 3}, \ldots, q_{2n+2}\)
on the right \(\pp^1\). This degeneration is illustrated in the following diagram:

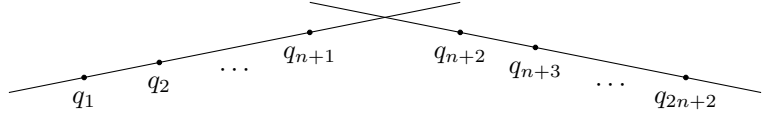
\begin{center}
\begin{tikzpicture}
\draw (-5, -1) -- (1, 0.2);
\draw (5, -1) -- (-1, 0.2);
\filldraw (-4, -0.8) circle[radius=0.03];
\filldraw (-3, -0.6) circle[radius=0.03];
\filldraw (-1, -0.2) circle[radius=0.03];
\filldraw (1, -0.2) circle[radius=0.03];
\filldraw (2, -0.4) circle[radius=0.03];
\filldraw (4, -0.8) circle[radius=0.03];
\draw (-4, -1.1) node{\(q_1\)};
\draw (-3, -0.9) node{\(q_2\)};
\draw (-1, -0.5) node{\(q_{n+1}\)};
\draw (-2, -0.7) node{\(\cdots\)};
\filldraw(1, -0.5) node{\(q_{n+2}\)};
\filldraw(2, -0.7) node{\(q_{n+3}\)};
\draw (3, -0.9) node{\(\cdots\)};
\filldraw(4, -1.1) node{\(q_{2n+2}\)};
\end{tikzpicture}
\end{center}

Such a map \(E \to \pp^1\) then degenerates to an admissible cover
from a marked curve whose stable model is \((E, p_1, p_2, \ldots, p_{2n+2})\).
One can construct two such admissible covers
(both with no infinitesimal deformations sending \(p_i\) to \(q_i\)): In one such cover, \(E\) maps to the left
component, with \(p_1, p_2, \ldots, p_{n+1}\) mapping to \(q_1, q_2, \ldots, q_{n+1}\),
and \(p_{n + 2}, p_{n + 3}, \ldots, p_{2n+2}\) mapping to the node;
each of the points \(p_{n + 2}, p_{n + 3}, \ldots, p_{2n+2}\) is then attached to a rational
tail mapping isomorphically onto the right component.
Similarly, in the other such cover, \(E\) maps to the right
component, with \(p_{n + 2}, p_{n + 3}, \ldots, p_{2n+2}\) mapping to \(q_{n + 2}, q_{n + 3}, \ldots, q_{2n+2}\),
and \(p_1, p_2, \ldots, p_{n+1}\) mapping to the node and attached to a rational
tail mapping isomorphically onto the left component.
These covers are pictured in the following diagrams:
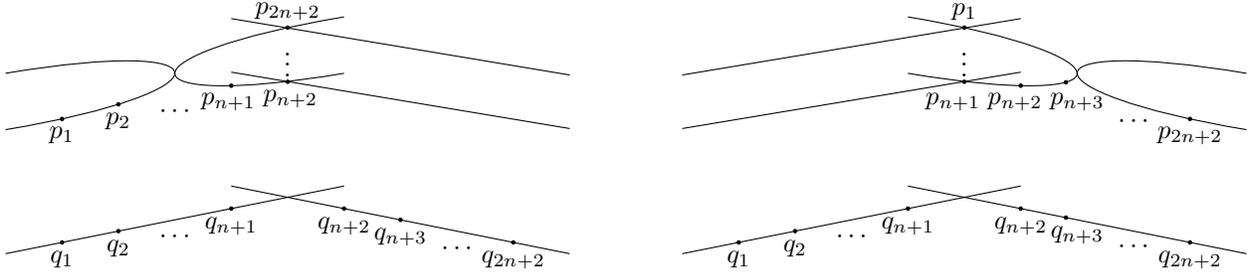
\begin{center}
\begin{tikzpicture}[scale=0.75]
\draw (1, 3.2) .. controls (-3, 2.53) and (-3, 1.53) .. (1, 2.2);
\draw (-5, 2.2) .. controls (-1, 2.87) and (-1, 1.87) .. (-5, 1.2);
\draw (-1, 3.17) -- (5, 2.17);
\draw (-1, 2.22) -- (5, 1.22);
\draw (-5, -1) -- (1, 0.2);
\draw (5, -1) -- (-1, 0.2);
\filldraw (-4, -0.8) circle[radius=0.03];
\filldraw (-3, -0.6) circle[radius=0.03];
\filldraw (-1, -0.2) circle[radius=0.03];
\filldraw (1, -0.2) circle[radius=0.03];
\filldraw (2, -0.4) circle[radius=0.03];
\filldraw (4, -0.8) circle[radius=0.03];
\filldraw (-4, 1.39) circle[radius=0.03];
\filldraw (-3, 1.65) circle[radius=0.03];
\filldraw (-1, 1.98) circle[radius=0.03];
\filldraw (0, 2.05) circle[radius=0.03];
\filldraw (0, 3.01) circle[radius=0.03];
\draw (-4, 1.1) node{\(p_1\)};
\draw (-3, 1.35) node{\(p_2\)};
\draw (-2, 1.53) node{\(\cdots\)};
\draw (-1.05, 1.7) node{\(p_{n+1}\)};
\draw (-4, -1.1) node{\(q_1\)};
\draw (-3, -0.9) node{\(q_2\)};
\draw (-1, -0.5) node{\(q_{n+1}\)};
\draw (-2, -0.7) node{\(\cdots\)};
\draw (0.05, 1.75) node{\(p_{n + 2}\)};
\draw (0, 3.3) node{\(p_{2n+2}\)};
\draw (0, 2.5) node{\(\vdots\)};
\filldraw(1, -0.5) node{\(q_{n+2}\)};
\filldraw(2, -0.7) node{\(q_{n+3}\)};
\draw (3, -0.9) node{\(\cdots\)};
\filldraw(4, -1.1) node{\(q_{2n+2}\)};
\end{tikzpicture}
\hfill
\begin{tikzpicture}[scale=0.75]
\draw (-1, 3.2) .. controls (3, 2.53) and (3, 1.53) .. (-1, 2.2);
\draw (5, 2.2) .. controls (1, 2.87) and (1, 1.87) .. (5, 1.2);
\draw (1, 3.17) -- (-5, 2.17);
\draw (1, 2.22) -- (-5, 1.22);
\draw (5, -1) -- (-1, 0.2);
\draw (-5, -1) -- (1, 0.2);
\filldraw (4, -0.8) circle[radius=0.03];
\filldraw (1.8, -0.36) circle[radius=0.03];
\filldraw (1, -0.2) circle[radius=0.03];
\filldraw (-1, -0.2) circle[radius=0.03];
\filldraw (-3, -0.6) circle[radius=0.03];
\filldraw (-4, -0.8) circle[radius=0.03];
\filldraw (4, 1.39) circle[radius=0.03];
\filldraw (1.8, 2.04) circle[radius=0.03];
\filldraw (1, 1.98) circle[radius=0.03];
\filldraw (0, 2.05) circle[radius=0.03];
\filldraw (0, 3.01) circle[radius=0.03];
\draw (4, 1.1) node{\(p_{2n+2}\)};
\draw (3, 1.35) node{\(\cdots\)};
\draw (2.0, 1.7) node{\(p_{n+3}\)};
\draw (0.9, 1.7) node{\(p_{n+2}\)};
\draw (4, -1.1) node{\(q_{2n+2}\)};
\draw (3.0, -0.88) node{\(\cdots\)};
\draw (1, -0.5) node{\(q_{n+2}\)};
\draw (2.0, -0.66) node{\(q_{n+3}\)};
\draw (-0.2, 1.7) node{\(p_{n+1}\)};
\draw (0, 3.3) node{\(p_1\)};
\draw (0, 2.5) node{\(\vdots\)};
\filldraw(-1, -0.5) node{\(q_{n+1}\)};
\filldraw(-3, -0.9) node{\(q_2\)};
\draw (-2, -0.7) node{\(\cdots\)};
\filldraw(-4, -1.1) node{\(q_1\)};
\end{tikzpicture}
\end{center}
In fact, these are the only two such admissible covers.
Indeed, the curve \(E\) must map to one of the two components of the above degeneration of the target \(\pp^1\), say without loss of generality to the left component.  Then \(p_{n+2}, \dots, p_{2n+2}\) must map to the node,
which we normalize to \([1:0]\).
Hence the map \(E \to \pp^1\) is given by \([s:1]\) for a section \(s \in H^0(\O(p_{n+2} + \cdots + p_{2n+2}))\).
Since \(p_1, \ldots, p_{n+1}\) are general, the evaluation map \(H^0(\O(p_{n+2} + \cdots + p_{2n+2})) \to \bigoplus_{i = 1}^{n+1} \O(p_{n+2} + \cdots + p_{2n+2})|_{p_i}\) is an isomorphism.
Hence \(s\) is uniquely determined.
We conclude that, when \(q_1, q_2, \ldots, q_{2n + 2} \in \pp^1\) are general,
there are exactly two such maps.
\end{proof}

\noindent
Write \(f_i \colon E \to \pp^1\) (for \(i \in \{1, 2\}\)) for these two maps,
and \(\bar{f} = (f_1, f_2) \colon E \to \pp^1 \times \pp^1\) for the resulting map.

\begin{lem} \label{ell_nodal}
In the setup of Lemma~\ref{map_of_ell_curve}, the map \(\bar{f}\) is
a general map from \(E\) to \(\pp^1\times \pp^1\) of bidegree \((n+1, n+1)\). In particular:
\begin{itemize}
\item \(\bar{f}\) is birational onto its image, and its image is nodal.
\item \(\bar{f}^* \O_{\pp^1\times \pp^1}(1, -1) \in \Pic^0(E)\) is general (and thus nontrivial).
\end{itemize}
\end{lem}
\begin{proof}
Fix a general elliptic curve \(E\).
Let \(f \colon E \to \pp^1 \times \pp^1\) be a general map.
Write \(\Delta \subset \pp^1 \times \pp^1\) for the diagonal.
Then the \(\{p_1, p_2, \ldots, p_{2n+2}\} = f^{-1}(\Delta) \subset E\),
and their images \(\{q_1, q_2, \ldots, q_{2n+2}\} \subset \pp^1\) under the composition of either projection with \(f\),
satisfy all the hypotheses of Lemma~\ref{map_of_ell_curve} except for possibly
genericity.

To complete the proof, it remains to check that there is no obstruction
to deforming \(f\) so that these points become general. In other words,
we must check \(H^1(f^* T_{\pp^1 \times \pp^1}(-\Delta)) = 0\).
But this is true because \(f\) is general and \(f^* T_{\pp^1 \times \pp^1}(-\Delta) \simeq f^*\O_{\pp^1 \times \pp^1}(1, -1) \oplus f^*\O_{\pp^1 \times \pp^1}(-1, 1)\).
\end{proof}

\begin{lem} \label{p1p2n} In the setup of Lemma~\ref{map_of_ell_curve}, we have
\[\bar{f}^* \O_{\pp^1 \times \pp^1}(1, 1) \simeq \O_E(p_1 + p_2 + \cdots + p_{2n+2}).\]
\end{lem}
\begin{proof}
The isomorphism class of the line bundle \(\bar{f}^* \O_{\pp^1 \times \pp^1}(1, 1)\) is independent of the moduli of the points \(q_1, \dots, q_{2n+2}\) because they vary in a rational base.  Hence we may calculate it in the degeneration of Lemma~\ref{map_of_ell_curve}, where the result clearly holds.
\end{proof}

\begin{lem} \label{push} In the setup of Lemma~\ref{map_of_ell_curve},
the pushforward of \(\O_E(p_1 + p_2 + \cdots + p_{2n+2})\)
along either map is perfectly balanced, i.e.,
\[(f_i)_* \O_E(p_1 + p_2 + \cdots + p_{2n+2}) \simeq \O_{\pp^1}(1)^{\oplus(n+1)}.\]
\end{lem}
\begin{proof}
Since \(f_i\) is of degree \(n+1\), the pushforward
\((f_i)_* \O_E(p_1 + p_2 + \cdots + p_{2n+2})\) is a rank \(n+1\)
vector bundle on \(\pp^1\), i.e., we can write
\[(f_i)_* \O_E(p_1 + p_2 + \cdots + p_{2n+2}) \simeq \bigoplus_{j = 1}^{n+1} \O_{\pp^1}(a_j).\]
The integers \(a_j\) satisfy
\[\sum a_j = \chi\left(\bigoplus_{j = 1}^{n+1} \O_{\pp^1}(a_j)\right) - (n+1) = \chi(\O_E(p_1 + p_2 + \cdots + p_{2n+2})) - (n+1) = 2n+2 - (n+1) = n+1,\]
so to see that \(a_j = 1\) for all \(j\), it suffices to see
that \(a_j < 2\) for all \(j\), i.e., that
\[0 = H^0\left(\bigoplus_{j = 1}^{n+1} \O_{\pp^1}(a_j - 2)\right) = H^0(\O_E(p_1 + p_2 + \cdots + p_{2n+2}) \otimes f_i^* \O_{\pp^1}(-2)),\]
or equivalently that
\[\O_E(p_1 + p_2 + \cdots + p_{2n+2}) \not\simeq \bar{f}^* \O_{\pp^1 \times \pp^1}(2, 0) = \bar{f}^* \O_{\pp^1 \times \pp^1}(1, 1) \otimes \bar{f}^* \O_{\pp^1 \times \pp^1}(1, -1),\]
which follows from Lemmas~\ref{ell_nodal} and~\ref{p1p2n}.
\end{proof}

Write \(r = 2n + 1\).
Applying Lemma~\ref{map_of_ell_curve}, there are exactly two maps \(f_i \colon E \to \pp^1 \simeq R\),
of degree \(n + 1\), sending \(\Gamma|_E\) to \(\Gamma|_R\).
Write \(\bar{f} = (f_1, f_2) \colon E \to \pp^1 \times \pp^1\).
Let \(S\) denote the blowup of \(\pp^1 \times \pp^1\) at the nodes of
the image of \(E\) under \(\bar{f}\),
so that the \(f_i\) give rise to an embedding \(f \colon E \hookrightarrow S\).  Writing \(F_1, \dots, F_{n^2-1}\) for the exceptional divisors, define
\[L \colonequals \O_S(n, n)\left(-\sum F_i\right) = K_S(1, 1)(E).\]

By adjunction on \(S\) and Lemma~\ref{p1p2n}, we have \(L|_E \simeq \O_S(1, 1)|_E \simeq \O_E(1)\).
Let \(\pi_1\colon S \to \pp^1\) and \(\pi_2\colon S \to \pp^1\) denote the two projections.
By intersection theory, for \(x \in \pp^1\), the restriction of the line bundle \(L(-E)\)
to the corresponding fiber \(\pi_i^{-1}(x)\) is
the (unique) line bundle of total degree \(-1\) that is
isomorphic to \(\O_{\pp^1}(-1)\) on any exceptional divisors lying over \(x\).
In particular, it has vanishing cohomology, so by the theorem on cohomology
and base-change, \((\pi_i)_* L(-E) = R^1 (\pi_i)_* L(-E) = 0\).
Combining this with Lemma~\ref{push}, we therefore have a natural identification
\begin{equation}(\pi_i)_* L \simeq (f_i)_* \O_E(1) \simeq \O_{\pp^1}(1)^{n + 1}.\end{equation}

The map \(S \to \pp^{2n + 1}\) via \(|L|\)
thus factors through a (uniquely defined) embedding
of \(\pp [(\pi_i)_* L] \simeq \pp^1 \times \pp^n\),
via the complete linear system of the relative \(\O(1)\) on \(\pp [(\pi_i)_* L]\),
which corresponds to \(|\O_{\pp^1 \times \pp^n}(1, 1)|\) on \(\pp^1 \times \pp^n\).
In particular, since \(S \to \pp [(\pi_i)_* L]\) is an embedding, so is
\(S \to \pp^{2n + 1}\).
Write \(\Sigma_i \subset \pp^{2n + 1}\) for the scroll
obtained as the image of the map \(\pp^1 \times \pp^n \to \pp^{2n+1}\).

\smallskip

\noindent
Putting all of this together, we can summarize this situation with the following diagram
of inclusions:
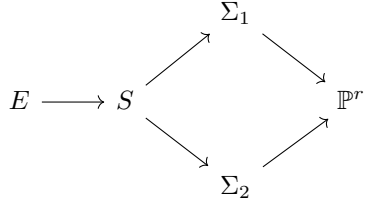
\begin{center}
\begin{tikzcd}
& & \Sigma_1 \arrow[rd] \\
E \arrow[r] & S \arrow[ru] \arrow[rd] & & \pp^r \\
& & \Sigma_2 \arrow[ru] \\
\end{tikzcd}
\end{center}

\begin{lem}\label{lem:sigmas_intersect_in_S}
The intersection \(\Sigma_1 \cap \Sigma_2\) coincides with \(S\).
\end{lem}
\begin{proof}
Let \(x_1, x_2 \in \pp^1\) be any two points, and write \(\Lambda_i\) for the fiber of
\(\Sigma_i\) over \(x_i\).
Note that \(\Lambda_i\) is the span of the divisor \(f_i^{-1}(x_i)\).

First suppose that \((x_1, x_2)\) is not a node of \(\bar{f}(E)\).
If \((x_1, x_2)\) does not lie on \(\bar{f}(E)\),
then the span of \(\Lambda_1\) and \(\Lambda_2\)
is the span of the divisor \(f_1^{-1}(x_1) + f_2^{-1}(x_2)\).
Since this divisor is linearly equivalent to \(O_E(1)\), the span is a hyperplane.
Otherwise, if \((x_1, x_2) = \bar{f}(y)\) lies on \(\bar{f}(E)\),
then the span of \(\Lambda_1\) and \(\Lambda_2\)
is the span of the divisor \(f_1^{-1}(x_1) + f_2^{-1}(x_2) - y\).
Since any \(2n + 1\) points on \(E\) are linearly general, this span is again a hyperplane.
Either way, since \(\Lambda_1\) and \(\Lambda_2\) span a hyperplane, they must meet at the single point that is the image of \((x_1, x_2)\) on \(S\).

Next suppose that \((x_1, x_2)\) is a node of \(\bar{f}(E)\),
say \((x_1, x_2) = \bar{f}(y_1) = \bar{f}(y_2)\).
Then the span of \(\Lambda_1\) and \(\Lambda_2\)
is the span of the divisor \(f_1^{-1}(x_1) + f_2^{-1}(x_2) - y_1 - y_2\).
Since any \(2n\) points on \(E\) are linearly general, this span is codimension~\(2\),
and so \(\Lambda_1\) and \(\Lambda_2\) meet along the line that is the image of the excpetional divisor over \((x_1, x_2)\).

Combining these two cases, we see that \(\Sigma_1 \cap \Sigma_2\) coincides with \(S\)
set-theoretically.
To upgrade this to a scheme-theoretic equality,
we must show that \(\Sigma_1\) and \(\Sigma_2\) are quasi-transverse
along \(S\), meaning that the tangent spaces to \(\Sigma_1\) and \(\Sigma_2\)
at points of \(S\) span a hyperplane.  (They cannot span all of \(\pp^r\), because \(\Sigma_1 \cap \Sigma_2\)
is pure of dimension \(2\).)
Away from the exceptional divisors, this is straight-forward: The tangent space to \(\Sigma_i\) contains \(\Lambda_i\), so it suffices to note that
\(\Lambda_1\) and \(\Lambda_2\) span a hyperplane.

It remains to consider an exceptional divisor \(M = \Lambda_1 \cap \Lambda_2\).
Write \(\Lambda\) for the span of \(\Lambda_1\) and \(\Lambda_2\).
As in the previous case, the span of the tangent spaces to the \(\Sigma_i\)
contains \(\Lambda\); however, in this case \(\Lambda\) is codimension~\(2\).
It thus remains to show that the natural map
\(N_{M / \Sigma_1} \oplus N_{M / \Sigma_2} \to N_{\Lambda/\pp^r}|_M\)
is everywhere nonzero along \(M\).
Recall that we write \(y_i\) for the points on \(E\)
lying over the node, i.e., satisfying
\(\bar{f}(y_1) = \bar{f}(y_2) = (x_1, x_2)\), so that \(M\) is the line
spanned by \(y_1\) and \(y_2\).
Let \(y_i^j\) be nontrivial first-order deformations of the \(y_i\)
satisfying \(f_j(y_1^j) = f_j(y_2^j)\). Such deformations are pictured in the following diagram:

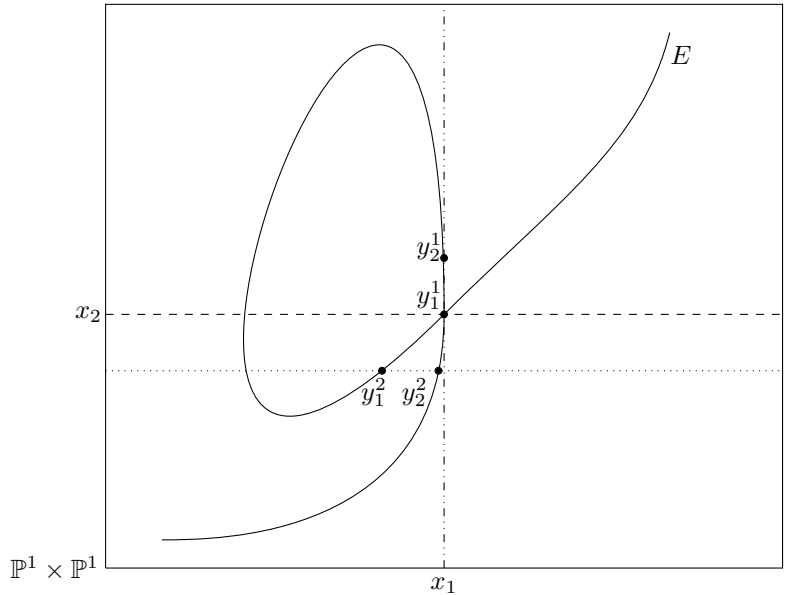
\begin{minipage}{0.95\textwidth}
\phantom{Beginning of figure}
\vspace{-16\baselineskip}
\begin{center}
\begin{tikzpicture}[scale=1.5]
\draw (0, 0) .. controls (1, 1) and (1.75, 1.5) .. (2, 2.5);
\draw (0, 0) .. controls (0, 7) and (-4, -4) .. (0, 0);
\draw (-3, -2.25) -- (-3, 2.75) -- (3, 2.75) -- (3, -2.25) -- (-3, -2.25);
\draw (0, 0) .. controls (0, -2) and (-2, -2) .. (-2.5, -2);
\draw (2.1, 2.3) node{\(E\)};
\draw[dashed] (-3, 0) -- (3, 0);
\draw[dash dot dot] (0, -2.25) -- (0, 2.75);
\draw[dotted] (-3, -0.5) -- (3, -0.5);
\filldraw (0, 0) circle[radius=0.03];
\filldraw (0, 0.5) circle[radius=0.03];
\draw (-0.13, 0.17) node{\(y_1^1\)};
\draw (-0.13, 0.6) node{\(y_2^1\)};
\draw (-0.62, -0.69) node{\(y_1^2\)};
\draw (-0.26, -0.69) node{\(y_2^2\)};
\filldraw (-0.55, -0.5) circle[radius=0.03];
\filldraw (-0.05, -0.5) circle[radius=0.03];
\draw (0, -2.4) node{\(x_1\)};
\draw (-3.15, 0) node{\(x_2\)};
\draw (-3.45, -2.25) node{\(\pp^1 \times \pp^1\)};
\end{tikzpicture}
\end{center}
\vspace{-7\baselineskip}
\phantom{End of figure.}
\end{minipage}

The lines joining \(y_1^j\) and \(y_2^j\) give first-order
deformations of \(M\)
in \(\Sigma_j\), i.e., sections \(\sigma_j\) of \(N_{M / \Sigma_j}\).
It suffices to see that the images of the \(\sigma_j\)
in \(N_{\Lambda / \pp^r}|_M\) do not simultaneously vanish
anywhere along~\(M\).

The span of \(\Lambda\) and the tangent line to \(E\) at \(y_1\)
is the span of the divisor \(f_1^{-1}(x_1) + f_2^{-1}(x_2) - y_2\).
Since any \(2n + 1\) points on \(E\) are linearly general,
this span is a hyperplane.
Similarly, the span of \(\Lambda\) and the tangent line to \(E\) at \(y_2\)
is a hyperplane.
Moreover, the span of \(\Lambda\) and the tangent lines to \(E\) at
\emph{both} \(y_1\) and \(y_2\) is the span of the divisor
\(f_1^{-1}(x_1) + f_2^{-1}(x_2)\), which is linearly equivalent to \(\O_E(1)\),
and therefore again spans a hyperplane.
Since this hyperplane contains the first two of these hyperplanes,
all three of these hyperplanes must be equal. Write \(\Lambda'\) for this
hyperplane.

By construction, the images of both \(\sigma_j\)
are nonzero sections in the subspace
\(N_{\Lambda / \Lambda'}|_M \simeq \O_M(1)\).
Because \(\bar{f}(E)\) is nodal, the
deformations \((y_1^1, y_2^1), (y_1^2, y_2^2)\) form a basis of \(T_{y_1} E \oplus T_{y_2} E\).
The images of these two sections thus form a basis of \(H^0(N_{\Lambda / \Lambda'}|_M) = H^0(\O_M(1))\),
and so do not simultaneously vanish anywhere along \(M\) as desired.
\end{proof}

The upshot of this is that we have a natural filtration of \(N_{E \cup R}\),
whose successive quotients are vector bundles of ranks \(1\), \(2n - 2\), and \(1\)
respectively:
\begin{equation} \label{filt}
0 \subset N_{E \cup R / S} \subset N_{E \cup R / \Sigma_1} + N_{E \cup R / \Sigma_2} \subset N_{E \cup R}.
\end{equation}

\begin{lem}\label{lem:c1_NR_in_Sigma}
We have
\(c_1(N_{E \cup R / \Sigma_j}|_R) = n(2n + 3) + 1\).
\end{lem}
\begin{proof}
The Picard group of \(\Sigma_j\)
is spanned by the class \(\gamma\) of one \(n\)-plane and the restriction of the hyperplane class \(h\) from \(\pp^{2n+1}\).
One computes that \(K_{\Sigma_j} = (n-1)\gamma - (n+1) h\) for such a scroll.
By adjunction we have
\begin{align*}
c_1(N_{R/\Sigma_j}) &= c_1(K_R) - c_1(K_S|_R) \\
&= -2 - ((n-1) \gamma - (n+1) h) \cdot R\\
&= -2 - (n-1) (\gamma \cdot R) + (n + 1) (h \cdot R) \\
&= -2 - (n-1) \cdot 1 + (n + 1) \cdot 2n \\
&= 2n^2 + n - 1.
\end{align*}
Therefore
\[c_1(N_{E \cup R / \Sigma_j}|_R) = c_1(N_{R/S}) + \# \Gamma = (2n^2 + n - 1) + (2n + 2) = n(2n + 3) + 1. \qedhere\]
\end{proof}

\begin{prop} \label{prop:HN} The vector bundle \(N_{E \cup R}|_R \simeq N_R[\posmodalong E]\)
is isomorphic to
\[\O_{\pp^1}(2n + 4) \oplus \O_{\pp^1}(2n + 3)^{\oplus (2n - 2)} \oplus \O_{\pp^1}(2n + 2).\]
Moreover, its HN-filtration
is precisely the restriction of the filtration \eqref{filt} to \(R\).
\end{prop}
\begin{proof}
By Lemma~\ref{lem:rest_onion_bal}, we have either
\(N_R[\posmodalong E] \simeq \O_{\pp^1}(2n + 4) \oplus \O_{\pp^1}(2n + 3)^{\oplus (2n - 2)} \oplus \O_{\pp^1}(2n + 2)\) or \(N_R[\posmodalong E] \simeq \O_{\pp^1}(2n + 3)^{\oplus 2n}\).  Lemma \ref{lem:c1_NR_in_Sigma} rules out the second case since
\(\O_{\pp^1}(2n + 3)^{\oplus 2n}\) admits no
subbundle of rank \(n\) and first Chern class \(n(2n + 3) + 1\).

Moreover, any subbundle of \(\O_{\pp^1}(2n + 4) \oplus \O_{\pp^1}(2n + 3)^{\oplus (2n - 2)} \oplus \O_{\pp^1}(2n + 2)\) of rank \(n\) and first Chern class \(n(2n + 3) + 1\)
contains \(O_{\pp^1}(2n + 4)\)
and is contained in \(O_{\pp^1}(2n + 4) \oplus \O_{\pp^1}(2n + 3)^{\oplus (2n - 2)}\).
Since the graded pieces of \eqref{filt} are the intersection
and span of the \(N_{E \cup R / \Sigma_j}\), they must therefore coincide with the HN-filtration.
\end{proof}

This provides the promised determination of \(N_R[\posmodalong E]\),
and the promised geometric construction of its HN-filtration.
This geometric description of the HN-filtration of \(N_{E \cup R}|_R\)
allows us to reduce interpolation for \(N_{E \cup R}\) to interpolation
for a modification of \(N_{E \cup R}|_E \simeq N_E[\posmodalong R]\) as follows:

\begin{lem}\label{lem:ell_norm_int1}
Let \(p\) and \(q\) be two distinct points of \(E \cap R\).
Then \(N_{E \cup R}\) satisfies interpolation provided that
\begin{equation} \label{Emod}
N_E[\posmodalong R][p \posmod N_{E/S}][q \negmod N_{E /\Sigma_1} + N_{E/\Sigma_2}]
\end{equation}
satisfies interpolation.
\end{lem}
\begin{proof}  
We imitate the basic idea of the proof of \cite[Lemma 8.8]{aly}.
Write \(\Gamma \colonequals E \cap R\), which has size \(r+1\).  Write \(x, y, z\) for \(3\) general points on \(R\). Twisting down, we have
\[N_{E \cup R}(-x-y-z)|_R \simeq \O_{\pp^1}(r-2) \oplus \O_{\pp^1}(r-1)^{\oplus r-3} \oplus \O_{\pp^1}(r).\]
Therefore the evaluation map
\[\ev_{R,\Gamma} \colon H^0(N_{E \cup R}|_R)\to N_{E \cup R}|_{\Gamma} \]
is injective when restricted to the subspace \(H^0(N_{E \cup R}|_R(-x-y-z))\).  Our aim is to suitably specialize the points \(x,y,z\) so as to be able to identify the subspace of sections of \(N_{E \cup R}|_E\)
\[V_{x,y,z} \colonequals \left\{\sigma \in H^0(E, N_{E \cup R}|_E) : \sigma|_\Gamma \in \Im\left(\ev_{R,\Gamma}|_{H^0(N_{E \cup R}|_R(-x-y-z))}\right) \right\}\]
 that glues to the image of \(H^0(N_{E \cup R}|_R(-x-y-z))\) under \(\ev_{R,\Gamma}\).  By Lemma \ref{lem:int_rest}, it suffices to show that this subspace of sections has the correct dimension and satisfies interpolation to conclude that \(N_{E \cup R}(-x-y-z)\) satisfies interpolation, which implies that \(N_{E \cup R}\) satisfies interpolation.

Since \(\# \Gamma = r + 1\), the evaluation map \(\ev_{R, \Gamma}\) restricted to  the sections of the largest factor \(\O_{\pp^1}(r)\) is already an isomorphism.  The evaluation on the other factors is not an isomorphism: On the \(\O_{\pp^1}(r-1)\) factors, the image is a codimension \(1\) subspace of \(\O_{\pp^1}(r-1)|_\Gamma\), and on the  \(\O_{\pp^1}(r-2)\) factor, the image is a codimension \(2\) subspace of \(\O_{\pp^1}(r-2)|_\Gamma\).  We will appropriately specialize so as to force these subspaces to be ``coordinate'' planes.

First limit \(x\) to \(p\).  The gluing data across the nodes (in particular at \(p\)) is fixed, and therefore the limiting codimension \(1\) subspace of \(\O_{\pp^1}(r-1)|_\Gamma\) contains the subspace \(\O_{\pp^1}(r-1)|_{\Gamma \setminus p}\oplus 0|_p\) of sections vanishing at \(p\).  Since this subspace has the correct dimension, it must be the flat limit.  Then limit \(y\) to \(q\).  In the limit, the codimension \(2\) subspace of \(\O_{\pp^1}(r-2)|_{\Gamma}\) must contain the subspace \(\O_{\pp^1}(r-2)|_{\Gamma \setminus \{p, q\}}\oplus 0|_p\oplus 0|_q\) of sections vanishing at \(p\) and \(q\).  Since this has the correct dimension, it must be the flat limit.  Since the HN-filtration on \(N_{E \cup R}|_R\) is the restriction of \eqref{filt} to \(R\),
this flat limit is
\[H^0\left(E, N_E[\posmodalong R][p \negmod N_{E/S}][q \negmod N_{E /\Sigma_1} + N_{E/\Sigma_2}]\right).\]
Since this subspace is the space of sections of a vector bundle, it suffices to show that this bundle satisfies interpolation.
To complete the proof, we note that \(\mu(N_E[\posmodalong R][p \negmod N_{E/S}][q \negmod N_{E /\Sigma_1} + N_{E/\Sigma_2}]) \geq 1\), and so it suffices by Lemma \ref{lem:int_large_degree} to prove interpolation after twisting up by \(p\).
\end{proof}

\subsection{\texorpdfstring{\boldmath The case \(r \geq 9\)}{The case r >= 9}}
By Lemma~\ref{lem:ell_norm_int1}, it suffices to show \eqref{Emod} satisfies interpolation.
On an elliptic curve, we can characterize which bundles satisfy interpolation
in terms of the Atiyah classification:

\begin{lem}\label{lem:JH_to_interpolation}
Let \(\sE\) be a vector bundle on an elliptic curve \(E\). Then \(\sE\) satisfies interpolation if and only if
there is a nonnegative integer \(a\) for which
every Jordan--Holder (JH) factor \(\sF\) of \(\sE\) satisfies
\begin{equation}\label{eq:JH_bound} a \leq \mu(\sF) \leq a+1 \qand \sF \not\simeq \O_E.\end{equation}
\end{lem}
\begin{proof}
By the Atiyah classification, every JH-factor of \(\sE\) is both a subbundle and quotient of \(\sE\).

First suppose \(\sE\) satisfies interpolation. Then \(\sE\) is nonspecial, so every JH-factor \(\sF\) is nonspecial,
or equivalently satisfies \(\mu(\sF) \geq 0\) and \(\sF \not\simeq \O_E\).
If no such nonnegative integer \(a\) exists, then there would be a positive integer \(b\)
and JH-factors \(\sF_1\) and \(\sF_2\) with \(\mu(\sF_1) < b < \mu(\sF_2)\).
This is a contradiction, since for general points \(p_1, p_2, \ldots, p_b \in E\), we would have
\[H^0(\sE(-p_1 - \cdots - p_b)) \neq 0 \qand H^1(\sE(-p_1 - \cdots - p_b)) \neq 0.\]

In the other direction, suppose there is a nonnegative integer \(a\)
for which every JH-factor \(\sF\) satisfies \eqref{eq:JH_bound}.
Then for general points \(p_1, p_2, \ldots, p_{a + 1} \in E\),
\[H^0(\sE(-p_1 - \cdots - p_{a + 1})) = 0 \qand H^1(\sE(-p_1 - \cdots - p_a)) = 0.\]
Therefore \(\sE\) satisfies interpolation.
\end{proof}

\begin{lem}\label{lem:JH_E_to_Emod}
Let \(\sE\) be a vector bundle on an elliptic curve \(E\), and let \(a\) and \(b\) be integers.  For two points \(p, q \in E\), consider subspaces \(\Delta \subseteq \sE|_p\) of rank \(1\) and \(\Lambda \subseteq \sE|_q\) of corank \(1\).  If every JH-factor \(\sF\) of \(\sE\) satsifies \(a < \mu(\sF) < b\),
then every JH-factor \(\sF'\) of
\(\sE' \colonequals \sE[p \posmodwithplus \Delta][q \negmod \Lambda]\)
satisfies \(a \leq \mu(\sF') \leq b\).
\end{lem}
\begin{proof}
Up to replacing \(\sE'\) with its dual, it suffices to show that every JH-factor \(\sF'\) of \(\sE'\) satisfies \(\mu(\sF') \leq b\).
Since every JH-factor is a subbundle, and \(\sE'\) is a subsheaf of \(\sE[p \posmodwithplus \Delta]\),
it suffices to show that every subsheaf \(\sF'\) of \(\sE[p \posmodwithplus \Delta]\) satisfies \(\mu(\sF') \leq b\).

If \(\sF'\) is a subsheaf of \(\sE\) we are done by assumption.
Otherwise, write \(x\) for the degree of \(\sF'\) and \(y\) for the rank of \(\sF'\) so that \(\mu(\sF') = x/y\).
Write \(\sF = \sF' \cap \sE\) for the corresponding subsheaf of \(\sE\).
Then \((x - 1)/y = \mu(\sF) < b\).
Since \(b\) is an integer, \(\mu(\sF') = x/y \leq b\) as desired.
\end{proof}

\noindent
Combining Lemmas~\ref{lem:JH_to_interpolation} and~\ref{lem:JH_E_to_Emod}, it suffices to prove:

\begin{prop} \label{prop:JHinterval}
If \(r \geq 9\) is odd, every JH-factor \(\sF\) of \(N_E[\posmodalong R]\) satisfies
\(r+4 < \mu(\sF) < r+5\).
\end{prop}

\noindent
Our proof of Proposition~\ref{prop:JHinterval} will be by induction on \(r\).

When \(r = 7\), we have \(\mu(N_E[\posmodalong R]) = 12 = r + 5\).
Since \(I(8, 1, 7, 0, 1)\) holds, we deduce from Lemma~\ref{lem:JH_to_interpolation}
that every JH-factor \(\sF\) of \(N_E[\posmodalong R]\)
has slope exactly \(r + 5\) in this case.
Although Proposition~\ref{prop:JHinterval} does not hold in this case,
it is close enough that we will be able to leverage it to establish
the case \(r = 9\).

In general, our strategy will be to use our inductive hypothesis
to show that Proposition~\ref{prop:JHinterval} is close enough to holding
that naturality of the HN-filtration forces it to hold exactly.

\begin{defin}
Let \(E\) be a genus \(1\) curve.  We say that a map
\[\Pic^aE \xrightarrow{f} \Pic^bE\]
is \defi{natural} if for any automorphism \(\theta \colon E \to E\), the following diagram commutes:
\begin{center}
\begin{tikzcd}
\Pic^a E \arrow{r}{f} \arrow{d}{\theta^*} & \Pic^bE \arrow{d}{\theta^*} \\
\Pic^a E \arrow{r}{f} & \Pic^bE
\end{tikzcd}
\end{center}
\end{defin}

\begin{prop}\label{prop:black_magic}
If \(\Pic^aE \to \Pic^bE\) is natural, then \(a \) divides \(b\).
\end{prop}
\begin{proof}
Translation by an \(a\)-torsion point acts as the identity on \(\Pic^aE\), and so it must also act as the identity on \(\Pic^bE\).
\end{proof}

\begin{proof}[Proof of Proposition~\ref{prop:JHinterval}]
It suffices to prove that
\(N_E[\posmodalong R]\) has no subbundles of slope \(r+5\) or more, and no quotient bundles of slope \(r+4\) or less.
We will prove this by induction on \(r\), using the case \(r = 7\) discussed above as our base case.
Our argument will consist of two steps:
\begin{enumerate}
\item We \emph{specialize} so that the statement of the proposition
becomes \emph{false}, but still close enough to true that
we can gather information about the possible limits of subbundles of large slope
(respectively quotient bundles of small slope).
\item Leveraging this information, we apply Proposition~\ref{prop:black_magic}
to the \emph{general} fiber.
\end{enumerate}

Our specialization will be of \(R\) to the union \(R^\circ = \bar{pq} \cup R^-\),
of a \(2\)-secant line \(\bar{pq}\) and an \((r-1)\)-secant rational curve \(R^-\) of degree \(r-2\) meeting \(\bar{p q}\)
at a single point.
Projection from \(\bar{pq}\) induces an exact sequence
\begin{equation}\label{eq:exact_projline} 0 \to \big[S \colonequals \O_E(1)(2p+q) \oplus \O_E(1)(2q + p)\big] \to N_E[\posmodalong R^\circ] \to \big[Q \colonequals N_{E(0,0;2)}(p+q)[\posmodalong \bar{R^-}]\big] \to 0.\end{equation}
The bundle \(S\) is perfectly balanced of slope \(r+4\).
The bundle \(Q\) is a twist of another instance of our problem in \(\pp^{r-2}\).
If \(r \geq 11\), then by induction, every
JH-factor of \(Q\) has slope strictly between \(r+4\) and \(r+5\);
if \(r=9\), then every JH-factor of \(Q\) has slope exactly \(r+5\).

We begin by showing \(N_E[\posmodalong R]\) has no quotient bundles of slope \(r+4\) or less.
Since \(Q\) has no quotient bundles of slope \(r + 4\) or less,
any such quotient must specialize to a quotient of \(S\),
and therefore must have slope exactly \(r + 4\) and rank at most \(2\).
Let \(G\) be the maximal such quotient of \(N_E[\posmodalong R]\) (i.e., on the general fiber).
Our above specialization of \(R\) shows that \(\mu(G) = r + 4\) and \(n \colonequals \rk G \leq \rk S = 2\).
The determinant \(\det G\) depends on the following data:
\begin{itemize}
\item A line bundle \(\O_E(1)\).
\item A basis for \(H^0(\O_E(1))\).
\item A hyperplane \(H \subset \pp H^0(\O_E(1))^\vee\).
\item A rational curve \(R \subset H\) of degree \(r - 1\) passing through the \(r + 1\) points of \(E \cap H\).
\end{itemize}
Except for the choice of line bundle \(\O_E(1)\), all of this data varies
in a rational family. Since any map from a rational variety to an abelian variety
is constant, \(\det G\) depends only on the choice of line bundle \(\O_E(1)\).
Extracting the determinant of \(G\) therefore gives a natural map
\[\Pic^{r+1} E \to \Pic^{n(r+4)}E.\]
Hence, by Proposition \ref{prop:black_magic}, we have \((r+1) \mid n(r+4)\),
and therefore \((r + 1) \mid 3n\). Since \(3n \leq 6 < 10 \leq r + 1\), we must have \(3n = 0\), i.e., \(n = 0\)
as desired.

We next show that \(N_E[\posmodalong R]\) has no subbundles of slope \(r+5\) or more.
If \(r > 9\), then the specialization \(N_E[\posmodalong R^\circ]\)
has no such subbundle, because \(S\) and \(Q\) do not.
It therefore remains only to consider the case \(r=9\),
in which every JH-factor of \(S\) has slope \(13\)
and every JH-factor of \(Q\) has slope \(14\).
Let \(G\) be the maximal such subbundle of \(N_E[\posmodalong R]\) (i.e., on the general fiber).
Our above specialization of \(R\) shows that \(\mu(G) = 14\) and \(n \colonequals \rk G \leq \rk Q = 6\).
Extracting the determinant of \(G\) therefore gives a natural map
\[\Pic^{10} E \to \Pic^{14n} E.\]
Hence, by Proposition \ref{prop:black_magic}, we have that \(10 \mid 14n\), so \(5 \mid n\).
If \(n = 0\), we are done, so suppose that \(n = 5\).

To obtain a contradiction, we analyze what happens in our specialization,
in which \(G\) specializes to a subbundle \(G^\circ\) of \(Q\) with slope \(14\)
and rank \(5\). We consider the determinant \(\det [G^\circ(-p-q)]\).
A priori this depends only on \(\O_E(1)\), \(p\), and \(q\)
(the remaining data varies in a rational family).
In fact, we claim it depends only on
\[\O_{E(0,0;2)}(1) = \O_E(1) - p - q.\]
Indeed, \(\det [G^\circ(-p-q)]\) is a product of JH-factors of
\(N_{E(0,0;2)}[\posmodalong \bar{R^-}]\), which is a discrete set of possibilities once we fix
\(\O_{E(0,0;2)}(1)\) and some additional data varying in a rational family.
As we fix \(\O_{E(0,0;2)}(1)\) and this additional data,
we may allow \(\{p, q\}\) to vary arbitrarily in \(E \times E\) by Lemma~\ref{proj-general}.
In this way, we obtain a map from \(E \times E\) to this discrete set, which must therefore be constant
because \(E \times E\) is connected.
Therefore \(\det [G^\circ(-p-q)]\) depends only on \(\O_{E(0,0;2)}(1)\) plus this additional data varying in a rational family,
and thus only on \(\O_{E(0,0;2)}(1)\).
The determinant of \(G^\circ(-p-q)\) therefore gives a natural map
\[\Pic^8 E \to \Pic^{60} E.\]
This is a contradiction by Proposition \ref{prop:black_magic}, since \(8\nmid 60\).
\end{proof}

\subsection{\boldmath The case \(r = 7\)}
To handle this case, we will first have to study
the restriction of \eqref{filt} to~\(E\), which we do for arbitrary odd \(r\).
This is a filtration of
\(N_E[\posmodalong R]\) whose successive quotients are:
\[N_{E/S}(1), \quad N_{S/\Sigma_1}|_E \oplus N_{S/\Sigma_2}|_E, \quad \text{and} \quad \frac{N_E}{N_{E / \Sigma_1} + N_{E / \Sigma_2}}.\]
We write \(H_i \colonequals f_i^* \O_{\pp^1}(1)\) for the corresponding hyperplane class.

\begin{prop}\label{prop:c1_NE_in_S}
We have \(N_{E/S} \simeq \O_E(3 - n)\) (and so \(N_{E/S}(1) \simeq \O_E(4 - n)\)).
\end{prop}
\begin{proof}
We consider the sequence (\emph{not} exact) of maps:
\[[N_{E/S} \simeq N_f] \to N_{\bar{f}} \to \bar{f}^* N_{\bar{f}(E) / \pp^1 \times \pp^1}.\]
By inspection both of these maps drop rank exactly at the points of \(E\) lying over the nodes of \(\bar{f}\).
Therefore their Chern classes lie in a linear progression, i.e.:
\begin{align*}
c_1(N_{E/S}) &= 2 \cdot c_1(N_{\bar{f}}) - c_1(\bar{f}^* N_{\bar{f}(E) / \pp^1 \times \pp^1}) \\
&= -2 c_1(\bar{f}^* K_{\pp^1 \times \pp^1}) - c_1(\bar{f}^* \O_{\pp^1 \times \pp^1}(n + 1, n + 1)) \\
&= -2 (-2 H_1 - 2H_2) - (n + 1)(H_1 + H_2) \\
&= (3 - n)(H_1 + H_2). \qedhere
\end{align*}
\end{proof}

\begin{prop}\label{prop:c1_NS_in_sigma}
We have \(c_1(N_{S/\Sigma_1}|_E \oplus N_{S/\Sigma_2}|_E) = \O_E(3n - 3)\).
\end{prop}
\begin{proof}
As in the proof of Proposition~\ref{prop:HN}, define the classes
\(\gamma_j\) and \(h\) in \(\Pic \Sigma_j\) to be 
the class of one \(n\)-plane, and the restriction of the hyperplane class from \(\pp^{2n+1}\), respectively.
By adjunction,
\[c_1(N_{E/\Sigma_j}) = -c_1(K_{\Sigma_j}|_E) = - \left((n - 1)\gamma_j - (n + 1) h\right) \cdot E.\]
Therefore,
\begin{align*}
c_1(N_{S/\Sigma_1}|_E \oplus N_{S/\Sigma_2}|_E) &= c_1(N_{E/\Sigma_1}) + c_1(N_{E/\Sigma_2}) - 2c_1(N_{E/S}) \\
&= -(n - 1)(\gamma_1 + \gamma_2) \cdot E + 2(n + 1) (h \cdot E) - 2(3 - n)(H_1 + H_2) \\
&= -(n - 1)(H_1 + H_2) + 2(n + 1) (H_1 + H_2) - 2(3 - n)(H_1 + H_2) \\
&= (3n - 3)(H_1 + H_2). \qedhere
\end{align*}
\end{proof}

\begin{prop}\label{prop:c1_NQ}
We have \(\frac{N_E}{N_{E / \Sigma_1} + N_{E / \Sigma_2}} \simeq \O_E(2)\).
\end{prop}
\begin{proof}
Since \(K_{\pp^{2n + 1}} = \O_{\pp^{2n + 1}}(-(2n + 2))\), we have
\(c_1(N_E) = \O_E(2n + 2)\). Combined with the previous two propositions,
this implies the statement of the lemma (\((2n + 2) - (3 - n) - (3n - 3) = 2\)).
\end{proof}

We now take \(r = 7\) (equivalently \(n = 3\)), and let \(p\) and \(q\) be points of \(E \cap R\).
By Lemma \ref{lem:ell_norm_int1}, it suffices to show interpolation for
\[N_E[\posmodalong R][p \posmod N_{E/S}][q \negmod N_{E /\Sigma_1} + N_{E/\Sigma_2}].\]
This bundle has slope \(12\), so it suffices to show that for a general effective divisor \(D\) of degree \(12\),
\begin{equation}\label{eq:suff_van}
H^0(N_E[\posmodalong R][p \posmod N_{E/S}][q \negmod N_{E /\Sigma_1} + N_{E/\Sigma_2}](-D)) = 0.
\end{equation}
Furthermore, \(I(8, 1, 7, 0, 1)\) holds, so
\(h^0(N_E[\posmodalong R](-D)) = h^1(N_E[\posmodalong R](-D)) = 0\),
which implies that \(h^0 (N_E[\posmodalong R][p \posmod N_{E/S}](-D)) = 1\).
Call the unique section \(\sigma\).  If there is any point \(q \in E \cap R \setminus \{p\}\) for which \(\sigma|_q \notin (N_{E /\Sigma_1} + N_{E/\Sigma_2})|_q\), then we have proved the desired vanishing \eqref{eq:suff_van}.
We may therefore assume that at all points of \(E \cap R \setminus \{p\}\), the value of \(\sigma\) lies in the subbundle
\(N_{E /\Sigma_1} + N_{E/\Sigma_2}\).
By Proposition \ref{prop:c1_NQ} we have the exact sequence
\begin{equation}\label{eq:in_Q_seq}0 \to \left(N_{E/\Sigma_1}+N_{E/\Sigma_2}\right)[\posmodalong R](-D) \to N_E[\posmodalong R](-D) \to \O_E(2)(-D) \to 0.\end{equation}
Since
\(\deg\left(\O_E(2)(-D)(-E\cap R +p)\right) = -3\),
we must have that \(\sigma\) comes from a section of
\[\left(N_{E/\Sigma_1}+N_{E/\Sigma_2}\right)[\posmodalong R][p \posmod N_{E/S}](-D).\]
It therefore suffices to show that for some \(p \in E \cap R\), this bundle has no global sections
(or, equivalently since the degree is \(-3\), such that \(h^1 = 3\)).  Using \eqref{eq:in_Q_seq}, we have \(h^1\left(\left(N_{E/\Sigma_1}+N_{E/\Sigma_2}\right)[\posmodalong R](-D)\right)=4\), so such a point \(p \in E \cap R\) exists unless making positive modifications towards \emph{all} points of \( E \cap R\) does not decrease the \(h^1\):
\[h^1\left(\left(N_{E/\Sigma_1}+N_{E/\Sigma_2}\right)[\posmodalong R](-D)[E \cap R \posmodwithplus N_{E/S}]\right)=4.\]
Equivalently, we are done unless 
\begin{equation}\label{eq:h0_8}h^0\left(\left(N_{E/\Sigma_1}+N_{E/\Sigma_2}\right)[\posmodalong R](-D)[E \cap R \posmodwithplus N_{E/S}]\right) =8.\end{equation}
Taking the sum of the two normal bundle exact sequences for \(S \hookrightarrow \Sigma_1\) and \(S\hookrightarrow \Sigma_2\) along \(E\)
yields
\begin{equation}\label{eq:SinSigmas} 0 \to N_{E/S}(2(E \cap R)) \to \left(N_{E/\Sigma_1}+ N_{E/\Sigma_2}\right)[\posmodalong R][E \cap R \posmodwithplus N_{E/S}]\to N_{S/\Sigma_1}|_E \oplus N_{S/\Sigma_2}|_E  \to 0.\end{equation}
Twisting down by \(D\), the line subbundle \(N_{E/S}(2(E \cap R))(-D)\) has degree \(4\), and hence \(4\) global sections and vanishing \(H^1\).  The quotient twisted down by \(D\)
(which has degree \(0\)) must therefore also have \(4\) global section in order for \eqref{eq:h0_8} to hold.
Furthermore, the two scrolls \(\Sigma_1\) and \(\Sigma_2\) are exchanged by monodromy
because the two maps \(f_i \colon E \to \pp^1\) in Lemma \ref{map_of_ell_curve} have degree \(n+1\), which is not a multiple of \(2n + 2\), and hence by Proposition \ref{prop:black_magic} they cannot be individually naturally defined.
Thus the two rank \(2\) bundles \(N_{S/\Sigma_i}|_E(-D)\) necessarily both have \(2\) sections.
If \(N_{S/\Sigma_i}|_E\) were indecomposable, then by the Atiyah classification, it would necessarily be an extension of a degree \(12\) line bundle \(M\) by itself.  As long as \(\O_E(D) \not\simeq M\), we would have \(h^0(N_{S/\Sigma_i}|_E(-D)) = 0\).
Therefore \(N_{S/\Sigma_i}|_E\) is a direct sum of line bundles. Since \(h^0(N_{S/\Sigma_i}|_E(-D)) = 2\):
\[N_{S/\Sigma_i}|_E \simeq L_{i1} \oplus L_{i2}\qquad \text{where} \qquad \deg(L_{i1}) = 14 \ \text{and} \ \deg(L_{i2}) = 10.\]
Then \(\det [L_{11} \oplus L_{21}]\) gives a natural map
\[\Pic^8 E \to \Pic^{28} E,\]
which is a contradiction by Proposition \ref{prop:black_magic}, since \(8 \nmid 28\).

\pagebreak

\appendix

\section{Code for Section~\ref{sec:most-sporadic} \label{app:code}}

\begin{lstlisting}
#!/usr/bin/python

XX = set([(5,2,3,0,0), (4,1,3,1,0), (4,1,3,0,1), (4,1,3,1,1),
          (6,2,4,0,0), (5,1,4,1,0),              (5,1,4,1,1), (5,1,4,2,1), (6,2,4,1,1),
          (7,2,5,0,0),              (6,1,5,0,1), (6,1,5,1,1)])

def good(d, g, r, l, m):
	if not (d >= g + r and 0 <= 2 * l <= r and 0 <= m <= (r + 1) * d - r * g - r * (r + 1)):
		return False
	if m == g == 0 and 2 * l < (1 - d) % (r - 1):
		return False
	if (d,g,r,l,m) in XX:
		return False
	return True

def n_lower_bound(d, g, r):
	if r % 2 == 0:
		return 3
	elif (d, g) == (r + 1, 1):
		return 4
	else:
		return 2

def erasable(r, s10, s11, s20, s21, w10, t1=0, t2=0, strong=True):
	if s10 == s11 == s20 == s21 == w10 == 0:
		return (strong and t2 == 0)

	if s10 > 0:
		if t1 + 1 < r - 1:
			if erasable(r, s10 - 1, s11, s20, s21, w10, t1 + 1, t2, strong):
				return True
		elif t2 < t1 + 1 == r - 1:
			if erasable(r, s10 - 1, s11, s20, s21, w10, t2, 0, True):
				return True

	if s11 > 0:
		if t1 + 1 < r - 1:
			if erasable(r, s10, s11 - 1, s20, s21, w10, t1 + 1, t2 + 1, strong):
				return True
		elif t2 + 1 < t1 + 1 == r - 1:
			if erasable(r, s10, s11 - 1, s20, s21, w10, t2 + 1, 0, True):
				return True

	if s20 > 0:
		if t1 + 2 < r - 1:
			if erasable(r, s10, s11, s20 - 1, s21, w10, t1 + 2, t2, strong):
				return True
		elif t2 < t1 + 2 == r - 1:
			if erasable(r, s10, s11, s20 - 1, s21, w10, t2, 0, True):
				return True
		elif t2 + 2 <= r - 1 <= t1 + 2:
			if erasable(r, s10, s11, s20 - 1, s21, w10, t2 + (t1 + 2 - (r-1)), 0, strong):
				return True

	if s21 > 0:
		if t1 + 2 < r - 1:
			if erasable(r, s10, s11, s20, s21 - 1, w10, t1 + 2, t2 + 1, strong):
				return True
		elif t2 + 1 < t1 + 2 == r - 1:
			if erasable(r, s10, s11, s20, s21 - 1, w10, t2 + 1, 0, True):
				return True
		elif t2 + 2 < t1 + 1 == r - 1:
			if erasable(r, s10, s11, s20, s21 - 1, w10, t2 + 2, 0, True):
				return True
		elif t1 + 1 == t2 + 2 == r - 1:
			if erasable(r, s10, s11, s20, s21 - 1, w10, 0, 0, True):
				return True

	if w10 > 0:
		if t1 + 1 < r - 1:
			if erasable(r, s10, s11, s20, s21, w10 - 1, t1 + 1, t2, False):
				return True
	
	return False

def can_induct(d, g, r, l, m):
	# Proposition 8.4
	if d >= g + 2 * r - 1 and good(d - (r - 1), g, r, l, m):
		return True

	# Proposition 8.6
	if m >= r - 1 and good(d, g, r, l, m - (r - 1)):
		return True

	deltanum = 2 * d + 2 * g - 2 * r + 2 * l + (r + 1) * m

	for lp in range(l + 1):
		for mp in range(m + 1):
			if (2 * mp + lp > r - 2) or (mp > 0 and r == 3):
				continue
			mbar = m - mp

			for dp in range(g + r, d + 1):
				if g == 0 and m != 0 and dp == g + r:
					continue

				in0 = lp + 2 * (d - dp)

				for sum_n in range(n_lower_bound(dp, g, r) * mp, (r - 1) * mp + 1, 2):
					lbar = l - lp + ((r - 1) * mp - sum_n) // 2

					# Proposition 8.2
					if (r - 1) * (in0 + sum_n - 1) + 1 <= deltanum <= (r - 1) * (in0 + sum_n + 1) - 1:
						if good(dp - 1, g, r - 1, lbar, mbar):
							return True

					# Proposition 8.3
					if mp < m and 2 * mp + lp < r - 2:
						if (r - 1) * (in0 + sum_n) + 1 <= deltanum <= (r - 1) * (in0 + sum_n + 2) - 1:
							if good(dp - 1, g, r - 1, lbar, mbar):
								if good(dp - 1, g, r - 1, lbar, mbar - 1): 
									if good(dp - 2, g, r - 2, lbar, mbar):
										return True
	
	# Proposition 8.7
	if l == 0 and m == 1:
		for epsilon in range(0, (d - g - r) // 2 + 1):
			if g > 0 or 2 * epsilon < d - g - r:
				if 2 * epsilon * (r - 1) + 2 <= deltanum <= (2 * epsilon + 2) * (r - 1) - 2:
					if good(d - 2 * epsilon - 2, g, r - 2, 0, 1):
						return True
	
	# Proposition 8.8		
	k = r // 2
	if k >= 3 and (d, g, r, l, m) == (4 * k + 1, 2 * k - 1, 2 * k + 1, 0, 1):
		if good(4 * k - 3, 2 * k - 2, 2 * k - 1, k - 3, 0):
			return True
	

	# Proposition 8.9
	if m == 0 and g >= 3 and r >= 6:
		for epsilon in range((d - g - r) // 3 + 1):
			if ((2 * epsilon + 2) * (r - 1) + 3 <= deltanum <= (2 * epsilon + 4) * (r - 1) - 3):
				if good(d - 3 * epsilon - 6, g - 3, r - 3, l + 1, 0):
					if good(d - 3 * epsilon - 6, g - 3, r - 3, l, 0):
						return True

	# Proposition 8.10		
	if m == 0 and g >= 1 and (r - 1) + 1 <= deltanum <= 3 * (r - 1) - 1:
		if good(d - 2, g - 1, r - 1, l + 1, 0):
			return True

	# Proposition 8.11
	if m == 0 and g >= 3 and r >= 6 and 3 * (r - 1) + 2 <= deltanum <= 5 * (r - 1) - 2:
		if good(d - 5, g - 3, r - 2, l + 1, 0) and good(d - 5, g - 3, r - 2, l, 0):
			return True

	## Proposition 11.2... ##
	for lp in range(l + 1):
		for mp in range(m + 1):
			if r == 3 and mp > 0:
				continue
			for mpp in range(m - mp + 1):
				mbar_max = m - mp
				mbar_min = m - mp - mpp

				for gp in range(g + 1):
					for dp in range(gp + r, d - g + gp + 1):
						if gp == 0 and m != 0 and dp == gp + r:
							continue

						for epsin in range(d - g - dp + gp + 1):
							epsout = d - g - dp + gp - epsin
							out = 2 * epsout + 3 * (g - gp) + m + mp + lp
							in0 = 2 * epsin + g - gp + mpp + lp + out // (r - 1)
			
							for sum_n in range(n_lower_bound(dp, g, r) * mp, (r - 1) * mp + 1, 2):
								lbar = l - lp + ((r - 1) * mp - sum_n) // 2

								if (r - 1) * (in0 + sum_n - 1) + 1 <= deltanum:
									if deltanum <= (r - 1) * (in0 + sum_n + 1) - 1:
										if erasable(r, lp + m - mp - mpp, epsout, mp, g - gp, mpp):
											ok = True
											for mbar in range(mbar_min, mbar_max + 1):
												if not good(dp - 1, gp, r - 1, lbar, mbar):
													ok = False
													break
											if ok:
												return True
			
	return False

base_cases = []
for r in range(3, 14):
	for g in range(r):
		for d in range(g + r, g + 2 * r):
			for l in range(r//2 + 1):
				for m in range(r):
					if (l, m) == (0, 0) and 2 * d + 2 * g == 3 * r - 1: #cf. Section 6
						continue
					if good(d, g, r, l, m):
						if not can_induct(d, g, r, l, m):
							base_cases.append((d, g, r, l, m))

for (d, g, r, l, m) in XX:
	if good(d, g, r, l, m + r - 1):
		if not can_induct(d, g, r, l, m + r - 1) and (d, g, r, l, m + r - 1) not in base_cases:
			base_cases.append((d, g, r, l, m + r - 1))

print(' & '.join([str(i) for i in base_cases]))
\end{lstlisting}

\pagebreak

\bibliographystyle{amsplain.bst}
\bibliography{interpolation.bib}
  
\end{document}